\documentclass{article}

\usepackage[margin=1in]{geometry}
\usepackage[T1]{fontenc}
\usepackage[english]{babel}
\usepackage[numbers,sort]{natbib} 
\usepackage[dvipsnames]{xcolor}
\usepackage{amsmath, amssymb, amsthm}
\usepackage{bbm}
\usepackage{blkarray}
\usepackage{cancel}
\usepackage{enumitem}
\usepackage{float}
\usepackage{bm}
\usepackage{graphicx}
\usepackage{mathrsfs}
\usepackage{microtype}
\usepackage{ragged2e}
\usepackage{thmtools}
\usepackage{tikz}
\usepackage[Symbolsmallscale]{upgreek}
\usepackage{caption}
\usepackage{subcaption}
\usepackage[pagebackref]{hyperref}
\hypersetup{citecolor=purple, colorlinks=true, linkcolor=purple}
\usepackage{fancyhdr}
\usepackage{algorithm}
\usepackage{algpseudocode}
\usepackage{comment}
\usepackage{subfiles}
\usepackage{mathtools} 
\mathtoolsset{showonlyrefs}

\DeclareUnicodeCharacter{2161}{II} 

\newtheorem{theorem}{Theorem}
\newtheorem{proposition}[theorem]{Proposition}

\newtheorem{lemma}[theorem]{Lemma}
\newtheorem{corollary}[theorem]{Corollary}

\newtheorem{assumption}[theorem]{Assumption}
\newtheorem{remark}[theorem]{Remark}

\newtheoremstyle{recall}
{3pt}{3pt}{}{}{\bfseries}{.}{.5em}
{\thmname{#1}\thmnote{ #3}}

\theoremstyle{recall}
\newtheorem*{assumptionrecall}{Assumption}

\title{ \Large \textbf{Efficient Online Learning in Interacting Particle Systems}}
\author{
Louis Sharrock\thanks{Department of Statistical Science, University College London. \texttt{l.sharrock@ucl.ac.uk}} \and 
Nikolas Kantas\thanks{Department of Mathematics, Imperial College London. \texttt{n.kantas@imperial.ac.uk}} \and 
Grigorios A. Pavliotis\thanks{Department of Mathematics, Imperial College London. \texttt{g.pavliotis@imperial.ac.uk}}
}

\date{}

\begin{document}
\maketitle

\begin{abstract}
We introduce a new method for online parameter estimation in stochastic interacting particle systems, based on continuous observation of a small number of particles from the system. Our method recursively updates the model parameters using a stochastic approximation of the gradient of the asymptotic log-likelihood, which is computed using the continuous stream of observations. Under suitable assumptions, we rigorously establish convergence of our method to the stationary points of the asymptotic log-likelihood of the interacting particle system. We consider asymptotics both in the limit as the time horizon $t\rightarrow\infty$, for a fixed and finite number of particles, and in the joint limit as the number of particles $N\rightarrow\infty$ and the time horizon $t\rightarrow\infty$. Under additional assumptions on the asymptotic log-likelihood, we also establish an $\mathrm{L}^2$ convergence rate and a central limit theorem. Finally, we present several numerical examples of practical interest, including a model for systemic risk, a model of interacting FitzHugh--Nagumo neurons, and a Cucker--Smale flocking model. Our numerical results corroborate our theoretical results, and also suggest that our estimator is effective even in cases where the assumptions required for our theoretical analysis do not hold.
\end{abstract}

\section{Introduction}

The study of interacting particle systems (IPSs) and their mean-field limits has a rich history, dating back to the seminal work of \citet{mckean1966class}; see \cite{vlasov1968vibrational,oelschlager1984martingale,meleard1996asymptotic,sznitman1991topics} for some other early references. In the last two decades, the probabilistic properties of such systems have been the subject of sustained interest, with many new results on well-posedness  \citep[e.g.,][]{chaudruderaynal2020strong,huang2019distribution}, existence and uniqueness \citep[e.g.,][]{jourdain2008nonlinear,bauer2018strong,mishura2020existence}, ergodicity \citep[e.g.,][]{bashiri2020longtime,bolley2013uniform,carrillo2006contractions,cattiaux2008probabilistic,eberle2019quantitative}, and the propagation of chaos \citep[e.g.,][]{durmus2020elementary,malrieu2001logarithmic,malrieu2003convergence,lacker2023sharp}. 

In parallel, significant attention has also been dedicated to the various applications of such models. These include, amongst others, statistical physics \citep{benedetto1997kinetic}, multi-agent systems \citep{benachour1998nonlinear}, mean-field games \citep{carmona2018probabilistic,cardaliaguet2018mean,cardaliaguet2019master}, stochastic control \citep{buckdahn2017meanfield}, filtering \citep{crisan2010approximate}, mathematical biology including neuroscience \citep{baladron2012meanfield} and structured models of population dynamics \citep{burger2007aggregation}, the social sciences including opinion dynamics \citep{chazelle2017wellposedness,goddard2022noisy} and cooperative behaviours \citep{canuto2012eulerian}, financial mathematics \citep{giesecke2020inference}, Bayesian inference \citep{liu2016stein}, and the analysis of mean-field neural networks \citep{hu2020meanfield,mei2018mean,rotskoff2022trainability,sirignano2020mean}.

More recently, there has been growing interest in the study of statistical inference for this class of processes \citep[e.g.,][]{kasonga1990maximum,bishwal2011estimation,giesecke2020inference,chen2021maximum,sharrock2022parameterestimation,sharrock2023online,dellamaestra2023lan,amorino2023parameter}, in both frequentist \citep[e.g.,][]{sharrock2022parameterestimation,amorino2023parameter} and Bayesian settings \citep[e.g.,][]{jasra2025bayesian,nickl2025bayesian}. One limitation of existing approaches is that, with certain notable exceptions \citep{pavliotis2022eigenfunction,genoncatalot2022inference,genoncatalot2023parametric}, it is generally assumed that it is possible to observe the entire IPS, or else multiple i.i.d. trajectories of the limiting McKean-Vlasov SDE (MVSDE). In cases where the number of particles is very large, however, this assumption may be unrealistic, or else associated with a prohibitive computational cost. 

In addition, most existing approaches are `offline' or `batch' methods, which can be impractical for large datasets where observations occur over a long time period. In particular, existing methods typically rely on optimisation of a function (e.g., the log-likelihood) of the entire observed data path, which can be impractically slow for long time periods, or for models which are costly to evaluate. One exception to this is \cite{sharrock2023online}, which introduced an efficient `online' or `recursive' estimator, and analysed its asymptotic properties. Unfortunately, the estimator in \cite{sharrock2023online} still relies on observation of multiple i.i.d. paths of the MVSDE, or multiple trajectories of the IPS. 

In this context, a natural question is whether online parameter estimation in IPSs remains possible when only a single particle, or a small number of particles, can be observed. In this paper, we answer this question in the affirmative.

\subsection{Contributions}
Our main contributions are summarised below.
 
\paragraph{Methodology.} We derive a new online estimator for statistical inference in ergodic IPSs (and the associated MVSDEs). Our estimator is based on minimising the asymptotic (both in time, and in the number of particles) negative log-likelihood of the IPS. It requires observation of the trajectories of just three particles, which suffices to form an asymptotically unbiased estimate of the mean-field interaction terms appearing in the gradient of the asymptotic log-likelihood. In comparison to existing online estimators, which assume it is possible to observe the trajectory of every particle from the IPS, our estimator offers significant computational advantages in the typical case where $N\gg 1$.

\paragraph{Theory.} We prove, under suitable assumptions, that the proposed estimator converges to the stationary points of the asymptotic log-likelihood of the IPS. Under additional assumptions on the asymptotic negative log-likelihood (e.g., strong convexity), we show that our estimator is consistent, in the sense that it converges in $\mathrm{L}^2$ to the true parameter $\theta_0$, and obtain a central limit theorem. We also present corresponding guarantees for a comparable estimator which requires observation of all particles from the IPS, extending existing theoretical results. This enables a detailed comparison of the theoretical properties of both estimators. In all cases, we consider asymptotics in the case where the number of particles is fixed and finite, and only the time horizon $t\to\infty$, and also in the case where both the number of particles $N\to\infty$ and the time horizon $t\to\infty$.

\paragraph{Applications.} We present extensive numerical results to demonstrate the performance of our estimator. Our numerical results corroborate our theoretical findings, and also suggest that our estimator is effective in cases where our assumptions demonstrably do not hold, e.g., in situations where there exist multiple invariant measures, or the diffusion coefficient in the particle dynamics is degenerate. We consider several models of practical interest, including a model for systemic risk, a model of interacting FitzHugh--Nagumo neurons, a Cucker--Smale flocking model, and a mean-field $\frac{3}{2}$ stochastic volatility model.

\subsection{Related Work}
\label{sec:related-work}
\paragraph{Parameter Estimation in IPSs and MVSDEs.} 
Until recent years, surprisingly little work had been devoted to the study of statistical inference for IPSs and MVSDEs. This is in stark contrast to the wealth of literature on parameter estimation in linear SDEs, i.e., diffusion processes whose coefficients do not depend on the law of the process \citep[e.g.,][]{borkar1982parameter,liptser2001statistics,kessler2012statistical,kutoyants2004statistical}. A notable exception is the pioneering work of \cite{kasonga1990maximum}, who established asymptotic properties (consistency, asymptotic normality) of the maximum likelihood estimator (MLE) for a system of weakly interacting particles in the limit as $N\to\infty$, based on continuous observation of all $N$ particles over a fixed time interval $[0,T]$. 

More recently, there has been a surge of interest in this topic. In particular, several authors have extended the results in \cite{kasonga1990maximum} in various directions \citep[e.g.,][]{bishwal2011estimation,giesecke2020inference,chen2021maximum,sharrock2022parameterestimation,sharrock2023online,dellamaestra2023lan,amorino2023parameter}. In particular, \citet{bishwal2011estimation} studied the case where only discrete observations of the system are available, and the parameter to be estimated is a function of time, while \citet{giesecke2020inference} established asymptotic properties (consistency, asymptotic normality, asymptotic efficiency) of an approximate MLE for a much broader class of interacting stochastic systems, widely applicable in financial mathematics, which additionally allow for discontinuous (i.e., jump) dynamics. \citet{chen2021maximum} established the optimal convergence rate for the MLE in an interacting particle system with linear interaction in the limit as both $N\to\infty$ and $T\to\infty$. \citet{sharrock2022parameterestimation} established the asymptotic properties of the MLE in the limit as $N\to\infty$ for a more general family of IPSs. Subsequently, \citet{sharrock2023online} introduced online (or recursive) MLEs for the parameters of an IPS or the associated MVSDE, and analysed their asymptotic properties in the limit as $T\to\infty$, and the joint limit as $T\to\infty$ and $N\to\infty$. \cite{dellamaestra2023lan} established the local asymptotic normality of the MLE, and obtained simple and explicit criteria for identifiability and non-degeneracy of the Fisher information matrix. Meanwhile, \citet{amorino2023parameter} studied joint parameter estimation for both the drift and diffusion coefficients, based on discrete observations of the IPS over a fixed time interval $[0,T]$. 

In a rather different direction, \citet{dellamaestra2022nonparametric} have considered non-parametric estimation of the drift term in a MVSDE, based on continuous observation of the associated IPS over a fixed time horizon. More specifically, the authors obtained adaptive estimators based on the solution map of the Fokker--Planck equation, and proved their optimality in a minimax sense. We refer to \cite{lu2019nonparametric,lang2021identifiability,lu2021learning,comte2022nonparametric,yao2022meanfield,amorino2025polynomial,belomestny2024nonparametric,comte2024nonparametric,nickl2025bayesian,pavliotis2025fourierbased} for other recent contributions on non-parametric (and semi-parametric) inference for IPSs. 

In most of the aforementioned works, statistical inference is based on direct observation of all $N$ particles in the IPS, or observation of $N$ i.i.d. trajectories of the limiting MVSDE. In cases where the number of particles is very large, however, this may be unrealistic, or entail a prohibitive computational cost. In this context, several authors have also studied estimation based on observation of a single particle from the IPS \citep{pavliotis2022eigenfunction,pavliotis2024method,pavliotis2025fourierbased}, or else a single trajectory of the limiting MVSDE \citep{genoncatalot2022inference,genoncatalot2023parametric,comte2024nonparametric}.  In particular, \citet{genoncatalot2022inference} studied parametric inference for a specific class of one-dimensional MVSDEs with no potential term and a polynomial interaction term, based on continuous observation of a single sample path on the time interval $[0,2T]$ in the stationary regime. \citet{genoncatalot2023parametric} considered a more general family of MVSDEs, and proposed an alternative pseudo-likelihood approach based on a kernel estimator of the invariant density. Meanwhile, \citet{pavliotis2022eigenfunction} established the asymptotic properties (asymptotic unbiasedness, asymptotic normality) of the eigenfunction martingale estimator as $N\to\infty$, based on discrete observations of a single trajectory of the IPS on a fixed time interval $[0,T]$, while \citet{pavliotis2024method} established consistency of the method of moments estimator as $N\to\infty$ and $T\to\infty$. Finally, \citet{pavliotis2025fourierbased} considered semi-parametric estimation of the interaction kernel based on observation of a single particle using a generalised Fourier expansion.

\paragraph{Online Parameter Estimation in Continuous-Time Processes.}
Even for linear SDEs, the literature on `online' or `recursive' parameter estimation is somewhat sparse, with some notable recent exceptions \citep{sirignano2017stochastic,surace2019online,sirignano2020stochastic,bhudisaksang2021online,sharrock2022twotimescale,sharrock2023twotimescale,sharrock2022theory}. The problem of recursive estimation in continuous time stochastic processes was first analysed by \citet{gerencser1984continuoustime}; see also  \citet{levanony1994recursive,gerencser2009recursive} for some other early references. More recently, this problem was revisited by \citet{sirignano2017stochastic,sirignano2020stochastic}, who proposed an online method---`stochastic gradient descent in continuous time'---for statistical inference in fully observed diffusion processes, and analysed its asymptotic properties (e.g., almost sure convergence, asymptotic normality). Their method can be seen as a form of continuous-time stochastic gradient descent with respect to the asymptotic (or average) log-likelihood of the diffusion process. This approach has since been extended to partially observed diffusion processes \citep{surace2019online,sharrock2023twotimescale,sharrock2022joint}, jump diffusion processes \citep{bhudisaksang2021online}, nonlinear diffusion processes \citep{sharrock2023online},  and stochastic processes driven by coloured noise \citep{pavliotis2025filtered}.

More recently, a related problem has also been studied by \citet{wang2024continuous,wang2022forward}. In particular, \citet{wang2024continuous,wang2022forward} considered the task of minimising a function of the stationary distribution of a parameterised (linear, in the sense of McKean) diffusion process. They introduced an efficient continuous-time stochastic gradient descent algorithm for this task, which continuously updates the parameters of the model using an estimate for the gradient of the stationary distribution, which is simultaneously updated using forward propagation of the derivatives of the diffusion process. They establish the convergence of their `online forward propagation algorithm' to the stationary points of the objective function for a broad class of diffusion processes, and demonstrate its efficacy in a number of applications.

\subsection{Paper Organisation}
The remainder of this paper is organised as follows. In Section \ref{sec:prelims}, we introduce the problem of interest. In Section \ref{sec:methodology}, we present our main methodological contributions. In Section \ref{sec:theory}, we state our assumptions and our main theoretical results. In Section \ref{sec:numerics}, we present several numerical examples illustrating our proposed methodology. Finally, in Section \ref{sec:conclusion}, we provide some concluding remarks. Additional material, including proofs of our main results, is provided in the Appendices. 

\section{Preliminaries}
\label{sec:prelims} 

\subsection{Notation}
We use $\langle\cdot,\cdot\rangle$ and $\|\cdot\|$ to denote, respectively, the Euclidean inner product and the Euclidean norm on $\mathbb{R}^d$. For matrices and higher order tensors, we use $\|\cdot\|$ to denote the Frobenius norm. Finally, we write $\|\cdot\|_p$ to denote the $\ell^p$ norm. We write $\mathcal{P}(\mathbb{R}^d)$ for the collection of all probability measures on $\mathbb{R}^d$ and $\smash{\mathcal{P}_p(\mathbb{R}^d) = \{\mu\in\mathcal{P}(\mathbb{R}^d): \int_{\mathbb{R}^d} \|x\|^p \mu(\mathrm{d}x)<\infty\}}$ for the collection of all probability measures on $\mathbb{R}^d$ with finite $p^{\text{th}}$ moment. In a slight abuse of notation, we will occasionally write $\mu(\|\cdot\|^p)=\int_{\mathbb{R}^d} \|x\|^p\mu(\mathrm{d}x)$ for the $p^{\text{th}}$ moment of $\mu$.  For $p\geq 1$, and $\mu,\nu\in\mathcal{P}_p(\mathbb{R}^d)$, we will write $\mathsf{W}_{p}(\mu,\nu)$ to denote the Wasserstein distance between $\mu$ and $\nu$, viz.
\begin{equation}
\mathsf{W}_{p}(\mu,\nu) = \inf_{\pi\in\Pi(\mu,\nu)}\left[\int_{\mathbb{R}^d\times\mathbb{R}^d}\|x-y\|^p\pi(\mathrm{d}x,\mathrm{d}y)\right]^{\frac{1}{p}},
\end{equation}
where $\Pi(\mu,\nu)$ denotes the set of all couplings of $\mu$ and $\nu$. That is, the set of all probability measures on $\mathbb{R}^d\times\mathbb{R}^d$ with marginals $\mu$ and $\nu$.

\subsection{The Model}
We consider a weakly interacting particle system (IPS) on $\mathbb{R}^d$, parameterised by $\theta\in\Theta$, of the form
\begin{align}
\mathrm{d}x_t^{\theta,i,N}&= \bigg[-\nabla V(\theta,x_t^{\theta,i,N}) - \frac{1}{N}\sum_{j=1}^N \nabla W(\theta,x_t^{\theta,i,N}-x_t^{\theta,j,N})\bigg]\mathrm{d}t + \sigma\mathrm{d}w_t^{i,N}, \quad t\geq 0, \label{eq:IPS} 
\end{align}
where $V(\theta,\cdot):\mathbb{R}^d\rightarrow\mathbb{R}$, $W(\theta,\cdot):\mathbb{R}^d\rightarrow\mathbb{R}$ are continuously differentiable functions, $\sigma\in\mathbb{R}^{d\times d}$ is a constant, invertible matrix, $w^{i,N}:=(w_t^{i,N})_{t\geq0}$ are a set of independent $\mathbb{R}^d$-valued standard Brownian motions, and $\Theta\subseteq\mathbb{R}^p$ is an open set. We assume that $\smash{(x_0^{i,N})_{i=1}^N}$ are a set of i.i.d. $\mathbb{R}^d$-valued random variables, with common law $\mu_0$, independent of $(w_t^{i,N})_{t\geq 0}$. We will commonly refer to $V$ as the \emph{confinement potential}, and to $W$ as the \emph{interaction potential}. 

For notational convenience, it will also be useful for us to introduce the \emph{drift function} $b:\mathbb{R}^p\times\mathbb{R}^d\times\mathbb{R}^d\rightarrow\mathbb{R}^d$, defined according to $b(\theta,x,y):= -\nabla V(\theta,x) - \nabla W(\theta,x-y)$. Using this notation, the IPS can be written as 
\noeqref{eq:ips-b-drift}
\begin{align}
    \mathrm{d}x_t^{\theta,i,N} &= 
    \Big[\frac{1}{N}\sum_{j=1}^N b(\theta,x_t^{\theta,i,N},x_t^{\theta,j,N})\Big]\mathrm{d}t+ \sigma\mathrm{d}w_t^{i,N}. \label{eq:ips-b-drift}
\end{align}
We will assume, throughout this paper, that there exists a true, static parameter $\theta_0\in\Theta$ which generates observations $(x_t^{i,N})_{t\geq 0}:= (x_t^{\theta_0,i,N})_{t\geq 0}$ of the IPS \eqref{eq:IPS}. Thus, we operate under the exact modelling regime, and in our notation will suppress the dependence of the observed path on the true parameter $\theta_0$. 

\begin{remark}
It will sometimes be useful to view the IPS in \eqref{eq:IPS} as an SDE on $(\mathbb{R}^d)^N$. In particular, suppose we write $\boldsymbol{x}_t^{\theta,N} = (x_t^{\theta,1,N},\dots,{x}_t^{\theta,N,N})^{\top}\in(\mathbb{R}^d)^N$. Then this process is the solution of
\begin{equation}
    \label{eq:vector-sde}
    \mathrm{d}\boldsymbol{x}_t^{\theta,N} = {B}^N(\theta,\boldsymbol{x}_t^{\theta,N})\mathrm{d}t + \Sigma_N \mathrm{d}\boldsymbol{w}_t^N,
\end{equation}
where $\Sigma_N = \boldsymbol{I}_N\otimes \sigma$, $\boldsymbol{w}^N = (w^{1,N},\dots,w^{N,N})^{\top}$ is a $(\mathbb{R}^d)^N$-valued standard Brownian motion, and the function $B^N(\theta,\cdot):(\mathbb{R}^d)^N\rightarrow(\mathbb{R}^d)^N$ is defined according to the form $ B^N(\theta,\boldsymbol{x}^{N}) = (B^{i,N}(\theta,\boldsymbol{x}^{N}),\dots,B^{N,N}(\theta,\boldsymbol{x}^{N}))^{\top}$, where, for each $i\in[N]:=\{1,\dots,N\}$, the function 
 $B^{i,N}(\theta,\cdot):(\mathbb{R}^d)^N\rightarrow\mathbb{R}^d$ is defined according to $B^{i,N}(\theta,\boldsymbol{x}^{N}) = \frac{1}{N}\sum_{j=1}^N b(\theta,x^{i,N},x^{j,N})$.
 \end{remark}

\subsection{The Mean-Field Model}
We are interested in the regime where the number of particles $N\gg 1$ so that, under appropriate conditions which we will later impose \citep[e.g.,][]{malrieu2001logarithmic,cattiaux2008probabilistic}, any single particle in the IPS can be well approximated by the solution of the limiting MVSDE, viz.
\begin{align}
\mathrm{d}x_t^{\theta}&= \bigg[
-\nabla V(\theta,x_t^{\theta}) - \int_{\mathbb{R}^d}\nabla W(\theta,x_t^{\theta}-y)\mu_t^{\theta}(\mathrm{d}y)
\bigg]\mathrm{d}t + \sigma\mathrm{d}w_t,\quad t\geq 0, \label{eq:MVSDE} 
\end{align}
where $w=(w_t)_{t\geq 0}$ is a standard $\mathbb{R}^d$-valued Brownian motion, and $\mu_t^{\theta} =\mathcal{L}(x_t^{\theta})$ denotes the law of $x_t^{\theta}$. This phenomenon is known as the \emph{propagation of chaos} \citep{sznitman1991topics,chaintron2022propagation,chaintron2022propagationa}.

\subsection{Model Assumptions}
\label{sec:model-assumptions}
We are now ready to state some initial assumptions on the data generating process. We begin with the following integrability assumption on the initial condition.

\begin{assumption}
\label{assumption:moments}
    The initial law satisfies $\mu_0\in\mathcal{P}_k(\mathbb{R}^d)$ for all $k\in\mathbb{N}$.
\end{assumption}

\noindent Meanwhile, regarding the true drift function, i.e., the drift function evaluated at the true parameter, we impose one of the following two sets of assumptions. 

\begin{assumption}
\label{assumption:drift}
    The functions $x\mapsto V(\theta_0,x)$ and $x\mapsto W(\theta_0,x)$ are twice continuously differentiable. In addition, they satisfy one of the following two conditions:
    \begin{enumerate}
        \item[(a)(i)] $V(\theta_0,\cdot)$ satisfies the $C(A,\alpha)$ condition. That is, there exist $A>0$, $\alpha\geq 0$ such that, for all $0<\varepsilon<1$, and for all $x,y\in\mathbb{R}^d$, 
        \begin{equation}
            (x - y)\cdot (\nabla V(\theta_0,x) - \nabla V(\theta_0,y) )\geq A\varepsilon^{\alpha}\left(\|x-y\|^2 - \varepsilon^2\right).
        \end{equation}
        In addition, $\nabla V(\theta_0,\cdot)$ is locally Lipschitz with polynomial growth, and $\nabla^2 V(\theta_0,\cdot)$ has polynomial growth. That is, there exist $C,m>0$ such that, for all $x,y\in\mathbb{R}^d$,  
        \begin{align}
            \|\nabla V(\theta_0,x) - \nabla V(\theta_0,y)\| &\leq C\|x-y\|(1+ \|x\|^m + \|y\|^m), \\
            \|\nabla^2 V(\theta_0,x)\| &\leq C(1 + \|x\|^m).
        \end{align}
        \item[(a)(ii)] $W(\theta_0,\cdot)$ is symmetric  and convex, $\nabla W(\theta_0,\cdot)$ is locally Lipschitz with polynomial growth, and $\nabla^2 W(\theta_0,\cdot)$ has polynomial growth. That is, there exist $C,m>0$ such that, for all $x,y\in\mathbb{R}^d$, 
        \begin{align}
            \|\nabla W(\theta_0,x) - \nabla W(\theta_0,y)\| &\leq C\|x-y\|(1+ \|x\|^m + \|y\|^m), \\
            \|\nabla^2 W(\theta_0,x)\| &\leq C(1 + \|x\|^m).
        \end{align}
    or
        \item[(b)(i)] $V(\theta_0,\cdot)=0$.
        \item[(b)(ii)] $W(\theta_0,\cdot)$ is symmetric,  $\nabla W(\theta_0,\cdot)$ is locally Lipschitz with polynomial growth, and $\nabla^2 W(\theta_0,\cdot)$ has polynomial growth. That is, there exist $C,m>0$ such that, for all $x,y\in\mathbb{R}^d$, 
        \begin{align}
            \|\nabla W(\theta_0,x) - \nabla W(\theta_0,y)\| &\leq C\|x-y\|(1+ \|x\|^m + \|y\|^m), \\
            \|\nabla^2 W(\theta_0,x)\| &\leq C(1 + \|x\|^m).
        \end{align}
        In addition, $W(\theta_0,\cdot)$ satisfies the $C(A,\alpha)$ condition. That is, there exist $A>0$, $\alpha\geq 0$,  such that, for all $0<\varepsilon<1$, and for all $x,y\in\mathbb{R}^d$, 
        \begin{equation}
            (x - y)\cdot (\nabla W(\theta_0,x) - \nabla W(\theta_0,y) )\geq A\varepsilon^{\alpha}\left(\|x-y\|^2 - \varepsilon^2\right).
        \end{equation}
    \end{enumerate}
\end{assumption}

These conditions ensure the existence and uniqueness of a strong solution to the (observed) IPS and the associated MVSDE \citep[][Theorem 2.6]{cattiaux2008probabilistic}, uniform-in-time moment bounds \citep[][Proposition 2.7]{cattiaux2008probabilistic}, uniform-in-time propagation of chaos \citep[][Theorem 3.1]{cattiaux2008probabilistic}, and the existence of, and convergence to, a unique invariant measure \citep[][Theorem 4.1]{cattiaux2008probabilistic}.\footnote{
\label{remark:centering}Under Assumption \ref{assumption:drift}(b), the required results in fact hold for a projected or centered version of the observed IPS, defined by $\smash{y_t^{i,N} = x_t^{i,N} - \frac{1}{N}\sum_{j=1}^N x_t^{j,N}}$ \citep[][Section 2]{cattiaux2008probabilistic}. This still defines a diffusion process, but now on the hyperplane $\mathcal{M}_N = \{\boldsymbol{x}^N\in(\mathbb{R}^d)^N: \sum_{i=1}^N x^{i,N} = 0\}$. For notational convenience, in the remainder we will state all results using the notation $\smash{(x_t^{i,N})_{t\geq 0}^{i\in[N]}}$, with this understood to mean the centered process $\smash{(y_t^{i,N})_{t\geq 0}^{i\in[N]}}$ when Assumption~\ref{assumption:drift}(b) holds.
}\textsuperscript{,}\footnote{
Under Assumption \ref{assumption:drift}(b), the MVSDE in fact admits a one-parameter family of invariant distributions, the parameter corresponding to the mean of the invariant distribution. Thus, the invariant distribution is unique once its expectation has been specified. Throughout this paper, we will assume that the mean is fixed, and without loss of generality set it to zero \cite[see, e.g.,][]{cattiaux2008probabilistic}.
} We provide a precise statement of these results in Appendix~\ref{app:existing-results}.

\begin{remark}
\label{rem:poc}
    More generally, our theoretical analysis remains valid under any conditions which guarantee uniform-in-time moment bounds, uniform-in-time propagation of chaos, and convergence (at a sufficiently fast rate) to an invariant measure. See, e.g., \cite{malrieu2001logarithmic,malrieu2003convergence} for some classical assumptions, and \cite{carrillo2020longtime,delgadino2023phase,lacker2018strong,lacker2023hierarchies,lacker2023sharp} for some more recent results. We choose to adopt the conditions introduced in \citet{cattiaux2008probabilistic} as they are at once sufficiently general to hold for many models of practical interest (see Section \ref{sec:numerics}), whilst not being so general as to demand a significant additional notational overhead. 
\end{remark}

Finally, we will impose the following regularity assumptions on the confinement potential and the interaction potential.

\begin{assumption}
\label{assumption:drift-grad}
    For all $\theta\in\Theta$, the functions $x\mapsto \nabla V(\theta,x)$ and $x\mapsto \nabla W(\theta,x)$ are continuously differentiable. For all $x\in\mathbb{R}^d$, the functions $\theta\mapsto \nabla V(\theta,x)$ and $\theta\mapsto \nabla W(\theta,x)$ are three times continuously differentiable.  In addition, for $k=0,1,2,3$:
    \begin{itemize}
        \item[(i)] There exist $C_k,m_k>0$ such that, for all $\theta\in\Theta$, and for all $x,y\in\mathbb{R}^d$, 
        \begin{align}
        \| \partial_{\theta}^k \nabla V(\theta,x)- \partial_{\theta}^k \nabla V (\theta,y)\|
        &\le C_k\|x-y\|\big(1+\|x\|^{m_k}+\|y\|^{m_k}\big) \\
        \| \partial_{\theta}^k \nabla V(\theta,0)\| &\leq C_k \label{eq:ass-MV-lip}
        \end{align}
        \item[(ii)] 
        There exist $C_k,m_k>0$ such that, for all $\theta\in\Theta$, and for all $x,y\in\mathbb{R}^d$,
        \begin{align}
        \| \partial_{\theta}^k\nabla W(\theta,x)- \partial_{\theta}^k\nabla W(\theta,y)\|
        &\le C_k\|x-y\|\big(1+\|x\|^{m_k}+\|y\|^{m_k}\big) \\
        \| \partial_{\theta}^k \nabla W(\theta,0)\| &\leq C_k
        \end{align}
    \end{itemize}
\end{assumption}

\begin{remark}
    In the case where $\Theta\subset\mathbb{R}^p$, these assumptions are very mild, and will hold for essentially all of the models that we encounter in practice. On the other hand, when $\Theta=\mathbb{R}^p$, they are somewhat restrictive. Given this, it is worth noting that it is possible to establish our results under weaker assumptions, which additionally allow for $x\mapsto \nabla V(\theta,x)$ and $x\mapsto \nabla W(\theta,x)$ to grow linearly in $\|\theta\|$ \citep[see, e.g.,][]{sirignano2020stochastic}. 
\end{remark}

\subsection{The Likelihood Function}
\label{sec:log-lik}
We are interested in online inference for the unknown parameter $\theta_0$. We will perform this task based on recursive maximisation of an appropriate likelihood function. 

\subsubsection{The Log-Likelihood of the Interacting Particle System}
\label{sec:log-lik-ips}
Let $\mathbb{P}_t^{\theta,N}$ denote the probability measure induced by the trajectories $\smash{(x_s^{\theta,i,N})_{s\in[0,t]}^{i\in[N]}}$ of the IPS. Then, using Girsanov's Theorem \citep[e.g.,][]{oksendal2003stochastic},
we have a log-likelihood function given (up to an additive constant) by \citep[e.g.,][]{kasonga1990maximum,dellamaestra2023lan}
\begin{align}
\label{eq:ips-log-likelihood}
\mathcal{L}_t^{N}(\theta)
&=\int_0^t  \langle B^N(\theta,\boldsymbol{x}_s^N), (\Sigma_N\Sigma_N^{\top})^{-1}\mathrm{d}\boldsymbol{x}_s^N\rangle - \frac{1}{2}\int_0^t \|B^N(\theta,\boldsymbol{x}_s^N)\|_{\Sigma_N\Sigma_N^{\top}}^2\mathrm{d}s \\
&=\sum_{i=1}^N\Big[ \int_0^t \big \langle  B^{i,N}(\theta,\boldsymbol{x}_s^N), (\sigma\sigma^{\top})^{-1}\mathrm{d}x_s^{i,N}\big\rangle - \frac{1}{2} \int_0^t \| B^{i,N}(\theta,\boldsymbol{x}_s^N)\|_{\sigma\sigma^{\top}}^2\mathrm{d}s \Big] \label{ips:log-likelihood-2} \\
&=\sum_{i=1}^N\Big[ \int_0^t \big \langle  B(\theta,x_s^{i,N},\mu_s^{N}), (\sigma\sigma^{\top})^{-1}\mathrm{d}x_s^{i,N}\big\rangle - \frac{1}{2} \int_0^t \| B(\theta,x_s^{i,N},\mu_s^{N})\|_{\sigma\sigma^{\top}}^2\mathrm{d}s \Big] \label{ips:log-likelihood}
\end{align}
 where $\smash{\mu_s^N=\frac{1}{N}\sum_{j=1}^N \delta_{x_s^{j,N}}}$ denotes the empirical law of the observed IPS, $\smash{B(\theta,x,\mu) = \int b(\theta,x,y)\mu(\mathrm{d}y)}$, and $\smash{\|z\|_{\sigma\sigma^{\top}}^2 := \langle z, z\rangle_{\sigma\sigma^{\top}}:=\langle z,(\sigma \sigma^{\top})^{-1} z \rangle}$.\footnote{Strictly speaking, to use Girsanov's theorem to define the log-likelihood in \eqref{eq:ips-log-likelihood} we must assume that, for all $\theta\in\Theta$, the likelihood ratio process
$\smash{Z_s^\theta:={\mathrm{d}\mathbb{P}_t^{\theta,N}}/{\mathrm{d}\mathbb{W}_t^N}\big|_{\mathcal{F}_s}}$,
$s\in[0,t]$ exists and is a martingale under $\mathbb{W}_t^N$, the unique law (on path space) of the driftless system $\mathrm{d}\boldsymbol{x}_s^N=\Sigma_N\,\mathrm{d}\boldsymbol{w}_s^N$, with the same initial condition as the original process. In the absence of this assumption, we can simply view the function in \eqref{eq:ips-log-likelihood} as a contrast function, and proceed in the same way.}\textsuperscript{,}\footnote{Under Assumption~\ref{assumption:drift}(b), we must consider the log-likelihood associated with the centered process $\smash{(y_s^{\theta,i,N})_{s\in[0,t]}^{i\in[N]}}$. This requires additional care, since the diffusion coefficient $\tilde\Sigma_N$ is now singular on $(\mathbb{R}^d)^N$. Nonetheless, with the addition of a small addendum to Assumption~\ref{assumption:drift}(b), one can show that the log-likelihood for the centered IPS takes precisely the same functional form as for the non-centered IPS. Thus, our subsequent methodological developments are applicable in this case. We discuss this further in Appendix~\ref{app:centered-ips}.}

 We are first interested in the asymptotic behaviour of this function as the time horizon $t\rightarrow\infty$, for a fixed and finite number of particles $N\in\mathbb{N}$. This is the subject of the following proposition.

\begin{proposition}
\label{prop:ips-likelihood-t-limit}
    Suppose that Assumption \ref{assumption:moments}, Assumption \ref{assumption:drift}, and Assumption~\ref{assumption:drift-grad} (with $k=0$) hold. Then, as $t\rightarrow\infty$, it holds that
    \noeqref{eq:IPS-log-likelihood-large-t-a}
    \begin{align}
    \label{eq:IPS-log-likelihood-large-t-a}
        \frac{1}{t}[\mathcal{L}_t^{N}(\theta) - \mathcal{L}_t^{N}(\theta_0)] &\xrightarrow[L^1]{\mathrm{a.s.}} -\frac{1}{2}\int_{(\mathbb{R}^d)^{N}} \Big[ \sum_{i=1}^N  \left|\left|B^{i,N}(\theta,\boldsymbol{x}^N) - B^{i,N}(\theta_0,\boldsymbol{x}^N)\right|\right|_{\sigma\sigma^{\top}}^2 \Big] \pi_{\theta_0}^N(\mathrm{d}\boldsymbol{x}^N).
    \end{align}
    where $\pi_{\theta_0}^N\in\mathcal{P}((\mathbb{R}^d)^N)$ denotes the unique invariant measure of the IPS evaluated at the true parameter $\theta_0$.
\end{proposition}

\begin{proof}
    See Appendix \ref{app:additional-proofs-log-lik-ips}.
\end{proof}

Meanwhile, the asymptotic behaviour of these functions as the number of particles $N\rightarrow\infty$, for a fixed and finite time horizon $t\in\mathbb{R}_{+}$, is established in the following proposition. 

\begin{proposition}
\label{prop:ips-likelihood-n-limit}
    Suppose that Assumption \ref{assumption:moments}, Assumption \ref{assumption:drift}, and Assumption~\ref{assumption:drift-grad} (with $k=0$) hold. Then, as $N\rightarrow\infty$, it holds that
    \noeqref{eq:IPS-log-likelihood-large-N-a}
    \begin{align}
        \frac{1}{N}\left[\mathcal{L}_t^N(\theta) - \mathcal{L}_t^N(\theta_0)\right] 
        &\stackrel{L^1}{\longrightarrow}-\frac{1}{2} \int_0^t \left[\int_{\mathbb{R}^d} \| B(\theta,x,\mu_s^{\theta_0}) - B(\theta_0,x,\mu_s^{\theta_0})\|_{\sigma\sigma^{\top}}^2\mu_s^{\theta_0}(\mathrm{d}x)\right]\mathrm{d}s
        \label{eq:IPS-log-likelihood-large-N-a}
    \end{align}
    where $\mu_s^{\theta_0} = \mathrm{Law}(x_s^{\theta_0})\in\mathcal{P}(\mathbb{R}^d)$ denotes the law of the MVSDE evaluated at the true parameter $\theta_0$.
\end{proposition}

\begin{proof}
    See Appendix \ref{app:additional-proofs-log-lik-ips}.
\end{proof}

Finally, we can characterise the behaviour of the log-likelihood of the IPS in the joint limit as $N\rightarrow\infty$ and $t\rightarrow\infty$. 

\begin{corollary}
\label{cor:ips-likelihood-n-t-limit}
    Suppose that Assumption \ref{assumption:moments}, Assumption \ref{assumption:drift}, and Assumption~\ref{assumption:drift-grad} (with $k=0$) hold. Then, as $N\rightarrow\infty$ and then $t\rightarrow\infty$, it holds that
    \begin{align}
         \frac{1}{Nt} \left[\mathcal{L}_t^N(\theta) - \mathcal{L}_t^{N}(\theta_0)\right]
         &\stackrel{L^1}{\longrightarrow} -\frac{1}{2}\int_{\mathbb{R}^d} \|B(\theta,x,\pi_{\theta_0}) - B(\theta_0,x,\pi_{\theta_0})\|_{\sigma\sigma^{\top}}^2\pi_{\theta_0}(\mathrm{d}x) \label{eq:asymptotic-ll-IPS}
    \end{align}
    where $\pi_{\theta_0}\in\mathcal{P}(\mathbb{R}^d)$ denotes the unique invariant measure of the MVSDE evaluated at the true parameter $\theta_0$.
\end{corollary}

\begin{proof}
    See Appendix \ref{app:additional-proofs-log-lik-ips}
\end{proof}

\subsubsection{The Log-Likelihood of the McKean--Vlasov SDE}
\label{sec:log-lik-mvsde}
Let $\mathbb{P}_t^{\theta}$ denote the probability measure induced by the solution $(x_s^{\theta})_{s\in[0,t]}$ of the MVSDE \eqref{eq:MVSDE}. Then, once more appealing to Girsanov's Theorem, we have a log-likelihood function given by \citep[e.g.,][Section 2.3]{dellamaestra2023lan}
\begin{equation}
    \mathcal{L}_t(\theta) = \int_0^t \big\langle B(\theta,x_s,\mu_s^{\theta}), (\sigma\sigma^{\top})^{-1}\mathrm{d}x_s\big\rangle - \frac{1}{2}\int_0^t \big\|B(\theta,x_s,\mu_s^{\theta})\big\|^2_{\sigma\sigma^{\top}}\mathrm{d}s \label{eq:MVSDE-log-likelihood}
\end{equation}
where $(x_s)_{s\geq 0}:=(x_s^{\theta_0})_{s\geq 0}$ denotes the path of the MVSDE at the true parameter $\theta_0$. In this case, we are just interested in the asymptotic behaviour of this function as the time horizon $t\rightarrow\infty$. In particular, we have the following result.

\begin{proposition}
\label{prop:mvsde-likelihood-t-limit}
    Suppose that Assumption \ref{assumption:moments}, Assumption \ref{assumption:drift}, and Assumption~\ref{assumption:drift-grad} (with $k=0$) hold.\footnote{For this result, we in fact require that Assumption~\ref{assumption:drift} holds for all $\theta\in\Theta$, and not just for $\theta=\theta_0$. This ensures the existence of the family of invariant measures $(\pi_{\theta})_{\theta\in\Theta}$.} Then, as $t\rightarrow\infty$, it holds that
    \begin{align}
        \frac{1}{t}\left[\mathcal{L}_t(\theta) - \mathcal{L}_t(\theta_0)\right]
        &\stackrel{L^1}{\longrightarrow} -\frac{1}{2}\int_{\mathbb{R}^d} \| B(\theta,x,\pi_{\theta}) - B(\theta_0,x,\pi_{\theta_0})\|_{\sigma\sigma^{\top}}^2\pi_{\theta_0}(\mathrm{d}x). \label{eq:asymptotic-ll-mvsde}
    \end{align}
    where $\pi_{\theta}, \pi_{\theta_0}\in\mathcal{P}(\mathbb{R}^d)$ denote the unique invariant measures of the MVSDE, evaluated at the parameter $\theta$ and the true parameter $\theta_0$, respectively.
\end{proposition}

\begin{proof}
See Appendix \ref{app:additional-proofs-log-lik-mvsde}.
\end{proof}

\begin{remark}
Curiously, the asymptotic log-likelihood of the IPS in the joint limit as $N\rightarrow\infty$ and $t\rightarrow\infty$, c.f. \eqref{eq:asymptotic-ll-IPS}, does not coincide with the asymptotic log-likelihood of the MVSDE as $t\rightarrow\infty$, c.f. \eqref{eq:asymptotic-ll-mvsde}. This being said, the two functions do coincide (and are both maximised) at the true parameter $\theta_0$. This disparity is perhaps a little surprising. Indeed, under the assumption of uniform-in-time propagation of chaos, we know that the dynamics of the IPS will converge to the dynamics of the McKean--Vlasov process as $N\rightarrow\infty$ and $t\rightarrow\infty$. The discrepancy arises because the log-likelihood of the IPS uses the empirical distribution of the observed system $\mu_t^N$, which converges to $\pi_{\theta_0}$, while the log-likelihood of the MVSDE uses the model distribution $\mu_t^{\theta}$, which converges to $\pi_{\theta}$. 
\end{remark}

\section{Methodology}
\label{sec:methodology}

Our goal is to recursively estimate the true parameter $\theta_{0}$ in real time, using the continuous stream of observations of a (subset of) the full collection of particles $\smash{(x_t^{i,N})_{t\geq 0}^{i\in[N]}}$ from the IPS.\footnote{In the case where Assumption~\ref{assumption:drift}(b) holds (i.e., the confinement potential is null), we must assume it is possible to observe (a subset of) the centered particles $\smash{(y_t^{i,N})_{t\geq 0}^{i\in[N]}}$. This ensures that the results required for our theoretical analysis (e.g., propagation-of-chaos, convergence to a unique invariant measure) continue to hold (see Remark~\ref{remark:centering}).} To achieve this task, we will seek to recursively minimise an appropriately chosen objective.

\subsection{The Objective Function}
 We are interested in the case where the number of particles $N\gg 1$, and thus any single particle in the IPS \eqref{eq:IPS} resembles a solution of the MVSDE \eqref{eq:MVSDE}. In this regime, there are two natural choices for the objective function.  

The first is the average \emph{negative} log-likelihood of the IPS, which under the conditions specified in Corollary~\ref{cor:ips-likelihood-n-t-limit} is given by:
\begin{align}
    \mathcal{L}(\theta) &:= \int_{\mathbb{R}^d} \frac{1}{2}\|B(\theta,x,\pi_{\theta_0})-B(\theta_0,x,\pi_{\theta_0})\|_{\sigma\sigma^{\top}}^2 \pi_{\theta_0}(\mathrm{d}x):=\int_{\mathbb{R}^d} L(\theta,x,\pi_{\theta_0})\pi_{\theta_0}(\mathrm{d}x). \label{eq:obj-func-1} \\ 
\intertext{\indent The second is the average negative log-likelihood of the limiting MVSDE, which under the conditions specified in Proposition~\ref{prop:mvsde-likelihood-t-limit} is given by:}
    \mathcal{J}(\theta)&:= \int_{\mathbb{R}^d} \frac{1}{2}\|B(\theta,x,\pi_{\theta})-B(\theta_0,x,\pi_{\theta_0})\|_{\sigma\sigma^{\top}}^2 \pi_{\theta_0}(\mathrm{d}x):=\int_{\mathbb{R}^d}J(\theta,x,\pi_{\theta},\pi_{\theta_0})\pi_{\theta_0}(\mathrm{d}x). \label{eq:obj-func-2}
\end{align}

\begin{remark}
These two functions are non-negative for all $\theta\in\Theta$ and, under standard identifiability assumptions \citep[e.g.,][Assumptions S2,S4]{genoncatalot2022inference}, uniquely minimised (and equal to zero) at the true parameter $\theta_0$.
\end{remark}

Interestingly, designing recursive maximum likelihood estimators with respect to the two objective functions in \eqref{eq:obj-func-1}, \eqref{eq:obj-func-2} will lead to rather different algorithms. In this paper, we will focus exclusively on algorithms designed with reference to the first objective function. The second will be the subject of a forthcoming paper \citep{sharrock2026recursive}.

\subsection{Gradient Descent in Continuous Time}
\label{sec:grad-log-lik}
In order to optimise the objective function in \eqref{eq:obj-func-1}, a natural approach is to simulate the corresponding \emph{gradient flow}, or \emph{curve of steepest descent}. Let $\smash{\theta_{\mathrm{init}}\in\Theta}$. Then, the gradient flow  $\smash{(\theta_t)_{t\geq 0}}$ of $\mathcal{L}$ is defined as the solution of 
\begin{align}
\label{eq:continuous-time-grad-descent-1}
\mathrm{d}\theta_t &= -\gamma_t \partial_{\theta}\mathcal{L}(\theta_t)\mathrm{d}t, 
\end{align}
where $\gamma_t:\mathbb{R}_{+}\rightarrow \mathbb{R}_{+}$ is a deterministic, positive, non-increasing function commonly referred to as the \emph{learning rate}. Thus, for all $t\geq 0$, $\smash{(\theta_t)_{t\geq 0}}$ follows the direction of steepest descent with respect to the asymptotic log-likelihood function of the IPS.

\begin{remark}
The definition of the gradient flow above differs slightly from the standard definition of a gradient flow \citep[e.g.,][]{santambrogio2017euclidean}, due to the additional inclusion of the learning rate $(\gamma_t)_{t\geq 0}$. Given this, we will sometimes instead refer to \eqref{eq:continuous-time-grad-descent-1} as ``gradient descent in continuous time'', in line with the taxonomy introduced in \cite{sirignano2017stochastic}. We can recover the standard definition of the gradient flow after a time reparameterisation. In particular, defining a new time coordinate as $\tau = \tau(t) = \int_0^{t}\gamma_{s}\mathrm{d}s$, we have $\mathrm{d}\theta_{\tau} =  -\partial_{\theta}\mathcal{L}({\theta}_{\tau})\mathrm{d}\tau$.
\end{remark}

\begin{remark}
\label{remark:constrained}
In order to account for the case where $\Theta\subsetneq\mathbb{R}^p$, we shall in fact use a modified version of this equation \cite[e.g.,][]{surace2019online}, setting 
\begin{align}
\label{eq:continuous-time-grad-descent-modified-1}
\mathrm{d}\theta_t &= \left\{\begin{array}{lll} -\gamma_t \partial_{\theta}\mathcal{L}(\theta_t)\mathrm{d}t & , & \theta_t\in\Theta, \\
0 & , & \theta_t\notin\Theta. \\
\end{array}
\right.
\end{align}
For notational convenience, in what follows we will always write update equations in the form \eqref{eq:continuous-time-grad-descent-1}. However, this should always be understood to mean \eqref{eq:continuous-time-grad-descent-modified-1}. 
\end{remark}

\subsubsection{The Gradient of the Asymptotic Log-Likelihood Functions}
In order to implement the gradient flow in \eqref{eq:continuous-time-grad-descent-1} we will first need to compute the gradient of the objective function. This is the subject of the following proposition.

\begin{proposition}
\label{prop:ips-likelihood-n-t-limit-grad}
    Suppose that Assumption \ref{assumption:moments}, Assumption \ref{assumption:drift}, and Assumption \ref{assumption:drift-grad} (with $k=0,1$) hold. Then the gradient of the negative asymptotic log-likelihood function $\mathcal{L}$ with respect to $\theta$ is given by\footnote{Here, and in the remainder, we use the convention that the gradient operator $\partial_{\theta}$ adds a contravariant dimension to the tensor field on which it acts. Thus, for example, $\partial_{\theta}\mathcal{L}(\theta)$ is a column vector, taking values in $\mathbb{R}^{p\times 1}$. Meanwhile, $G(\theta,x,\mu):=\partial_{\theta}B(\theta,x,\mu)$ is a matrix, taking values in $\mathbb{R}^{p\times d}$.}
    \begin{equation}
     \partial_{\theta}\mathcal{L}(\theta) = \int_{\mathbb{R}^d} H(\theta,x,\pi_{\theta_0}) \pi_{\theta_0}(\mathrm{d}x) 
        \label{eq:asymptotic-ll-IPS-grad}
    \end{equation}
    where $H:\mathbb{R}^p\times\mathbb{R}^d\times\mathcal{P}(\mathbb{R}^d)\rightarrow\mathbb{R}^{p\times 1}$ is given by
    \begin{align}
    \label{eq:h1}
    H(\theta,x,\mu) &:= G(\theta,x,\mu) (\sigma\sigma^{\top})^{-1}\left(B(\theta,x,\mu) - B(\theta_0,x,\mu)\right),
    \end{align}
    and $G:\mathbb{R}^p\times\mathbb{R}^d\times\mathcal{P}(\mathbb{R}^d)\rightarrow\mathbb{R}^{p\times d}$ is given by
    \begin{align}
        G(\theta,x,\mu) &:=\partial_{\theta}{B(\theta,x,\mu)} = \int_{\mathbb{R}^d} \partial_{\theta} b(\theta,x,y)\mu(\mathrm{d}y). \label{eq:g1}
    \end{align}
\end{proposition}

\begin{proof}
    See Appendix~\ref{app:additional-proofs-grad-log-lik}.
\end{proof}

\subsection{Stochastic Gradient Descent in Continuous Time}
\label{sec:sgdct}
In practice, we cannot simulate the gradient flow in \eqref{eq:continuous-time-grad-descent-1} directly, even after a suitable time-discretisation. In particular, it is not possible to compute $\smash{\partial_{\theta}\mathcal{L}}$, since this gradient is given by an expectation with respect to the unknown invariant measure $\pi_{\theta_0}$. To proceed, we thus seek a stochastic estimate for $\smash{\partial_{\theta}\mathcal{L}}$. Ideally, we would like to be able to compute this estimate in an online fashion, based on the continuous stream of observations.

\subsubsection{A Stochastic Estimate of the Gradient of the Asymptotic Log-Likelihood}
Below, we provide a formal derivation of one such estimate; the use of this estimate will later be justified rigorously. We begin with the observation that, due to ergodicity and the convergence of $\mu_s^{\theta_0}\rightarrow\pi_{\theta_0}$ as $s\rightarrow\infty$, it holds that
\begin{align}
\partial_{\theta} \mathcal{L}(\theta)& \stackrel{L^1}{=}  
\displaystyle \lim_{t\rightarrow\infty}\frac{1}{t} \left[ \int_{0}^t G(\theta,x_s^{i},\mu_s^{\theta_0}) 
(\sigma\sigma^{\top})^{-1} \left(B(\theta,x_s^{i},\mu_s^{\theta_0}) - B(\theta_0,x_s^{i},\mu_s^{\theta_0})\right)\mathrm{d}s\right].
\label{IPS_estimator_4-a} 
\intertext{where $(x_s^{i})_{s\geq 0}$ is a solution of the MVSDE in \eqref{eq:MVSDE}, driven by the same Brownian motion and with the same initial conditions as the solution $(x_s^{i,N})_{s\geq 0}$ of the IPS in \eqref{eq:IPS}. Substituting $B(\theta_0,x_s^{i},\mu_s^{\theta_0})\mathrm{d}s =  \mathrm{d}x_s^{i} -\sigma\mathrm{d}w_s^{i}$ from \eqref{eq:MVSDE}, and noting that the additional martingale term converges to zero both a.s. and in $L^1$ under our conditions, it follows that}
\partial_{\theta}\mathcal{L}(\theta)&\stackrel{L^1}{=} \lim_{t\rightarrow\infty} \frac{1}{t}\left[ \int_0^t G(\theta,x_s^{i},\mu_s^{\theta_0}) (\sigma\sigma^{\top})^{-1} \big[B(\theta,x_s^{i},\mu_s^{\theta_0})\mathrm{d}s - \mathrm{d}x_s^{i}\big]  \right]
\intertext{Finally, due to uniform-in-time propagation of chaos, which guarantees that $\mu_s^{N}\rightarrow\mu_s^{\theta_0}$ and $x_s^{i,N}\rightarrow x_s^{i}$ in $\mathrm{L}^2$ (and hence $\mathrm{L}^1$) as $N\rightarrow\infty$, for all $s\geq 0$, we have} 
\partial_{\theta}\mathcal{L}(\theta)&\stackrel{L^1}{=} \lim_{t\rightarrow\infty} \lim_{N\rightarrow\infty} \frac{1}{t} \left[ \int_0^t G(\theta,x_s^{i,N},\mu_s^{N}) (\sigma\sigma^{\top})^{-1} \big[ B(\theta,x_s^{i,N},\mu_s^{N})\mathrm{d}s - \mathrm{d}x_s^{i,N}\big]  \right] \label{IPS_estimator_6-a}  
\end{align}
This expression suggests that, when the number of particles $N\gg 1$, a natural stochastic estimate for $\partial_{\theta}\mathcal{L}(\theta_t)\mathrm{d}t$ is given by
\begin{align}
\partial_{\theta}\mathcal{L}(\theta_t)\mathrm{d}t &\approx G(\theta_t,x_t^{i,N},\mu_t^{N})(\sigma\sigma^{\top})^{-1}\big[ B(\theta_t,x_t^{i,N},\mu_t^{N})\mathrm{d}t - \mathrm{d}x_t^{i,N}\big] \label{eq:stochastic_estimate-a}
\end{align}

\subsubsection{The Algorithm}
\label{sec:algorithm-1}
Substituting this expression into \eqref{eq:continuous-time-grad-descent-1},  we obtain our first algorithm for optimising the objective function $\mathcal{L}$. Let $\bar{\theta}_{\mathrm{init}}^{i,N}\in\Theta$. Then, for $t\geq 0$, update
\begin{align}
\mathrm{d}\bar{\theta}_t^{i,N} &= -\gamma_t G(\bar{\theta}_t^{i,N},x_t^{i,N},{\mu}_t^N) (\sigma\sigma^{\top})^{-1}\big[ B(\bar{\theta}_t^{i,N},x_t^{i,N},{\mu}_t^N)\mathrm{d}t - \mathrm{d}x_t^{i,N}\big], \label{eq:IPS_update1-a} 
\end{align}
It is instructive to rewrite the update equation in \eqref{eq:IPS_update1-a} in a different form, which emphasises the connection with the objective function $\mathcal{L}$. In particular, after substituting the particle dynamics from \eqref{eq:IPS}, and performing some simple algebraic manipulations, we can rewrite \eqref{eq:IPS_update1-a} as
\begin{align}
 \mathrm{d}\bar{\theta}_t^{i,N} 
&= -\gamma_t G(\bar{\theta}_t^{i,N},x_t^{i,N},\mu_t^N) (\sigma\sigma^{\top})^{-1} \big(\big[B(\bar{\theta}_t^{i,N},x_t^{i,N},\mu_t^N) - B(\theta_0,x_t^{i,N},\mu_t^N)\big]\mathrm{d}t - \sigma \mathrm{d}w_t^{i,N}\big) \\[1.5mm]
 &= \underbrace{-\gamma_tH(\bar{\theta}_t^{i,N},x_t^{i,N},\mu_t^N)\mathrm{d}t}_{\text{noisy descent term}} + \underbrace{\gamma_tG(\bar{\theta}_t^{i,N},x_t^{i,N},\mu_t^N) \sigma^{-\top} \mathrm{d}w_t^{i,N}}_{\text{noise term}}, \label{theta_ideal_2-a} \\
&= -\underbrace{\gamma_t\partial_{\theta}\mathcal{L}(\bar{\theta}_t^{i,N})\mathrm{d}t}_{\text{true descent term}} - \underbrace{\gamma_t (H(\bar{\theta}_t^{i,N},x_t^{i,N},\mu_t^N)-\partial_{\theta}\mathcal{L}(\bar{\theta}_t^{i,N}))\mathrm{d}t}_{\text{fluctuations term}} + \underbrace{\gamma_tG(\bar{\theta}_t^{i,N},x_t^{i,N},\mu_t^N)\sigma^{-\top} \mathrm{d}w_t^{i,N}}_{\text{noise term}}. \label{theta_ideal_2-a-i-v0}
\end{align}

\begin{remark}
The estimator $\smash{(\bar{\theta}_t^{i,N})_{t\geq 0}}$ defined in \eqref{eq:IPS_update1-a} coincides with one of the estimators proposed in \cite{sharrock2023online}. The obvious disadvantage of this estimator is that it requires observation of the full system $\smash{(x_t^{i,N})_{t\geq 0}^{i\in[N]}}$ of interacting particles, via the empirical law $\smash{\mu_t^N = \frac{1}{N}\sum_{j=1}^N \delta_{x_t^{j,N}}}$. In cases where the number of particles $N$ is very large, $\smash{(\bar{\theta}_t^{i,N})_{t\geq 0}}$ may therefore be prohibitively expensive to implement.
\end{remark}

\subsection{Stochastic Gradient Descent in Continuous Time: A New Approach}

We now seek an alternative estimator which does not require observation of the entire set of interacting particles.

\subsubsection{A New Expression for the Gradients of the Asymptotic Log-Likelihood}
In order to obtain such an estimator, we will first obtain an alternative form for the asymptotic log-likelihood and its gradients. We begin by expanding the integrand in our existing expressions for the asymptotic negative log-likelihood function in \eqref{eq:obj-func-1}, which yields
    \begin{align}
    \mathcal{L}(\theta)
    &=\int_{\mathbb{R}^d} \tfrac{1}{2}\|B(\theta,x,\pi_{\theta_0})-B(\theta_0,x,\pi_{\theta_0})\|_{\sigma\sigma^{\top}}^2 \pi_{\theta_0}(\mathrm{d}x) \\[-.5mm]
    &= \int_{(\mathbb{R}^d)^3}\tfrac{1}{2} \big\langle b(\theta,x,y) - B(\theta_0,x,\pi_{\theta_0}), b(\theta,x,z) - B(\theta_0,x,\pi_{\theta_0}) \big\rangle_{\sigma\sigma^{\top}} \pi_{\theta_0}^{\otimes 3}(\mathrm{d}x,\mathrm{d}y,\mathrm{d}z)  \\[-.5mm]
    &:= \int_{(\mathbb{R}^d)^3} \ell(\theta,x,y,z,\pi_{\theta_0}) \pi_{\theta_0}^{\otimes 3}(\mathrm{d}x,\mathrm{d}y,\mathrm{d}z). \label{eq:ell-def}
    \end{align}
Similarly, expanding the integrand in our existing expression for the gradient of the asymptotic negative log-likelihood function in \eqref{eq:asymptotic-ll-IPS-grad}, it is possible to show that
\begin{align}
\partial_{\theta} \mathcal{L}(\theta)
&= \int_{\mathbb{R}^d} \big[ G(\theta,x,\pi_{\theta_0}) \big] (\sigma\sigma^{\top})^{-1} \big[B(\theta,x,\pi_{\theta_0}) - B(\theta_0,x,\pi_{\theta_0})\big] \pi_{\theta_0}(\mathrm{d}x)  \\
&=\int_{(\mathbb{R}^d)^3} \big[g(\theta,x,y)\big]  (\sigma\sigma^{\top})^{-1} \big[b(\theta,x,z) - B(\theta_0,x,\pi_{\theta_0})\big] \pi_{\theta_0}^{\otimes 3}(\mathrm{d}x,\mathrm{d}y,\mathrm{d}z) \label{eq:l1-explicit-v0} \\
&:=\int_{(\mathbb{R}^d)^3} h(\theta,x,y,z,\pi_{\theta_0}) \pi_{\theta_0}^{\otimes 3}(\mathrm{d}x,\mathrm{d}y,\mathrm{d}z), \label{eq:l1-explicit}  
\end{align}
where, in the second line, we have defined $g(\theta,x,y) := \partial_{\theta}b(\theta,x,y)$. This alternative representation of the gradient of the asymptotic negative log-likelihood will provide the starting point for our new, more efficient stochastic estimate.

\subsubsection{A New Stochastic Estimate of the Gradient of the Asymptotic Log-Likelihood} 
Once again, we present a formal derivation, deferring a rigorous theoretical treatment to the sequel. Similar to before, due to ergodicity and the convergence of $\mu_s^{\theta_0}\rightarrow \pi_{\theta_0}$ as $s\rightarrow\infty$, we have that
\begin{align}
\partial_{\theta} \mathcal{L}(\theta)
&\stackrel{L^1}{=}\lim_{t\rightarrow\infty}\frac{1}{t}\Big[\int_0^t \Big[g(\theta,x_s^{i},x_s^{j})\Big](\sigma\sigma^{\top})^{-1} \Big[b(\theta,x_s^{i},x_s^{k}) - B(\theta_0,x_s^{i}, \mu_s^{\theta_0})\Big]\mathrm{d}s\Big], \label{IPS_estimator2_4-a}
\end{align}
where $(x_s^{i})_{s\geq 0}$, $(x_s^{j})_{s\geq 0}$, $(x_s^{k})_{s\geq 0}$ are three independent solutions of the MVSDE, driven by Brownian motions $(w_s^{i})_{s\geq 0}$, $(w_s^{j})_{s\geq 0}$, $(w_s^{k})_{s\geq 0}$. Substituting $B(\theta_0,x_s^{i},\mu_s^{\theta_0})\mathrm{d}s = \mathrm{d}x_s^{i} - \sigma\mathrm{d}w_s^{i}$, and using the fact that the additional martingale term converges to zero, we have that
\begin{align}
\partial_{\theta} \mathcal{L}(\theta)
&\stackrel{L^1}{=}\lim_{t\rightarrow\infty}\frac{1}{t}\Big[\int_0^t \big[g(\theta,x_s^i,x_s^j)\big](\sigma\sigma^{\top})^{-1}  \big[b(\theta,x_s^i,x_s^k)\mathrm{d}s - \mathrm{d}x_s^i \big] \Big]  \label{IPS_estimator2_5-a}  
\end{align}
Finally, under the assumption of uniform-in-time propagation of chaos, it follows from the previous display that
\begin{align}
\partial_{\theta} \mathcal{L}(\theta)
&\stackrel{L^1}{=}\lim_{t\rightarrow\infty}\lim_{N\rightarrow\infty}\frac{1}{t}\Big[\int_0^t \big[g(\theta,x_s^{i,N},x_s^{j,N})\big](\sigma\sigma^{\top})^{-1} \big[b(\theta,x_s^{i,N},x_s^{k,N})\mathrm{d}s -\mathrm{d}x_s^{i,N}\big] \Big]\label{IPS_estimator2_7-a} 
\end{align}
 where $(x_s^{i,N})_{s\geq 0}$, $(x_s^{j,N})_{s\geq 0}$, and $(x_s^{k,N})_{s\geq 0}$ are the trajectories of three distinct particles from the observed IPS. This expression suggests that, for $N\gg 1$, an alternative stochastic estimate for $\partial_{\theta}\mathcal{L}(\theta_t)\mathrm{d}t$ is given by
 \begin{align}
\partial_{\theta}{{\mathcal{L}}}(\theta_t)\mathrm{d}t&\approx\big[g(\theta_t,x_t^{i,N}, x_t^{j,N})\big]  (\sigma\sigma^{\top})^{-1} \big[b(\theta_t,x_t^{i,N}, x_t^{k,N})\mathrm{d}t-\mathrm{d}x_t^{i,N}\big].
\end{align}

\subsubsection{A New Algorithm}
\label{sec:algorithm-2}
Substituting this expression into \eqref{eq:continuous-time-grad-descent-1},  we obtain an alternative algorithm for optimising the objective function $\mathcal{L}$. Let $\smash{\theta_{\mathrm{init}}^{i,j,k,N}\in\Theta}$. Then, for $t\geq 0$, evolve
\begin{align}
\mathrm{d}\theta_t^{i,j,k,N} &=  -\gamma_t \big[g(\theta_t^{i,j,k,N},x_t^{i,N}, x_t^{j,N})\big]  (\sigma\sigma^{\top})^{-1} \big[b(\theta_t^{i,j,k,N},x_t^{i,N},x_t^{k,N})\mathrm{d}t -\mathrm{d}x_t^{i,N}\big]. \label{eq:IPS_update2-a-v0}
\end{align}
Once again, it is instructive to rewrite this algorithm in a slightly different form. In this case, following similar manipulations to before, we have that
\begin{align}
 \mathrm{d}\theta_t^{i,j,k,N} 
&= -\underbrace{\gamma_t\partial_{\theta}\mathcal{L}(\theta_t^{i,j,k,N})\mathrm{d}t}_{\text{true descent term}} - \underbrace{\gamma_t (h(\theta_t^{i,j,k,N},x_t^{i,N},x_t^{j,N},x_t^{k,N},\mu_t^N) - \partial_{\theta}\mathcal{L}(\theta_t^{i,j,k,N}))\mathrm{d}t}_{\text{fluctuations term}} \label{theta_ideal_2-a-i-ii} \\
&~~~~~+ \underbrace{\gamma_t g(\theta_t^{i,j,k,N},x_t^{i,N},x_t^{j,N})\sigma^{-\top} \mathrm{d}w_t^{i,N}}_{\text{noise term}}.  \nonumber
\end{align}

\begin{remark}
    In some sense, one can view the estimator $\smash{(\bar{\theta}_t^{i,N})_{t\geq 0}}$ from Section~\ref{sec:algorithm-1}, as defined in \eqref{eq:IPS_update1-a}, as an ``averaged'' version of our new estimator $\smash{(\theta_t^{i,j,k,N})_{t\geq 0}}$, as defined in \eqref{eq:IPS_update2-a-v0}. Indeed, comparing the two update equations, wherever a pair or a triplet of particles appear in \eqref{eq:IPS_update2-a-v0}, an average over all of the particles appears in \eqref{eq:IPS_update1-a}. 
\end{remark}

\begin{remark}
The estimator $(\theta_t^{i,j,k,N})_{t\geq 0}$, as defined by \eqref{eq:IPS_update2-a-v0}, only depends on observations of three distinct particles $\smash{(x_t^{i,N})_{t\geq 0}}$, $\smash{(x_t^{j,N})_{t\geq 0}}$, and $\smash{(x_t^{k,N})_{t\geq 0}}$, regardless of the total number of particles $N$ in the data-generating IPS. Thus, in the typical case where $N\gg1$, it is much less costly to implement than the estimator $\smash{(\bar{\theta}_t^{i,N})_{t\geq 0}}$ defined in \eqref{eq:IPS_update1-a}, which requires observation of all particles. Nonetheless, we will show that the two estimators share many of the same theoretical properties as $N\rightarrow\infty$.
\end{remark}

\subsection{Extensions}
\label{sec:extensions}
At the price of an increased computational cost, it is possible to define variants of both of our estimators which enjoy improved convergence guarantees. Let $M\in[N]$. Define $\Pi = \{i_1,\dots,i_M\}\subseteq[N]$ as an ordered subset of the particles, and $\mathcal{C}(\Pi)\subseteq[N]^{3}$ as the set of cyclic triplets in $\Pi$, so that $M=|\Pi| = |\mathcal{C}(\Pi)|$.\footnote{To be precise, $\mathcal C(\Pi) := \{(i_\ell,i_{\ell+1},i_{\ell+2}) : \ell=1,\dots,M\}$, where indices are taken cyclically, so that $i_{M+1}=i_1$ and $i_{M+2}=i_2$.}\textsuperscript{,}\footnote{In the cases where $M=1$ with $\Pi = \{i_{1}\}$, or $M=2$ with $\Pi=\{i_1,i_2\}$, it is clear that the set of cyclic triplets $\mathcal{C}(\Pi)$ is not well defined. In these cases, we first define an extended set of indices $\Pi^{*} = \{i_1,i_2,i_3\}$ by choosing the required number of fixed auxiliary indices from $[N]\cap \Pi^c$, with $|\Pi^{*}| = M_{*} = \max\{3,M\}$. This means, in particular, that the set of cyclic triplets $\mathcal{C}(\Pi^{*})$ \emph{is} well defined. We then \emph{redefine} $\mathcal{C}(\Pi)$ as the subset of cyclic triplets from $\mathcal{C}(\Pi^{*})$ whose first index lies in the original set $\Pi$. This ensures that $|\mathcal C(\Pi)|=M$ and that each $(i,j,k)\in\mathcal C(\Pi)$ consists of three distinct indices.} We can then define two new estimators according to
\begin{align}
\mathrm{d}\bar{\theta}^{N,M}_t &= -\gamma_t \frac{1}{M}\sum_{i\in\Pi} 
\left[\big[ G(\bar{\theta}^{N,M}_t,x_t^{i,N},\mu_t^N)\big] (\sigma\sigma^{\top})^{-1} 
\big[ B(\bar{\theta}^{N,M}_t,x_t^{i,N},\mu_t^N)\mathrm{d}t - \mathrm{d}x_t^{i,N}\big]\right], 
\label{eq:IPS_update1-a-explicit-with-M} \\
\mathrm{d}\theta^{N,M}_t &=  -\gamma_t \frac{1}{M}\sum_{(i,j,k)\in\mathcal{C}(\Pi)}
\left[ \big[g(\theta^{N,M}_t,x_t^{i,N}, x_t^{j,N})\big]  (\sigma\sigma^{\top})^{-1} 
\big[b(\theta^{N,M}_t,x_t^{i,N},x_t^{k,N})\mathrm{d}t -\mathrm{d}x_t^{i,N}\big]\right].
\label{eq:IPS_update2-a-with-M}
\end{align}
Thus, in particular, we can view $\smash{(\bar{\theta}^{N,M}_t)_{t\geq 0}}$ and $\smash{({\theta}^{N,M}_t)_{t\geq 0}}$ as the estimators whose drifts are obtained by averaging the drifts defining the estimators in Section~\ref{sec:algorithm-1} and Section~\ref{sec:algorithm-2} over multiple \emph{primary trajectories} $(x_t^{i,N})^{i\in\Pi}_{t\geq 0}$. Naturally, the first of these estimators still requires us to observe the trajectories of every particle from the IPS. Meanwhile, the second estimator now requires us to observe $M_{*} = \max\{3,M\}$ trajectories.

\begin{remark}
In the case where the number of primary trajectories is equal to the total number of particles (i.e., $M=N$), the estimator $\smash{(\bar{\theta}_t^{N,M})_{t\geq 0}}$ coincides with the second estimator proposed in \cite{sharrock2023online}. This, in some sense, is the estimator which uses the maximal possible amount of information available from observations of the IPS. In particular, in comparison to $\smash{(\bar{\theta}_t^{i,N}})_{t\geq 0}$, which only uses the information from particles other than the $i^{\mathrm{th}}$ particle via the empirical measure $(\mu_t^N)_{t\geq 0}$, $\smash{(\bar{\theta}_t^{N,N})_{t\geq 0}}$ explicitly uses the sample paths $(x_t^{i,N})_{t\geq 0}^{i\in[N]}$ of all of the particles. Naturally, this also means that it is the most computationally costly estimator to implement. 
\end{remark}

\noindent We will later show that the convergence rates for $\smash{(\bar{\theta}^{N,M}_t})_{t\geq 0}$ and $\smash{({\theta}^{N,M}_t})_{t\geq 0}$ improve on the convergence rates for $\smash{(\bar{\theta}^{i,N}_t})_{t\geq 0}$ and $\smash{({\theta}^{i,j,k,N}_t})_{t\geq 0}$ by a factor of $M$ in one of the constants (see Theorem~\ref{theorem:main-theorem-1-2} vs Corollary~\ref{corollary:main-theorem-1-2-finite-n-with-m}). In this sense, the averaging mechanism provides a rather quantifiable way in which to balance the computational cost and the finite-time accuracy of the online estimation procedure.

\section{Theoretical Results}
\label{sec:theory}

In this section, we present our main results regarding the convergence of the estimators introduced in Sections~\ref{sec:algorithm-1}, \ref{sec:algorithm-2}, and \ref{sec:extensions}.

\subsection{Preliminary Results}
\label{sec:theory-add-notation}

We will first require some additional notation, and some auxiliary results. We begin by defining two finite-particle approximations to our original objective function (see Section~\ref{sec:methodology}). In particular, we set
\noeqref{eq:partial-lik-2}
\begin{align}
    \mathcal{L}^{i,N}(\theta) &:=\int_{(\mathbb{R}^d)^N}L(\theta,x^{i,N},\mu^N)\pi_{\theta_0}^N(\mathrm{d}\boldsymbol{x}^N) := \int_{(\mathbb{R}^d)^N} L^{i,N}(\theta,\boldsymbol{x}^N)\pi_{\theta_0}^{N}(\mathrm{d}\boldsymbol{x}^N)\label{eq:partial-lik-1} \\
    {\mathcal{L}}^{i,j,k,N}(\theta) &:=\int_{(\mathbb{R}^d)^N}\ell(\theta,\boldsymbol{x}^{i,j,k,N},\mu^N)\pi_{\theta_0}^{N}(\mathrm{d}\boldsymbol{x}^{N}) := \int_{(\mathbb{R}^d)^N} \ell^{i,j,k,N}(\theta,\boldsymbol{x}^N)\pi_{\theta_0}^{N}(\mathrm{d}\boldsymbol{x}^N), \label{eq:partial-lik-2} 
\end{align}
where $\boldsymbol{x}^{i,j,k,N} = (x^{i,N},x^{j,N},x^{k,N})$. These functions correspond to the time-averages of two negative \emph{pseudo log-likelihood} or \emph{contrast} functions (see Proposition~\ref{prop:ips-likelihood-t-limit-recall}, Appendix~\ref{app:additional-results-prop-inf-n}). We can also characterise the gradients of these functions. 

\begin{proposition}
\label{prop:asymptotic-partial-log-lik-grad}
Suppose that Assumption \ref{assumption:moments}, Assumption \ref{assumption:drift}, and Assumption \ref{assumption:drift-grad} (with $k=0,1$) hold. Then the gradients of the negative asymptotic pseudo log-likelihood functions $\mathcal{L}^{i,N}$ and $\mathcal{L}^{i,j,k,N}$ with respect to $\theta$ are given by
\begin{align}
    \partial_{\theta}\mathcal{L}^{i,N}(\theta) &= \int_{(\mathbb{R}^d)^N}H(\theta,x^{i,N},\mu^N)\pi_{\theta_0}^N(\mathrm{d}\boldsymbol{x}^N) := \int_{(\mathbb{R}^d)^N} H^{i,N}(\theta,\boldsymbol{x}^N)\pi_{\theta_0}^{N}(\mathrm{d}\boldsymbol{x}^N) \label{eq:incomplete-log-lik-1-grad} \\[-1mm]
    \partial_{\theta}{\mathcal{L}}^{i,j,k,N}(\theta) &= \int_{(\mathbb{R}^d)^N}h(\theta,\boldsymbol{x}^{i,j,k,N},\mu^N)\pi_{\theta_0}^{N}(\mathrm{d}\boldsymbol{x}^N) := \int_{(\mathbb{R}^d)^N} h^{i,j,k,N}(\theta,\boldsymbol{x}^N)\pi_{\theta_0}^{N}(\mathrm{d}\boldsymbol{x}^N). \label{eq:incomplete-log-lik-2-grad}
\end{align}
\end{proposition}

\begin{proof}
    See Appendix \ref{app:additional-proofs-theory-add-notation}
\end{proof}

These functions can be viewed as finite-particle approximations to the gradient of our original objective function (see Proposition~\ref{prop:ips-likelihood-n-t-limit-grad}, Section~\ref{sec:grad-log-lik}).  This notion is made precise in the following result, which establishes that both finite-particle gradients converge (uniformly) to the mean-field gradient as $N\rightarrow\infty$, and characterises the rate at which this convergence takes place. 

\begin{proposition}
\label{prop:inf-n-convergence-1}
    Suppose that Assumption \ref{assumption:moments}, Assumption \ref{assumption:drift}, and Assumption \ref{assumption:drift-grad} (with $k=0,1$) hold. Then, for all $N\in\mathbb{N}$, for all distinct $i,j,k\in[N]$, 
    there exist finite constants $K_1,K_1^{\dagger},K_2,K_2^{\dagger}<\infty$ 
    such that
\begin{align}
    \sup_{\theta\in\Theta}\big\| \partial_{\theta}\mathcal{L}^{i,N}(\theta) - \partial_{\theta} \mathcal{L}(\theta)\big\| 
    &\leq K_1\rho(N) + \frac{K_2}{N^{\frac{1}{2(1+\alpha)}}}\label{eq:l-i-convergence-rate} \\
    \sup_{\theta\in\Theta}\big\| \partial_{\theta}\mathcal{L}^{i,j,k,N}(\theta) - \partial_{\theta} \mathcal{L}(\theta)\big\| 
    &\leq K_1^{\dagger}\rho(N) + \frac{K_2^{\dagger}}{N^{\frac{1}{2(1+\alpha)}}}, \label{eq:l-i-j-k-convergence-rate} 
\end{align}
where the function $\rho:\mathbb{N}\rightarrow\mathbb{R}_{+}$ is defined according to 
    \begin{equation}
    \rho(N)= \left\{
    \begin{array}{lll} 
    N^{-\frac{1}{4}} & \text{if} & d<4 \\
    N^{-\frac{1}{4}}[\log(1+N)]^{\frac{1}{2}} & \text{if} & d=4 \\
    N^{-\frac{1}{d}} & \text{if} & d > 4.
    \end{array}
    \right. \label{eq:rho}
    \end{equation}
\end{proposition}

\begin{proof}
    See Appendix \ref{app:additional-proofs-theory-add-notation}.
\end{proof}

 \begin{remark}
The two contributions to the rates in \eqref{eq:l-i-convergence-rate} - \eqref{eq:l-i-j-k-convergence-rate} have distinct origins. The term $\smash{\rho(N)}$ corresponds to the standard $\mathsf{W}_2$ empirical measure rate \citep[][Theorem 1]{fournier2015rate}, while the term $\smash{N^{-\frac{1}{2(1+\alpha)}}}$ is inherited from the propagation-of-chaos rate in \citep[][Theorem 3.1]{cattiaux2008probabilistic} (see also Theorem~\ref{thm:poc}, Appendix~\ref{app:existing-results}). Consequently, if one strengthens the assumptions in a way that improves either (i) the empirical-measure concentration rate or (ii) the propagation-of-chaos rate (see Remark~\ref{rem:poc}, Section~\ref{sec:model-assumptions}), then the bounds in \eqref{eq:l-i-convergence-rate} - \eqref{eq:l-i-j-k-convergence-rate} can be improved accordingly, yielding faster overall convergence rates for the finite-particle gradients. 
\end{remark}

\subsection{Main Results}
\label{sec:main-results-first-estimator}
We are now ready to state our main results. In all cases, we will require the following standard assumption on the learning rate. This is the continuous-time analogue of the standard step-size condition used in the convergence analysis of stochastic approximation algorithms in discrete time \citep[e.g.,][]{robbins1951stochastic,sirignano2017stochastic}.

\begin{assumption}
\label{assumption:learning-rate}
    The learning rate $\gamma_t:\mathbb{R}_{+}\rightarrow\mathbb{R}_{+}$ is a positive, non-increasing function which satisfies $\int_0^{\infty}\gamma_t\mathrm{d}t = \infty$, $\int_0^{\infty}\gamma_t^2\mathrm{d}t<\infty$, $\int_0^{\infty}|\dot{\gamma}_t|\mathrm{d}t<\infty$, and $\lim_{t\rightarrow\infty}\gamma_t t^{p}=0$ for some $p>0$.
\end{assumption}

\subsubsection{Convergence} 
\label{sec:main-results-convergence}
We begin by characterising the asymptotic behaviour of the estimators $\smash{(\bar{\theta}^{i,N}_t)_{t\geq 0}}$ and $\smash{(\theta^{i,j,k,N}_t)_{t\geq 0}}$ in the limit as the time horizon $t\rightarrow\infty$, given a fixed and finite number of particles. In particular, the following proposition establishes the convergence of the two estimators to the stationary points of the two pseudo negative log-likelihood functions $\mathcal{L}^{i,N}$ and $\mathcal{L}^{i,j,k,N}$ defined above as $t\rightarrow\infty$. 

\begin{proposition}
    \label{prop:inf-t-convergence-1}
    Suppose that Assumption \ref{assumption:moments}, Assumption~\ref{assumption:drift}, Assumption~\ref{assumption:drift-grad}, and Assumption~\ref{assumption:learning-rate} hold. Let $N\in\mathbb{N}$, and let $i,j,k\in[N]$ be distinct. Suppose that $\mathbb{P}(\bar{\theta}_t^{i,N}\in\Theta~\forall t\geq 0)=\mathbb{P}(\theta_t^{i,j,k,N}\in\Theta~\forall t\geq 0)=1$. Then it holds almost surely that
    \begin{align}
    \lim_{t\rightarrow\infty}\|\partial_{\theta}\mathcal{L}^{i,N}(\bar{\theta}_t^{i,N}) \| &= 0 \\ \lim_{t\rightarrow\infty}\|\partial_{\theta}\mathcal{L}^{i,j,k,N}(\theta_t^{i,j,k,N}) \| &= 0.
    \end{align}
\end{proposition}

\begin{proof}
    See Appendix \ref{app:additional-proofs-main-results-convergence}.
\end{proof}

\begin{remark}
    The assumption that the iterates remain in the admissible set for all times is automatically satisfied in the unconstrained case where $\Theta=\mathbb{R}^p$. On the other hand, when $\Theta\subset \mathbb{R}^p$, this assumption is not automatic. In this case, our conclusions still hold \emph{conditional} on $\smash{\{\bar{\theta}_t^{i,N}\in\Theta~\forall t\geq 0\}}$ and $\smash{\{{\theta}_t^{i,j,k,N}\in\Theta~\forall t\geq 0\}}$, respectively. Conversely, conditional on the events $\smash{\{\bar{\theta}_t^{i,N}\in\Theta~\forall t\geq 0\}}^{\mathsf{c}}$ and $\smash{\{{\theta}_t^{i,j,k,N}\in\Theta~\forall t\geq 0\}}^{\mathsf{c}}$, we have that $\lim_{t\rightarrow\infty}\bar{\theta}_t^{i,N}=\partial\Theta$ and $\lim_{t\to\infty}{\theta}_t^{i,j,k,N}=\partial\Theta$. This follows from the definition of our dynamics: once the iterates hit the boundary, they remain there for all times (see Remark~\ref{remark:constrained}).
\end{remark}

We next establish that the estimators $\smash{(\bar{\theta}^{i,N}_t)_{t\geq 0}}$ and $\smash{(\theta^{i,j,k,N}_t)_{t\geq 0}}$ both converge to the stationary points of the asymptotic negative log-likelihood function $\mathcal{L}$ as first the time-horizon $t\rightarrow\infty$ and then the number of particles $N\rightarrow\infty$.

\begin{theorem}
\label{theorem:main-result-1}
    Suppose that Assumption \ref{assumption:moments}, Assumption~\ref{assumption:drift}, Assumption~\ref{assumption:drift-grad}, and Assumption~\ref{assumption:learning-rate} hold. Let $N\in\mathbb{N}$ and let $i,j,k\in[N]$ be distinct. Suppose that $\mathbb{P}(\bar{\theta}_t^{i,N}\in\Theta~\forall t\geq 0)=\mathbb{P}(\theta_t^{i,j,k,N}\in\Theta~\forall t\geq 0)=1$ for all $N\in\mathbb{N}$. Then it holds almost surely that
    \begin{align}
    \label{eq:final-limit-1}
\lim_{N\rightarrow\infty}\limsup_{t\rightarrow\infty}\|\partial_{\theta}\mathcal{L}(\bar{\theta}_t^{i,N})\|&=0 \\ \lim_{N\rightarrow\infty}\limsup_{t\rightarrow\infty}\|\partial_{\theta}\mathcal{L}(\theta_t^{i,j,k,N})\|&=0. \label{eq:final-limit-2}
    \end{align}
\end{theorem}

\begin{proof}
    See Appendix \ref{app:additional-proofs-main-results-convergence}.
\end{proof}

\subsubsection{Convergence Rates}
\label{sec:main-results-convergence-rates}
Under some additional assumptions, we can also obtain an $L^2$ convergence rate. We will first require some additional conditions on the learning rate. These conditions, which resemble those introduced in \cite{sirignano2020stochastic}, ensure that fluctuation terms which appear in the ODE governing the $L^2$ distance to the optimiser vanish sufficiently quickly as $t\rightarrow\infty$. They are satisfied, for example, by the standard choice $\gamma_t = \gamma_0 (1+t)^{-\beta}$, given $\gamma_0>0$ and $\beta\in(1/2,1)$ 

\begin{assumption}
    \label{assumption:learning-rate-v2}
Let $\Phi_{s,t} = \exp(-2\zeta\int_{s}^{t} \gamma_u\mathrm{d}u)$, where $\zeta\in\{\eta,\eta_{i,N},\eta_{i,j,k,N}\}$ is equal to the strong convexity constant in Assumption \ref{assumption:convexity-finite-n} or Assumption~\ref{assumption:convexity-infinite-n}, depending on the result at hand. The learning rate $\gamma_t:\mathbb{R}_{+}\rightarrow\mathbb{R}_{+}$ satisfies
\begin{align}
    \int_{0}^{t}\gamma^2_s\Phi_{s,t}\mathrm{d}s = O(\gamma_t), \quad \int_{0}^{t}|\dot{\gamma}_s|\Phi_{s,t}\mathrm{d}s=O(\gamma_t),  \quad &\int_0^{t}\gamma_s\Phi_{s,t}\mathrm{d}s = O(1), \quad \Phi_{0,t} = O(\gamma_t).
\end{align}
as $t\rightarrow\infty$. In addition, writing $a_s:\mathbb{R}_{+}\rightarrow\mathbb{R}_{+}$ for the function which characterises the rate of convergence to the invariant distribution (see Theorem~\ref{thm:invariant-distribution-2}), the learning rate satisfies
\begin{alignat}{3}
    &\Phi_{0,t}^{\frac{1}{2}} = o(\gamma_t^{\frac{1}{2}}), \quad &&\int_0^t  \gamma_s \Phi_{s,t}^{\frac{1}{2}}  \mathrm{d}s = O(1), \quad &&\int_0^t  \gamma_s^2 \Phi_{s,t}^{\frac{1}{2}}  \mathrm{d}s = o(\gamma_t^{\frac{1}{2}}),  \\
    &\int_0^t \gamma_s^{\frac{5}{2}}\Phi_{s,t}\mathrm{d}s = o(\gamma_t), \quad &&\int_0^t \gamma_s\Phi_{s,t}^{\frac{1}{2}}a_s^{\frac{1}{2}}\mathrm{d}s = o(\gamma_t^{\frac{1}{2}}), \quad  &&\int_0^t  \Phi_{s,t}\,\gamma_s^2a_s^{\frac{1}{2}(1-\varepsilon)}  \mathrm{d}s = o(\gamma_t), ~\varepsilon\in(0,1)
\end{alignat}
\end{assumption}

In addition to this condition on the learning rate, we will now need to assume that the relevant objective function is strongly convex. We provide two alternative assumptions. The first, which relates to the finite-particle functions $\mathcal{L}^{i,N}$ and $\mathcal{L}^{i,j,k,N}$, is relevant to the case where the time horizon $t\rightarrow\infty$, with the number of particles $N$ assumed fixed and finite. The second, which relates to the limiting function $\mathcal{L}$, is relevant to the case where both the time horizon $t\rightarrow\infty$ and the number of particles $N\rightarrow\infty$.

\begin{assumption}
\label{assumption:convexity-finite-n}
    Let $N\in\mathbb{N}$, and let $i,j,k\in[N]$ be distinct. The functions $\mathcal{L}^{i,N}$ and $\mathcal{L}^{i,j,k,N}$ are strongly convex with constants $\eta_{i,N}>0$ and $\eta_{i,j,k,N}>0$. 
\end{assumption}

\begin{assumption}
\label{assumption:convexity-infinite-n}
    The function $\mathcal{L}$ is strongly convex with constant $\eta>0$.
\end{assumption}

Once again, we begin by providing a result which characterises the asymptotic convergence rate of the estimators $(\bar{\theta}_t^{i,N})_{t\geq 0}$ and $(\theta_t^{i,j,k,N})_{t\geq 0}$ in the limit as the time horizon $t\rightarrow\infty$, given a fixed and finite number of particles. In particular, the following theorem establishes an $L^2$ convergence rate for $(\bar{\theta}_t^{i,N})_{t\geq 0}$ and $(\theta_t^{i,j,k,N})_{t\geq 0}$, assuming strong convexity of the asymptotic (in time), finite-particle, incomplete-data negative log-likelihoods.

\begin{theorem}
\label{theorem:main-theorem-1-2-finite-n}
    Suppose that Assumption \ref{assumption:moments}, Assumption~\ref{assumption:drift}, Assumption~\ref{assumption:drift-grad}, Assumption~\ref{assumption:learning-rate}, Assumption~\ref{assumption:learning-rate-v2}, Assumption~\ref{assumption:convexity-finite-n} hold. Let $N\in\mathbb{N}$, and let $i,j,k\in[N]$ be distinct. Suppose that $\mathbb{P}(\bar{\theta}_t^{i,N}\in\Theta~\forall t\geq 0)=\mathbb{P}(\theta_t^{i,j,k,N}\in\Theta~\forall t\geq 0)=1$. Suppose also that $\sup_{t\geq 0}\mathbb{E}[\|\bar{\theta}_t^{i,N}\|^l]<\infty$ and $\sup_{t\geq 0}\mathbb{E}[\|{\theta}_t^{i,j,k,N}\|^l] <\infty$ for all $l\in\mathbb{N}$. Finally, suppose that $\Theta$ is convex. Then, for sufficiently large $t\in\mathbb{R}_{+}$, there exist positive constants $K_{1},K_2>0$, $K_{1}^{\dagger},K_2^{\dagger}>0$, such that 
    \begin{align}
        &\mathbb{E}\left[\|\bar{\theta}_t^{i,N} - \theta_0^{i,N}\|^2\right]  \leq (K_{1}+K_{2})\gamma_t  \label{eq:average-l2-rate-1} \\
        &\mathbb{E}\left[\|\theta_t^{i,j,k,N}-\theta_0^{i,j,k,N}\|^2\right] \leq (K_{1}^{\dagger}+K_{2}^{\dagger})\gamma_t  \label{eq:non-average-l2-rate-1}
    \end{align}
    where $\theta_0^{i,N}$ and $\theta_0^{i,j,k,N}$ denote the unique minimisers of $\mathcal{L}^{i,N}$ and $\mathcal{L}^{i,j,k,N}$, respectively. Moreover, writing $\theta_0$ for the true parameter, there exist constants $K_3^{\dagger},K_4^{\dagger}>0$ such that
    \begin{align}
        &\mathbb{E}\left[\|\bar{\theta}_t^{i,N} - \theta_0\|^2\right]  \leq (K_{1}+K_{2})\gamma_t \label{eq:average-l2-rate-2} \\
        &\mathbb{E}\left[\|\theta_t^{i,j,k,N}-\theta_0\|^2\right] \leq 2(K_{1}^{\dagger}+K_{2}^{\dagger})\gamma_t + 2 \frac{1}{(\eta_{i,j,k,N})^2} \left[K_3^{\dagger}\rho^2(N) + \frac{K_4^{\dagger}}{N^{\frac{1}{1+\alpha}}}\right]. \label{eq:non-average-l2-rate-2}
    \end{align}
\end{theorem}

\begin{proof}
    See Appendix \ref{app:additional-proofs-main-results-convergence-rates}.
\end{proof}  

\begin{remark}
The uniform bounded-moments assumption on $\smash{(\bar\theta_t^{i,N})_{t\ge0}}$ and $\smash{(\theta_t^{i,j,k,N})_{t\ge0}}$ is automatic whenever $\Theta$ is compact, since then $\smash{\sup_{t\ge0}\|\bar\theta_t^{i,N}\|}$ and $\smash{\sup_{t\ge0}\|\theta_t^{i,j,k,N}\|}$ are almost surely bounded. In the unconstrained case $\Theta=\mathbb{R}^p$, such bounds can be established under standard \emph{dissipativity} conditions, via a comparison theorem \citep[e.g.,][]{ikeda1977comparison}. See, e.g., \cite{sirignano2020stochastic,sharrock2023online} for some specific examples.
\end{remark}

\begin{remark}
    The bounds in \eqref{eq:average-l2-rate-1} - \eqref{eq:non-average-l2-rate-1} show that, for each fixed $N\in\mathbb{N}$, the estimators converge to the minimisers $\theta_0^{i,N}$ and $\theta_0^{i,j,k,N}$ of the finite particles objectives $\mathcal{L}^{i,N}$ and $\mathcal{L}^{i,j,k,N}$, respectively, at a rate determined by the learning rate $\gamma_t$. Meanwhile, the bounds in \eqref{eq:average-l2-rate-2} - \eqref{eq:non-average-l2-rate-2} show that the convergence of the two estimators w.r.t. the true parameter $\theta_0$ is quantitatively different. In particular, \eqref{eq:average-l2-rate-2} implies that, for each fixed $N\in\mathbb{N}$, in the limit as $t\rightarrow\infty$, the ``averaged'' estimator is asymptotically \emph{unbiased}, while the ``non-averaged'' (i.e., three particle) estimator is asymptotically \emph{biased}. This is a consequence of the fact that $\smash{\theta_0^{i,N}}$ always coincides with the true parameter $\theta_0$, and so consistency w.r.t. $\theta_0^{i,N}$ in \eqref{eq:average-l2-rate-1} immediately implies consistency w.r.t. $\theta_0$ in \eqref{eq:average-l2-rate-2}. On the other hand, $\theta_0^{i,j,k,N}$ is generally not equal to $\theta_0$, and thus \eqref{eq:non-average-l2-rate-2} contains an additional bias term in comparison to \eqref{eq:non-average-l2-rate-1}, which only vanishes as $N\rightarrow\infty$. In practice, this suggests the following trade-off. In cases where $N$ is small, the additional finite-particle bias term for the non-averaged estimator will be significant, and thus the averaged estimator is likely to be preferable. On the other hand, when $N$ is moderate to large, the additional bias term will be negligible, and thus the non-averaged estimator is likely to be preferable, given its substantially reduced computational cost and observational requirements. 
\end{remark}

\begin{corollary}
\label{corollary:main-theorem-1-2-finite-n-with-m}
    Suppose that Assumption \ref{assumption:moments}, Assumption~\ref{assumption:drift}, Assumption~\ref{assumption:drift-grad}, Assumption~\ref{assumption:learning-rate}, Assumption~\ref{assumption:learning-rate-v2}, and Assumption~\ref{assumption:convexity-finite-n} hold. Let $N\in\mathbb{N}$, $M\in[N]$, and let $i,j,k\in[N]$ be distinct. Suppose that $\mathbb{P}(\bar{\theta}_t^{N,M}\in\Theta~\forall t\geq 0)=\mathbb{P}(\theta_t^{N,M}\in\Theta~\forall t\geq 0)=1$. Suppose also that $\sup_{t\geq 0}\mathbb{E}[\|\bar{\theta}_t^{N,M}\|^l]<\infty$ and $\sup_{t\geq 0}\mathbb{E}[\|{\theta}_t^{N,M}\|^l] <\infty$ for all $l\in\mathbb{N}$. Finally, suppose that $\Theta$ is convex. Then, for sufficiently large $t\in\mathbb{R}_{+}$, there exists positive constants $K_{1},K_2>0$, $K_{1}^{\dagger},K_2^{\dagger}>0$ such that
    \noeqref{eq:average-l2-rate-1-with-m}
    \begin{align}
        &\mathbb{E}\left[\|\bar{\theta}_t^{N,M} - \theta_0^{i,N}\|^2\right]  \leq (K_{1}+\frac{K_{2}}{M})\gamma_t  \label{eq:average-l2-rate-1-with-m} \\
        &\mathbb{E}\left[\|\theta_t^{N,M}-\theta_0^{i,j,k,N}\|^2\right] \leq (K_{1}^{\dagger}+\frac{K_{2}^{\dagger}}{M})\gamma_t  \label{eq:non-average-l2-rate-1-with-m}
    \end{align}
    where $\theta_0^{i,N}$ and $\theta_0^{i,j,k,N}$ denote the unique minimisers of $\mathcal{L}^{i,N}$ and $\mathcal{L}^{i,j,k,N}$, respectively. Moreover, writing $\theta_0$ for the true parameter, there exist constants $K_3^{\dagger},K_4^{\dagger}>0$
    such that, 
    \begin{align}
        &\mathbb{E}\left[\|\bar{\theta}_t^{N,M} - \theta_0\|^2\right]  \leq (K_{1}+\frac{K_{2}}{M})\gamma_t \label{eq:average-l2-rate-2-with-m} \\
        &\mathbb{E}\left[\|\theta_t^{N,M}-\theta_0\|^2\right] \leq 2(K_{1}^{\dagger}+\frac{K_{2}^{\dagger}}{M})\gamma_t + 2 \frac{1}{(\eta_{i,j,k,N})^2} \left[K_3^{\dagger}\rho^2(N) + \frac{K_4^{\dagger}}{N^{\frac{1}{1+\alpha}}}\right]. \label{eq:non-average-l2-rate-2-with-m}
    \end{align}
\end{corollary}

\begin{proof}
    See Appendix \ref{app:additional-proofs-main-results-convergence-rates}.
\end{proof}

We next characterise the asymptotic convergence rate of $(\bar{\theta}_t^{i,N})_{t\geq 0}$ and $(\theta_t^{i,j,k,N})_{t\geq 0}$ as both $t\rightarrow\infty$ and $N\rightarrow\infty$. This means, in particular, that we assume now convexity of the asymptotic (in time and in the number of particles) complete-data negative log-likelihood $\mathcal{L}$, rather than the asymptotic (in time, but not in the number of particles) pseudo negative log-likelihoods $\mathcal{L}^{i,N}$ or $\mathcal{L}^{i,j,k,N}$. 

\begin{theorem}
\label{theorem:main-theorem-1-2}
    Suppose that Assumption \ref{assumption:moments}, Assumption~\ref{assumption:drift}, Assumption~\ref{assumption:drift-grad}, Assumption~\ref{assumption:learning-rate}, Assumption~\ref{assumption:learning-rate-v2}, and Assumption~\ref{assumption:convexity-infinite-n} hold. Let $N\in\mathbb{N}$ and let $i,j,k\in[N]$ be distinct. Suppose that $\mathbb{P}(\bar{\theta}_t^{i,N}\in\Theta~\forall t\geq 0)=\mathbb{P}(\theta_t^{i,j,k,N}\in\Theta~\forall t\geq 0)=1$ for all $N\in\mathbb{N}$. Suppose also that $\sup_{t\geq 0}\mathbb{E}[\|\bar{\theta}_t^{i,N}\|^l]<\infty$ and $\sup_{t\geq 0}\mathbb{E}[\|{\theta}_t^{i,j,k,N}\|^l] <\infty$ for all $l\in\mathbb{N}$, and for all $N\in\mathbb{N}$. Finally, suppose that $\Theta$ is convex. Then, for sufficiently large $t\in\mathbb{R}_{+}$, there exists positive constants $K_{1},K_2,K_3,K_4>0$, $K_{1}^{\dagger},K_2^{\dagger},K_3^{\dagger},K_4^{\dagger}>0$ such that 
    \begin{align}
        &\mathbb{E}\left[\|\bar{\theta}_t^{i,N} - \theta_0\|^2\right]  \leq (K_{1}+K_{2})\gamma_t + K_3 \rho(N) + \frac{K_4}{N^{\frac{1}{2(1+\alpha)}}} \\
        &\mathbb{E}\left[\|\theta_t^{i,j,k,N}-\theta_0\|^2\right] \leq (K_{1}^{\dagger}+K_{2}^{\dagger})\gamma_t + K_3^{\dagger} \rho(N) + \frac{K_4^{\dagger}}{N^{\frac{1}{2(1+\alpha)}}}
    \end{align}
    Suppose, in addition, that $\sup_{\theta\in\Theta}\|\partial_{\theta}^2\mathcal{L}^{i,N}(\theta) - \partial_{\theta}^2\mathcal{L}(\theta)\|_{\mathrm{op}}\leq \delta_{i,N}$ and $\sup_{\theta\in\Theta}\|\partial_{\theta}^2\mathcal{L}^{i,j,k,N}(\theta) - \partial_{\theta}^2 \mathcal{L}(\theta)\|_{\mathrm{op}}\leq \delta_{i,j,k,N}$, where  $0<\delta_{i,N}<\eta$ and $0<\delta_{i,j,k,N}<\eta$ for sufficiently large $N$. Then 
    \begin{align}
        &\mathbb{E}\left[\|\bar{\theta}_t^{i,N} - \theta_0\|^2\right]  \leq (K_{1}+K_{2})\gamma_t \label{eq:average-l2-rate-2-inf-n} \\
        &\mathbb{E}\left[\|\theta_t^{i,j,k,N}-\theta_0\|^2\right] \leq 2(K_{1}^{\dagger}+K_{2}^{\dagger})\gamma_t + 2 \frac{1}{(\eta-\delta_{i,j,k,N})^2} \left[K_3^{\dagger}\rho^2(N) + \frac{K_4^{\dagger}}{N^{\frac{1}{1+\alpha}}}\right]. \label{eq:non-average-l2-rate-2-inf-n}
    \end{align}
\end{theorem}

\begin{proof}
    See Appendix \ref{app:additional-proofs-main-results-convergence-rates}.
\end{proof}

\begin{remark}
The additional conditions required for the second set of statements in Theorem~\ref{theorem:main-theorem-1-2} are precisely the conditions required to transfer strong convexity of the limiting objective $\mathcal L$ to strong convexity of the finite-particle objectives. This allows us to recover the same $L^2$ rates as in the fixed-$N$ setting (see Theorem~\ref{theorem:main-theorem-1-2-finite-n}), up to replacing the strong convexity constant by $\eta-\delta_{i,N}$ or $\eta - \delta_{i,j,k,N}$. These additional conditions are very mild. In particular, under our existing assumptions, one can obtain (as in Proposition~\ref{prop:inf-n-convergence-1}) bounds of the form
\begin{alignat}{2}
    \sup_{\theta\in\Theta}\big\| \partial_{\theta}^2\mathcal{L}^{i,N}(\theta) - \partial_{\theta}^2 \mathcal{L}(\theta)\big\|_{\mathrm{op}} 
    &\leq \delta_{i,N}, \qquad &&\delta_{i,N} =\tilde K_1\rho(N) + \frac{\tilde K_2}{N^{\frac{1}{2(1+\alpha)}}}\label{eq:l-i-convergence-rate-hess} \\
    \sup_{\theta\in\Theta}\big\| \partial_{\theta}^2\mathcal{L}^{i,j,k,N}(\theta) - \partial_{\theta}^2 \mathcal{L}(\theta)\big\|_{\mathrm{op}} 
    &\leq \delta_{i,j,k,N},\qquad &&\delta_{i,j,k,N}=\tilde K_1^{\dagger}\rho(N) + \frac{\tilde K_2^{\dagger}}{N^{\frac{1}{2(1+\alpha)}}}. \label{eq:l-i-j-k-convergence-rate-hess} 
\end{alignat}
Thus, in particular, $\delta_{i,N}\to0$ and $\delta_{i,j,k,N}\to0$ as $N\to\infty$, so the conditions hold for all sufficiently large $N$.
\end{remark}

\subsubsection{Central Limit Theorem}
\label{sec:main-results-clt}

Finally, we obtain a central limit theorem. Once again, we begin in the case where the number of particles $N\in\mathbb{N}$ is fixed and finite, and we just consider asymptotics as $t\rightarrow\infty$. Similar to above, for this result, we will assume strong convexity of the finite-particle asymptotic incomplete-data negative log-likelihood, $\mathcal{L}^{i,N}$ or $\mathcal{L}^{i,j,k,N}$. In this case, we have the following result. 

\begin{theorem}
\label{theorem:clt}
    Suppose that Assumption \ref{assumption:moments}, Assumption~\ref{assumption:drift}, Assumption~\ref{assumption:drift-grad}, Assumption~\ref{assumption:learning-rate}, Assumption~\ref{assumption:learning-rate-v2}, and Assumption~\ref{assumption:convexity-finite-n} hold. Let $N\in\mathbb{N}$, and let $i,j,k\in[N]$ be distinct. Suppose that $\mathbb{P}(\bar{\theta}_t^{i,N}\in\Theta~\forall t\geq 0)=\mathbb{P}(\theta_t^{i,j,k,N}\in\Theta~\forall t\geq 0)=1$. Suppose, in addition, that $\smash{\sup_{t\geq 0}\mathbb{E}[\|\bar{\theta}_t^{i,N}\|^l]<\infty}$ and $\smash{\sup_{t\geq 0}\mathbb{E}[\|{\theta}_t^{i,j,k,N}\|^l] <\infty}$ for all $l\in\mathbb{N}$. Finally, suppose that $\Theta$ is convex. Then it holds that
    \begin{align}
        & \gamma_{t}^{-\frac{1}{2}}\big(\bar{\theta}^{i,N}_t - \theta_0^{i,N}\big) \stackrel{\mathrm{d}}{\longrightarrow} \mathcal{N}(0,\bar{\Sigma}^{i,N}) \label{eq:clt-average-1} \\
        & \gamma_{t}^{-\frac{1}{2}}\big(\theta^{i,j,k,N}_t - \theta_0^{i,j,k,N}\big) \stackrel{\mathrm{d}}{\longrightarrow} \mathcal{N}(0,\bar{\Sigma}^{i,j,k,N}) \label{eq:clt-non-average-1}
    \end{align}
    where $\theta_0^{i,N}$ and $\theta_0^{i,j,k,N}$ denote the unique minimisers of $\mathcal{L}^{i,N}$ and $\mathcal{L}^{i,j,k,N}$, respectively. 
    The limiting covariance matrices $\bar{\Sigma}^{i,N}$ and $\bar{\Sigma}^{i,j,k,N}$ are given by
    \begin{align}
        \bar{\Sigma}^{i,N} &= \lim_{t\rightarrow\infty} \left[{\gamma_t}^{-1}\int_0^t \gamma_s^2 \Phi_{s,t}^{*,i,N} \bar{\Gamma}^{i,N}(\theta_0^{i,N})  \Phi_{s,t}^{*,i,N,\top} \mathrm{d}s\right] \\
        \bar{\Sigma}^{i,j,k,N} &= \lim_{t\rightarrow\infty} \left[{\gamma_t}^{-1}\int_0^t \gamma_s^2 \Phi_{s,t}^{*,i,j,k,N} \bar{\Gamma}^{i,j,k,N}(\theta_0^{i,j,k,N})  \Phi_{s,t}^{*,i,j,k,N,\top} \mathrm{d}s\right]
    \end{align}
    where $\Phi^{*,i,N}_{s,t}\in\mathbb{R}^{p\times p}$ and $\Phi^{*,i,j,k,N}_{s,t}\in\mathbb{R}^{p\times p}$ are given by $\Phi^{*,i,N}_{s,t}= \exp[-\nabla^2\mathcal{L}^{i,N}(\theta_0^{i,N})\int_{s}^{t}\gamma_u\mathrm{d}u]$ and $\Phi^{*,i,j,k,N}_{s,t} = \exp[-\nabla^2\mathcal{L}^{i,j,k,N}(\theta_0^{i,j,k,N})\int_{s}^{t}\gamma_u\mathrm{d}u]$, and where $\bar{\Gamma}^{i,N}:\mathbb{R}^p\rightarrow\mathbb{R}^{p\times p}$ and $\bar{\Gamma}^{i,j,k,N}:\mathbb{R}^p\rightarrow\mathbb{R}^{p\times p}$ are given by
\begin{align}
    \bar{\Gamma}^{i,N}(\theta) &= \int_{(\mathbb{R}^d)^N} \Gamma^{i,N}(\theta,\boldsymbol{x}^N)\pi_{\theta_0}^N(\mathrm{d}\boldsymbol{x}^N) \\
     \bar{\Gamma}^{i,j,k,N}(\theta) &= \int_{(\mathbb{R}^d)^N} \Gamma^{i,j,k,N}(\theta,\boldsymbol{x}^N)\pi_{\theta_0}^N(\mathrm{d}\boldsymbol{x}^N) 
\end{align}
 with $\Gamma^{i,N}: \mathbb{R}^p\times(\mathbb{R}^d)^N\rightarrow\mathbb{R}^{p\times p}$ and $\Gamma^{i,j,k,N}: \mathbb{R}^p\times(\mathbb{R}^d)^N\rightarrow\mathbb{R}^{p\times p}$ given by
 \begin{align}
    \Gamma^{i,N}(\theta,\boldsymbol{x}^N)&= \left(G^{i,N}(\theta,{\boldsymbol{x}}^N)(\sigma\sigma^{\top})^{-1}E_i^{\top}  - \partial_{\boldsymbol{x}^N}v^{i,N}(\theta,{\boldsymbol{x}}^N) \right) \big(I_N\otimes (\sigma\sigma^{\top})\big) \\[-1mm]
    &~~~~\big(G^{i,N}(\theta,{\boldsymbol{x}}^N)(\sigma\sigma^{\top})^{-1}E_i^{\top}  - \partial_{\boldsymbol{x}^N}v^{i,N}(\theta,{\boldsymbol{x}}^N) \big)^{\top} \\[2mm]
    \Gamma^{i,j,k,N}(\theta,\boldsymbol{x}^N)&:=\left(g^{i,j,N}(\theta,{\boldsymbol{x}}^N)(\sigma\sigma^{\top})^{-1}E_i^{\top}  - \partial_{\boldsymbol{x}^N}v^{i,j,k,N}(\theta,{\boldsymbol{x}}^N) \right) \big(I_N\otimes (\sigma\sigma^{\top})\big) \\[-1mm]
    &~~~~\big(g^{i,j,N}(\theta,{\boldsymbol{x}}^N)(\sigma\sigma^{\top})^{-1}E_i^{\top}  - \partial_{\boldsymbol{x}^N}v^{i,j,k,N}(\theta,{\boldsymbol{x}}^N) \big)^{\top}. \notag
 \end{align}
where we use the shorthand $\smash{G^{i,N}(\theta,\boldsymbol{x}^N):=G(\theta,x^{i,N},\mu^N)}$ and $\smash{g^{i,j,N}(\theta,\boldsymbol{x}^N):=g(\theta,x^{i,N},x^{j,N})}$; $E_i\in\mathbb{R}^{dN\times d}$ is the matrix which selects the $i^{\mathrm{th}}$ component of a vector $\boldsymbol{x}^N = (x^{1,N},\dots,x^{N,N})^{\top}$, and $v^{i,N}(\theta,\boldsymbol{x}^N)$ and $v^{i,j,k,N}(\theta,\boldsymbol{x}^N)$ denote the solutions of the Poisson equations
\begin{align}
    \mathcal{A}_{\boldsymbol{x}^N} v^{i,N}(\theta,\boldsymbol{x}^N)  &= \partial_{\theta}\mathcal{L}^{i,N}(\theta) - H^{i,N}(\theta,\boldsymbol{x}^N), \quad \int_{(\mathbb{R}^d)^N} v^{i,N}(\theta,\boldsymbol{x}^N) \pi_{\theta_0}^{N}(\mathrm{d}\boldsymbol{x}^N)=0 \\
    \mathcal{A}_{\boldsymbol{x}^N} v^{i,j,k,N}(\theta,\boldsymbol{x}^N)  &=  \partial_{\theta}\mathcal{L}^{i,j,k,N}(\theta) - h^{i,j,k,N}(\theta,\boldsymbol{x}^N), \quad \int_{(\mathbb{R}^d)^N} v^{i,j,k,N}(\theta,\boldsymbol{x}^N) \pi_{\theta_0}^{N}(\mathrm{d}\boldsymbol{x}^N)=0.
\end{align}
\end{theorem}

\begin{proof}
See Appendix \ref{app:additional-proofs-main-results-clt}.
\end{proof}

Finally, we consider the situation where also the number of particles $N\rightarrow\infty$. In this setting, the natural assumption is once more that the asymptotic (both in time and in particles) complete-data negative log-likelihood is strongly convex. 

\begin{theorem}
\label{theorem:clt-inf-n}
    Suppose that Assumption \ref{assumption:moments}, Assumption~\ref{assumption:drift}, Assumption~\ref{assumption:drift-grad}, Assumption~\ref{assumption:learning-rate}, Assumption~\ref{assumption:learning-rate-v2}, and Assumption~\ref{assumption:convexity-infinite-n} hold. Let $N\in\mathbb{N}$, and let $i,j,k\in[N]$ be distinct. Suppose that $\mathbb{P}(\bar{\theta}_t^{i,N}\in\Theta~\forall t\geq 0)=\mathbb{P}(\theta_t^{i,j,k,N}\in\Theta~\forall t\geq 0)=1$. Suppose, in addition, that $\smash{\sup_{t\geq 0}\mathbb{E}[\|\bar{\theta}_t^{i,N}\|^l]<\infty}$ and $\smash{\sup_{t\geq 0}\mathbb{E}[\|{\theta}_t^{i,j,k,N}\|^l] <\infty}$ for all $l\in\mathbb{N}$. Suppose also that $\Theta$ is convex. Finally, suppose that $N=N(t)\rightarrow\infty$ as $t\rightarrow\infty$ at the rate $\rho(N) + {N^{-\frac{1}{2(1+\alpha)}}} = o(\gamma_t^{\frac{1}{2}})$, where $\rho:\mathbb{N}\rightarrow\mathbb{R}_{+}$ is the function defined in \eqref{eq:rho}. Then it holds that
    \begin{align}
        & \gamma_{t}^{-\frac{1}{2}}\big(\bar{\theta}_t^{i,N} - \theta_0\big) \stackrel{\mathrm{d}}{\longrightarrow} \mathcal{N}(0,\bar{\Sigma}^{i}) \label{eq:clt-average-1-inf-n} \\
        & \gamma_{t}^{-\frac{1}{2}}\big(\theta_t^{i,j,k,N} - \theta_0\big) \stackrel{\mathrm{d}}{\longrightarrow} \mathcal{N}(0,\bar{\Sigma}^{i,j,k}) \label{eq:clt-non-average-1-inf-n}
    \end{align}
    The limiting covariance matrices $\bar{\Sigma}^{i}$ and $\bar{\Sigma}^{i,j,k}$ are given by
    \begin{align}
        \bar{\Sigma}^{i} &= \lim_{t\rightarrow\infty} \left[{\gamma_t}^{-1}\int_0^t \gamma_s^2 \Phi_{s,t}^{*} \bar{\Gamma}^{i}(\theta_0)  \Phi_{s,t}^{*,\top} \mathrm{d}s\right] \\
        \bar{\Sigma}^{i,j,k} &= \lim_{t\rightarrow\infty} \left[{\gamma_t}^{-1}\int_0^t \gamma_s^2 \Phi^{*}_{s,t} \bar{\Gamma}^{i,j,k}(\theta_0)  \Phi_{s,t}^{*,\top} \mathrm{d}s\right]
    \end{align}
where $\Phi^{*}_{s,t}\in\mathbb{R}^{p\times p}$ is given by $\Phi^{*}_{s,t} = \exp[-\nabla^2\mathcal{L}(\theta_0)\int_{s}^{t}\gamma_u\mathrm{d}u]$, where $\bar{\Gamma}^{i}:\mathbb{R}^p\rightarrow\mathbb{R}^{p\times p}$ and $\bar{\Gamma}^{i,j,k}:\mathbb{R}^p\rightarrow\mathbb{R}^{p\times p}$ are given by
\begin{align}
    \bar{\Gamma}^{i}(\theta) &= \int_{\mathbb{R}^d} \Gamma^{i}(\theta,{x}^i)\pi_{\theta_0}(\mathrm{d}{x}^i) \\
    \bar{\Gamma}^{i,j,k}(\theta) &= \int_{(\mathbb{R}^d)^3} \Gamma^{i,j,k}(\theta,\boldsymbol{x}^{(i,j,k)})\pi_{\theta_0}^{\otimes 3}(\mathrm{d}\boldsymbol{x}^{(i,j,k)}) 
\end{align}
with $\Gamma^{i}:\mathbb{R}^p\times\mathbb{R}^d\rightarrow\mathbb{R}^{p\times p}$ and $\Gamma^{i,j,k}:\mathbb{R}^p\times(\mathbb{R}^d)^3\rightarrow\mathbb{R}^{p\times p}$ given by
\begin{align}
    \Gamma^{i}(\theta,x^{i})&=\left(G(\theta,x^{i},\pi_{\theta_0}) (\sigma\sigma^{\top})^{-1}  - \partial_{{x}^i}v^{i}(\theta,{x}^{i}) \right) (\sigma\sigma^{\top})  \\
    &~~~~\left(G(\theta,x^{i},\pi_{\theta_0}) (\sigma\sigma^{\top})^{-1}  - \partial_{{x}^i}v^{i}(\theta,{x}^{i}) \right)^{\top} \notag \\[2mm]
    \Gamma^{i,j,k}(\theta,\boldsymbol{x}^{(i,j,k)})&=\left(g(\theta,\boldsymbol{x}^{(i,j)}) (\sigma\sigma^{\top})^{-1} D_i^{\top} - \partial_{\boldsymbol{x}}v^{i,j,k}(\theta,\boldsymbol{x}^{(i,j,k)}) \right) \Big(I_3 \otimes (\sigma\sigma^{\top})\Big)  \\
    &~~~~\left(g(\theta,\boldsymbol{x}^{(i,j)})(\sigma\sigma^{\top})^{-1} D_i^{\top}  - \partial_{\boldsymbol{x}}v^{i,j,k}(\theta,\boldsymbol{x}^{(i,j,k)}) \right)^{\top}, \notag
\end{align}
where $D_i\in\mathbb{R}^{3d\times d}$ is the matrix which selects the $x_i^{\mathrm{th}}$ component of a vector $\boldsymbol{x}^{i,j,k} = (x^{i},{x}^j,{x}^k)^{\top}$, and where $v^{i}(\theta,{x}^i)$ and $v^{i,j,k}(\theta,\boldsymbol{x}^{(i,j,k)})$ denote the solutions of the Poisson equations
\begin{align}
    \mathcal{A}_{{x}^i} v^{i}(\theta,{x}^i)  &= \partial_{\theta}\mathcal{L}(\theta) - H(\theta,{x}^i,\pi_{\theta_0}), \quad \int_{\mathbb{R}^d} v^{i}(\theta,{x}^i) \pi_{\theta_0}(\mathrm{d}{x}^i)=0 \\
    \mathcal{A}_{\boldsymbol{x}^{(i,j,k)}} v^{i,j,k}(\theta,\boldsymbol{x}^{(i,j,k)})  &=  \partial_{\theta}\mathcal{L}(\theta) - h(\theta,\boldsymbol{x}^{(i,j,k)},\pi_{\theta_0}), \quad \int_{(\mathbb{R}^d)^3} v^{i,j,k}(\theta,\boldsymbol{x}^{(i,j,k)}) \pi_{\theta_0}^{\otimes 3}(\mathrm{d}\boldsymbol{x}^{(i,j,k)})=0.
\end{align}
\end{theorem}

\begin{proof}
See Appendix \ref{app:additional-proofs-main-results-clt}.
\end{proof}

\section{Numerical Results}
\label{sec:numerics}
In this section, we present numerical experiments to illustrate the performance of the proposed estimators. We consider examples that satisfy the assumptions of the previous section, as well as examples that violate one or more of these assumptions (e.g., unique invariant measure, non-degenerate diffusion coefficient). In all cases, unless otherwise specified, we discretise the SDEs using a standard Euler-Maruyama scheme, with constant time-step $\Delta t = 0.1$.  We perform all experiments using a MacBook Pro 16'' (2021) laptop with Apple M1 Pro chip and 16GB of RAM.

\subsection{Quadratic Confinement, Quadratic Interaction}
We begin by considering a one-dimensional IPS with quadratic confinement potential and quadratic interaction potential, parametrised by $\theta=(\theta_1,\theta_2)^{\top}\in\mathbb{R}^2$, namely
\begin{equation}    
\label{eq:IPS_linear_model}
    \mathrm{d}x_t^{\theta,i,N} = \left[-\theta_1 x_t^{\theta,i,N} - \frac{\theta_2}{N} \sum_{j=1}^N \left(x_t^{\theta,i,N} - x_t^{\theta,j,N}\right)\right] \mathrm{d}t + \sigma\mathrm{d}w_t^{i,N},
\end{equation}
where $\sigma>0$ is a (known) diffusion coefficient, and $w^{i,N}=(w_t^{i,N})_{t\geq 0}$ are a set of independent standard Brownian motions.  In this model, we can interpret $\theta_1$ as a \emph{confinement parameter}, which determines the rate at which each particle is driven towards zero, and $\theta_2$ as an \emph{interaction parameter}, which determines the strength of interaction between the particles.  In this case, the online parameter update equations in \eqref{eq:IPS_update1-a} and \eqref{eq:IPS_update2-a-v0} take the form
\begin{align}
\mathrm{d}\begin{bmatrix} \bar{\theta}_{t,1}^{i,N} \\ \bar{\theta}_{t,2}^{i,N} \end{bmatrix} &= 
-\gamma_t 
\begin{bmatrix} -x_t^{i,N} \\[1mm] -(x_t^{i,N} - \bar{x}_t^{N}) \end{bmatrix}
(\sigma\sigma^{\top})^{-1}\Bigg[ 
\Bigg(-\bar{\theta}_{t,1}^{i,N} x_t^{i,N} - \bar{\theta}_{t,2}^{i,N}(x_t^{i,N} - \bar{x}_t^{N} )\Bigg)
\mathrm{d}t - \mathrm{d}x_t^{i,N}\Bigg] \label{eq:IPS_linear_update1} \\
\mathrm{d}\begin{bmatrix} \theta_{t,1}^{i,j,k,N} \\ \theta_{t,2}^{i,j,k,N} \end{bmatrix} &= 
-\gamma_t 
\begin{bmatrix} -x_t^{i,N} \\ -(x_t^{i,N} - x_t^{j,N}) \end{bmatrix}
(\sigma\sigma^{\top})^{-1}\Bigg[ 
\left(-\theta_{t,1}^{i,j,k,N} x_t^{i,N} - \theta_{t,2}^{i,j,k,N}(x_t^{i,N} - x_t^{k,N} )\right)
\mathrm{d}t - \mathrm{d}x_t^{i,N}\Bigg] \label{eq:IPS_linear_update2}
\end{align}
where $\smash{\bar{x}_t^N = \frac{1}{N}\sum_{j=1}^N x_t^{j,N}}$ denotes the empirical mean of the particles.  For our first experiment, we assume that the true parameters are given by $\smash{\theta_0 = (1.0, 0.2)^{\top}}$. Meanwhile, the initial parameter estimates are given by $\smash{\theta_{\mathrm{init},1}\sim \mathcal{U}[1.5,2.5]}$ and $\smash{\theta_{\mathrm{init},2}\sim\mathcal{U}[0.5,1.0]}$, respectively. We simulate trajectories from the IPS with $N=50$ particles and for $T=10000$ iterations, with initial condition $\smash{x_0^{i,N}\sim\mathcal{N}(0,1)}$. Finally, for both estimators, we use a constant parameter-wise learning rate $\gamma = (\gamma_1,\gamma_2) = (8\times 10^{-3}, 5\times 10^{-3})^{\top}$. 

\begin{figure}[b!]
  \centering
  \begin{subfigure}{0.325\linewidth}
    \includegraphics[width=\linewidth]{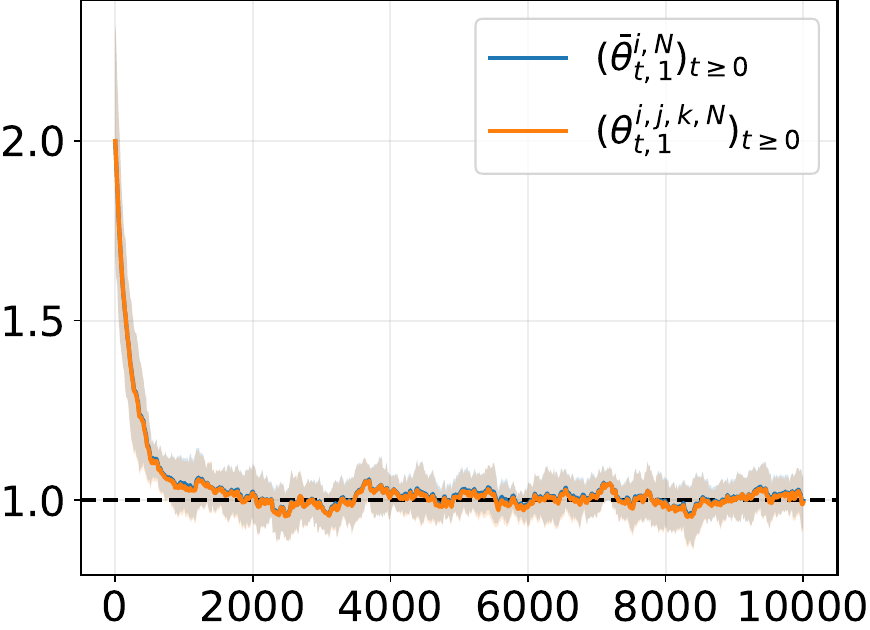}
    \caption{$\theta_1$.}
    \label{fig:1a}
  \end{subfigure}\hfill
  \begin{subfigure}{0.325\linewidth}
    \includegraphics[width=\linewidth]{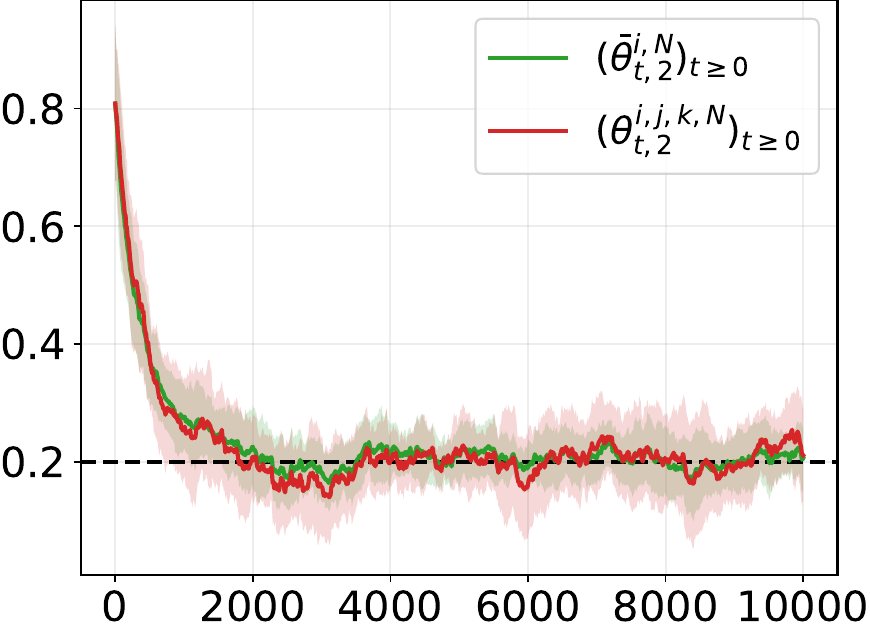}
    \caption{$\theta_2$.}
    \label{fig:1b}
  \end{subfigure}
  \begin{subfigure}{0.325\linewidth}
    \includegraphics[width=\linewidth]{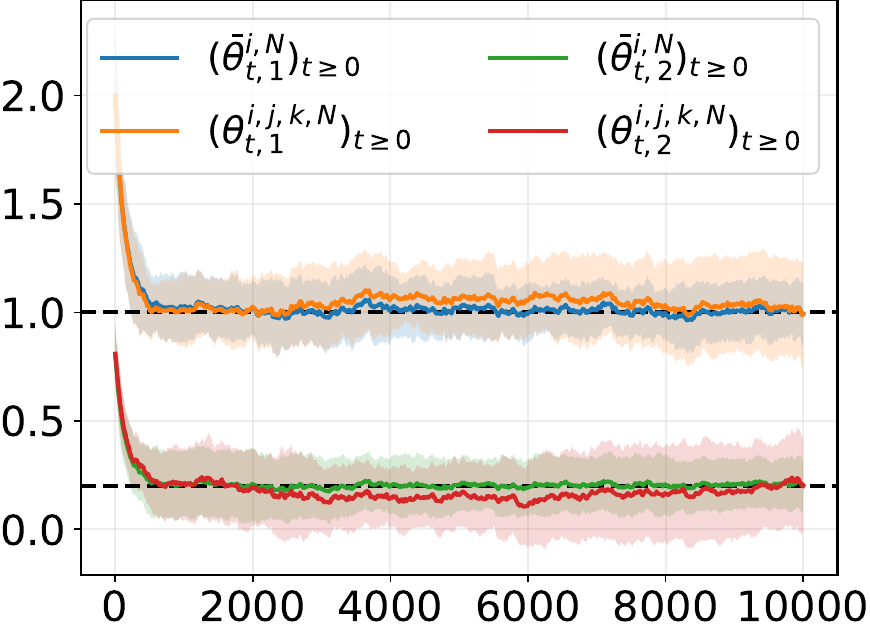}
    \caption{$(\theta_1,\theta_2)^{\top}$.}
    \label{fig:1c}
  \end{subfigure}
  \captionsetup{width=\textwidth}
  \caption{\textbf{Online parameter estimation for a model with quadratic confinement potential and quadratic interaction potential}. We plot the sequence of online parameter estimates $\smash{(\bar\theta_t^{i,N})_{t\geq 0}}$ and $\smash{(\theta_t^{i,j,k,N})_{t\geq 0}}$, as defined by the update equations in \eqref{eq:IPS_linear_update1} and \eqref{eq:IPS_linear_update2}. The true parameters are given by $\smash{\theta_0 = (1.0, 0.2)^{\top}}$. The initial parameter estimates are given by $\smash{\theta_{\mathrm{init},1}\sim \mathcal{U}[1.5,2.5]}$ and $\smash{\theta_{\mathrm{init},2}\sim\mathcal{U}[0.5,1.0]}$.}
  \label{fig:1}
\end{figure}

The performance of the two estimators is illustrated in Figure~\ref{fig:1}. In both the case where only one of the parameters is to be estimated (Fig.~\ref{fig:1a}, Fig.~\ref{fig:1b}), and the case where both of the parameters are to be jointly estimated (Fig.~\ref{fig:1c}), the sequence of online parameter estimates converges to the true parameter(s). Comparing the performance of the two estimators, we distinguish between several cases. In the case where only the confinement parameter is to be estimated (Fig.~\ref{fig:1a}), the evolution of both parameter estimates (blue, orange) is essentially identical. In the case where only the interaction parameter is to be estimated (Fig.~\ref{fig:1b}), the variance of the first estimator (green) is slightly smaller than the variance of the second estimator (red). Finally, when both parameters are to be estimated (Fig.~\ref{fig:1c}), the variance of the first estimator (blue, green) is reduced in comparison to the variance of the second estimator (orange, red) for both parameters. 

In Figure~\ref{fig:2}, we continue to investigate the performance of the two online parameter estimators, now as a function of the number of particles in the data-generating IPS. Our results indicate that the $\mathrm{L}^2$ error of the ``averaged'' estimator is essentially constant with respect to the number of particles, while the error of the ``non-averaged'' estimator decreases as the number of particles increases. This is entirely consistent with our theoretical results. In particular,  Theorem~\ref{theorem:main-theorem-1-2-finite-n} indicates that $\smash{\bar{\theta}_t^{i,N}\rightarrow\theta_0}$ as $\smash{t\rightarrow\infty}$, for any fixed and finite $N\in\mathbb{N}$. On the other hand, $\smash{\theta_t^{i,j,k,N}\rightarrow\theta_0}$ only in the joint limit as $t\rightarrow\infty$ and $N\rightarrow\infty$. In other words, the ``averaged'' estimator is consistent as $t\rightarrow\infty$, for any fixed $N\in\mathbb{N}$, while the non-averaged estimator is only consistent in the joint limit as both $t\rightarrow\infty$ and $N\rightarrow\infty$. 

\begin{figure}[t!]
  \centering
  \begin{subfigure}{0.45\linewidth}
    \includegraphics[width=\linewidth]{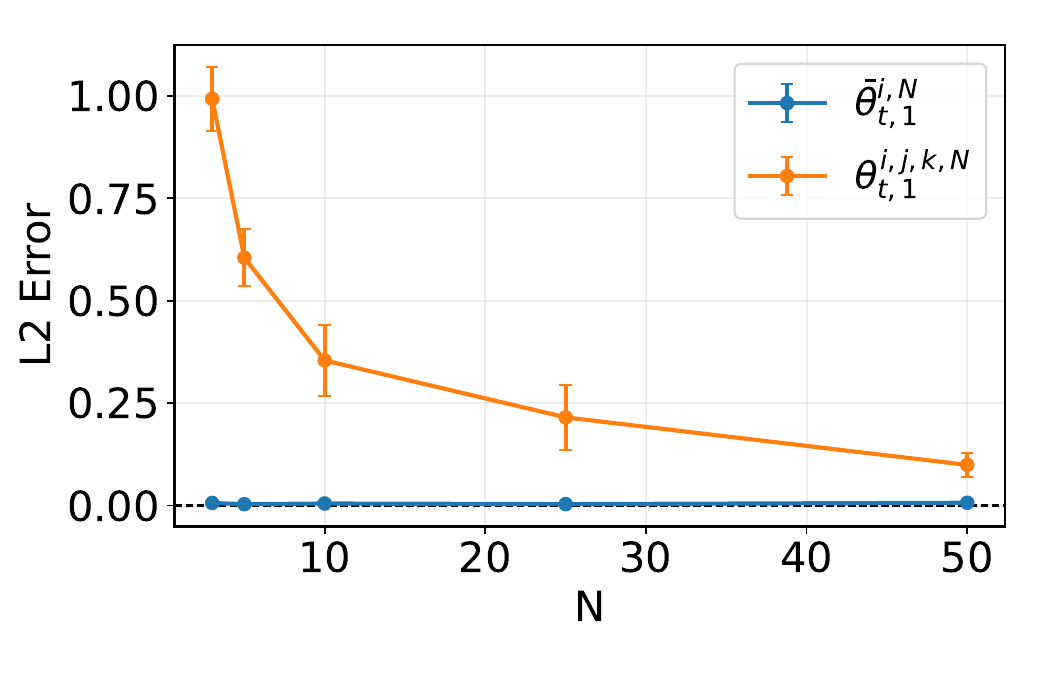}
    \caption{$\theta_1$.}
    \label{fig:2a}
  \end{subfigure}
  \begin{subfigure}{0.45\linewidth}
    \includegraphics[width=\linewidth]{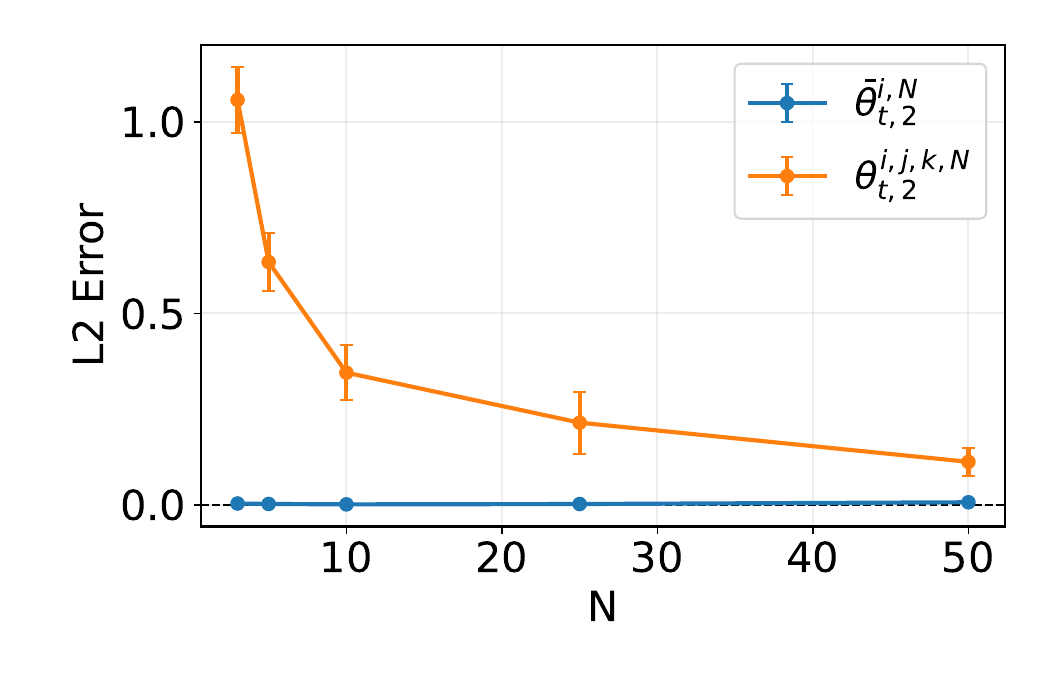}
    \caption{$\theta_2$.}
    \label{fig:2b}
  \end{subfigure}
  \captionsetup{width=\textwidth}
  \caption{\textbf{The $\mathrm{L}^2$ error of the averaged and the non-averaged estimators, for a model with quadratic confinement potential and quadratic interaction potential.} We plot the $\mathrm{L}^2$ error for both estimators after $T=50,000$ iterations, for $N\in\{3,5,10,25,50\}$ particles.}
  \label{fig:2}
\end{figure}

It is worth noting that, as with any gradient-based method, our estimators are somewhat sensitive to the choice of the learning rate. This sensitivity is particularly acute in the case where both parameters are estimated jointly, due to the non-identifiability of the parameter vector $\theta = (\theta_1,\theta_2)^{\top}$ in the mean-field limit as the number of particles $N\rightarrow\infty$ \citep[see, e.g.,][Section 5]{sharrock2023online}. While, in theory, both parameters are identifiable for any finite value of $N$, in reality, even for moderate values of $N$ (e.g., $N\sim 20$), the likelihood attains close to its maximal value for any $(\theta_1,\theta_2)$ satisfying $\theta_{1} + \theta_{2} = \theta_{0,1} + \theta_{0,2}$, where $\theta_0 = (\theta_{0,1},\theta_{0,2})^{\top}$ denotes the true parameter. This phenomenon is visualised in Figures~\ref{fig:3} - \ref{fig:4}, where we plot the time-averaged finite-particle (pseudo) likelihoods $\mathcal{L}^{i,N}$ and $\mathcal{L}^{i,j,k,N}$ of the IPS for $N\in\{3,10,20\}$. In practical terms, the result is that, for poorly chosen values of the learning rate, our estimators may converge quickly to a value $\theta_{*} = (\theta_{*,1},\theta_{*,2})^{\top}$ which satisfies $\theta_{*,1} + \theta_{*,2} = \theta_{0,1} + \theta_{0,2}$, but for which $\theta_{*,1}\neq  \theta_{0,1}$ and $\theta_{*,2}\neq \theta_{0,2}$ even approximately. Consequently, it may then take a \emph{very} long time to converge to the true parameter.

\begin{figure}[b!]
  \centering
  \begin{subfigure}{0.325\linewidth}
    \includegraphics[width=\linewidth]{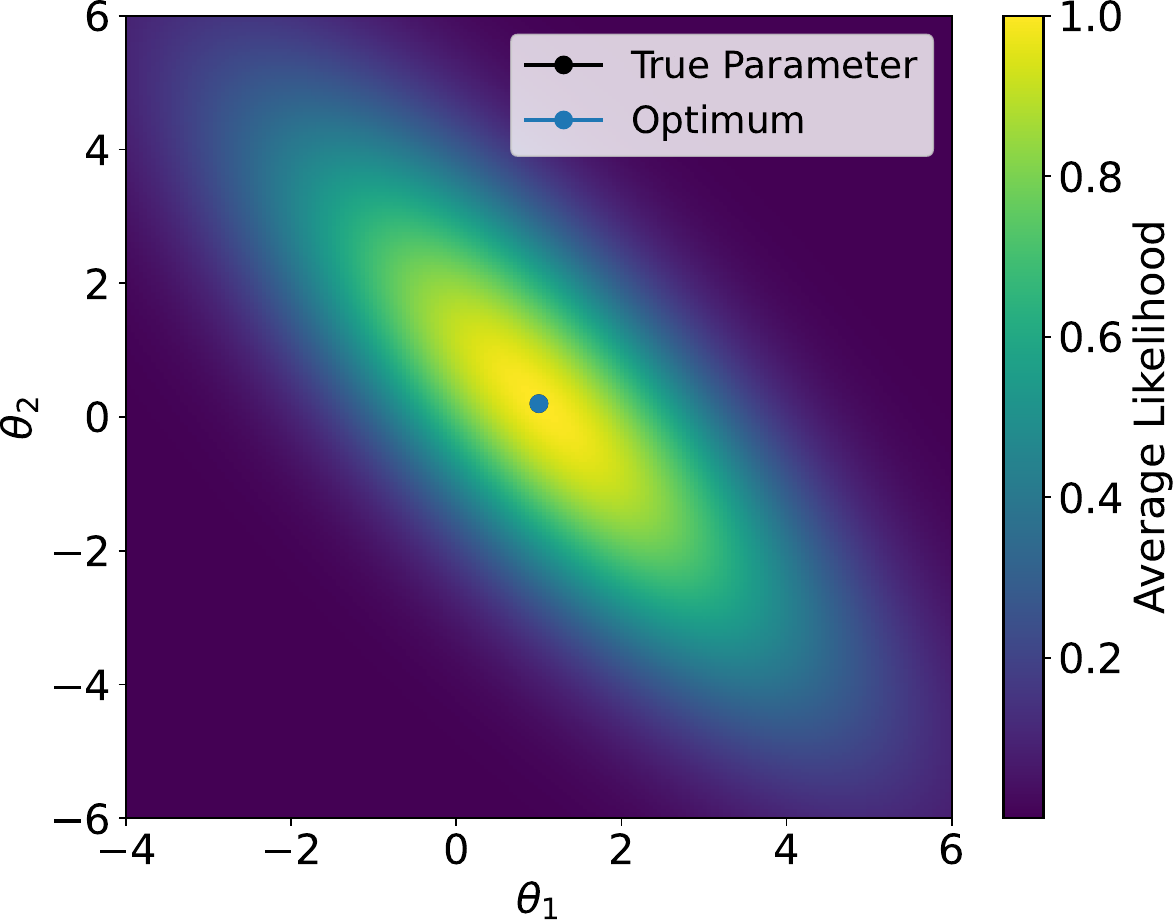}
    \caption{$N=3$.}
    \label{fig:3a}
  \end{subfigure}\hfill
  \begin{subfigure}{0.325\linewidth}
    \includegraphics[width=\linewidth]{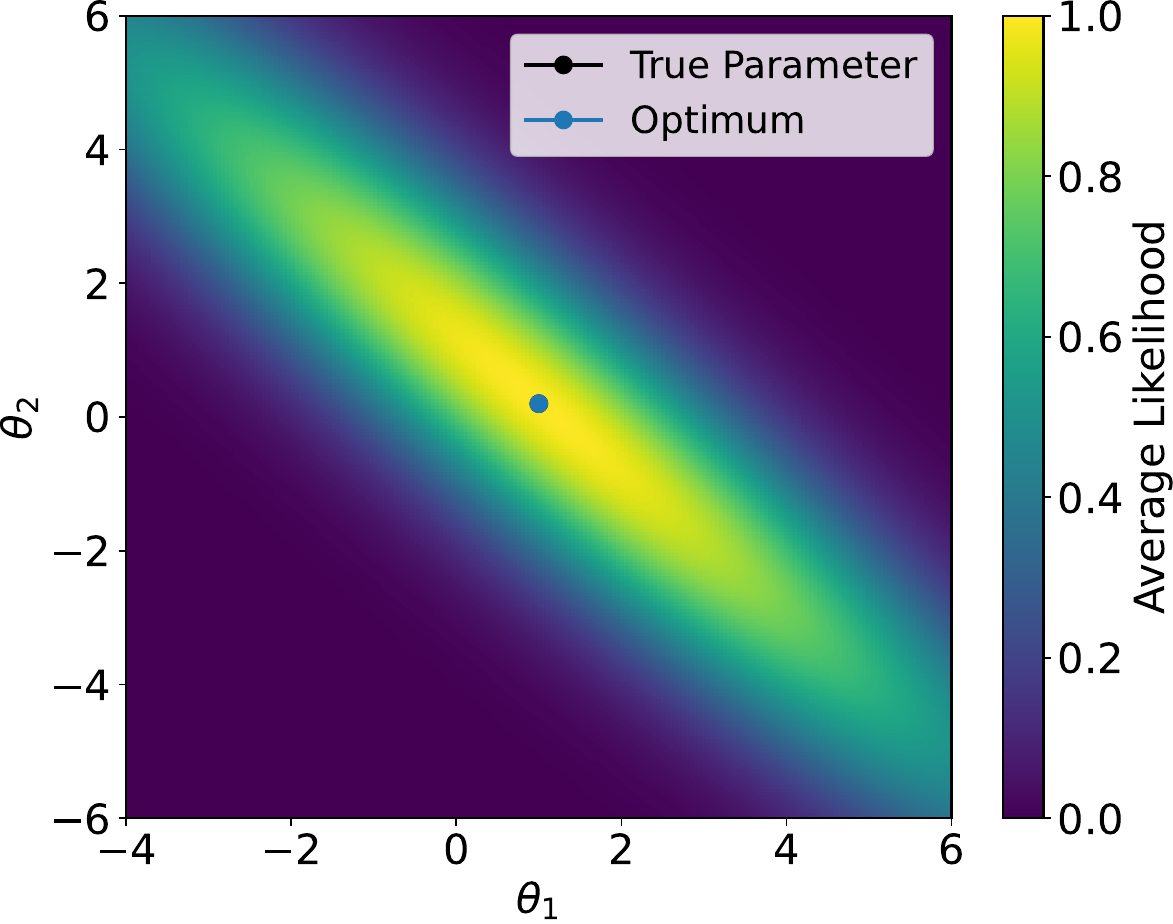}
    \caption{$N=10$.}
    \label{fig:3b}
  \end{subfigure}
  \begin{subfigure}{0.325\linewidth}
    \includegraphics[width=\linewidth]{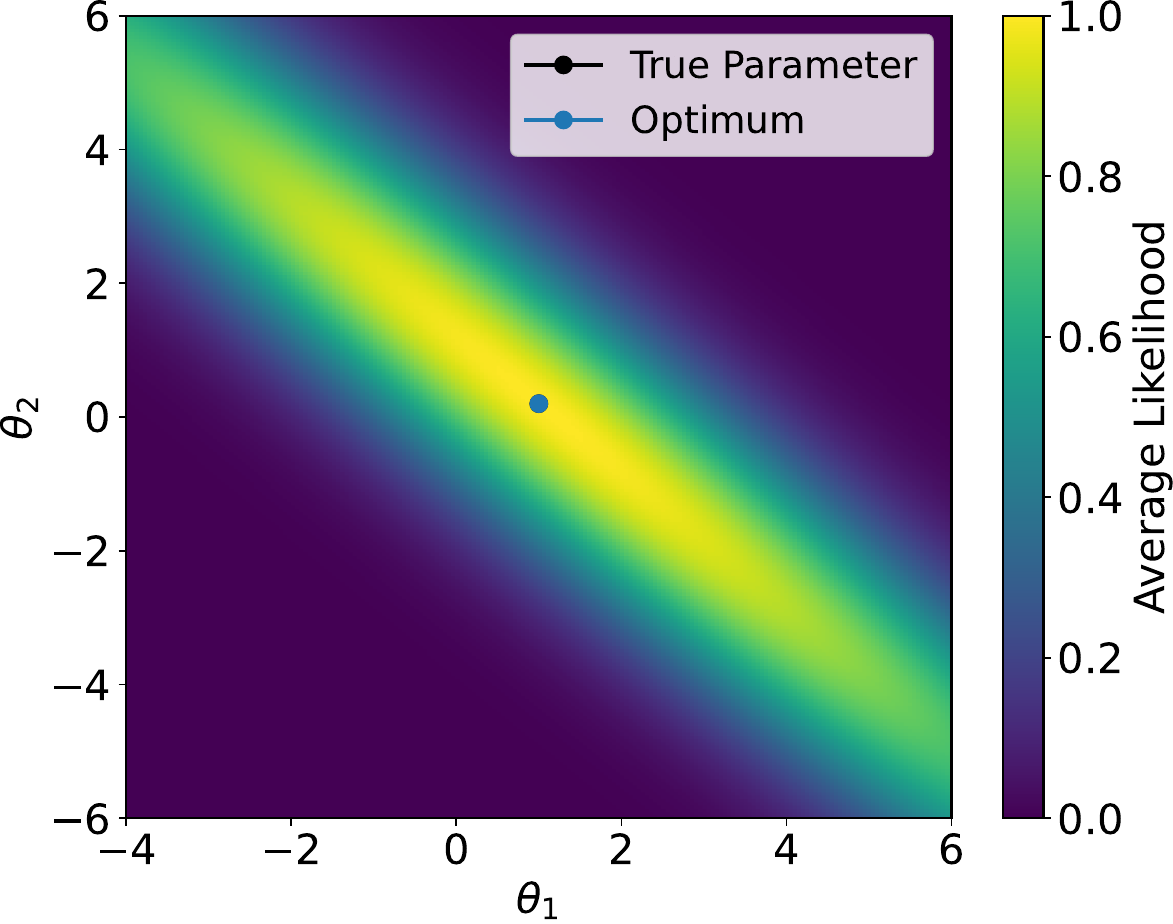}
    \caption{$N=20$.}
    \label{fig:3c}
  \end{subfigure}
  \captionsetup{width=\textwidth}
  \caption{\textbf{The asymptotic pseudo log-likelihood function $\mathcal{L}^{i,N}$ for a model with quadratic confinement potential and quadratic interaction potential.} We plot the time-averaged likelihood function of the IPS for $N\in\{3,10,20\}$.}
  \label{fig:3}
\end{figure}

\begin{figure}[t!]
  \centering
  \begin{subfigure}{0.325\linewidth}
    \includegraphics[width=\linewidth]{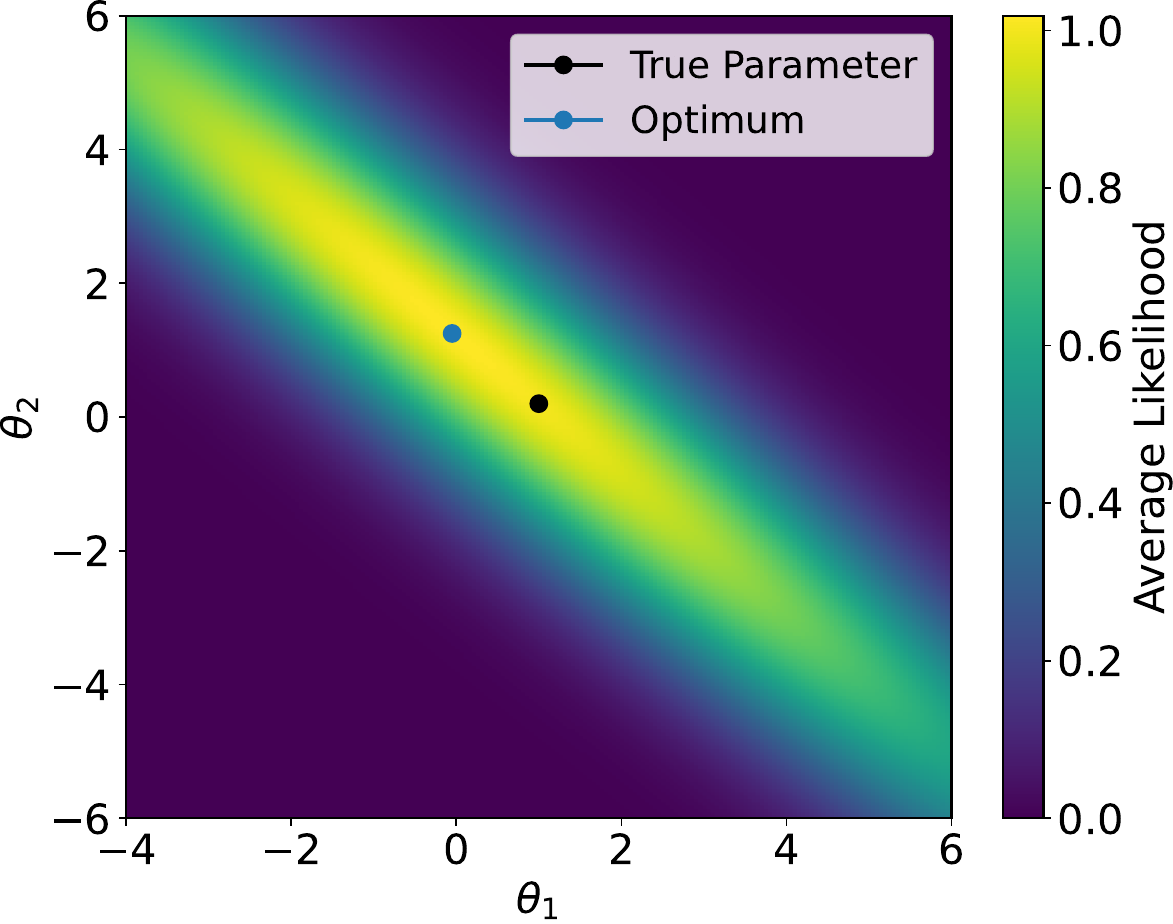}
    \caption{$N=3$.}
    \label{fig:4a}
  \end{subfigure}\hfill
  \begin{subfigure}{0.325\linewidth}
    \includegraphics[width=\linewidth]{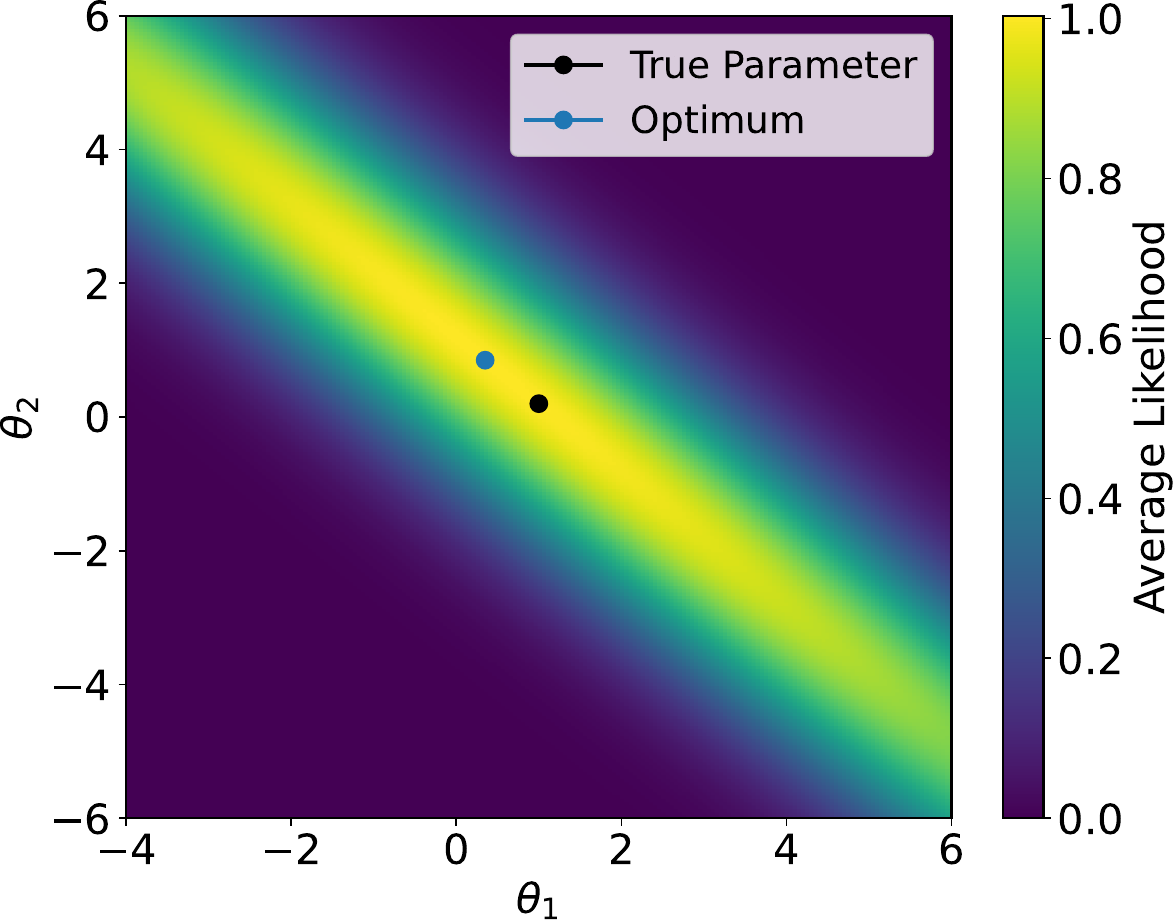}
    \caption{$N=10$.}
    \label{fig:4b}
  \end{subfigure}
  \begin{subfigure}{0.325\linewidth}
    \includegraphics[width=\linewidth]{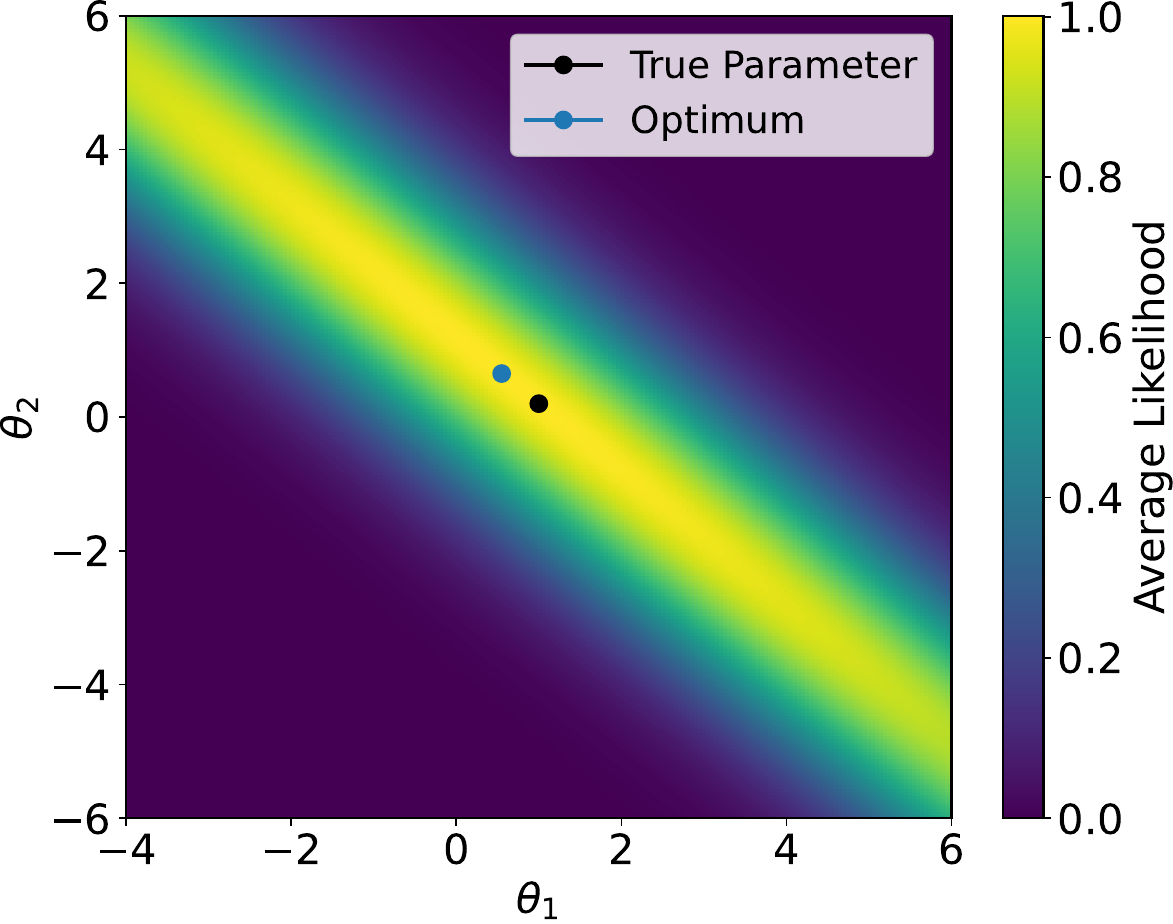}
    \caption{$N=20$.}
    \label{fig:4c}
  \end{subfigure}
  \captionsetup{width=\textwidth}
  \caption{\textbf{The asymptotic pseudo log-likelihood function $\mathcal{L}^{i,j,k,N}$ for a model with quadratic confinement potential and quadratic interaction potential.} We plot the time-averaged likelihood function of the IPS for $N\in\{3,10,20\}$.}
  \label{fig:4}
\end{figure}

\subsection{Double Well Confinement Potential, Quadratic Interaction Potential}
We next consider a model consisting of a double-well confinement potential, and a quadratic (i.e., Curie-Weiss) interaction potential, parametrised by $\theta = (\theta_1,\theta_2,\theta_3)^{\top}\in\mathbb{R}^3$. That is, $V(\theta,x) = \frac{\theta_1}{4}x^4 - \frac{\theta_2}{2} x^2$ and $W(\theta,x) = \frac{\theta_3}{2}x^2$. In this case, the IPS reads
\begin{equation}    
\label{eq:IPS_double_well_model}
    \mathrm{d}x_t^{\theta,i,N} = \Big[-\Big(\theta_1 (x_t^{\theta,i,N})^3 - \theta_{2} x_t^{\theta,i,N}\Big)  - \frac{\theta_3}{N} \sum_{j=1}^N \Big(x_t^{\theta,i,N} - x_t^{\theta,j,N}\Big)\Big] \mathrm{d}t + \sigma\mathrm{d}w_t^{i,N}.
\end{equation}
We will assume that the interaction parameter $\theta_3$ is known, and consider estimation of the confinement parameters $(\theta_1,\theta_2)^{\top}$. The update equations for these two online parameter estimators are given by
\small
\begin{align}
\hspace{-4mm} \mathrm{d}\begin{bmatrix} \bar{\theta}_{t,1}^{i,N} \\ \bar\theta_{t,2}^{i,N} \end{bmatrix} &= 
-\gamma_t 
\begin{bmatrix} -(x_t^{i,N})^3 \\[1mm] x_t^{i,N}  \end{bmatrix}
(\sigma\sigma^{\top})^{-1}\Bigg[  
\Big(-\left(\bar\theta_{t,1}^{i,N} (x_t^{i,N})^3 - \bar\theta_{t,2}^{i,N} x_t^{i,N}\right) - \theta_3(x_t^{i,N} - \bar{x}_t^{N} )\Big)
\mathrm{d}t - \mathrm{d}x_t^{i,N} \Bigg] \label{eq:IPS_bistable_update1} \\
\hspace{-4mm} \mathrm{d}\begin{bmatrix} \theta_{t,1}^{i,j,k,N} \\ \theta_{t,2}^{i,j,k,N} \end{bmatrix} &= 
-\gamma_t 
\begin{bmatrix} -(x_t^{i,N})^3 \\[1mm] x_t^{i,N}  \end{bmatrix}
(\sigma\sigma^{\top})^{-1}\Bigg[ 
\left(-\left(\theta_{t,1}^{i,j,k,N} (x_t^{i,N})^3 - \theta_{t,2}^{i,j,k,N} x_t^{i,N}\right) - \theta_{3} (x_t^{i,N} - x_t^{k,N} )\right)
\mathrm{d}t - \mathrm{d}x_t^{i,N}\Bigg] \hspace{-.5mm} \label{eq:IPS_bistable_update2}
\end{align}
\normalsize
where, once again, $\smash{\bar{x}_t^N = \frac{1}{N}\sum_{j=1}^N x_t^{j,N}}$ denotes the empirical mean of the particles. For our first experiment, we will suppose that the true parameter is given by $\theta_0 = (\theta_{0,1},\theta_{0,2},\theta_{0,3})^{\top} = (1.0,2.0,2.0)^{\top}$. Meanwhile, we will consider two values of $\sigma\in\{1.0,2.0\}$. The reason for this is that the mean-field limit of this model exhibits a phase transition \citep[e.g.,][]{gomes2018mean,dawson1983critical}: for values of $\sigma>\sigma_{c}$, the model admits a unique invariant distribution, while for values of $\sigma<\sigma_c$, a continuous phase transition occurs and there exist two stationary states. In our case, the critical noise strength is given by $\sigma_{c}\approx 1.9$, and thus the two considered values of $\sigma\in\{1.0,2.0\}$ place us in both regimes. Similar to before, we simulate trajectories from the IPS with $N\in\{3,10,50\}$ particles and for $T=5000$ iterations, with initial condition $x_0^{i,N}\sim\mathcal{N}(0,1)$. We use a constant learning rate, this time given by $\gamma = (\gamma_1,\gamma_2) = (2\times 10^{-3}, 2\times 10^{-2})^{\top}$. 

Illustrative results for this experiment are reported in Figures~\ref{fig:5} and \ref{fig:6}. We make several observations. First, similar to before, given a suitably chosen value of the learning rate, the sequence of online parameter estimates converges to the true values of the parameters. Second, aside from their transient behaviours, both estimators appear agnostic as to whether $\sigma<\sigma_c$ or $\sigma>\sigma_c$, suggesting that our methodology can be applied even in the absence of a unique invariant distribution. Finally, these results are once more consistent with our theory (i.e., Theorem~\ref{theorem:main-theorem-1-2-finite-n}). To be specific, for each considered value of $N$, the averaged estimator (blue, green) converges to the true parameters $(\theta_{0,1},\theta_{0,2})^{\top}$ as $t\rightarrow\infty$. On the other hand, the non-averaged estimator (orange, red) exhibits a persistent bias as $t\rightarrow\infty$ when $N$ is small (Fig.~\ref{fig:5a} and Fig.~\ref{fig:6a}), which diminishes as $N$ is increased (Fig.~\ref{fig:5c} and Fig.~\ref{fig:6c}). 

\begin{figure}[t!]
  \centering
  \begin{subfigure}{0.325\linewidth}
    \includegraphics[width=\linewidth]{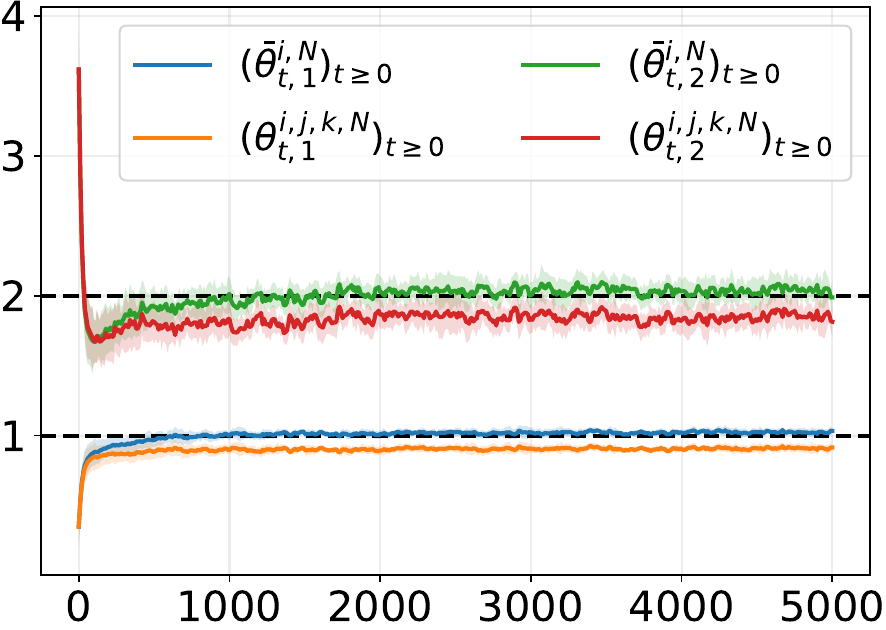}
    \caption{$N=3$.}
    \label{fig:5a}
  \end{subfigure}
  \hfill
  \begin{subfigure}{0.325\linewidth}
    \includegraphics[width=\linewidth]{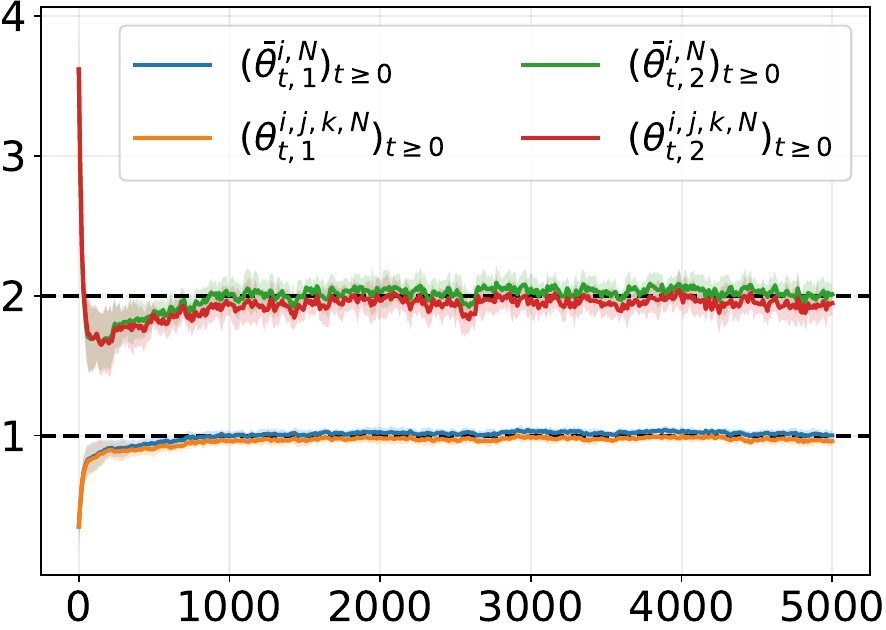}
    \caption{$N=10$.}
    \label{fig:5b}
  \end{subfigure}
  \hfill
  \begin{subfigure}{0.325\linewidth}
    \includegraphics[width=\linewidth]{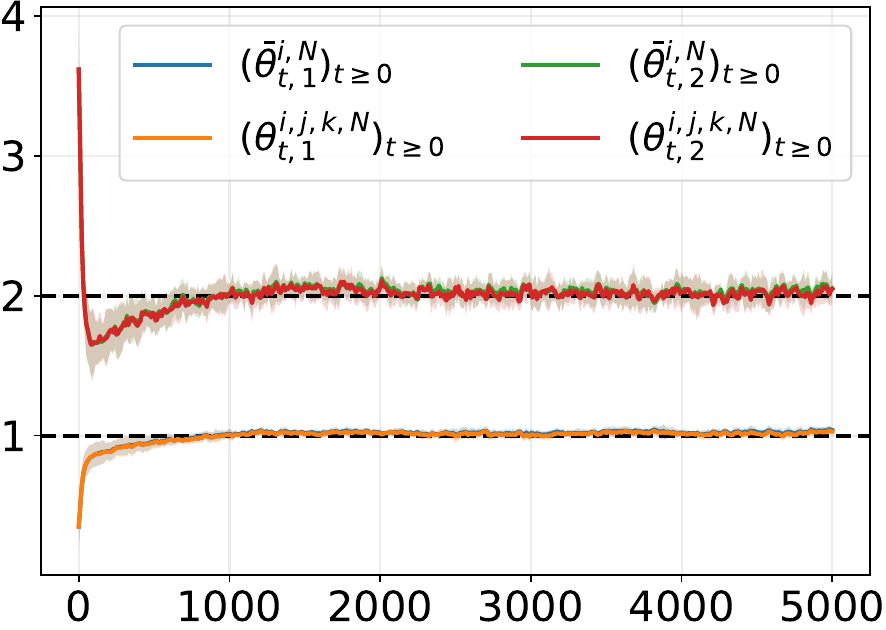}
    \caption{$N=50$.}
    \label{fig:5c}
  \end{subfigure}
  \captionsetup{width=\textwidth}
  \caption{\textbf{Online parameter estimation for a model with double-well confinement potential and quadratic interaction potential}. We plot the sequence of online parameter estimates, as defined by the update equations in \eqref{eq:IPS_bistable_update1} and \eqref{eq:IPS_bistable_update2}. The true parameters (black, dashed) are given by $\smash{\theta_0 = (1.0, 2.0, 2.0)^{\top}}$, with the third of these parameters assumed known. The noise coefficient is given by $\sigma=1.0$. The initial parameter estimates are given by $\smash{\theta_{\mathrm{init},1}\sim \mathcal{U}[0.1,0.6]}$ and $\smash{\theta_{\mathrm{init},2}\sim\mathcal{U}[3.0,4.0]}$.}
  \label{fig:5}
\end{figure}

\begin{figure}[t!]
  \centering
  \begin{subfigure}{0.325\linewidth}
    \includegraphics[width=\linewidth]{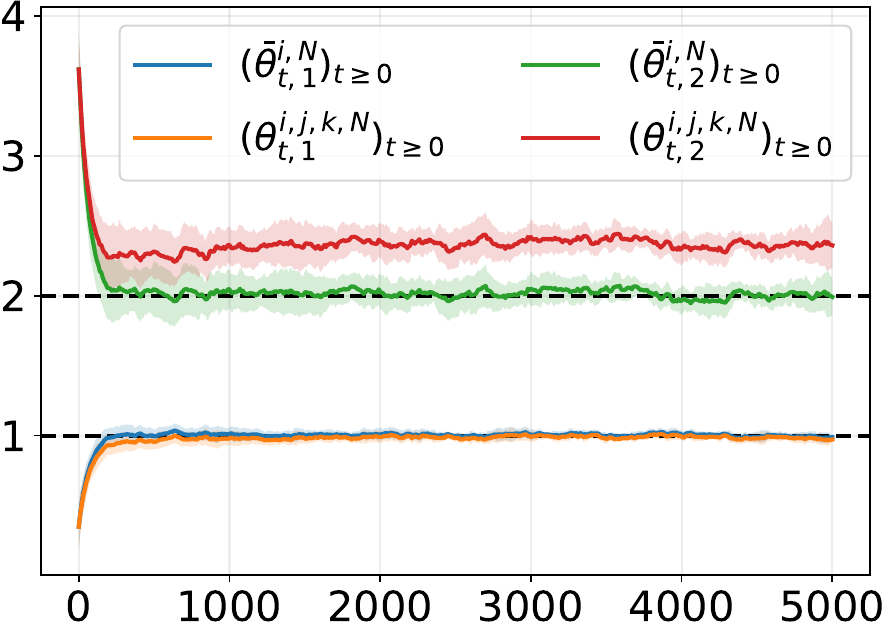}
    \caption{$N=3$.}
    \label{fig:6a}
  \end{subfigure}
  \hfill
  \begin{subfigure}{0.325\linewidth}
    \includegraphics[width=\linewidth]{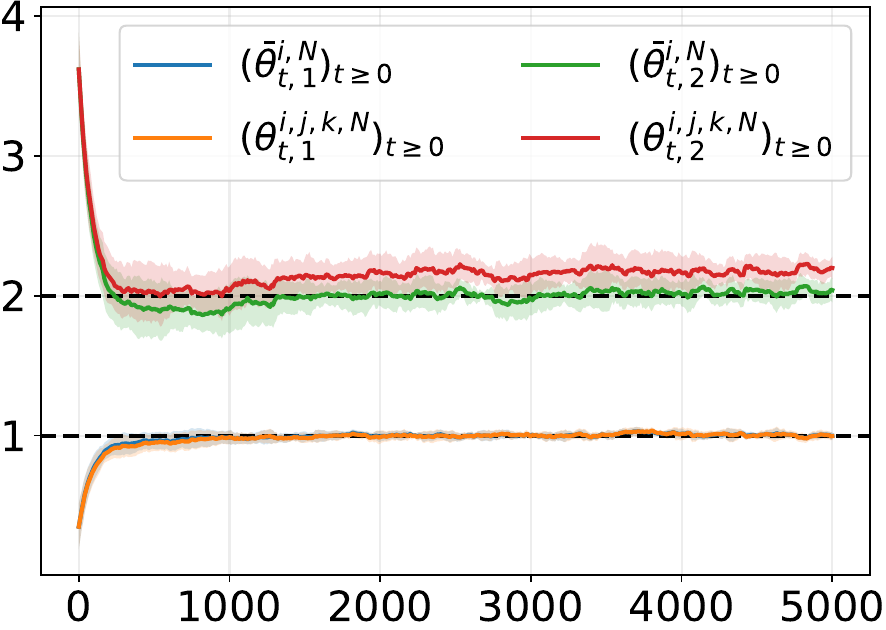}
    \caption{$N=10$.}
    \label{fig:6b}
  \end{subfigure}
  \hfill
  \begin{subfigure}{0.325\linewidth}
    \includegraphics[width=\linewidth]{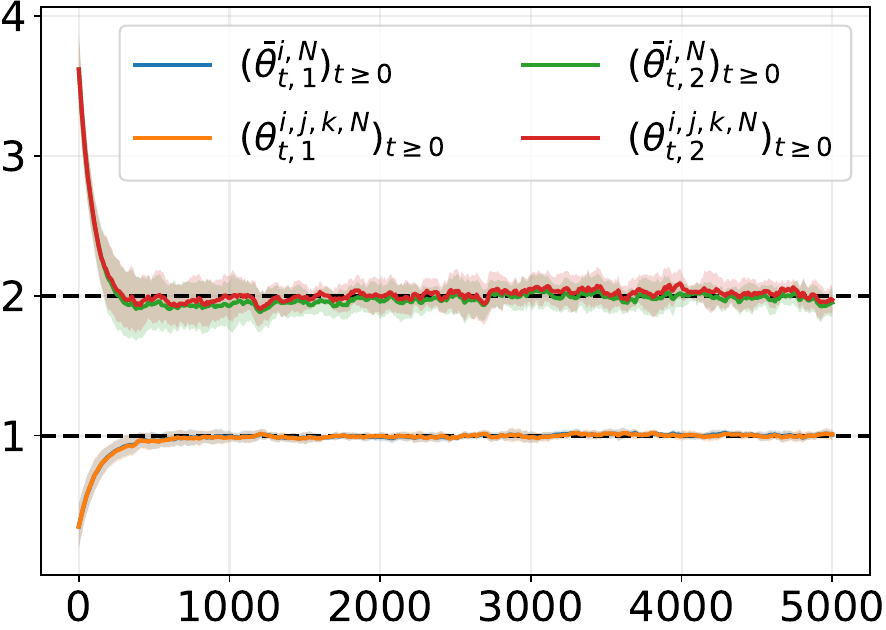}
    \caption{$N=50$.}
    \label{fig:6c}
  \end{subfigure}
  \captionsetup{width=\textwidth}
  \caption{\textbf{Online parameter estimation for a model with double-well confinement potential and quadratic interaction potential}. We plot the sequence of online parameter estimates, as defined by the update equations in \eqref{eq:IPS_bistable_update1} and \eqref{eq:IPS_bistable_update2}. The true parameters (black, dashed) are given by $\smash{\theta_0 = (1.0, 2.0, 2.0)^{\top}}$, with the third of these parameters assumed known. The noise coefficient is given by $\sigma=2.0$. The initial parameter estimates are given by $\smash{\theta_{\mathrm{init},1}\sim \mathcal{U}[0.1,0.6]}$ and $\smash{\theta_{\mathrm{init},2}\sim\mathcal{U}[3.0,4.0]}$.}
  \label{fig:6}
\end{figure}

Once again, it is worth emphasising the importance of well chosen learning rates. In this case, it is the \emph{relative} size of the learning rate(s) for the two parameters that plays a particularly important role. Similar to the first example, this is a consequence of the fact that, for large values of $N$, the objective function is somewhat ill conditioned (see Figure~\ref{fig:6a}). In principle, one could alleviate this by incorporating standard techniques from the optimisation literature into our update equations, e.g., preconditioning, adaptive step-sizes, parameter-free methods, etc. \citep[e.g.,][]{hinton2012neural}. We provide an initial demonstration of one such approach (RMSProp) in Figure~\ref{fig:7}, with a more detailed study of such extensions left to future work. 

\begin{figure}[t!]
  \centering
  \begin{subfigure}{0.45\linewidth}
    \includegraphics[width=\linewidth]{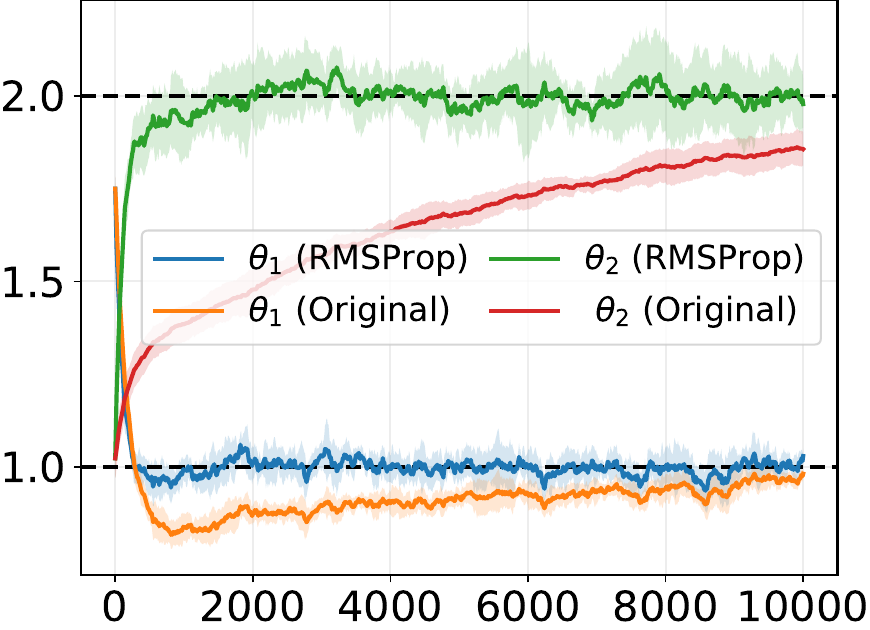}
    \caption{The sequence of online parameter estimates (coloured, solid) and the true parameters (black, dashed), with and without RMSProp.}
    \label{fig:7a}
  \end{subfigure}
  \hfill
  \begin{subfigure}{0.45\linewidth}
    \includegraphics[width=\linewidth]{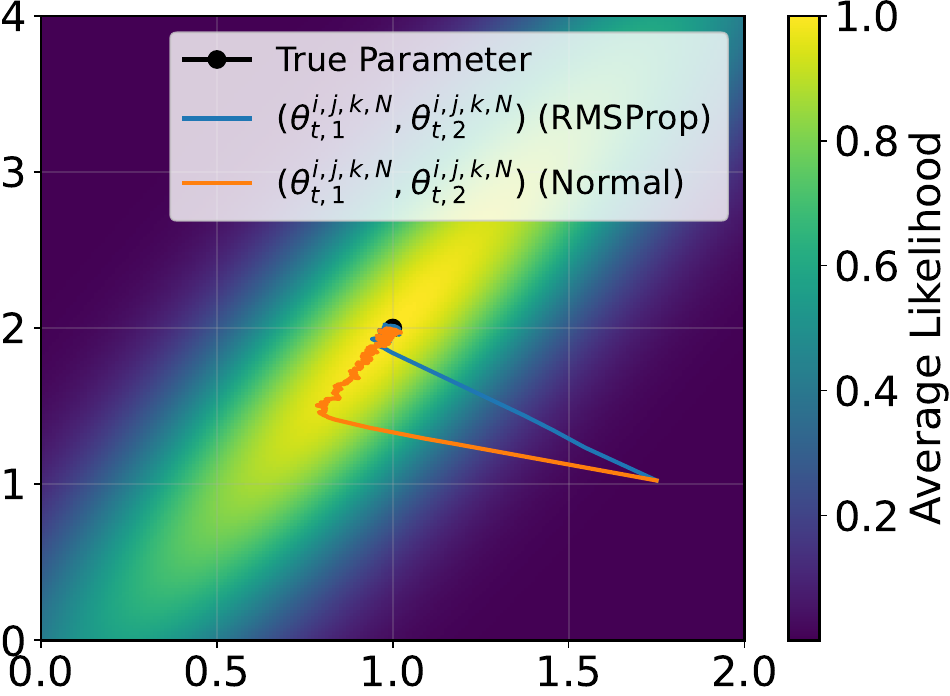}
    \caption{The trajectory of the mean parameter estimates (coloured) and the true parameter (black), overlaid on the asymptotic likelihood function.}
    \label{fig:7b}
  \end{subfigure}
  \captionsetup{width=\textwidth}
  \caption{\textbf{Online parameter estimation for a model with double-well confinement potential and quadratic interaction potential}. We plot the sequence of online parameter estimates $\smash{(\theta_{t,1}^{i,j,k,N},\theta_{t,2}^{i,j,k,N})_{t\geq 0}}$, as defined by the update equation in \eqref{eq:IPS_bistable_update2}, as well as a modified version which incorporates RMSProp \citep[e.g.,][]{hinton2012neural}. The true parameters are once again given by $\smash{\theta_0 = (1.0, 2.0, 2.0)^{\top}}$, with the final parameter assumed known. The noise coefficient is given by $\sigma=2.0$. The initial parameter estimates are now given by $\smash{\theta_{\mathrm{init},1}\sim \mathcal{U}[1.7,1.8]}$ and $\smash{\theta_{\mathrm{init},2}\sim\mathcal{U}[0.9,1.1]}$. In this case, the learning rate is given by $\smash{\gamma=(\gamma_1,\gamma_2) = (2\times 10^{-3},2\times 10^{-3})^{\top}}$.}
  \label{fig:7}
\end{figure}

\subsection{Stochastic FitzHugh--Nagumo Model}
\label{sec:fitzhugh}
We next consider a stochastic FitzHugh--Nagumo model, parametrised by $\theta = (\theta_1,\theta_2,\theta_3, \theta_4)^{\top}\in\mathbb{R}^4$, and defined by
\begin{align}
\mathrm{d}x_t^{\theta,i,N} &= \Big[ \theta_1 \Big(x_t^{\theta,i,N} - \frac{1}{3}(x_t^{\theta,i,N})^3 - y_t^{\theta,i,N}\Big) - \frac{\theta_2}{N}\sum_{j=1}^N (x_t^{\theta,i,N} - x_t^{\theta,j,N})\Big] \mathrm{d}t + \sigma\mathrm{d}w_t^{i,N} \label{eq:fitzhugh_nagumo1_IPS} \\
\mathrm{d}y_t^{\theta,i,N}& = \Big[ x_t^{\theta,i,N} + \theta_3 - \theta_4 y_{t}^{\theta,i,N} \Big] \mathrm{d}t, \label{eq:fitzhugh_nagumo2_IPS} 
\end{align}
This model originates in neuroscience, modelling the evolution of a collection of neurons of FitzHugh--Nagumo type, each being represented by its voltage $x_t^i$ and recovery variable $y_t^i$, and coupled through a linear mean-field interaction which corresponds to a coupling via electrical synapses. We refer to \cite{luccon2021periodicity,baladron2012meanfield} for further details.

\begin{remark}
This model is degenerate, since the noise acts only on the first equation. It is therefore not possible to use Girsanov's theorem to obtain a likelihood. Fortunately, our methodology can still be applied after a minor modification to our objective function. In particular, we will simply no longer weight the inner product in the objective by the inverse of the diffusion coefficient (since this is now undefined). Under standard identifiability assumptions, the average of the resulting contrast function is still uniquely minimised at the true parameter $\theta_0$, and thus this provides a suitable objective function for the statistical learning procedure. In practice, in terms of the parameter update equations, the only difference is that the term $(\sigma\sigma^{\top})^{-1}$ is now replaced by the identity.
\end{remark}

We report illustrative results for these estimators in the case that the first three parameters are to be (jointly) estimated, and the final parameter is known and fixed equal to the ground truth. We assume that the true parameter $\theta_0 = (\theta_{0,1},\theta_{0,2},\theta_{0,3},\theta_{0,4})^{\top}$ has first three components drawn at random according to $\theta_{0,1}\sim\mathcal{U}[0.0,1.0]$, $\theta_{0,2}\sim \mathcal{U}[0.0,1.0]$ and $\theta_{0,3}\sim\mathcal{U}[0.0,1.0]$, with $\theta_{0,4}=1.0$. Meanwhile, the initial parameter estimates are given by $\theta_{1,\mathrm{init}}\sim\mathcal{U}[1.0,2.0]$, $\theta_{2,\mathrm{init}}\sim \mathcal{U}[0.0,1.0]$ and $\theta_{3,\mathrm{init}}\sim\mathcal{U}[0.0,0.5]$. We simulate trajectories from the IPS with $N=50$ particles and for $T=10000$ iterations, with initial condition $x_0^{i,N}\sim\mathcal{N}(0,1)$. We use a constant learning rate, this time given by $\gamma = (\gamma_1,\gamma_2,\gamma_3) = (1\times 10^{-3}, 1\times 10^{-3}, 1\times 10^{-3})^{\top}$. 

\begin{figure}[t!]
  \centering
  \begin{subfigure}{0.325\linewidth}
    \includegraphics[width=\linewidth]{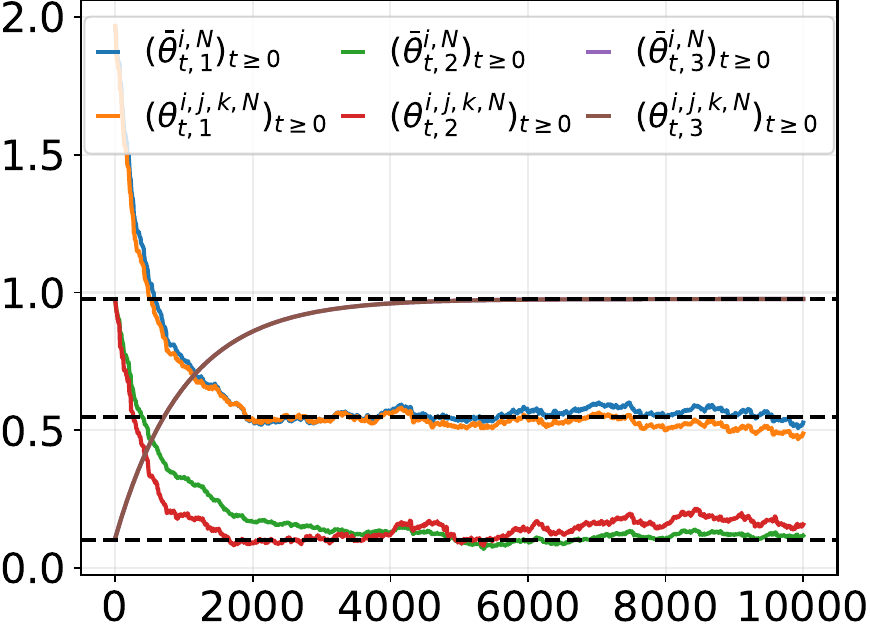}
    \caption{$N=3$ (Trajectory 1).}
    \label{fig:8a}
  \end{subfigure}\hfill
  \begin{subfigure}{0.325\linewidth}
    \includegraphics[width=\linewidth]{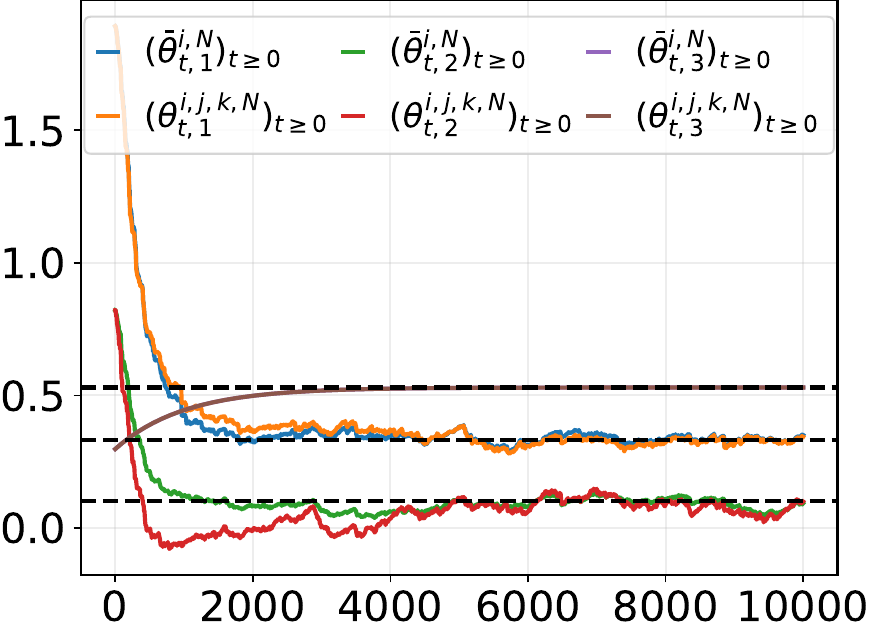}
    \caption{$N=3$ (Trajectory 2).}
    \label{fig:8b}
  \end{subfigure}
  \begin{subfigure}{0.325\linewidth}
    \includegraphics[width=\linewidth]{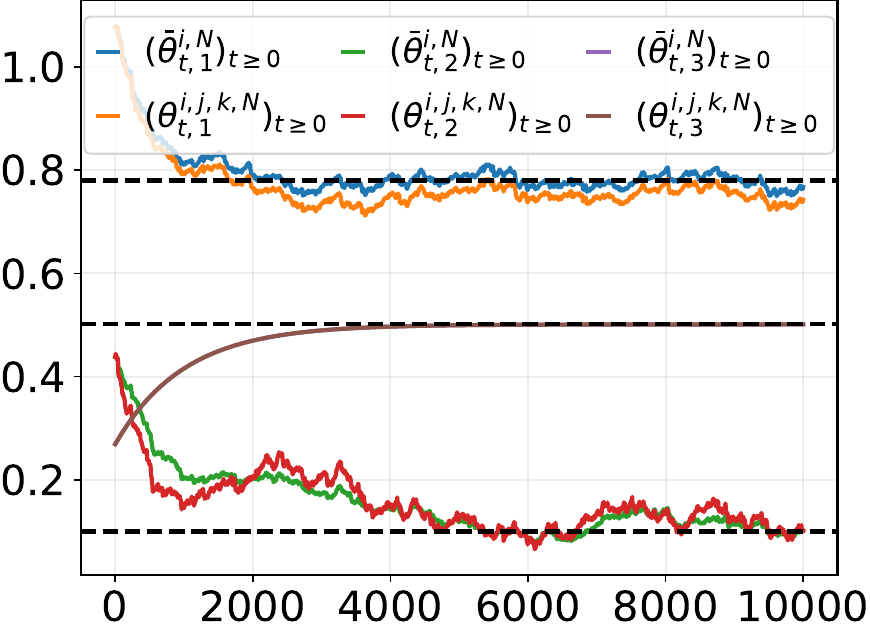}
    \caption{$N=3$ (Trajectory 3).}
    \label{fig:8c}
  \end{subfigure}
  ~\\[3mm]
    \begin{subfigure}{0.325\linewidth}
    \includegraphics[width=\linewidth]{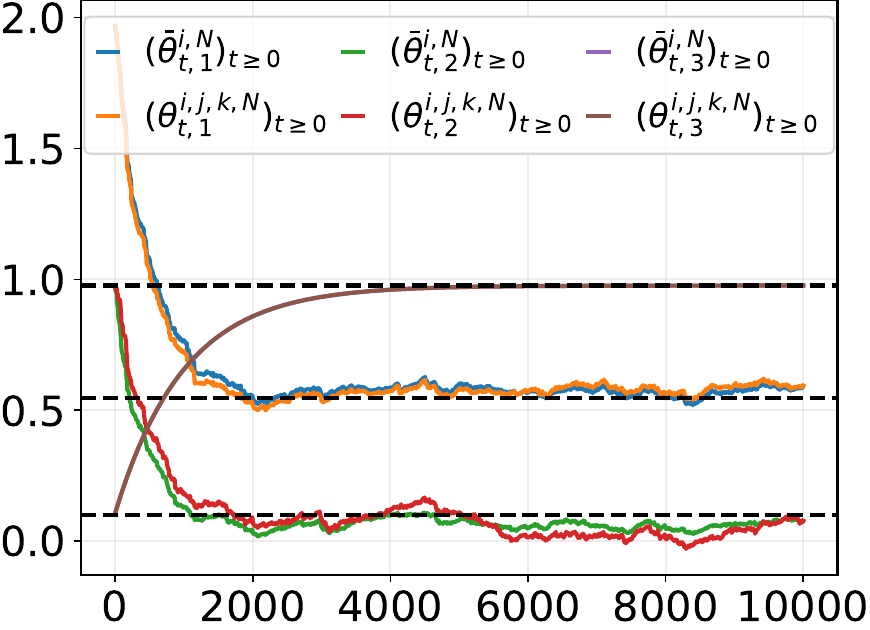}
    \caption{$N=50$ (Trajectory 1).}
    \label{fig:8d}
  \end{subfigure}\hfill
  \begin{subfigure}{0.325\linewidth}
    \includegraphics[width=\linewidth]{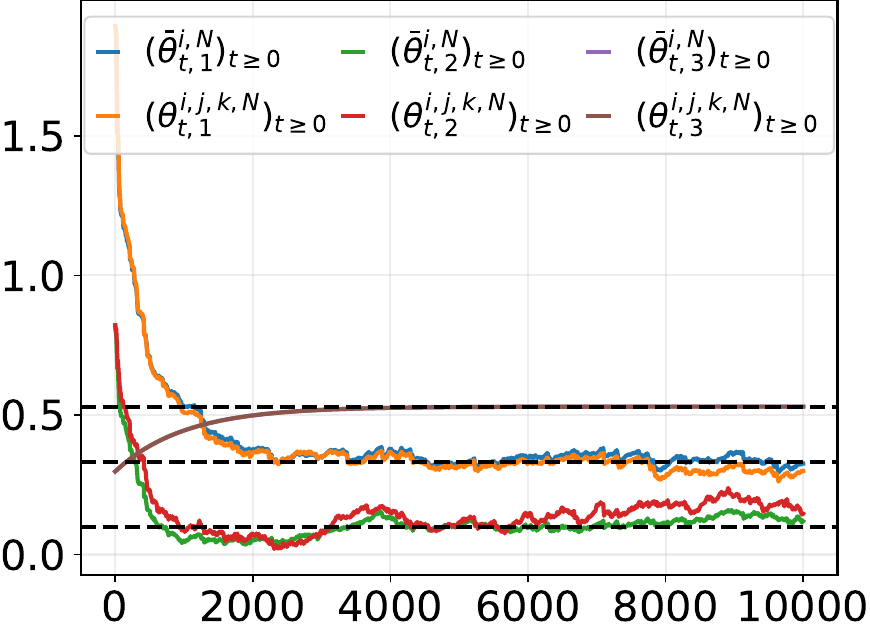}
    \caption{$N=50$ (Trajectory 2).}
    \label{fig:8e}
  \end{subfigure}
  \begin{subfigure}{0.325\linewidth}
    \includegraphics[width=\linewidth]{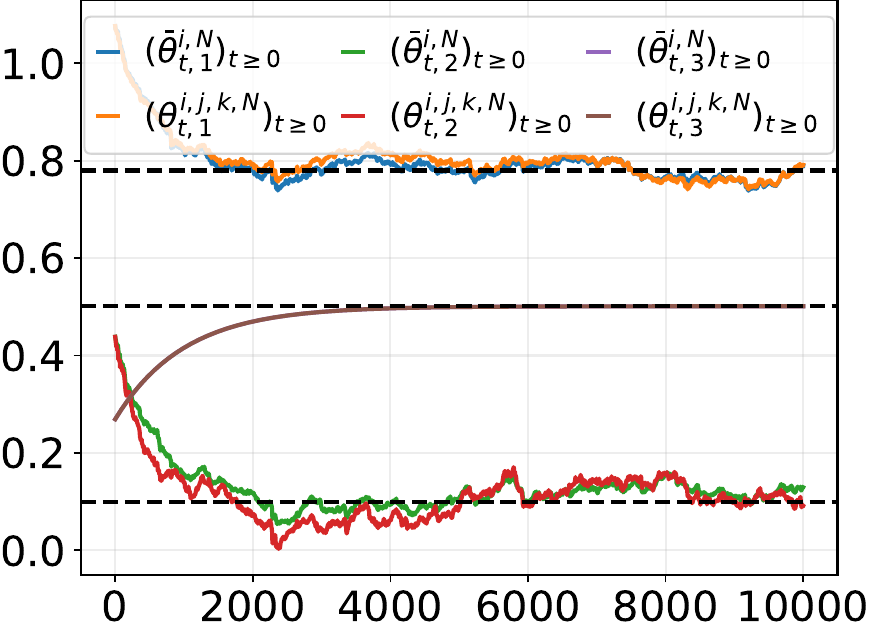}
    \caption{$N=50$ (Trajectory 3).}
    \label{fig:8f}
  \end{subfigure}
  \captionsetup{width=\textwidth}
  \caption{\textbf{Online parameter estimation for the stochastic FitzHugh--Nagumo model}. We plot three trajectories of the online parameter estimates $\smash{(\bar{\theta}_t^{i,N})_{t\geq 0}}$ and $\smash{(\theta_t^{i,j,k,N})_{t\geq 0}}$, each corresponding to a different random initial condition, true parameter value, and random seed, for $N\in\{3,50\}$. The true parameters are given by $\theta_{0,1}\sim\mathcal{U}[0.0,1.0]$, $\theta_{0,2}\sim \mathcal{U}[0.0,1.0]$ and $\theta_{0,3}\sim\mathcal{U}[0.0,1.0]$, and $\theta_{0,4}=1.0$, with the final parameter assumed known. The initial parameter estimates are given by $\smash{\theta_{\mathrm{init},1}\sim \mathcal{U}[1.0,2.0]}$,  $\smash{\theta_{\mathrm{init},2}\sim\mathcal{U}[0.0,1.0]}$, and $\smash{\theta_{\mathrm{init},3}\sim\mathcal{U}[0.0,0.5]}$.}
  \label{fig:8}
\end{figure}

Our results are shown in Figure~\ref{fig:8}. On this occasion, we report three representative individual trajectories of the online parameter estimates, corresponding to three different initial parameter values, true parameter values, and random seeds. In this case there is little to distinguish between the performance of the ``averaged'' estimator and the ``non-averaged'' estimator, even at the level of the individual trajectories. In addition, both the averaged and the non-averaged estimators converge to the true parameter values, regardless of the number of particles. Given our theoretical results (i.e., Theorem~\ref{theorem:main-theorem-1-2}), this suggests that there is little or no disagreement between the true parameter value $\theta_0$ and the minimiser $\theta_{0}^{i,j,k,N}$ of the pseudo log-likelihood $\mathcal{L}^{i,j,k,N}$.

\subsection{Stochastic Kuramoto Model}
We next consider the stochastic Kuramoto model, also known as the Kuramoto--Shinomoto--Sakaguchi model \citep[e.g.,][]{bertini2010dynamical,kuramoto1981rhythms,acebron2005kuramoto,sakaguchi1988phase}. In particular, we consider the IPS defined according to
\begin{equation}
    \mathrm{d}x_t^{\theta,i,N} = -\frac{\theta}{N}\sum_{j=1}^N \sin(x_t^{\theta,i,N}-x_t^{\theta,j,N})\,\mathrm{d}t + \sigma\mathrm{d}w_t^{i,N}.
\end{equation}
where $\theta\in\mathbb{R}$ is the coupling strength. This system of interacting particles models the synchronisation of noisy oscillators interacting through their phases, and finds application in various fields including physics, chemistry, and biology (see, e.g., \cite{acebron2005kuramoto} and references therein). Similar to previous examples, the mean-field limit of this model exhibits a phase transition \citep[e.g.,][]{bertini2010dynamical}. In particular, when $\sigma>\sigma_c$, for some critical noise strength $\sigma_c$, the noise dominates and there is a unique invariant distribution (i.e., the uniform distribution). On the other hand, when $\sigma<\sigma_c$, there exists a family of non-trivial coherent equilibria, and the population tends to synchronise. Equivalently, given a fixed value of $\sigma>0$: there is a unique invariant distribution when $\theta<\theta_c$, and multiple invariant distributions when $\theta>\theta_c$, for some critical coupling strength $\theta_c:=\sigma^2$. 

We illustrate the performance of our estimators in Figure~\ref{fig:10}. Similar to before, we simulate trajectories from the IPS with $N=50$ particles and for $T=10000$ iterations. Now, however, we consider two \emph{time-varying} specifications of the true parameter:
\begin{align}
    \theta_{0,t} &= \begin{cases}
    \theta_{0,1}, & t\in[0,5000),\\
    \theta_{0,2}, & t\in[5000,10000],
    \end{cases}
    \qquad \text{or} \qquad \theta_{0,t} = \theta_{0,1} + (\theta_{0,2} - \theta_{0,1}) \frac{t}{10000}, \label{eq:true-param-kuramoto}
\end{align}
where $\theta_{0,1} = 1.5$ and $\theta_{0,2} = 0.2$. We also assume that $\sigma=1.0 \implies \theta_c = \sigma^2 = 1.0$. Thus, in particular, for certain values of $t$, the coupling strength is above its critical value (since $\theta_{0,1}>\theta_{c}$), while at others it is below the critical value (since $\theta_{0,2}<\theta_{c}$). While, strictly speaking, this scenario is outside the scope of our theoretical results, it demonstrates another advantage of our online estimation procedure in comparison to a batch or offline approach. In particular, our estimators are able to accurately track changes in the true parameter in real time. 

\begin{figure}[t!]
  \centering
  \begin{subfigure}{0.42\linewidth}
    \includegraphics[width=\linewidth]{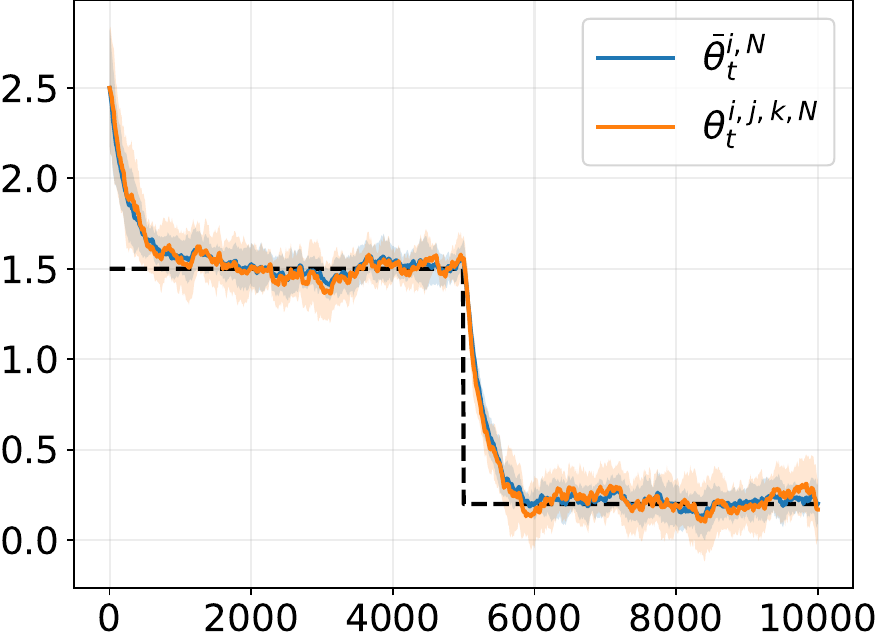}
    \caption{Changepoint.}
    \label{fig:10a}
  \end{subfigure}
  \hspace{10mm}
  \begin{subfigure}{0.42\linewidth}
    \includegraphics[width=\linewidth]{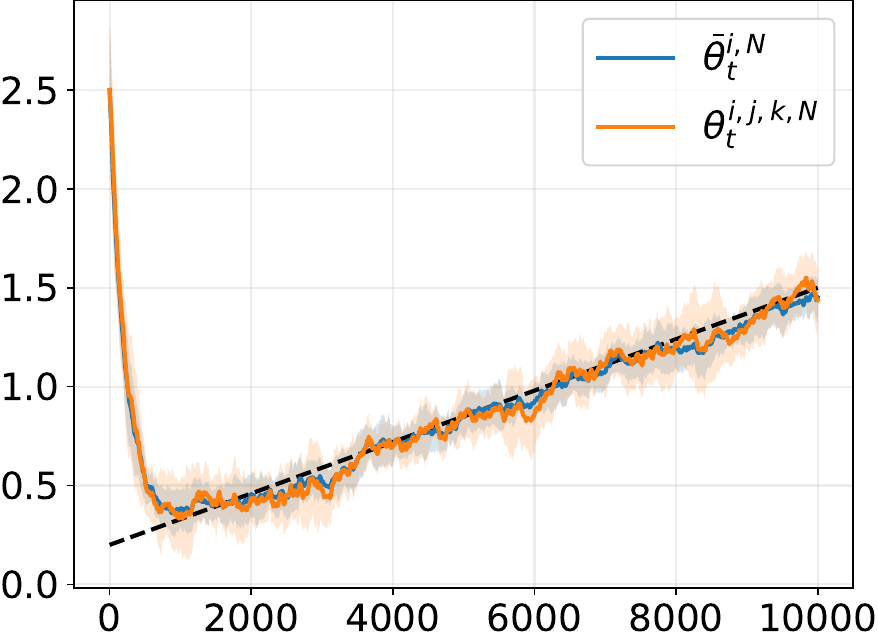}
    \caption{Linear Interpolation.}
    \label{fig:10b}
  \end{subfigure}
  \captionsetup{width=\textwidth}
  \caption{\textbf{Online parameter estimation for the stochastic Kuramoto model}. We plot the sequence of online parameter estimates $\smash{(\bar{\theta}_t^{i,N})_{t\geq 0}}$ and $\smash{(\theta_t^{i,j,k,N})_{t\geq 0}}$. The true time-varying parameter is given by the two specifications in \eqref{eq:true-param-kuramoto}. Meanwhile, the initial parameter estimate is given by $\smash{\theta_{\mathrm{init}}\sim \mathcal{U}[2,3]}$.}
  \label{fig:10}
\end{figure}

\subsection{Stochastic Cucker--Smale Model}
Our next model is a stochastic Cucker--Smale flocking model \citep[e.g.,][]{cucker2007emergent,cucker2007mathematics,cucker2008flocking,cattiaux2018stochastic,ahn2010stochastic}, parametrised by $\theta = (\theta_1,\theta_2,\theta_3)^{\top}\in\mathbb{R}^2\times\mathbb{R}_{+}$, and defined according to
\begin{align}
\mathrm{d}x_t^{\theta,i,N} & = v_t^{\theta,i,N}\mathrm{d}t \label{eq:cucker_smale_IPS1} \\
\mathrm{d}v_t^{\theta,i,N} & = -\Big[\theta_1 x_t^{\theta,i,N} + \frac{\theta_2}{N} \sum_{j=1}^N \psi^{ij}_t(\theta_3)(v_t^{\theta,i,N} - v_t^{\theta,j,N}) \Big]\mathrm{d}t + \sigma\mathrm{d}w_t^{i,N}, \label{eq:cucker_smale_IPS2}
\end{align}
where $\psi_t^{ij}$ is a non-negative function known as the \emph{communication rate}, which in this case we define according to $\smash{\psi^{ij}_t(\theta_3) = \psi(\theta_3,\|x_t^{\theta,i,N}-x_t^{\theta,j,N}\|^2)}$ with $\smash{\psi(\theta_3,u) = (1+u)^{-\theta_3}}$. This model, which originates in \cite{cucker2007emergent,cucker2007mathematics}, is intended to describe the self-organisation of individuals within a population, each individual being represented by its position and velocity $(x_t^{i},v_t^{i})\in\mathbb{R}^d\times\mathbb{R}^d$. Similar to the FitzHugh--Nagumo model (see Section \ref{sec:fitzhugh}), the stochastic Cucker--Smale model is degenerate, since the noise acts only on the second variable. We thus proceed as in Section \ref{sec:fitzhugh}, replacing the weighting with respect to the diffusion coefficient with the identity. 

We report illustrative results for this model in Figure~\ref{fig:11}. Once more, we investigate numerically the performance of the estimators as a function of the number of particles $N$. In this case, we assume that the true parameters are given by $\theta_{0} = (\theta_{0,1},\theta_{0,2},\theta_{0,3})^{\top} =(0.2,1.0,0.5)^{\top}$. We estimate the parameters $\theta_2$ and $\theta_3$ separately, with the other parameters fixed and equal to the true value. The initial parameter estimates are then given by $\theta_{2,\mathrm{init}}\sim\mathcal{U}[2,3]$ and $\theta_{3,\mathrm{init}}\sim\mathcal{U}[0,0.2]$. We simulate trajectories from the IPS with $N\in\{3,5,50\}$ particles and for $T=5000$ iterations, with initial conditions $\smash{x_0^{i,N}\sim\mathcal{N}(0,1)}$ and $\smash{v_{0}^{i,N}\sim\mathcal{N}(0,1)}$. Finally, we use constant learning rates, given by $\smash{\gamma = (\gamma_2,\gamma_3)^{\top} = (0.01,0.005)^{\top}}$.

\begin{figure}[t!]
  \centering
  \begin{subfigure}{0.32\linewidth}
    \includegraphics[width=\linewidth]{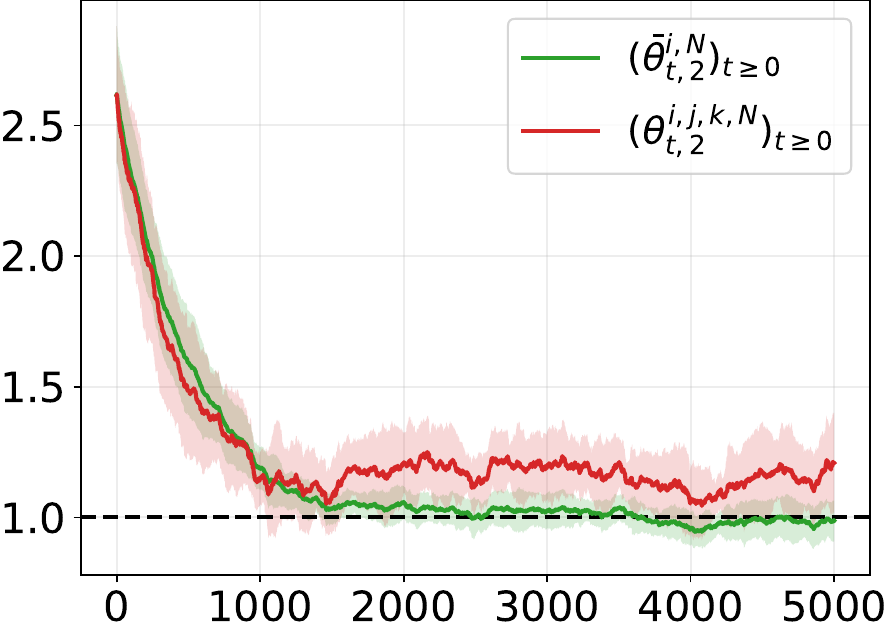}
    \caption{$N=3$.}
    \label{fig:11a}
  \end{subfigure}
  \hfill
  \begin{subfigure}{0.32\linewidth}
    \includegraphics[width=\linewidth]{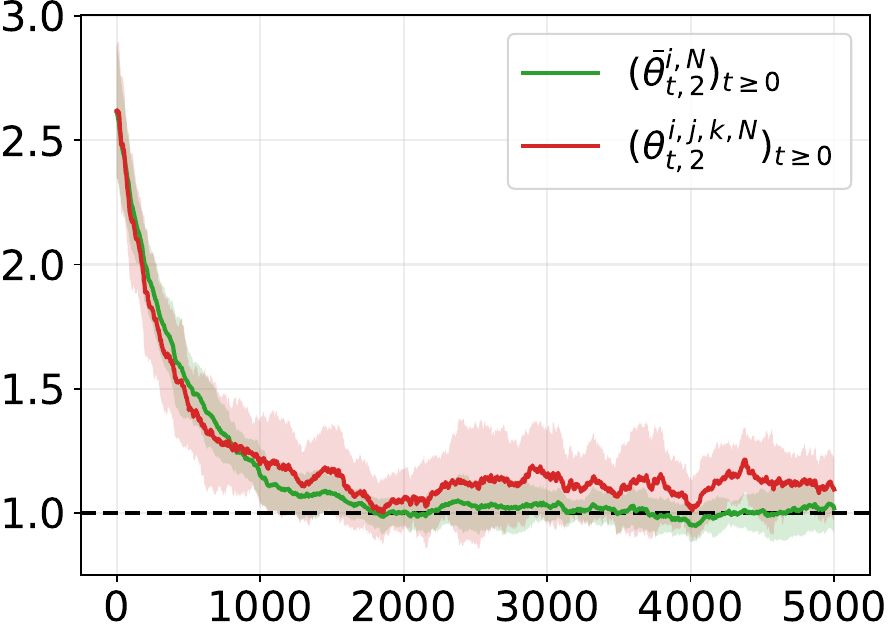}
    \caption{$N=5$.}
    \label{fig:11b}
  \end{subfigure}
  \hfill
  \begin{subfigure}{0.32\linewidth}
    \includegraphics[width=\linewidth]{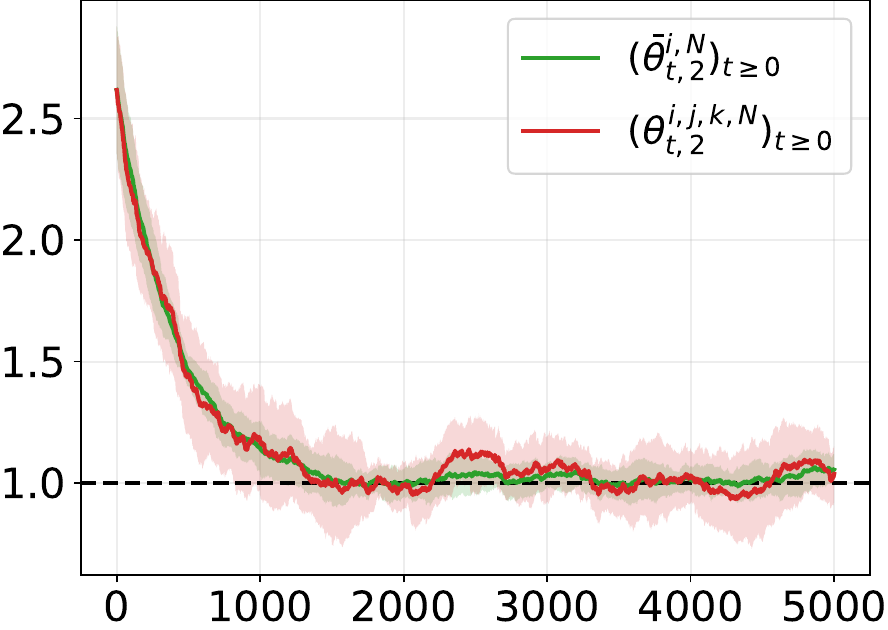}
    \caption{$N=50$.}
    \label{fig:11c}
  \end{subfigure}
  \hfill
  \vspace{3mm}
  \begin{subfigure}{0.32\linewidth}
    \includegraphics[width=\linewidth]{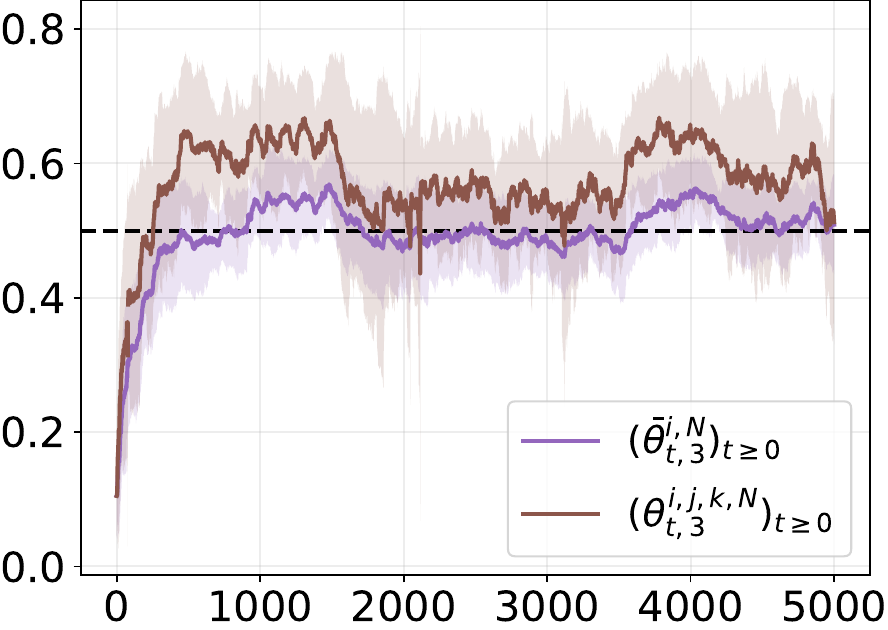}
    \caption{$N=3$.}
    \label{fig:11d}
  \end{subfigure}
  \hfill
  \begin{subfigure}{0.32\linewidth}
    \includegraphics[width=\linewidth]{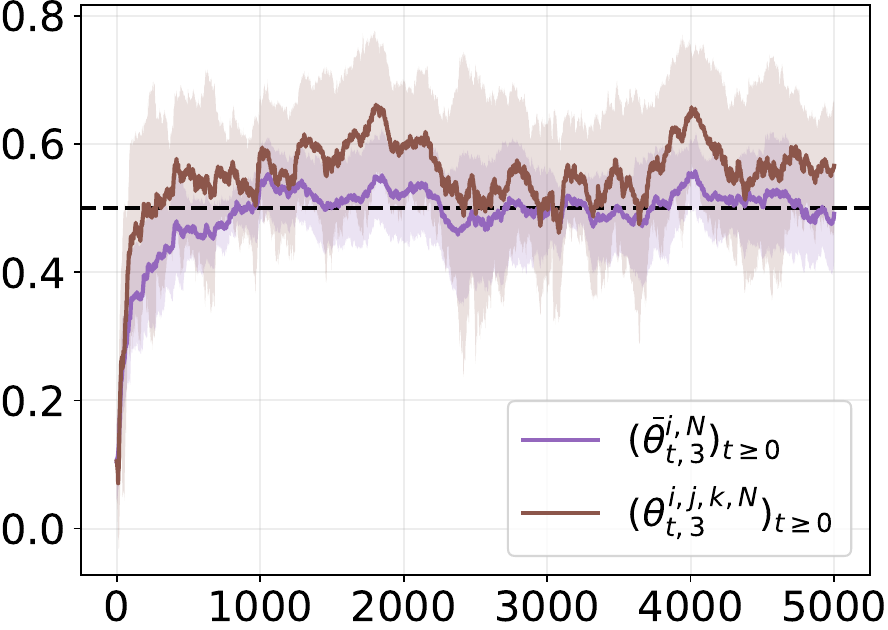}
    \caption{$N=5$.}
    \label{fig:11e}
  \end{subfigure}
  \hfill
  \begin{subfigure}{0.32\linewidth}
    \includegraphics[width=\linewidth]{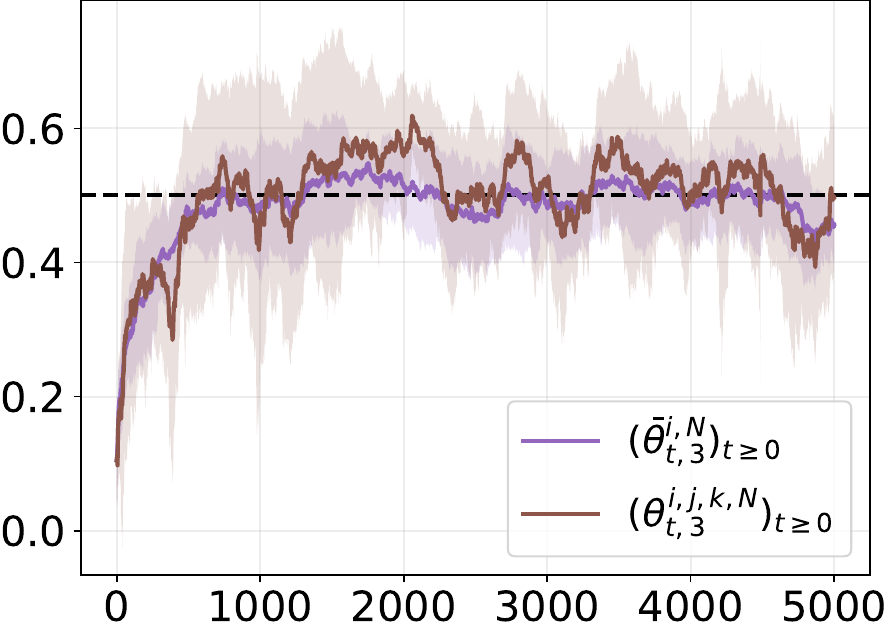}
    \caption{$N=50$.}
    \label{fig:11f}
  \end{subfigure}
  \hfill
  \captionsetup{width=\textwidth}
  \caption{\textbf{Online parameter estimation for the stochastic Cucker--Smale model}. We plot the sequence of online parameter estimates $\smash{(\bar\theta_{t,2}^{i,N})_{t\geq 0}}$ and $\smash{(\theta_{t,2}^{i,j,k,N})_{t\geq 0}}$ (top panels), and $\smash{(\bar\theta_{t,3}^{i,N})_{t\geq 0}}$ and $\smash{(\theta_{t,3}^{i,j,k,N})_{t\geq 0}}$ (bottom panels), for $N\in\{3,5,50\}$. The true parameters are given by $\smash{\theta_{0} = (\theta_{0,1},\theta_{0,2},\theta_{0,3})^{\top} =(0.2,1.0,0.5)^{\top}}$. The initial parameter estimates are given by $\smash{\theta_{\mathrm{init},2}\sim \mathcal{U}[2,3]}$ and $\smash{\theta_{\mathrm{init},3}\sim\mathcal{U}[0,0.2]}$.}
  \label{fig:11}
\end{figure}

Our numerical results once more highlight the different behaviour of the two online estimators. In particular, the performance of the averaged estimator (green, purple) is insensitive to the number of particles in the IPS. This is true both for the confinement parameter (results omitted), and for both of the interaction parameters. By contrast, the asymptotic bias of the non-averaged estimator (red, brown) decreases as the number of particles increases. This is evident in Figure~\ref{fig:11}, where the non-averaged estimator overestimates the true interaction parameter(s) when the number of particles in the IPS is small (Fig.~\ref{fig:11a} - \ref{fig:11b} and Fig.~\ref{fig:11d} - \ref{fig:11e}), but this asymptotic bias vanishes when the number of particles is sufficiently large (Fig.~\ref{fig:11c} and Fig.~\ref{fig:11f}). It is worth emphasising that this performance improvement is not a function of the number of \emph{observed} particles but rather of the number of particles in the underlying data generating process.

\subsection{Mean-Field \texorpdfstring{$\frac{3}{2}$}{3/2} Stochastic Volatility Model}
Finally, we consider a one-dimensional $\frac{3}{2}$ stochastic volatility model \citep[e.g.,][]{kumar2022wellposedness}, parametrised by $\theta=(\theta_1,\theta_2,\theta_3)^{\top}\in\mathbb{R}^3$ and $\eta:=\eta_1 \in\mathbb{R}_{+}$, which can be written as
\begin{equation}
\mathrm{d}x_t^{\theta,\eta,i,N} = - \Big[ x_t^{\theta,\eta,i,N} (\theta_1 |x_t^{\theta,\eta,i,N}| - \theta_2) + \theta_3 \frac{1}{N}\sum_{j=1}^N (x_t^{\theta,\eta,i,N} -x_t^{\theta,\eta,j,N}) \Big]\mathrm{d}t + \eta_1 |x_t^{\theta,\eta,i,N}|^{\frac{3}{2}}\mathrm{d}w_t^{i,N} \label{eq:stochastic_volatility_IPS} \vspace{-2mm}
\end{equation}
where, once again, $(w_t^{i,N})_{t\geq 0}^{i\in[N]}$ are a set of independent standard Brownian motions. This model, which represents a reparametrisation of the one introduced in  \cite[][Section 5]{kumar2022wellposedness}, can be viewed as the mean-field extension of the well-known $\frac{3}{2}$ model, which is used for pricing VIX options and modelling certain (non-affine) stochastic volatility processes \citep[e.g.,][]{goard2013stochastic}.

\begin{remark}
    In this model, we would like to estimate unknown parameters appearing in both the drift and the diffusion. It is now no longer possible to use Girsanov's theorem to obtain a likelihood, since the path measures corresponding to different values of the parameters appearing in the diffusion are, in general, mutually singular. Nonetheless, subject to a small modification, it is still possible to apply our methodology; see also \cite[][Section 4]{sirignano2017stochastic}. Consider, in the general case, an IPS parametrised by $\theta\in\mathbb{R}^p$ and $\eta\in\mathbb{R}^m$ of the form
    \begin{equation}
        \mathrm{d}x_t^{\theta,\eta,i,N} = \Big[\underbrace{\frac{1}{N}\sum_{j=1}^Nb(\theta,x_t^{\theta,\eta,i,N},x_t^{\theta,\eta,j,N})}_{B[\theta,x_t^{\theta,\eta,i,N},\mu_t^{\theta,\eta,N}]}\Big]\mathrm{d}t + \Big[\underbrace{\frac{1}{N}\sum_{j=1}^N \sigma(\eta,x_t^{\theta,\eta,i,N},x_t^{\theta,\eta,j,N})}_{\Sigma(\eta,x_t^{\theta,\eta,i,N},\mu_t^{\theta,\eta,N})}\Big]\mathrm{d}w_t^{i,N}. \vspace{-2mm}
    \end{equation}
    We will suppose, as before, that there exists true but unknown static parameters $\theta_0\in\Theta\subseteq\mathbb{R}^p$ and $\eta_0\in\mathcal{H}\subseteq\mathbb{R}^m$ which generate the observed paths. To estimate the parameters, we can now simply consider the modified objective functions
    \begin{align}
    \tilde{\mathcal{L}}(\theta) &:= \int_{\mathbb{R}^d} \frac{1}{2}\|B(\theta,x,\pi_{\theta_0,\eta_0})-B(\theta_0,x,\pi_{\theta_0,\eta_0})\|^2 \pi_{\theta_0,\eta_0}(\mathrm{d}x) \\
    \tilde{\mathcal{J}}(\eta) &:= \int_{\mathbb{R}^d} \frac{1}{2}\|\Sigma(\eta,x,\pi_{\theta_0,\eta_0})\Sigma^{\top}(\eta,x,\pi_{\theta_0,\eta_0}) -\Sigma(\eta_0,x,\pi_{\theta_0,\eta_0})\Sigma^{\top}(\eta_0,x,\pi_{\theta_0,\eta_0})\|^2 \pi_{\theta_0,\eta_0}(\mathrm{d}x)
    \end{align}
    Following similar steps to before (see Sections~\ref{sec:algorithm-1} - \ref{sec:algorithm-2}), we can define online estimators based on these objectives. For the drift parameters, we simply obtain an unweighted version of the update equations in \eqref{eq:IPS_update1-a} and \eqref{eq:IPS_update2-a-v0}, with $(\sigma\sigma^{\top})^{-1}$ replaced by the identity. Meanwhile, for the diffusion parameters, we obtain
    \begin{align}
    \label{eq:diffusion-update-average}
        \mathrm{d}\bar\eta_t^{i,N}&= -\delta_t \nabla_{\eta}\left(\Sigma\Sigma^{\top}(\bar\eta_t^{i,N},x_t^{i,N},\mu_t^N)\right)  \left( \Sigma\Sigma^{\top}(\bar\eta_t^{i,N},x_t^{i,N},\mu_t^N) \mathrm{d}t - \mathrm{d}\langle x^{i,N}, x^{i,N}\rangle_t \right) \\
    \label{eq:diffusion-update-non-average}
        \mathrm{d}\eta_t^{i,j,k,N}&= -\delta_t \nabla_{\eta}\left(\sigma\sigma^{\top}(\eta_t^{i,j,k,N},x_t^{i,N},x_t^{j,N}) \right) \left( \sigma\sigma^{\top}(\eta_t^{i,j,k,N},x_t^{i,N},x_t^{k,N}) \mathrm{d}t - \mathrm{d}\langle x^{i,N}, x^{i,N}\rangle_t \right) 
    \end{align}
    where $(\delta_t)_{t\geq 0}$ is the learning rate for the diffusion parameters. Arguing as before, we can establish results for these estimators analogous to those proved in the previous sections (see Sections~\ref{sec:main-results-convergence} - \ref{sec:main-results-clt}).
\end{remark}

We can now write down the online parameter estimate(s) for the mean-field $\frac{3}{2}$ stochastic volatility model. Based on an unweighted version of \eqref{eq:IPS_update1-a} or \eqref{eq:IPS_update2-a-v0}, for the drift parameters we have
\begin{align}
\mathrm{d}\bar\theta_t^{i,N} &= 
\gamma_t
\begin{bmatrix} -x_t^{i,N} |x_t^{i,N}|  \\ x_t^{i,N} \\  -(x_t^{i,N} - \bar{x}_t^N)\end{bmatrix}
\big[\mathrm{d}x_t^{i,N} - 
\big(-x_t^{i,N}(\bar\theta_{t,1}^{i,N}|x_t^{i,N}| - \bar\theta_{t,2}^{i,N}) - \bar\theta_{t,3}^{i,N} (x_t^{i,N} - \bar{x}_t^N) \big)
\mathrm{d}t \big] \label{eq:stochastic_volatility_update1_ips} \\
\mathrm{d}\theta_t^{i,j,k,N} &= 
\gamma_t 
\begin{bmatrix} -x_t^{i,N} |x_t^{i,N}|  \\ x_t^{i,N} \\  -(x_t^{i,N} - {x}_t^{j,N})\end{bmatrix}
\big[\mathrm{d}x_t^{i,N} - \big(-x_t^{i,N}(\theta_{t,1}^{i,j,k,N}|x_t^{i,N}| - \theta_{t,2}^{i,j,k,N}) - \theta_{t,3}^{i,j,k,N} (x_t^{i,N} - {x}_t^{k,N}) \big)
\mathrm{d}t \big]. \label{eq:stochastic_volatility_update2_ips}
\end{align}

Meanwhile, according to \eqref{eq:diffusion-update-average} and \eqref{eq:diffusion-update-non-average}, the update equation for the diffusion parameter is the same in both cases. In particular, writing $\eta_t^{i,N}$ for either $\bar\eta_t^{i,N}$ or $\eta_t^{i,j,k,N}$, we have
\begin{align}
\mathrm{d}\eta_t^{i,N} &=  \delta_t \big[2 \eta_t^{i,N} |x_t^{i,N}|^{3}\big]  \big[\mathrm{d}\langle x^{i,N}, x^{i,N}\rangle_t - (\eta_t^{i,N})^2 |x_t^{i,N}|^{3} \mathrm{d}t \big]. \label{eq:eta_update_ips}
\end{align}
We report illustrative results for this model in Figure~\ref{fig:12}. We assume that the true parameters are given by $\smash{\theta_{0} = (\theta_{0,1},\theta_{0,2},\theta_{0,3})^{\top} =(2.7,2.3,1.0)^{\top}}$ and $\eta_{0} = 0.7$, while the initial parameter estimates are given by $\smash{\theta_{\mathrm{init},1}\sim \mathcal{U}[1.0,1.5]}$, $\smash{\theta_{\mathrm{init},2}\sim\mathcal{U}[3.5,4.0]}$, $\smash{\theta_{\mathrm{init},3}\sim\mathcal{U}[0.0,0.2]}$ and $\eta_{\mathrm{init}}\sim\mathcal{U}[1.5,2.0]$. We simulate trajectories from the IPS with $N\in\{3,10,50\}$ particles and $T=5000$ iterations, with $\smash{x_0^{i,N}\sim\mathcal{N}(0,1)}$. Finally, we use constant learning rates, given by $\smash{\gamma = (\gamma_1,\gamma_2,\gamma_3)^{\top} = (0.01,0.01,0.05)^{\top}}$ and $\smash{\delta = 0.01}$, respectively. 

\begin{figure}[b!]
  \centering
  \begin{subfigure}{0.31\linewidth}
    \includegraphics[width=\linewidth]{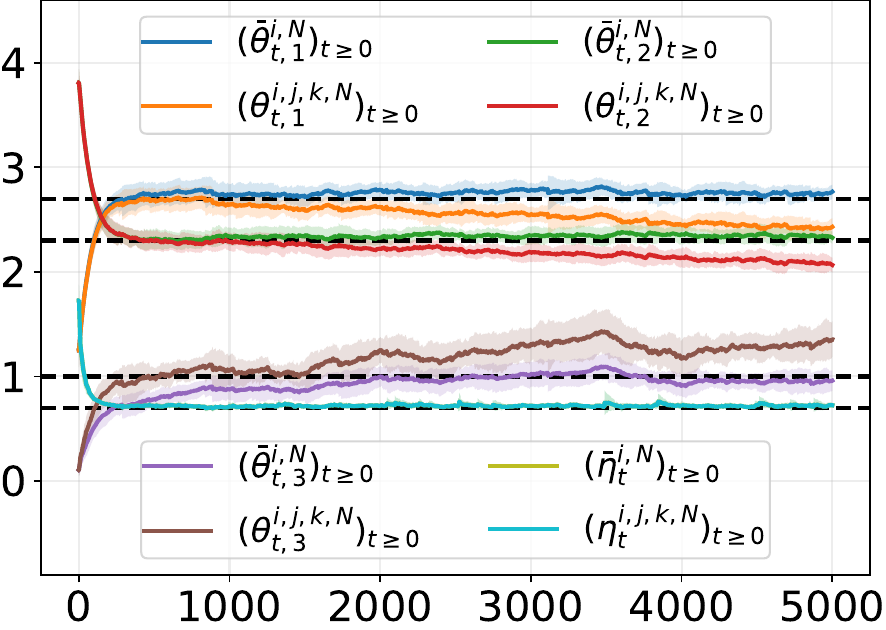}
    \caption{$N=3$.}
    \label{fig:12a}
  \end{subfigure}
  \hfill
  \begin{subfigure}{0.31\linewidth}
    \includegraphics[width=\linewidth]{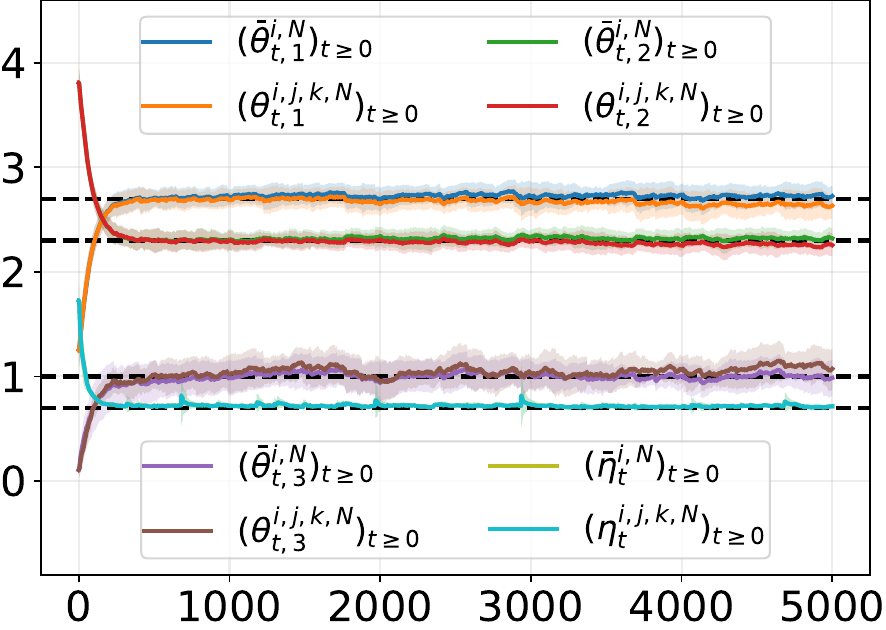}
    \caption{$N=10$.}
    \label{fig:12b}
  \end{subfigure}
  \hfill
  \begin{subfigure}{0.31\linewidth}
    \includegraphics[width=\linewidth]{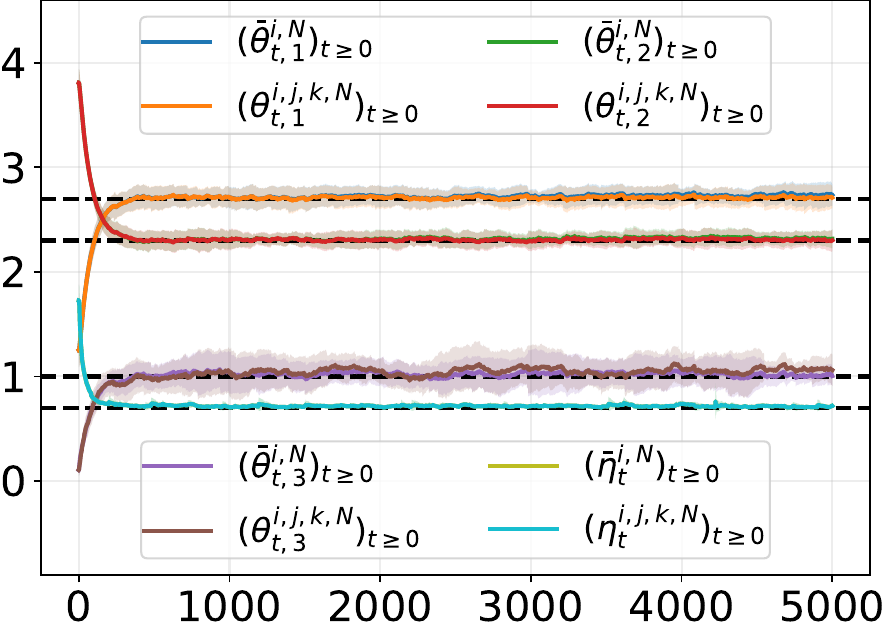}
    \caption{$N=50$.}
    \label{fig:12c}
  \end{subfigure}
  \captionsetup{width=\textwidth}
  \caption{\textbf{Online parameter estimation for the mean-field $\frac{3}{2}$ stochastic volatility model}. We plot the sequence of online parameter estimates as defined by the update equations in \eqref{eq:stochastic_volatility_update1_ips}, \eqref{eq:stochastic_volatility_update2_ips}, and \eqref{eq:eta_update_ips}. The true parameters are given by $\smash{\theta_{0} = (\theta_{0,1},\theta_{0,2},\theta_{0,3})^{\top} =(2.7,2.3,1.0)^{\top}}$ and $\eta_{0} = 0.7$. The initial parameter estimates are given by $\smash{\theta_{\mathrm{init},1}\sim \mathcal{U}[1.0,1.5]}$, $\smash{\theta_{\mathrm{init},2}\sim\mathcal{U}[3.5,4.0]}$, $\smash{\theta_{\mathrm{init},3}\sim\mathcal{U}[0.0,0.2]}$ and $\eta_{\mathrm{init}}\sim\mathcal{U}[1.5,2.0]$.}
  \label{fig:12}
  \vspace{-4mm}
\end{figure}

\section{Conclusion}
\label{sec:conclusion}

In this paper, we introduced new algorithms for online parameter estimation in interacting particle systems (IPSs), based on continuous observation of a small number of particles from the system. In comparison to previous approaches, which required observation of the entire IPS, our approach offers a significant computational advantage. Under mild assumptions, we established convergence of our proposed estimator to the stationary points of an asymptotic log-likelihood function. Under additional conditions (e.g., strong convexity), we also established an $\mathrm{L}^2$ convergence rate and a central limit theorem. 

There are a number of natural extensions to the work presented here. In terms of theory, an interesting question is whether it is possible to extend our results to the hypoelliptic setting \citep[e.g.,][]{iguchi2025parameter}. Regarding methodology, a natural extension is to generalise our approach to the nonparametric or semiparametric setting, i.e., to the case where the functional form of the drift (or diffusion) is not known \citep[e.g.,][]{belomestny2023semiparametric}. Finally, it would be of interest to extend our approach to the partially observed setting \citep[e.g.,][]{jasra2025bayesian}.

 \section*{Acknowledgements} G.A.P. is partially supported by an ERC-EPSRC Frontier Research Guarantee through Grant No. EP/X038645, ERC Advanced Grant No. 247031, and a Leverhulme Trust Senior Research Fellowship, SRF\textbackslash{}R1\textbackslash{}241055.

\bibliography{references}

\appendix 


\section{Existing Results}
\label{app:existing-results}
In this section we recall some classical results on the IPS and the associated MVSDE $\vphantom{(w_t^{i,N})_{t\geq 0}^{i\in[N]}}$ which hold under our standing assumptions. The proofs of these results can be found in \cite{cattiaux2008probabilistic} $\vphantom{(w_t^{i,N})_{t\geq 0}^{i\in[N]}}$. We begin by recalling some notation. Let $(x_t^{i,N})_{t\geq 0}^{i\in[N]}$ denote the solutions of the observed IPS, with initial conditions $(x_0^{i,N})^{i\in[N]} \sim \mu_0^{\otimes N}$. Let $(x_t^{i})^{i\in[N]}_{t\geq 0}$ denote the family of independent solutions of the corresponding MVSDE, driven by the same Brownian motions $(w_t^{i,N})_{t\geq 0}^{i\in[N]}$ and with the same random initial conditions $(x_0^{i})^{i\in[N]} \sim\mu_0^{\otimes N}$ as the IPS.

\begin{theorem}[Moment Bounds]
\label{thm:moment-bounds}
Suppose that Assumption \ref{assumption:moments} and Assumption \ref{assumption:drift} hold. Then, for all $k\geq 1$, there exists a constant $C_k>0$ such that, for all $i=1,\dots,N$, 
\begin{align}
    &\sup_{t\geq 0} \mathbb{E}\left[\|x_t^{i,N}\|^{2k}\right] \leq C_k \left[ 1 + \mu_0(\|x\|^{2mk})\right] \\
    &\sup_{t\geq 0} \mathbb{E}\left[\|x_t^{i}\|^{2k}\right] \leq C_k \left[ 1 + \mu_0(\|x\|^{2mk})\right].
\end{align}
\end{theorem}

\begin{theorem}[Propagation of Chaos]
\label{thm:poc}
Suppose that Assumption \ref{assumption:moments} and Assumption \ref{assumption:drift} hold. Then, for all $N\in\mathbb{N}$, there exists a constant $K>0$ such that, for all $i=1,\dots,N$,
\begin{equation}
    \sup_{t\geq 0}\mathbb{E}\left[\|x_t^{i,N} - x_t^{i}\|^2\right] \leq \frac{K}{N^{\frac{1}{1+\alpha}}}.
\end{equation}
\end{theorem}

\begin{theorem}[Ergodicity]
\label{thm:invariant-distribution}
Suppose that Assumption \ref{assumption:moments} and Assumption \ref{assumption:drift} hold. Then the observed IPS and the corresponding MVSDE are ergodic, with unique invariant distributions $\pi_{\theta_0}^N\in\mathcal{P}((\mathbb{R}^d)^N)$ and $\pi_{\theta_0}\in\mathcal{P}(\mathbb{R}^d)$.
\end{theorem}

We now introduce some additional notation. Let $(\bar{x}_t^{i,N})^{i\in[N]}_{t\geq 0}$ denote the solutions of the observed IPS, driven by the same Brownian motions $(w_t^{i,N})_{t\geq 0}^{i\in[N]}$ as above, but now with the initial condition $(\bar{x}_0^{i,N})^{i\in[N]}\sim\pi_{\theta_0}^{N}$. Let $(\bar{x}_t^{i})^{i\in[N]}_{t\geq 0}$ denote independent solutions of the MVSDE, driven by the same Brownian motions $(w_t^{i,N})_{t\geq 0}^{i\in[N]}$ as above, but now with the initial condition $(\bar{x}_0^{i})^{i\in[N]}\sim\pi_{\theta_0}^{\otimes N}$. Finally, let $\gamma_0^{*,N}\in\Gamma(\mu_0^{\otimes N},\pi_{\theta_0}^N)$ denote the exchangeable optimal coupling of $\mu_0^{\otimes N}$ and $\pi_{\theta_0}^N$, and $\gamma_0^{*}\in\Gamma(\mu_0,\pi_{\theta_0})\vphantom{(\bar{x}_0^{i})^{i\in[N]}\sim\pi_{\theta_0}^{\otimes N}}$ denote the optimal coupling of $\mu_0$, and $\pi_{\theta_0}\vphantom{(\bar{x}_t^{i})^{i\in[N]}_{t\geq 0}}$, both with respect to the quadratic cost. $\vphantom{(\bar{x}_t^{i})^{i\in[N]}_{t\geq 0}}$.

\begin{theorem}[Convergence to Invariant Measure]
\label{thm:invariant-distribution-2}
Suppose that Assumption \ref{assumption:moments} and Assumption \ref{assumption:drift} hold. Then, for all $t\geq 0$, and for all $i=1,\dots,N$, it holds that
\begin{align}
&\mathsf{W}_2^2(\mu_t^{i,N},\pi_{\theta_0}^{i,N}) \le \mathbb{E}_{\gamma_0^{*,N}}[\|{{x}}_t^{i,N} - \bar{{x}}_t^{i,N}\|^2] = \frac{1}{N}\sum_{j=1}^N\mathbb{E}_{\gamma_0^{*,N}}[\|x_t^{j,N} - \bar{x}_t^{j,N}\|^2] \leq a_t(\frac{1}{\sqrt{N}} \mathsf{W}_2(\mu_0^{\otimes N},\pi_{\theta_0}^{N}) ).\\
&\mathsf{W}_2^2(\mu_t, \pi_{\theta_0}) \leq \mathbb{E}_{\gamma_0^{*}}\big[\|x_t^{i} - \bar{x}_t^{i}\|^2\big] \leq a_t(\mathsf{W}_2(\mu_0,\pi_{\theta_0})),
\end{align}
where the function $\smash{a_t:\mathbb{R}_{+}\rightarrow\mathbb{R}_{+}}$ is defined according to
\begin{equation}
\label{eq:rate-function}
    a_t(x) = \left\{\begin{array}{lll} \Big[x^{-\alpha} + A\big(\tfrac{\alpha}{2+\alpha}\big)^{1+\frac{\alpha}{2}}t\Big]^{-\frac{2}{\alpha}} & , & \alpha>0 \\[2mm]
    C^2x^2 e^{-2At} & , & \alpha = 0.
    \end{array}\right.
\end{equation}
\end{theorem}

\section{The Centered Interacting Particle System}
\label{app:centered-ips}

Under Assumption~\ref{assumption:drift}(b), the results in Appendix~\ref{app:existing-results} hold for a \emph{projected} or \emph{centered} version of the observed IPS, defined by $\smash{y_t^{i,N} = x_t^{i,N} - \frac{1}{N}\sum_{j=1}^N x_t^{j,N}}$ \citep[][Section 2]{cattiaux2008probabilistic}. This still defines a diffusion process, but now on the hyperplane $\mathcal{M}_N = \{\boldsymbol{x}^N\in(\mathbb{R}^d)^N: \sum_{i=1}^N x^{i,N} = \mathbf{0}\}$. In particular, the SDE governing the dynamics of the (observed) centered IPS $(\boldsymbol{y}_t^N)_{t\geq 0}$ is given by 
 \begin{equation}
     \mathrm{d}\boldsymbol{y}_t^{N} = B^N(\theta_0,\boldsymbol{y}_t^{N})\mathrm{d}t + \tilde{\Sigma}_N\mathrm{d}\boldsymbol{w}_t^N
 \end{equation}
 where $\tilde{\Sigma}_N = P_N\otimes \sigma$, $P_N:=\mathbf{I}_N - \frac{1}{N}\mathbf{1}\mathbf{1}^{\top}$ denotes the orthogonal projection operator, and $\mathbf{1}\in\mathbb{R}^N$ denotes the all-ones vector \citep[e.g.,][Section~2]{cattiaux2008probabilistic}. 

\subsection{The Log-Likelihood} 
In this case, we must consider the log-likelihood associated with the centered process. This requires additional care, since the diffusion coefficient $\tilde\Sigma_N$ is now singular on $(\mathbb{R}^d)^N$. Fortunately, we can still apply Girsanov's theorem on $\mathcal{M}_N$, since $\tilde\Sigma_N$ is non-degenerate when restricted to this space. This yields 
\begin{align}
\label{eq:ips-log-likelihood-centered}
\mathcal{L}_t^{N}(\theta)
&=\int_0^t  \langle B^N(\theta,\boldsymbol{y}_s^N), (\tilde\Sigma_N\tilde\Sigma_N^{\top})^{+}\mathrm{d}\boldsymbol{y}_s^N\rangle - \frac{1}{2}\int_0^t \langle B^N(\theta,\boldsymbol{y}_s^N), (\tilde\Sigma_N\tilde\Sigma_N^{\top})^{+}B^N(\theta,\boldsymbol{y}_s^N)\rangle\mathrm{d}s 
\end{align}
where $(\cdot)^{+}$ denotes the pseudo-inverse (i.e., Moore--Penrose inverse). In our case, the relevant pseudo-inverse simplifies as $(\tilde{\Sigma}_N \tilde{\Sigma}_N^{\top})^{+} = ((P_N\otimes \sigma)(P_N\otimes \sigma)^{\top})^{+} = (P_N\otimes(\sigma\sigma^\top))^{+} = P_N\otimes(\sigma\sigma^{\top})^{-1}$, where we have used the fact that $P_N=P_N^{+}$ since it is an orthogonal projection. Thus, in particular, \eqref{eq:ips-log-likelihood-centered} can be rewritten as
\begin{align}
\label{eq:ips-log-likelihood-centered-v2}
\mathcal{L}_t^{N}(\theta)
&=\int_0^t  \langle B^N(\theta,\boldsymbol{y}_s^N), P_N\otimes(\sigma\sigma^{\top})^{-1}\mathrm{d}\boldsymbol{y}_s^N\rangle - \frac{1}{2}\int_0^t \langle B^N(\theta,\boldsymbol{y}_s^N), P_N\otimes(\sigma\sigma^{\top})^{-1}B^N(\theta,\boldsymbol{y}_s^N)\rangle\mathrm{d}s 
\end{align}
In order for our methodology to remain applicable in the centered case, we require the log-likelihood to factorise in a similar fashion to \eqref{ips:log-likelihood-2} - \eqref{ips:log-likelihood}. In particular, we must be able to simplify \eqref{eq:ips-log-likelihood-centered-v2} as
\begin{align}
\mathcal{L}_t^{N}(\theta)
&=\sum_{i=1}^N\Big[ \int_0^t \big \langle  B^{i,N}(\theta,\boldsymbol{y}_s^N), (\sigma\sigma^{\top})^{-1}\mathrm{d}y_s^{i,N}\big\rangle - \frac{1}{2} \int_0^t \left|\left| B^{i,N}(\theta,\boldsymbol{y}_s^N)\right|\right|_{\sigma\sigma^{\top}}^2\mathrm{d}s \Big]  \label{ips:log-likelihood-2-centered} \\
&=\sum_{i=1}^N\Big[ \int_0^t \big \langle  B(\theta,y_s^{i,N},\mu_s^{N}), (\sigma\sigma^{\top})^{-1}\mathrm{d}y_s^{i,N}\big\rangle - \frac{1}{2} \int_0^t \left|\left| B(\theta,y_s^{i,N},\mu_s^{N})\right|\right|_{\sigma\sigma^{\top}}^2\mathrm{d}s \Big] \label{ips:log-likelihood-centered}
\end{align}
where now $\smash{\mu_s^N = \frac{1}{N}\sum_{j=1}^N\delta_{y_s^{j,N}}}$ denotes the empirical law of the observed centered IPS, with all other notation as before. This is exactly the same coordinate-wise functional form as in the non-centered case. A sufficient condition for this simplification to hold is that both $\smash{\mathrm{d}\boldsymbol{y}_s^N}$ and
$\smash{B^N(\theta,\boldsymbol{y}_s^N)}$ take values in $\mathcal{M}_N$, so that $P_N$ acts as the
identity on these vectors. More precisely, if $\boldsymbol{u},\boldsymbol{v}\in\mathcal{M}_N$,
then $(P_N\otimes I_d)\boldsymbol{u}=\boldsymbol{u}$ and $(P_N\otimes I_d)\boldsymbol{v}=\boldsymbol{v}$,
and hence
\begin{equation}
\langle \boldsymbol{u}, (P_N\otimes(\sigma\sigma^\top)^{-1})\boldsymbol{v}\rangle
= \langle \boldsymbol{u}, (\mathbf{I}_N\otimes(\sigma\sigma^\top)^{-1})\boldsymbol{v}\rangle
= \sum_{i=1}^N \langle u^i,(\sigma\sigma^\top)^{-1}v^i\rangle.
\end{equation}
Applying this with $\boldsymbol{u}=B^N(\theta,\boldsymbol{y}_s^N)$ and
$\boldsymbol{v}=\mathrm{d}\boldsymbol{y}_s^N$ for the first inner product in \eqref{eq:ips-log-likelihood-centered-v2}, and with $\boldsymbol{u}=\boldsymbol{v}=B^N(\theta,\boldsymbol{y}_s^N)$ for the second inner product in \eqref{eq:ips-log-likelihood-centered-v2}, it is clear that \eqref{eq:ips-log-likelihood-centered-v2} simplifies to \eqref{ips:log-likelihood-2-centered}. It remains to ensure that, for all $\theta\in\Theta$, the function $\boldsymbol{x}^N\mapsto B^N(\theta,\boldsymbol{x}^N)$ is \emph{centered}. That is, $\sum_{i=1}^N B^{i,N}(\theta,\boldsymbol{y})=0$ for all $\boldsymbol{y}\in\mathcal{M}_N$. A sufficient condition for this is the following addenda to Assumption~\ref{assumption:drift}(b).

\begin{assumptionrecall}[\ref{assumption:drift}(b*)]
\label{assumption:drift-recall}
The functions $x\mapsto V(\theta_0,x)$ and $x\mapsto W(\theta_0,x)$ satisfy all of the existing conditions of Assumption~\ref{assumption:drift}(b). In addition, 
    \begin{enumerate}[labelindent=0pt, labelwidth=!, leftmargin=*, align=left]
        \item[(b*)(i)] For all $\theta\in\Theta$, $V(\theta,\cdot)=0$.
        \item[(b*)(ii)] For all $\theta\in\Theta$, $W(\theta,\cdot)$ is symmetric.
    \end{enumerate}
\end{assumptionrecall}

\noindent Thus, in particular, we now assume that the confinement potential is null and the interaction potential is symmetric for all $\theta\in\Theta$, rather than just at the true parameter $\theta_0$. Since the factorisation of the log-likelihood is essential for our subsequent methodological developments, we will henceforth assume that Assumption~\ref{assumption:drift}(b*) always subsumes the relevant parts of Assumption~\ref{assumption:drift}(b). 

\section{Additional Proofs}
\label{app:additional-proofs}


\subsection{Proofs for Section \ref{sec:log-lik}}
\label{app:additional-proofs-log-lik}

\subsubsection{Proofs for Section \ref{sec:log-lik-ips}}
\label{app:additional-proofs-log-lik-ips}
\begin{proof}[Proof of Proposition \ref{prop:ips-likelihood-t-limit}]
    Using the definition of the log-likelihood function in \eqref{ips:log-likelihood-2}, and the data-generating process in \eqref{eq:IPS}, we have that
    \begin{align}
        \frac{1}{t}\left[\mathcal{L}_t^N(\theta) - \mathcal{L}_t^N(\theta_0)\right]
        &=  -\sum_{i=1}^N\Big[\frac{1}{t}\int_0^t L^{i,N}(\theta,\boldsymbol{x}_s^N) \mathrm{d}s \Big] + \sum_{i=1}^N \Big[\frac{1}{t}  \int_0^t \big\langle \Delta B^{i,N}(\theta,\boldsymbol{x}_s^N), \sigma \,\mathrm{d}w_s^{i,N}\big\rangle_{\sigma\sigma^{\top}} \Big]. \label{eq:particle-lik-decomp-v2}
    \end{align}
    with $\smash{L^{i,N}(\theta,\boldsymbol{x}^N) = \frac{1}{2}\|B^{i,N}(\theta,\boldsymbol{x}^N) - B^{i,N}(\theta_0,\boldsymbol{x}^N)\|^2_{\sigma\sigma^{\top}}}$ and $\smash{\Delta B^{i,N}(\theta,\boldsymbol{x}^N) = B^{i,N}(\theta,\boldsymbol{x}^N) - B^{i,N}(\theta_0,\boldsymbol{x}^N)}$.
    We begin by considering the first term in \eqref{eq:particle-lik-decomp-v2}. By Theorem \ref{thm:invariant-distribution}, the IPS is ergodic, and admits a unique invariant measure $\pi_{\theta_0}^N \in \mathcal{P}((\mathbb{R}^d)^N)$. In addition, due to Theorem~\ref{thm:moment-bounds} (i.e., uniform-in-time moment bounds for the IPS) and Corollary~\ref{cor:empirical-L-lip} (i.e., the polynomial growth property of $L^{i,N}$), it follows that $L^{i,N}(\theta,\boldsymbol{x}^N)\in \mathrm{L}^1(\pi_{\theta_0}^N)$. Thus, by the ergodic theorem \citep[e.g.,][Theorem 4.2]{khasminskii2012stochastic}, we have that 
    \begin{align}
        \frac{1}{t} \int_0^t L^{i,N}(\theta,\boldsymbol{x}_s^N)\mathrm{d}s 
        \stackrel{\mathrm{a.s.}}{\longrightarrow}
        \int_{(\mathbb{R}^d)^{N}}  L^{i,N}(\theta,\boldsymbol{x}^N) \pi_{\theta_0}^N(\mathrm{d}\boldsymbol{x}^N) \label{eq:b-diff-convergence}
    \end{align}
    as $t\rightarrow\infty$. We now show that second term in \eqref{eq:particle-lik-decomp-v2} converges a.s. to zero. To do so, let us define the continuous local martingales $\smash{M_{i,t}^N = \int_0^t \langle  \Delta B^{i,N}(\theta,\boldsymbol{x}_s^N) , (\sigma\sigma^{\top})^{-1}\sigma \,\mathrm{d}w_s^{i,N}\rangle}$, with quadratic variations given by $\smash{\langle M_{i}^N\rangle_{t} = \int_0^t \|\Delta B^{i,N}(\theta,\boldsymbol{x}_s^N) \|_{\sigma\sigma^{\top}}^2 \mathrm{d}s = \int_0^t 2L^{i,N}(\theta,\boldsymbol{x}_s^N) \mathrm{d}s}$. Reasoning as above, the ergodic theorem yields
    \begin{align}
        \frac{1}{t}\langle M_{i}^N\rangle_{t} 
        &\stackrel{\mathrm{a.s.}}{\longrightarrow} \int_{(\mathbb{R}^d)^N} 2L^{i,N}(\theta,\boldsymbol{x}^N) \pi_{\theta_0}^N(\mathrm{d}\boldsymbol{x}^N) <\infty
    \end{align}
    as $t\rightarrow\infty$. It follows, using the strong law of large numbers for continuous local martingales \citep[e.g.,][Theorem 1.3.4]{mao2008stochastic} that 
    \begin{equation} 
    \frac{1}{t}M_{i,t}^N = \frac{1}{t}\int_0^t \langle  \Delta B^{i,N}(\theta,\boldsymbol{x}_s^N) , (\sigma\sigma^{\top})^{-1}\sigma\, \mathrm{d}w_s^{i,N}\rangle \stackrel{\mathrm{a.s.}}{\longrightarrow} 0. \label{eq:mart-convergence}
    \end{equation}
    as $t\rightarrow\infty$. Finally, combining \eqref{eq:b-diff-convergence} and \eqref{eq:mart-convergence}, and summing over $i\in[N]$, we have the required a.s. convergence result. It remains to show that the convergence also holds in $\mathbb \mathrm{L}_1$. By Corollary~\ref{cor:empirical-L-lip} (i.e., polynomial growth of $L^{i,N}$) and Theorem~\ref{thm:moment-bounds} (i.e., uniform-in-time moment bounds for the IPS), for each $\delta>0$ there exists $K_{\delta}<\infty$ such that $\smash{\sup_{s\ge0}\mathbb{E}\big[|L^{i,N}(\theta,\boldsymbol x_s^N)|^{1+\delta}\big]<K_{\delta}}$. Thus, using Jensen's inequality, it holds uniformly in $t\geq 1$ that
\begin{equation}
\mathbb{E}\Big[\Big|\frac1t\int_0^t L^{i,N}(\theta,\boldsymbol{x}_s^N)\,\mathrm ds\Big|^{1+\delta}\Big]
\le \frac1t\int_0^t \mathbb{E}\big[|L^{i,N}(\theta,\boldsymbol x_s^N)|^{1+\delta}\big]\,\mathrm ds
\le K_{\delta}.
\end{equation}
It follows that the family of random variables $\{\frac{1}{t}\int_0^t L^{i,N}(\theta,\boldsymbol{x}_s^N)\,\mathrm ds\}_{t\geq 1}$ is uniformly integrable. This, combined with the a.s. convergence already established in \eqref{eq:b-diff-convergence}, and Vitali's theorem, yields
\begin{align}
        \frac{1}{t} \int_0^t L^{i,N}(\theta,\boldsymbol{x}_s^N)\mathrm{d}s \stackrel{\mathrm{L}^1}{\longrightarrow}
        \int_{(\mathbb{R}^d)^{N}}  L^{i,N}(\theta,\boldsymbol{x}^N) \pi_{\theta_0}^N(\mathrm{d}\boldsymbol{x}^N). \label{eq:b-diff-convergence-l1}
    \end{align}
For the martingale term, using the Burkholder-Davis-Gundy (BDG) inequality and Jensen's inequality, we have (allowing the value of the constant $K$ to increase from line to line) that
\begin{align}
\mathbb{E}\Big[\Big|\frac{1}{t}M_{i,t}^N\Big|\Big]\le \frac{K}{t}\mathbb{E}\big[\langle M_i^N\rangle_t^{1/2}\big] &\le \frac{K}{t}\Big(\mathbb{E}[\langle M_i^N\rangle_t]\Big)^{1/2}= \frac{K}{t}\Big(\int_0^t \mathbb{E}[2L^{i,N}(\theta,\boldsymbol{x}_s^N)]\,\mathrm ds\Big)^{1/2}
\end{align}
Once more using Corollary~\ref{cor:empirical-L-lip} (i.e., polynomial growth of $L^{i,N}$) and Theorem~\ref{thm:moment-bounds} (i.e., uniform-in-time moment bounds for the IPS), there exists $K<\infty$ such that $\smash{\sup_{s\ge0}\mathbb{E}\big[|L^{i,N}(\theta,\boldsymbol x_s^N)|\big]}<K$. Combining this with the previous display, it follows that $\smash{\mathbb{E}[|\frac{1}{t}M_{i,t}^N|]\le \frac{K}{t} t^{\frac{1}{2}} = Kt^{-\frac{1}{2}}}$, which implies in particular that
\begin{equation}
\label{eq:mart-convergence-l1}
\frac{1}{t}M_{i,t}^N = \frac{1}{t}\int_0^t \langle  \Delta B^{i,N}(\theta,\boldsymbol{x}_s^N) , \sigma \,\mathrm{d}w_s^{i,N}\rangle_{\sigma\sigma^{\top}}\stackrel{\mathrm{L}^1}{\longrightarrow}0.
\end{equation}
Combining \eqref{eq:mart-convergence-l1} and \eqref{eq:b-diff-convergence-l1}, and summing over $i\in[N]$, we obtain the stated convergence in $\mathbb \mathrm{L}_1$. 
\end{proof}

\begin{proof}[Proof of Proposition \ref{prop:ips-likelihood-n-limit}]
We begin similarly to the previous proof. In particular, from the definition of the log-likelihood in \eqref{ips:log-likelihood}, we have 
\begin{align}
\label{eq:log-likelihood-ips-recall}
        \frac{1}{N}\left[\mathcal{L}_t^N(\theta) - \mathcal{L}_t^N(\theta_0)\right] 
        &= -\frac{1}{N}\sum_{i=1}^N\big[ \int_0^t L(\theta,x_s^{i,N},\mu_s^N)\mathrm{d}s \big] +\frac{1}{N}\sum_{i=1}^N \big[ \int_0^t \big\langle  \Delta B(\theta,x_s^{i,N},\mu_s^N), \sigma \,\mathrm{d}w_s^{i,N}\big\rangle_{\sigma\sigma^{\top}} \big].
    \end{align}
We will establish convergence in $\mathrm{L}^1$, which will also imply convergence in probability. We begin with the first term. We would like to show that
\begin{align}
    &\mathbb{E}\Big[\Big|\frac{1}{N}\sum_{i=1}^N\int_0^t L(\theta,x_s^{i,N},\mu_s^N)  \mathrm{d}s - \int_0^t \int_{\mathbb{R}^d} L(\theta,x_s,\mu_s)  \mu_s(\mathrm{d}x) \mathrm{d}s\Big|\Big] \\
    &=\mathbb{E}\Big[\Big|\int_0^t \Big[\frac{1}{N}\sum_{i=1}^NL(\theta,x_s^{i,N},\mu_s^N) -  \int_{\mathbb{R}^d} L(\theta,x,\mu_s) \,\mu_s(\mathrm{d}x)\Big]\, \mathrm{d}s\Big|\Big] \stackrel{N\rightarrow\infty}{\longrightarrow} 0.
    \label{eq:required-limit}
\end{align}
First note that the LHS of \eqref{eq:required-limit} is bounded by $\smash{ \int_0^t \mathbb{E}[ |\frac{1}{N}\sum_{i=1}^NL(\theta,x_s^{i,N},\mu_s^N) -  \int_{\mathbb{R}^d} L(\theta,x,\mu_s) \,\mu_s(\mathrm{d}x)|\, ] \mathrm{d}s}$. We thus seek an upper bound for the integrand. Using the triangle inequality, we have 
\begin{align}
    &\mathbb{E}\Big[ \Big|\frac{1}{N}\sum_{i=1}^NL(\theta,x_s^{i,N},\mu_s^N) -  \int_{\mathbb{R}^d} L(\theta,x,\mu_s) \,\mu_s(\mathrm{d}x)\Big|\, \Big] \\
    &\leq\frac{1}{N}\sum_{i=1}^N\mathbb{E}\Big[ \Big|L(\theta,x_s^{i,N},\mu_s^N) -  L(\theta,x_s^{i},\mu_s)\Big| \Big]  + \mathbb{E}\Big[ \Big|\frac{1}{N}\sum_{i=1}^N\Big(L(\theta,x_s^{i},\mu_s) - \int_{\mathbb{R}^d} L(\theta,x,\mu_s) \,\mu_s(\mathrm{d}x)\Big)\Big|\, \Big] \label{eq:phi-decomp}
\end{align}
where $\smash{(x_s^{i})_{s\geq 0}^{i\in[N]}}$ denote $N$ independent solutions of the MVSDE, driven by the same Brownian motions $\smash{(w_s^{i,N})_{s\geq 0}^{i\in[N]}}$ as the interacting particles $\smash{(x_s^{i,N})_{s\geq 0}^{i\in[N]}}$ (i.e., the standard synchronous coupling). Due to Lemma \ref{lem:Phi-L1-stability} (i.e., the fact that $(x,\mu)\mapsto L(\theta,x,\mu)$ is locally Lipschitz with polynomial growth), the Cauchy-Schwarz inequality, and Theorem~\ref{thm:moment-bounds} (i.e., uniform-in-time moment bounds for the IPS and the MVSDE), there exists a constant $K<\infty$ such that, for all $s\geq 0$,  
\begin{align}
\label{eq:Phi-L1-bound-recall}
\sup_{s\geq 0} \mathbb{E}\Big[\big|L(\theta,x_s^{i,N},\mu_s^N)-L(\theta,x_s^i,\mu_s)\big|\Big]
&\le
K\sup_{s\geq 0}\Big[
\big(\mathbb{E}\left[\|x_s^{i,N}-x_s^i\|^2\right]\big)^{1/2}
+
\big(\mathbb{E}\left[\,\mathsf{W}_2^2(\mu_s^N,\mu_s)\right]\big)^{1/2}
\Big]
\end{align}
By Theorem \ref{thm:poc} (i.e., uniform-in-time propagation-of-chaos), there exists a constant $K_1<\infty$ such that, for each $i\in[N]$, 
\begin{equation}
    \sup_{s\geq 0}\mathbb{E}\left[\|x_s^{i,N}-x_s^i\|^2\right]\leq 
\frac{K_1}{N^{\frac{1}{1+\alpha}}} \label{eq:phi-bound-1}
\end{equation}
Meanwhile, using the triangle inequality, once more Theorem \ref{thm:poc} (i.e., uniform-in-time propagation-of-chaos), and also now Theorem 1 in \cite{fournier2015rate},  there exist constants $K_{2,1},K_{2,2}<\infty$ such that
\begin{align}
    \sup_{s\geq 0} \mathbb{E}\big[\mathsf{W}_2^2(\mu_s^N,\mu_s)\big] \leq \sup_{s\geq 0}\left(2 \mathbb{E}\big[\mathsf{W}_2^2(\mu_s^N, \mu_s^{[N]})\big] + 2 \mathbb{E}\big[\mathsf{W}_2^2(\mu_s^{[N]},\mu_s)\big]\right)
    \leq \frac{2K_{2,1}}{N^{\frac{1}{1+\alpha}}} + 2K_{2,2}\rho^2(N). \label{eq:phi-bound-2}
\end{align}
Substituting \eqref{eq:phi-bound-1} and \eqref{eq:phi-bound-2} back into \eqref{eq:Phi-L1-bound-recall}, it follows that 
\begin{equation}
    \sup_{s\geq 0} \mathbb{E}\Big[\big|L(\theta,x_s^{i,N},\mu_s^N)-L(\theta,x_s^i,\mu_s)\big|\Big]\le K\left[\left(\frac{K_1}{N^{\frac{1}{1+\alpha}}}\right)^{\frac{1}{2}}  + \left(\frac{2K_{2,1}}{N^{\frac{1}{1+\alpha}}} + 2K_{2,2}\rho^2(N)\right)^{\frac{1}{2}}\right].
\end{equation}
We now turn our attention to the second term in \eqref{eq:phi-decomp}. Using the Cauchy-Schwarz inequality, and the fact that $\smash{(x_s^{i})_{s\geq 0}^{i\in[N]}}$ are \emph{independent} solutions of the MVSDE, we have for all $s\geq 0$ that
\begin{align}
&\mathbb{E}\Big[\Big|\frac{1}{N}\sum_{i=1}^N \Big(L(\theta,x_s^{i},\mu_s)-\int_{\mathbb{R}^d} L(\theta,x_s,\mu_s)\mu_s(\mathrm{d}x)\Big)\Big|\Big] \le \mathbb{E}\Big[\Big|\frac{1}{N}\sum_{i=1}^N \Big(L(\theta,x_s^{i},\mu_s)-\int_{\mathbb{R}^d} L(\theta,x_s,\mu_s)\mu_s(\mathrm{d}x)\Big)\Big|^2\Big]^{1/2} \notag \\ 
&=\Big[\frac{1}{N^2}\sum_{i=1}^N \mathrm{Var}(L(\theta,x_s,\mu_s))\Big]^{1/2} = \Big[\frac{1}{N}\mathrm{Var}(L(\theta,x_s,\mu_s))\Big]^{1/2} \leq \frac{1}{N^{\frac{1}{2}}}\Big[\mathbb{E}\left[(L(\theta,x_s,\mu_s))^2\right]\Big]^{1/2} \leq \frac{K_3}{N^{\frac{1}{2}}} \label{eq:Phi-L1-bound-recall-v2}
\end{align}
where in the final display we have used Lemma \ref{lem:Phi-L1-stability} (i.e., the polynomial growth of $L$), and Theorem~\ref{thm:moment-bounds} (i.e., the bounded moments of the MVSDE). Finally, substituting the bounds in \eqref{eq:Phi-L1-bound-recall} and \eqref{eq:Phi-L1-bound-recall-v2} two bounds back into \eqref{eq:phi-decomp}, we have that
\begin{align}
&\int_0^t\mathbb{E}\Big[\Big|\frac{1}{N}\sum_{i=1}^N L(\theta,x_s^{i,N},\mu_s^N)  - \frac{1}{N}\sum_{i=1}^N \int_{\mathbb{R}^d} L(\theta,x,\mu_s) \,\mu_s(\mathrm{d}x)\Big|\Big]\mathrm{d}s 
\stackrel{N\rightarrow\infty}{\longrightarrow} 0. \label{eq:large-n-limit-0}
\end{align}
We now turn our attention to the martingale term in \eqref{eq:log-likelihood-ips-recall}.  To establish the desired limit, we need to show that $\smash{\mathbb{E}[|\frac{1}{N}\sum_{i=1}^N\int_0^t \langle \Delta B(\theta,x_s^{i,N},\mu_s^N), (\sigma\sigma^{\top})^{-1}\sigma\,\mathrm{d}w_s^{i,N}\rangle |^2] \stackrel{N\rightarrow\infty}{\longrightarrow} 0}$. First note that the martingales
$\int_0^\cdot \langle \Delta B(\theta,x_s^{i,N},\mu_s^N),(\sigma\sigma^{\top})^{-1}\sigma\,\mathrm{d}w_s^{i,N}\rangle$ are orthogonal for different $i$, since the Brownian motions $(w^{i,N})_{i=1}^N$ are independent. It follows, using this and the It\^{o} isometry, that 
\begin{align}
\label{eq:martingale-ito-0}
&\mathbb{E}\Big[\Big|\frac{1}{N}\sum_{i=1}^N\int_0^t \left\langle \Delta B(\theta,x_s^{i,N},\mu_s^N), (\sigma\sigma^{\top})^{-1}\sigma\,\mathrm{d}w_s^{i,N}\right\rangle \Big|^2\Big] \\
&= \frac{1}{N^2}\sum_{i=1}^N
\int_0^t \mathbb{E}\Big[\big\| \Delta B(\theta,x_s^{i,N},\mu_s^N) \big\|_{\sigma\sigma^{\top}}^2\Big]\mathrm{d}s \le \frac{c_\sigma^2}{N^2}\sum_{i=1}^N
\int_0^t \mathbb{E}\left[\big\|\Delta B(\theta,x_s^{i,N},\mu_s^N) \big\|^2\right]\mathrm{d}s,
\label{eq:martingale-ito}
\end{align}
where $c_\sigma := \|\sigma^\top(\sigma\sigma^\top)^{-1}\|_{\mathrm{op}}<\infty$ is a constant depending only on $\sigma$. Meanwhile, by Corollary~\ref{cor:empirical-drift-lip} (i.e., the polynomial growth of $B$) and Theorem \ref{thm:moment-bounds} (i.e, moment bounds for the IPS, uniform in time), there exist a constant $K<\infty$ such that $\sup_{s\geq 0}\mathbb{E}[\|\Delta B(\theta,x_s^{i,N},\mu_s^N) \|^2] \le K$. 
Substituting this into \eqref{eq:martingale-ito-0} - \eqref{eq:martingale-ito}, we have, as required, that
\begin{equation}
\mathbb{E}\Big[\Big|\frac{1}{N}\sum_{i=1}^N\int_0^t \left\langle B(\theta,x_s^{i,N},\mu_s^N) - B(\theta_0,x_s^{i,N},\mu_s^N), \sigma\,\mathrm{d}w_s^{i,N}\right\rangle_{\sigma\sigma^{\top}} \Big|^2\Big]
\leq \frac{c_{\sigma}^2}{N^2}\sum_{i=1}^N \int_0^t K\mathrm{d}s =\frac{K c_\sigma^2\,t}{N}\stackrel{N\rightarrow\infty}{\longrightarrow}0. \label{eq:large-n-limit-2}
\end{equation}
\end{proof}

\begin{proof}[Proof of Corollary \ref{cor:ips-likelihood-n-t-limit}] 
    We begin with the observation that, by essentially the same argument as the one used in the proof of Proposition \ref{prop:ips-likelihood-n-limit}, we have that
    \begin{align}
        \lim_{N\rightarrow\infty}\frac{1}{Nt}\left[\mathcal{L}_t^N(\theta) - \mathcal{L}_t^N(\theta_0)\right]  &\stackrel{\mathrm{L}^1}{=}-\frac{1}{t} \int_0^t \int_{\mathbb{R}^d} L(\theta,x,\mu_s) \mu_s(\mathrm{d}x)\,\mathrm{d}s.
        \label{eq:IPS-log-likelihood-large-N-a-recall} 
    \end{align}
    It remains to establish $\smash{\lim_{t\rightarrow\infty}\frac{1}{t}\int_0^t \int_{\mathbb{R}^d} L(\theta,x,\mu_s)\mu_s(\mathrm{d}x)\,\mathrm{d}s= \int_{\mathbb{R}^d}L(\theta,x,\pi_{\theta_0})\pi_{\theta_0}(\mathrm{d}x)}$. To establish this limit, we begin by using the triangle inequality to write 
    \begin{align}
    \label{eq:phi-new-triangle-ineq}
        &\left|\frac{1}{t}\int_0^t \int_{\mathbb{R}^d} L(\theta,x,\mu_s)\mu_s(\mathrm{d}x)\,\mathrm{d}s - \int_{\mathbb{R}^d}L(\theta,x,\pi_{\theta_0})\pi_{\theta_0}(\mathrm{d}x)\right| \leq I_t^{(1)} + I_t^{(2)} 
    \end{align}
    where the two quantities on the RHS are defined as $\smash{I_t^{(1)}:=\frac{1}{t}\int_0^t \int_{\mathbb{R}^d} \left|L(\theta,x,\mu_s) - L(\theta,x,\pi_{\theta_0})\right|\mu_s(\mathrm{d}x)\,\mathrm{d}s}$ and $\smash{I_t^{(2)} :=\frac{1}{t}\int_0^t \left|\int_{\mathbb{R}^d} L(\theta,x,\pi_{\theta_0})\mu_s(\mathrm{d}x)-\int_{\mathbb{R}^d}L(\theta,x,\pi_{\theta_0})\pi_{\theta_0}(\mathrm{d}x)\right|\mathrm{d}s}$. For the first term, using Lemma \ref{lem:Phi-L1-stability}, there exists a constant $K<\infty$ and an integer $q\geq 1$ such that, for all $s\geq 0$, $|L(\theta,x,\mu_s)-L(\theta,x,\pi_{\theta_0})|\le K\,\mathsf{W}_2(\mu_s,\pi_{\theta_0})(1+\|x\|^q+\mu_s(\|\cdot\|^q)+\pi_{\theta_0}(\|\cdot\|^q))$. Taking expectations, using Cauchy--Schwarz, and Theorem~\ref{thm:moment-bounds} (i.e., moment bounds of the MVSDE, uniform-in-time), we obtain
\begin{equation}
\int_{\mathbb{R}^d} \big|L(\theta,x_s,\mu_s)-L(\theta,x_s,\pi_{\theta_0})\big| \mu_s(\mathrm{d}x)
\le K\,\mathsf{W}_2(\mu_s,\pi_{\theta_0}), \label{eq:phi-exp-bound}
\end{equation}
for some finite constant $K<\infty$. Next, using Theorem~\ref{thm:invariant-distribution-2} (i.e., convergence to the invariant distribution of the MVSDE), we have $\mathsf{W}_2(\mu_s,\pi_{\theta_0})\rightarrow 0$ as $s\rightarrow \infty$. Using this, substituting the bound in \eqref{eq:phi-exp-bound} back into \eqref{eq:phi-new-triangle-ineq}, and using Ces\`aro's Theorem, it follows that
    \begin{align}
I_t^{(1)}:= \frac{1}{t}\int_0^t \int_{\mathbb{R}^d} \left|L(\theta,x,\mu_s) - L(\theta,x,\pi_{\theta_0})\right|\mu_s(\mathrm{d}x)\,\mathrm{d}s
&\le \frac{K}{t}\int_0^t \mathsf{W}_2(\mu_s,\pi_{\theta_0})\,\mathrm{d}s \stackrel{t\rightarrow \infty}{\longrightarrow 0}.
\label{eq:i1-bound-final}
\end{align}
    For the second term, let $\gamma_s\in\Pi(\mu_s,\pi_{\theta_0})$ be any coupling between $\mu_s$ and $\pi_{\theta_0}$. Then, using the triangle inequality (integral version), we have that
    \begin{align}
    \label{eq:i2-bound}
         &\frac{1}{t}\int_0^t \left|\int_{\mathbb{R}^d} L(\theta,x,\pi_{\theta_0})\mu_s(\mathrm{d}x)-\int_{\mathbb{R}^d}L(\theta,x,\pi_{\theta_0})\pi_{\theta_0}(\mathrm{d}x)\right|\mathrm{d}s \\
        & = \frac{1}{t}\int_0^t \left|\int_{\mathbb{R}^d\times\mathbb{R}^d} \left[L(\theta,x,\pi_{\theta_0}) - L(\theta,y,\pi_{\theta_0})\right]\gamma_s(\mathrm{d}x,\mathrm{d}y)\right| \leq \frac{1}{t}\int_0^t \int_{\mathbb{R}^d\times\mathbb{R}^d} \left|L(\theta,x,\pi_{\theta_0}) - L(\theta,y,\pi_{\theta_0})\right|\gamma_s(\mathrm{d}x,\mathrm{d}y) \notag
    \end{align}
    By Lemma~\ref{lem:Phi-L1-stability}, there exist a constant $K<\infty$ and an integer $q\ge 1$ such that for all $x,y\in\mathbb R^d$, $|L(\theta,x,\pi_{\theta_0}) - L(\theta,y,\pi_{\theta_0})|\le K\,\|x-y\|(1+\|x\|^q+\|y\|^q+\pi_{\theta_0}(\|\cdot\|^q))$. Using this fact, Cauchy-Schwarz, and Theorem~\ref{thm:moment-bounds} (i.e., moment bounds for the MVSDE, uniform-in-time), it follows that
\begin{align}
&\int_{\mathbb{R}^d\times\mathbb{R}^d} \Big|L(\theta,x,\pi_{\theta_0}) - L(\theta,y,\pi_{\theta_0})\Big|\gamma_s(\mathrm{d}x,\mathrm{d}y) \\
&\le K\Big(\int \|x-y\|^2\,\gamma_s(dx,dy)\Big)^{1/2}
\Big(\int \Big(1+\|x\|^q+\|y\|^q+\pi_{\theta_0}(\|\cdot\|^q)\Big)^2\,\gamma_s(dx,dy)\Big)^{1/2} \\
&\le K\left(\int \|x-y\|^2\,\gamma_s(dx,dy)\right)^{1/2}.
\end{align}
where as elsewhere, the value of the constant is allowed to increase from line to line. In particular, this inequality holds for the optimal coupling $\gamma_s^{*}\in\Pi(\mu_s,\pi_{\theta_0})$, in which case it rewrites as
\begin{equation}
\int_{\mathbb{R}^d\times\mathbb{R}^d} \left|L(\theta,x,\pi_{\theta_0}) - L(\theta,y,\pi_{\theta_0})\right|\gamma_s^{*}(\mathrm{d}x,\mathrm{d}y)\le K\,\mathsf{W}_2(\mu_s,\pi_{\theta_0}).
\end{equation}
Substituting this bound into \eqref{eq:i2-bound}, using Theorem~\ref{thm:invariant-distribution-2} (i.e., convergence to the invariant distribution), and once again the fact that $\mathsf{W}_2(\mu_s,\pi_{\theta_0})\rightarrow 0$ and C\'esaro's Theorem, it follows as required that
\begin{align}
I_t^{(2)}:=\frac{1}{t}\int_0^t \Big|\int_{\mathbb{R}^d} L(\theta,x,\pi_{\theta_0})\mu_s(\mathrm{d}x)-\int_{\mathbb{R}^d}L(\theta,x,\pi_{\theta_0})\pi_{\theta_0}(\mathrm{d}x)\Big|\mathrm{d}s
&\le \frac{1}{t}\int_0^t K \mathsf{W}_2(\mu_s,\pi_{\theta_0})\,\mathrm{d}s
\stackrel{t\rightarrow \infty}{\longrightarrow 0}.
\label{eq:i2-bound-final}
\end{align}\end{proof}

\subsubsection{Proofs for Section \ref{sec:log-lik-mvsde}}
\label{app:additional-proofs-log-lik-mvsde}

\begin{proof}[Proof of Proposition \ref{prop:mvsde-likelihood-t-limit}] Using the definition of the log-likelihood of the MVSDE in \eqref{eq:MVSDE-log-likelihood}, and the MVSDE in \eqref{eq:MVSDE}, we have that
\begin{align}
\label{eq:i1-i2-decomp} 
\frac{1}{t}\left[\mathcal{L}_t(\theta) - \mathcal{L}_t(\theta_0)\right] &= -\frac{1}{t} \Big[ \int_0^t J(\theta,x_s,\mu_s^{\theta},\mu_s^{\theta_0})\mathrm{d}s\Big] + \frac{1}{t}\Big[\int_0^t \big\langle  \Delta B(\theta,x_s,\mu_s^{\theta},\mu_s^{\theta_0}) , (\sigma\sigma^{\top})^{-1}\sigma\, \mathrm{d}w_s\big\rangle\Big].
\end{align}
where $\smash{J(\theta,x,\mu^{\theta},\mu^{\theta_0}) :=\frac{1}{2}\|B(\theta,x,\mu^{\theta}) - B(\theta_0,x,\mu^{\theta_0})\|_{\sigma\sigma^{\top}}^2}$ and $\smash{\Delta B(\theta,x,\mu^{\theta},\mu^{\theta_0}) = B(\theta,x,\mu^{\theta}) - B(\theta_0,x,\mu^{\theta_0})}$. We start by considering the first term in \eqref{eq:i1-i2-decomp}. We would like to show that in $\mathrm{L}^1$ 
    \begin{equation}
    \label{eq:j-limit}
        \lim_{t\rightarrow\infty}\frac{1}{t}\int_0^t J(\theta,x_s,\mu_s^{\theta},\mu_s^{\theta_0})\mathrm{d}s = \int_{\mathbb{R}^d} J(\theta,x,\pi_{\theta},\pi_{\theta_0})\pi_{\theta_0}(\mathrm{d}x).
    \end{equation}
    We need to show that $\smash{\mathbb{E}[|\frac{1}{t}\int_0^t  J(\theta,x_s,\mu_s^{\theta},\mu_s^{\theta_0})\mathrm{d}s - \int_{\mathbb{R}^d}J(\theta,x,\pi_{\theta},\pi_{\theta_0})\pi_{\theta_0}(\mathrm{d}x)|] \stackrel{t\rightarrow\infty}{\longrightarrow}0}$. To do so, we begin by using the triangle inequality to write
    \begin{align}
        &\Big|\frac{1}{t}\int_0^t J(\theta,x_s,\mu_s^{\theta},\mu_s^{\theta_0})\mathrm{d}s - \int_{\mathbb{R}^d} J(\theta,x,\pi_{\theta},\pi_{\theta_0})\pi_{\theta_0}(\mathrm{d}x)\Big| \leq H_t^{(1)} + H_t^{(2)} + H_t^{(3)}, \label{eq:h-t-decomp}
    \end{align}
    where  $\smash{H_t^{(1)}:=\frac{1}{t}\int_0^t |J(\theta,x_s,\mu_s^{\theta},\mu_s^{\theta_0}) - J(\theta,x_s,\pi_{\theta},\pi_{\theta_0})|\mathrm{d}s}$, $\smash{{H_t^{(2)}:=\frac{1}{t} \int_0^t|J(\theta,x_s,\pi_{\theta},\pi_{\theta_0}) - J(\theta,\bar{x}_s,\pi_{\theta},\pi_{\theta_0})|\mathrm{d}s}}$, and $\smash{{H_t^{(3)}:=|\frac{1}{t}\int_0^t J(\theta,\bar{x}_s,\pi_{\theta},\pi_{\theta_0}) \mathrm{d}s - \int_{\mathbb{R}^d} J(\theta,x,\pi_{\theta},\pi_{\theta_0})\pi_{\theta_0}(\mathrm{d}x)|}}$, and where $(x_t)_{t\geq 0}$ and $(\bar{x}_t)_{t\ge 0}$ denote two solutions of the MVSDE, both driven by the same Brownian motion, but initialized with $x_0\sim{\mu_0}$ and $\bar{x}_0\sim \pi_{\theta_0}$, respectively.\footnote{We note that $\mathrm{Law}(\bar{x}_t)=\pi_{\theta_0}$ for all $t\ge 0$, since $\pi_{\theta_0}$ is the unique stationary distribution of the MVSDE \citep[e.g.,][]{genoncatalot2023parametric}.It follows, in particular, that $(\bar{x}_t)_{t\geq 0}$ is a positive recurrent ergodic diffusion in its stationary regime \citep[][Proposition 2]{genoncatalot2023parametric}, given by $\smash{\mathrm{d}\bar{x}_t = B(\theta_0,\bar{x}_t,\pi_{\theta_0})\,\mathrm{d}t + \sigma\,\mathrm{d}w_t}$ with $\smash{\bar{x}_0\sim \pi_{\theta_0}}$. \label{footnote:ergodic}} We assume that $(x_0,\bar{x}_0)\sim \gamma_0^{*}$, where $\gamma_0^{*}\in\Gamma(\mu_0,\pi_{\theta_0})$ is the optimal coupling between the initial conditions $\mu_0$ and $\pi_{\theta_0}$ w.r.t. quadratic cost. 
    
    We begin by bounding $\smash{H_t^{(1)}}$. Under our assumptions, there exists a constant $K<\infty$ and an integer $q\geq 1$ such that, for all $s\geq 0$, $|J(\theta,x,\mu_s^{\theta},\mu_s^{\theta_0})-J(\theta,x,\pi_{\theta},\pi_{\theta_0})|\le K\,(\mathsf{W}_2(\mu_s^{\theta},\pi_{\theta}) +\mathsf{W}_2(\mu_s^{\theta_0},\pi_{\theta_0}))(1+\|x\|^q+\mu_s^{\theta}(\|\cdot\|^q)+ \mu_s^{\theta_0}(\|\cdot\|^q)+\pi_{\theta}(\|\cdot\|^q) + \pi_{\theta_0}(\|\cdot\|^q))$. Furthermore, under the additional assumption in Proposition~\ref{prop:mvsde-likelihood-t-limit} (i.e., Assumption~\ref{assumption:drift} holds for all $\nu\in\Theta$), Theorem~\ref{thm:moment-bounds} (i.e., moment bounds of the MVSDE, uniform-in-time) hold for each fixed $\nu\in\Theta$ and thus, in particular, for $\nu\in\{\theta,\theta_0\}$. It follows that, for all $s\geq 0$, 
\begin{equation} \mathbb{E}\big[\big|J(\theta,x_s,\mu_s^{\theta},\mu_s^{\theta_0})-J(\theta,x_s,\pi_{\theta},\pi_{\theta_0})\big|\big] 
\le K\,\big(\mathsf{W}_2(\mu_s^{\theta},\pi_{\theta}) + \mathsf{W}_2(\mu_s^{\theta_0},\pi_{\theta_0})\big), \label{eq:j-exp-bound}
\end{equation}
for some finite constant $0<K<\infty$. Next, under the additional Assumption in Proposition~\ref{prop:mvsde-likelihood-t-limit} (i.e., Assumption~\ref{assumption:drift} holds for all $\nu\in\Theta$), the results of Theorem~\ref{thm:invariant-distribution-2} (i.e., convergence to the invariant distribution of the MVSDE) hold with $\theta_0$ replaced by any $\nu\in\Theta$. Thus, in particular, we have $\mathsf{W}_2(\mu_s^{\nu},\pi_{\nu})\rightarrow 0$ as $s\rightarrow\infty$ for all $\nu\in\Theta$. Using this result, the bound in \eqref{eq:j-exp-bound}, the definition of $\smash{H_t^{(1)}}$, and C\'esaro's Theorem, we thus have
    \begin{align}
\mathbb{E}[H_t^{(1)}]
&\le \mathbb{E}\Big[\frac{1}{t}\int_0^t |J(\theta,x_s,\mu_s^{\theta},\mu_s^{\theta_0}) - J(\theta,x_s,\pi_{\theta},\pi_{\theta_0})|\mathrm{d}s\Big] \stackrel{t\rightarrow \infty}{\longrightarrow 0}.
\label{eq:h1-bound-final}
\end{align}
We next consider $H_t^{(2)}$. Similar to above, using a minor generalization of Lemma \ref{lem:Phi-L1-stability}, there exists a constant $K<\infty$ and an integer $q\geq 1$ such that, for all $s\geq 0$, $|J(\theta,x,\pi_{\theta},\pi_{\theta_0})-J(\theta,\bar{x},\pi_{\theta},\pi_{\theta_0})|\le K\,
\|x-\bar{x}\|(1+\|x\|^q+\pi_{\theta}(\|\cdot\|^q) + \pi_{\theta_0}(\|\cdot\|^q))$. Taking expectations, using Cauchy-Schwarz, and Theorem~\ref{thm:moment-bounds} (i.e., uniform in time moment bounds for the MVSDE) for $\nu\in\{\theta,\theta_0\}$, it follows that, for all $s\geq 0$,
\begin{align}
& \mathbb{E}\big[|J(\theta,x_s,\pi_{\theta},\pi_{\theta_0})-J(\theta,\bar{x}_s,\pi_{\theta},\pi_{\theta_0})|\big] \leq K \big(\mathbb{E}\left[\|x_s-\bar x_s\|^2\right]\big)^{1/2}.
\label{eq:Phi-x-L1-simplified-final}
\end{align}
Meanwhile, by Theorem~\ref{thm:invariant-distribution-2} (i.e., convergence to the invariant distribution of the MVSDE), we have that $\mathbb{E}[\|x_s-\bar{x}_s\|^2]\rightarrow 0$ as $s\rightarrow\infty$. Since, by Theorem~\ref{thm:moment-bounds} (i.e., moment bounds for the MVSDE, uniform-in-time), the integrand is uniformly bounded in $\mathrm{L}^1$, it follows by Ces\`aro's theorem that
\begin{align}
\mathbb{E}[H_t^{(2)}]:=\mathbb{E}\Big[\frac{1}{t} \int_0^t\left|J(\theta,x_s,\pi_{\theta},\pi_{\theta_0}) - J(\theta,\bar{x}_s,\pi_{\theta},\pi_{\theta_0})\right|\mathrm{d}s\Big]\stackrel{t\rightarrow\infty}{\longrightarrow}0. \label{eq:h2-final-limit}
\end{align}
We now consider $\smash{H_t^{(3)}}$. By Theorem \ref{thm:moment-bounds} (i.e., moment bounds for the MVSDE, uniform-in-time) the function $x\mapsto J(\theta,x,\pi_{\theta},\pi_{\theta_0})\in \mathrm{L}^1(\pi_{\theta_0})$. In addition, $(\bar{x}_t)_{t\geq 0}$ is a positive recurrent ergodic diffusion process in its stationary regime (see Footnote~\ref{footnote:ergodic}). We can thus apply the ergodic theorem (e.g., Theorem 4.2, \citealt{khasminskii2012stochastic}; Theorem 17.0.1, \citealt{meyn2009markov}) to conclude that, both a.s. and in $\mathrm{L}^1$, $\smash{\frac{1}{t}\int_0^t J(\theta,\bar{x}_s,\pi_{\theta},\pi_{\theta_0}) \mathrm{d}s \stackrel{t\rightarrow\infty}{\longrightarrow} \int_{\mathbb{R}^d} J(\theta,x,\pi_{\theta},\pi_{\theta_0})\pi_{\theta_0}(\mathrm{d}x)}$. Thus, in particular, we have shown that 
    \begin{equation}
        \mathbb{E}[H_t^{(3)}]:=\mathbb{E}\Big[\Big|\frac{1}{t}\int_0^t J(\theta,\bar{x}_s,\pi_{\theta},\pi_{\theta_0}) \mathrm{d}s - \int_{\mathbb{R}^d} J(\theta,x,\pi_{\theta},\pi_{\theta_0})\pi_{\theta_0}(\mathrm{d}x)\Big|\Big]\stackrel{t\rightarrow\infty}{\longrightarrow} 0. \label{eq:h3-final-limit}
    \end{equation}
    Taking expectations in \eqref{eq:h-t-decomp}, and substituting the bounds in \eqref{eq:h1-bound-final}, \eqref{eq:h2-final-limit}, and \eqref{eq:h3-final-limit}, it follows at last that $\smash{\mathbb{E}[|\frac{1}{t}\int_0^t  J(\theta,x_s,\mu_s^{\theta},\mu_s^{\theta_0})\mathrm{d}s - \int_{\mathbb{R}^d}J(\theta,x,\pi_{\theta},\pi_{\theta_0})\pi_{\theta_0}(\mathrm{d}x)|] \stackrel{t\rightarrow\infty}{\longrightarrow}0}$. This establishes the limit in \eqref{eq:j-limit}.

    It remains to establish $\mathrm{L}^1$ convergence of the second term in \eqref{eq:i1-i2-decomp} to zero. To do so, define the continuous local martingale $\smash{M_t := \int_0^t \big\langle \Delta B(\theta,x_s,\mu_s^{\theta},\mu_s^{\theta_0}),\,(\sigma\sigma^\top)^{-1}\sigma\,\mathrm{d}w_s\big\rangle}$. By the It\^o isometry, we then have $\smash{\mathbb{E}[M_t^2]= \int_0^t \mathbb{E}[\|B(\theta,x_s,\mu_s^{\theta}) - B(\theta_0,x_s,\mu_s^{\theta_0})\|^2_{\sigma\sigma^\top}]\mathrm{d}s = 2 \int_0^t \mathbb{E}[J(\theta,x_s,\mu_s^{\theta},\mu_s^{\theta_0})]\,\mathrm{d}s}$. Using the PGP of the function $J$, and Theorem~\ref{thm:moment-bounds} (i.e., the moment bounds for the MVSDE, uniform-in-time), there exists a constant $K<\infty$ such that $\sup_{s\geq 0}\mathbb{E}[|J(\theta,x_s,\mu_s^{\theta},\mu_s^{\theta_0})|]\le K$. We thus have, as required, that
\begin{equation}
\mathbb{E}\big[\big|\frac{1}{t}M_t\big|\big]
\le \big(\mathbb{E}\big[\big|\frac{1}{t}M_t\big|^2\big]\big)^{1/2}
= \frac{1}{t}\big(\mathbb{E}[M_t^2]\big)^{\frac{1}{2}}
\le \frac{1}{t}\sqrt{Kt} = \frac{K}{\sqrt{t}} \stackrel{t\to\infty}{\longrightarrow}0.
\end{equation}
This establishes $\mathrm{L}^1$ convergence of the second term in \eqref{eq:i1-i2-decomp} to zero, and thus completes the proof.
\end{proof}

\subsection{Proofs for Section \ref{sec:grad-log-lik}}
\label{app:additional-proofs-grad-log-lik}

\begin{proof}[Proof of Proposition \ref{prop:ips-likelihood-n-t-limit-grad}]
Recall, from \eqref{eq:obj-func-1}, that the negative asymptotic log-likelihood of the IPS is defined according to $\smash{\mathcal L(\theta)=\int_{\mathbb{R}^d} L(\theta,x,\pi_{\theta_0})\,\pi_{\theta_0}(\mathrm{d}x)}$, where $L(\theta,x,\mu):=\frac12\|B(\theta,x,\mu)-B(\theta_0,x,\mu)\|_{\sigma\sigma^{\top}}^2$. Using the chain rule, it is straightforward to compute the derivative of the integrand as
\begin{equation}
\partial_{\theta} L(\theta,x,\pi_{\theta_0})
=
\partial_{\theta}B(\theta,x,\pi_{\theta_0})(\sigma\sigma^{\top})^{-1}\big(B(\theta,x,\pi_{\theta_0})-B(\theta_0,x,\pi_{\theta_0})\big),
\label{eq:grad-pointwise}
\end{equation}
where $\smash{\partial_{\theta}B(\theta,x,\pi_{\theta_0})\in\mathbb{R}^{p\times d}}$. By Lemma~\ref{lem:weighted-lip-g}, $\smash{(x,y)\mapsto \partial_{\theta}b(\theta,x,y)}$ satisfies a polynomial growth property, uniformly in $\theta\in\Theta$. Meanwhile, by Theorem~\ref{thm:moment-bounds}, $\pi_{\theta_0}$ has finite moments of all orders. Thus, by the dominated convergence theorem (DCT), we have $\smash{\partial_\theta B(\theta,x,\pi_{\theta_0}) = \int_{\mathbb{R}^d} \partial_{\theta} b(\theta,x,y)\,\pi_{\theta_0}(\mathrm dy)
=:G(\theta,x,\pi_{\theta_0})}$ for each $x\in\mathbb{R}^d$. Substituting this into \eqref{eq:grad-pointwise} yields 
\begin{equation}
\partial_{\theta} L(\theta,x,\pi_{\theta_0})
= G(\theta,x,\pi_{\theta_0}) (\sigma\sigma^{\top})^{-1} (B(\theta,x,\pi_{\theta_0}) - B(\theta_0,x,\pi_{\theta_0})) =: H(\theta,x,\pi_{\theta_0}). 
\end{equation}
It remains to justify that we can differentiate under the integral sign in the asymptotic log-likelihood function. By Lemma~\ref{lem:Phi-hH-stability}, the map $x\mapsto \partial_{\theta}L(\theta,x,\pi_{\theta_0})$ satisfies a polynomial growth property, uniformly over $\theta\in\Theta$. By Theorem~\ref{thm:moment-bounds}, $\pi_{\theta_0}$ has finite moments of all orders. Therefore, once more using the DCT, we conclude as required that $\smash{\partial_{\theta} \mathcal L(\theta)=\int_{\mathbb{R}^d} \partial_{\theta} L(\theta,x,\pi_{\theta_0})\,\pi_{\theta_0}(\mathrm{d}x)=\int_{\mathbb{R}^d} H(\theta,x,\pi_{\theta_0})\,\pi_{\theta_0}(\mathrm{d}x)}$.
\end{proof}

\subsection{Proofs for Section \ref{sec:theory-add-notation}}
\label{app:additional-proofs-theory-add-notation}

\begin{proof}[Proof of Proposition \ref{prop:asymptotic-partial-log-lik-grad}]
We begin by proving the statement for $\mathcal L^{i,N}$ in \eqref{eq:incomplete-log-lik-1-grad}. This part of the proof is very similar to the proof of Proposition~\ref{prop:ips-likelihood-n-t-limit-grad}. Recall, from \eqref{eq:partial-lik-1}, that  $\smash{\mathcal{L}^{i,N}(\theta)=\int_{(\mathbb R^d)^N} L^{i,N}(\theta,\boldsymbol x^N)\,\pi_{\theta_0}^N(\mathrm d\boldsymbol x^N)}$, where $L^{i,N}(\theta,\boldsymbol x^N):=
\frac{1}{2}\|B^{i,N}(\theta,\boldsymbol x^N)-B^{i,N}(\theta_0,\boldsymbol x^N)\|_{\sigma\sigma^\top}^2$. By the chain rule, for all fixed $\boldsymbol x^N\in(\mathbb R^d)^N$, we can compute the derivative of the integrand as 
\begin{equation}
\partial_{\theta} L^{i,N}(\theta,\boldsymbol x^N)
=
\partial_{\theta} B^{i,N}(\theta,\boldsymbol x^N)\,(\sigma\sigma^\top)^{-1}\,
\big(B^{i,N}(\theta,\boldsymbol x^N)-B^{i,N}(\theta_0,\boldsymbol x^N)\big).
\label{eq:grad-LiN-config}
\end{equation}
In addition, recalling that $\smash{B^{i,N}(\theta,\boldsymbol x^N)=\frac1N\sum_{j=1}^N b(\theta,x^{i,N},x^{j,N})}$, we can also compute $\smash{\partial_{\theta} B^{i,N}(\theta,\boldsymbol x^N)=}$ $\smash{
\frac1N\sum_{j=1}^N \partial_{\theta}b(\theta,x^{i,N},x^{j,N})
=
\int_{\mathbb R^d} \partial_{\theta} b(\theta,x^{i,N},y)\,\mu^N(\mathrm{d}y)
=G(\theta,x^{i,N},\mu^N) = G^{i,N}(\theta,\boldsymbol{x}^N)}$.
Substituting this into \eqref{eq:grad-LiN-config}, we arrive at
\begin{equation}
    \partial_{\theta} L^{i,N}(\theta,\boldsymbol x^N)= G^{i,N}(\theta,\boldsymbol{x}^N)(\sigma\sigma^\top)^{-1}\,
\big(B^{i,N}(\theta,\boldsymbol x^N)-B^{i,N}(\theta_0,\boldsymbol x^N)\big) =: H^{i,N}(\theta,\boldsymbol x^N).
\end{equation}
It remains to justify differentiation under the outer integral in the definition of $\mathcal{L}^{i,N}$. By Corollary~\ref{cor:empirical-H-lip}, there exist a constant $K<\infty$ and an integer $q\ge1$ such that for all $\theta\in\Theta$, all $N\in\mathbb N$, and all $\boldsymbol x^N\in(\mathbb R^d)^N$, $\smash{\|\partial_{\theta} L^{i,N}(\theta,\boldsymbol x^N)\| = \|H^{i,N}(\theta,\boldsymbol x^N)\| \le
K(1+\|x^{i,N}\|^q+\frac1N\sum_{j=1}^N\|x^{j,N}\|^q)}$.
By Theorem~\ref{thm:moment-bounds}, $\pi_{\theta_0}^N$ has finite moments of all orders. By the DCT, we may thus differentiate under the integral sign to obtain $\smash{\partial_{\theta} \mathcal L^{i,N}(\theta)=\int_{(\mathbb R^d)^N}\partial_{\theta} L^{i,N}(\theta,\boldsymbol x^N)\,\pi_{\theta_0}^N(\mathrm d\boldsymbol x^N)=\int_{(\mathbb R^d)^N} H^{i,N}(\theta,\boldsymbol x^N)\,\pi_{\theta_0}^N(\mathrm d\boldsymbol x^N)}$.

We now turn to establish the result for $\mathcal{L}^{i,j,k,N}$ in \eqref{eq:incomplete-log-lik-2-grad}. Recall that $\smash{\mathcal L^{i,j,k,N}(\theta)=\int_{(\mathbb R^d)^N}\ell^{i,j,k,N}(\theta,\boldsymbol x^N)\,\pi_{\theta_0}^N(\mathrm d\boldsymbol x^N)}$, where $\smash{\ell^{i,j,k,N}(\theta,\boldsymbol x^N):=
\frac12\langle
b(\theta,x^{i,N},x^{j,N})-B\!(\theta_0,x^{i,N},\mu^N),\, b(\theta,x^{i,N},x^{k,N})-B\!(\theta_0,x^{i,N},\mu^N) \rangle_{\sigma\sigma^\top}}$. Using the product rule, we can compute the gradient of the integrand as
\begin{align}
\partial_{\theta} \ell^{i,j,k,N}(\theta,\boldsymbol x^N)&=\frac12\big(
\partial_{\theta} b(\theta,x^{i,N},x^{j,N})\,(\sigma\sigma^\top)^{-1}\big[b(\theta,x^{i,N},x^{k,N})-B(\theta_0,x^{i,N},\mu^N)\big]\\
&\hspace{6mm}+\partial_{\theta} b(\theta,x^{i,N},x^{k,N})\,(\sigma\sigma^\top)^{-1}\big[b(\theta,x^{i,N},x^{j,N})-B(\theta_0,x^{i,N},\mu^N)\big]\big)\nonumber\\
&=\frac12\Big(h^{i,j,k,N}(\theta,\boldsymbol x^N)+h^{i,k,j,N}(\theta,\boldsymbol x^N)\Big),
\label{eq:grad-ell-ijk-config}
\end{align}
Once more, it remains to justify differentiation under the integral sign. By Corollary~\ref{cor:empirical-H-lip}, there exist a constant $K<\infty$ and an integer $q\ge1$ such that, for all $\theta\in\Theta$, for all $N\in\mathbb{N}$, and for all $\boldsymbol{x}^N\in(\mathbb{R}^d)^N$, $\smash{\|\partial_{\theta} \ell^{i,j,k,N}(\theta,\boldsymbol x^N)\| \le
K(1+\sum_{a\in\{i,j,k\}}\|x^{a,N}\|^q+\frac1N\sum_{a=1}^N\|x^{a,N}\|^q)}$. Similar to above, the right-hand side of this bound is integrable with respect to $\pi_{\theta_0}^N$ by Theorem~\ref{thm:moment-bounds}. Thus, using the DCT, we can conclude that 
\begin{align}
\partial_{\theta} \mathcal L^{i,j,k,N}(\theta)=\int_{(\mathbb R^d)^N} \!\!\! \partial_{\theta} \ell^{i,j,k,N}(\theta,\boldsymbol x^N)\,\pi_{\theta_0}^N(\mathrm d\boldsymbol x^N)=\frac12\int_{(\mathbb R^d)^N}\!\!\Big(h^{i,j,k,N}(\theta,\boldsymbol x^N)+h^{i,k,j,N}(\theta,\boldsymbol x^N)\Big)\,\pi_{\theta_0}^N(\mathrm d\boldsymbol x^N).
\notag 
\end{align}
Finally, using the definition of $h^{i,j,k,N}$, and the exchangeability of $\pi_{\theta_0}^N$, the two integrals coincide. We thus have, as required, $\smash{\partial_{\theta} \mathcal L^{i,j,k,N}(\theta)=\int_{(\mathbb R^d)^N} h^{i,j,k,N}(\theta,\boldsymbol x^N)\,\pi_{\theta_0}^N(\mathrm d\boldsymbol x^N)}$.
\end{proof}

\begin{proof}[Proof of Proposition \ref{prop:inf-n-convergence-1}]
We will first require some additional notation. We begin by defining two \emph{pseudo log-likelihood} functions $\mathcal{L}_t^{i,N}:\mathbb{R}^p\rightarrow\mathbb{R}$ and $\mathcal{L}_t^{i,j,k,N}:\mathbb{R}^p\rightarrow\mathbb{R}$ for the IPS according to the identities
\begin{align}
\mathcal{L}_t^{i,N}(\theta) - \mathcal{L}_t^{i,N}(\theta_0) 
&:=  -\int_0^t L(\theta,x_s^{i,N},\mu_s^N)\,\mathrm{d}s + \int_0^t \big\langle \Delta B(\theta,x_s^{i,N},\mu_s^N) , \sigma \mathrm{d}w_s^{i,N}\big\rangle_{\sigma\sigma^{\top}} \label{eq:IPS-incomplete-likelihood-1} \\[-1mm]
\mathcal{L}_t^{i,j,k,N}(\theta) - \mathcal{L}_t^{i,j,k,N}(\theta_0) &\stackrel{+c}{:=} -\int_{0}^t \ell(\theta,x_s^{i,N},x_s^{j,N},x_s^{k,N},\mu_s^N) \mathrm{d}s + \int_0^t \big\langle  \Delta b(\theta,x_s^{i,N},x_s^{j,N}) , \sigma \mathrm{d}w_s^{i,N}\big\rangle_{\sigma\sigma^{\top}} \label{eq:IPS-incomplete-likelihood-2} 
\end{align}
where $\smash{\stackrel{+c}{:=}}$ indicates that the definition is up to an additive constant, determined by the constraint the right-hand side must equal zero at $\theta=\theta_0$. We note that this constant is independent of $\theta$, and thus vanishes after applying $\partial_{\theta}$. We next define the functions $\smash{\mathcal{L}_t^{[i,N]}:\mathbb{R}^p\rightarrow\mathbb{R}}$ and $\smash{\mathcal{L}_t^{[i,j,k,N]}:\mathbb{R}^p\rightarrow\mathbb{R}}$ according to 
\begin{align}
\mathcal{L}_t^{[i,N]}(\theta) - \mathcal{L}_t^{[i,N]}(\theta_0) 
&:=  -\int_0^t L(\theta,x_s^{i},\mu_s^{[N]})\,\mathrm{d}s + \int_0^t \big\langle \Delta B(\theta,x_s^{i},\mu_s^{[N]}) , \sigma \mathrm{d}w_s^{i}\big\rangle_{\sigma\sigma^{\top}} \label{eq:mckean-copy-incomplete-likelihood-1} \\
\mathcal{L}_t^{[i,j,k,N]}(\theta) - \mathcal{L}_t^{[i,j,k,N]}(\theta_0) &\stackrel{+c}{:=} -\int_{0}^t \ell(\theta,x_s^{i},x_s^{j},x_s^{k},\mu_s^{[N]}) \mathrm{d}s + \int_0^t \big\langle  \Delta b(\theta,x_s^{i},x_s^{j}) , \sigma \mathrm{d}w_s^{i}\big\rangle_{\sigma\sigma^{\top}}  \label{eq:mckean-copy-incomplete-likelihood-2} 
\end{align}
where $\smash{(x_t^{i})^{i\in[N]}_{t\geq 0}}$ denotes an independent family of solutions of the MVSDE, driven by the same Brownian motions as the corresponding particles $\smash{(x_t^{i,N})_{t\geq0}^{i\in[N]}}$, and with the same initial conditions (i.e., the synchronous coupling), and where $\smash{\mu_t^{[N]} = \frac{1}{N}\sum_{j=1}^N \delta_{x_t^{j}}}$. Finally, we define the functions $\mathcal{L}_t^{i}:\mathbb{R}^p\rightarrow\mathbb{R}$ and $\mathcal{L}_t^{i,j,k}:\mathbb{R}^p\rightarrow\mathbb{R}$ via
\begin{align}
\mathcal{L}_t^{i}(\theta) - \mathcal{L}_t^{i}(\theta_0) 
&:=  -\int_0^t L(\theta,x_s^{i},\mu_s)\,\mathrm{d}s + \int_0^t \big\langle \Delta B(\theta,x_s^{i},\mu_s) , \sigma \mathrm{d}w_s^{i}\big\rangle_{\sigma\sigma^{\top}} \label{eq:mckean-incomplete-likelihood-1} \\
\mathcal{L}_t^{i,j,k}(\theta) - \mathcal{L}_t^{i,j,k}(\theta_0) &\stackrel{+c}{:=} -\int_{0}^t \ell(\theta,x_s^{i},x_s^{j},x_s^{k},\mu_s) \mathrm{d}s + \int_0^t \big\langle  \Delta b(\theta,x_s^{i},x_s^{j}) , \sigma \mathrm{d}w_s^{i}\big\rangle_{\sigma\sigma^{\top}}. \label{eq:mckean-incomplete-likelihood-2} 
\end{align}
 
\noindent We can now prove the stated results. Using the triangle inequality, and the fact the LHS is deterministic (it is an integral w.r.t. the invariant measure of the IPS), so that $\|\partial_{\theta}\mathcal{L}^{i,N}(\theta) - \partial_{\theta}\mathcal{L}(\theta) \|=\mathbb{E}[\|\partial_{\theta}\mathcal{L}^{i,N}(\theta) - \partial_{\theta}\mathcal{L}(\theta) \|]$ and $\|\partial_{\theta}\mathcal{L}^{i,j,k,N}(\theta) - \partial_{\theta}\mathcal{L}(\theta) \|=\mathbb{E}[\|\partial_{\theta}\mathcal{L}^{i,j,k,N}(\theta) - \partial_{\theta}\mathcal{L}(\theta) \|]$, we have
\begin{align}
\|\partial_{\theta}\mathcal{L}^{i,N}(\theta) - \partial_{\theta}\mathcal{L}(\theta) \| &\leq  \mathbb{E}\left[\|\partial_{\theta}\mathcal{L}^{i,N}(\theta) - \tfrac{1}{t}\partial_{\theta}{\mathcal{L}}_t^{i,N}(\theta) \|\right] + \mathbb{E} \left[\|\tfrac{1}{t}\partial_{\theta}\mathcal{L}_t^{i,N}(\theta) - \tfrac{1}{t}\partial_{\theta}{\mathcal{L}}_t^{[i,N]}(\theta) \|\right]  \nonumber \\
&+\mathbb{E} \left[\|\tfrac{1}{t}\partial_{\theta}{\mathcal{L}}_t^{[i,N]}(\theta) - \tfrac{1}{t}\partial_{\theta}{\mathcal{L}}_t^{i}(\theta) \|\right] +\mathbb{E}\left[\|\tfrac{1}{t}\partial_{\theta}{\mathcal{L}}_t^{i}(\theta) - \partial_{\theta}\mathcal{L}(\theta)\|\right] \label{eq:grad-l-decomp-1} \\
\|\partial_{\theta}\mathcal{L}^{i,j,k,N}(\theta) - \partial_{\theta}\mathcal{L}(\theta) \| 
&\leq \mathbb{E}\left[\|\partial_{\theta}\mathcal{L}^{i,j,k,N}(\theta) - \tfrac{1}{t}\partial_{\theta}{\mathcal{L}}_t^{i,j,k,N}(\theta) \|\right] + \mathbb{E} \left[\|\tfrac{1}{t}\partial_{\theta}\mathcal{L}_t^{i,j,k,N}(\theta) - \tfrac{1}{t}\partial_{\theta}{\mathcal{L}}_t^{[i,j,k,N]}(\theta) \|\right]  \nonumber \\
&+\mathbb{E} \left[\|\tfrac{1}{t}\partial_{\theta}{\mathcal{L}}_t^{[i,j,k,N]}(\theta) - \tfrac{1}{t}\partial_{\theta}{\mathcal{L}}_t^{i,j,k}(\theta) \|\right] +\mathbb{E}\left[\|\tfrac{1}{t}\partial_{\theta}{\mathcal{L}}_t^{i,j,k}(\theta) - \partial_{\theta}\mathcal{L}(\theta)\|\right]. \label{eq:grad-l-decomp-2}
\end{align}
By Lemma \ref{prop:asymptotic-partial-log-lik-grad-l1-convergence} and Lemma \ref{prop:asymptotic-partial-log-lik-grad-l1-convergence-mvsde} (see Appendix \ref{app:additional-results-prop-inf-n}), we have that 
\begin{alignat}{2}
    \lim_{t\rightarrow\infty} \mathbb{E}\big[\|\partial_{\theta}\mathcal{L}^{i,N}(\theta) - \tfrac{1}{t}\partial_{\theta}{\mathcal{L}}_t^{i,N}(\theta) \|\big] &= 0, \qquad &&\lim_{t\rightarrow\infty} \mathbb{E}\big[\|\tfrac{1}{t}\partial_{\theta}{\mathcal{L}}_t^{i}(\theta) - \partial_{\theta}\mathcal{L}(\theta)\|\big]  = 0 \label{eq:l-i-n-bounds}\\
    \lim_{t\rightarrow\infty} \mathbb{E}\big[\|\partial_{\theta}\mathcal{L}^{i,j,k,N}(\theta) - \tfrac{1}{t}\partial_{\theta}{\mathcal{L}}_t^{i,j,k,N}(\theta) \|\big] &= 0 ,\qquad  &&\lim_{t\rightarrow\infty} \mathbb{E}\big[\|\tfrac{1}{t}\partial_{\theta}{\mathcal{L}}_t^{i,j,k}(\theta) - \partial_{\theta}\mathcal{L}(\theta)\|\big] = 0 \label{eq:l-i-j-k-n-bounds}
\end{alignat}
By Lemma \ref{lemma:main-theorem-1-2-lemma-1} and Lemma \ref{lemma:main-theorem-1-2-lemma-2} (see Appendix \ref{app:additional-results-prop-inf-n}), 
there exist constants $K_1,K_1^{\dagger},K_2,K_2^{\dagger}$ such that, for all $t\geq t_0>0$ (e.g., $t_0=1$), and for all $N\in\mathbb{N}$, 
    \begin{align} 
        \limsup_{t\rightarrow\infty} \mathbb{E}\big[\big\| \tfrac{1}{t}\partial_{\theta}\mathcal{L}_t^{[i,N]}(\theta) - \tfrac{1}{t}\partial_{\theta}{\mathcal{L}}_t^{i}(\theta)\big\|\big] &\leq K_1\rho(N)
        \label{eq:lt-quant-v1-v2-recall} \\
        \limsup_{t\rightarrow\infty}\mathbb{E}\big[\big\| \tfrac{1}{t}\partial_{\theta}\mathcal{L}_t^{i,N}(\theta) - \tfrac{1}{t}\partial_{\theta}\mathcal{L}_t^{[i,N]}(\theta)\big\|\big] &\leq K_2 N^{-\frac{1}{2(1+\alpha)}} 
        \label{eq:lt-quant-v1-recall}
    \end{align}
and
    \begin{align} 
        \limsup_{t\rightarrow\infty}\mathbb{E}\big[\big\| \tfrac{1}{t}\partial_{\theta}\mathcal{L}_t^{[i,j,k,N]}(\theta) - \tfrac{1}{t}\partial_{\theta}{\mathcal{L}}_t^{i,j,k}(\theta)\big\|\big] &\leq K_1^{\dagger}\rho(N). \label{eq:lt-quant-v2-v2-recall} \\
        \limsup_{t\rightarrow\infty}\mathbb{E}\big[\big\| \tfrac{1}{t}\partial_{\theta}\mathcal{L}_t^{i,j,k,N}(\theta) - \tfrac{1}{t}\partial_{\theta}\mathcal{L}_t^{[i,j,k,N]}(\theta)\big\|\big] &\leq K_2^{\dagger}N^{-\frac{1}{2(1+\alpha)}}. 
        \label{eq:lt-quant-v2-recall}
    \end{align}
Taking $\limsup_{t\rightarrow\infty}$ of \eqref{eq:grad-l-decomp-1} and \eqref{eq:grad-l-decomp-2}, substituting  the bounds in \eqref{eq:l-i-n-bounds}, \eqref{eq:lt-quant-v1-v2-recall}, \eqref{eq:lt-quant-v1-recall},  or \eqref{eq:l-i-j-k-n-bounds}, \eqref{eq:lt-quant-v2-v2-recall}, \eqref{eq:lt-quant-v2-recall}, and noting that the bound holds uniformly over $\theta\in\Theta$, we have the stated result. 
\end{proof}

\subsection{Proofs for Section \ref{sec:main-results-first-estimator}}
\label{app:additional-proofs-main-results}


\subsubsection{Proofs for Section \ref{sec:main-results-convergence}}
\label{app:additional-proofs-main-results-convergence}

\begin{proof}[Proof of Proposition \ref{prop:inf-t-convergence-1}]
     We follow closely the proof of \cite[Theorem 2.4]{sirignano2017stochastic}, adapted appropriately to the current setting. Let $\kappa>0$. Define the sequence of stopping times $0=\sigma_0\leq \tau_1\leq \sigma_1\leq \tau_2\leq \sigma_2\leq \dots$ according to
     \small
    \begin{align}
        \tau_r &= \inf \Big\{t>\sigma_{r-1}: \|\partial_{\theta}\mathcal{L}^{i,N}(\bar{\theta}_t^{i,N})\|\geq \kappa\Big\} \label{eq:stop-times-i-eq-1}\\[2mm]
        \sigma_r &= \sup \Big\{t\geq\tau_r: \frac{1}{2} \|\partial_{\theta}\mathcal{L}^{i,N}(\bar{\theta}_{\tau_r}^{i,N}) \|\leq \|\partial_{\theta}\mathcal{L}^{i,N}(\bar{\theta}_{s}^{i,N}) \|\leq 2 \|\partial_{\theta}\mathcal{L}^{i,N}(\bar{\theta}_{\tau_r}^{i,N}) \| ~~\forall s\in[\tau_r,t], \textstyle \int_{\tau_r}^t \gamma_s\mathrm{d}s \leq \lambda \Big\}. \label{eq:stop-times-i-eq-2} 
    \intertext{or}
        \tau_r &= \inf \Big\{t>\sigma_{r-1}: \|\partial_{\theta}\mathcal{L}^{i,j,k,N}(\theta_t^{i,j,k,N})\|\geq \kappa\Big\} \label{eq:stop-times-ijk-eq-1} \\[2mm]
        \sigma_r &= \sup \Big\{t\geq\tau_r: \frac{1}{2} \|\partial_{\theta}\mathcal{L}^{i,j,k,N}(\theta_{\tau_r}^{i,j,k,N}) \|\leq \|\partial_{\theta}\mathcal{L}^{i,j,k,N}(\theta_{s}^{i,j,k,N}) \|\leq 2 \|\partial_{\theta}\mathcal{L}^{i,j,k,N}(\theta_{\tau_r}^{i,j,k,N}) \| ~~\forall s\in[\tau_r,t], \textstyle \int_{\tau_r}^t \gamma_s\mathrm{d}s \leq \lambda \Big\}. \label{eq:stop-times-ijk-eq-2}
    \end{align} 
    \normalsize
    We will first prove the result for $(\theta_t^{i,j,k,N})_{t\geq 0}$, using the stopping times defined in \eqref{eq:stop-times-ijk-eq-1} - \eqref{eq:stop-times-ijk-eq-2}. We will consider two subcases. First, suppose that there are a finite number of stopping times $\tau_r$. In this case, there exists a finite $t_0$ such that $\smash{\|\partial_{\theta}\mathcal{L}^{i,j,k,N}(\theta_t^{i,j,k,N})\|<\kappa}$ for all $t\geq t_0$. Since $\kappa>0$ can be chosen arbitrarily small, this implies that $\smash{\lim_{t\rightarrow\infty}\|\partial_{\theta}\mathcal{L}^{i,j,k,N}(\theta_t^{i,j,k,N})\|=0}$. Second, suppose there are an infinite number of stopping times $\tau_r$. By Lemma \ref{lemma:main-theorem-lemma-3} and Lemma \ref{lemma:main-theorem-lemma-4}, there exist $0<\beta_1<\beta$ such that, for sufficiently large $k$, it holds almost surely that
    \begin{equation}
        \mathcal{L}^{i,j,k,N}(\theta_{\sigma_r}^{i,j,k,N}) - \mathcal{L}^{i,j,k,N}(\theta_{\tau_r}^{i,j,k,N})\leq -\beta, \qquad \mathcal{L}^{i,j,k,N}(\theta_{\tau_r}^{i,j,k,N}) - \mathcal{L}^{i,j,k,N}(\theta_{\sigma_{r-1}}^{i,j,k,N})\leq \beta_1. \label{eq:l-decrease-increase-bounds}
    \end{equation}
    It follows, choosing $r_0\in\mathbb{N}$ such that \eqref{eq:l-decrease-increase-bounds} holds for all $r\geq r_0$, that
    \begin{align}
        &\mathcal{L}^{i,j,k,N}(\theta_{\tau_{n+1}}^{i,j,k,N}) - \mathcal{L}^{i,j,k,N}(\theta_{\tau_{r_0}}^{i,j,k,N}) = \sum_{k=k_0}^n\left[\mathcal{L}^{i,j,k,N}(\theta_{\tau_{r+1}}^{i,j,k,N}) - \mathcal{L}^{i,j,k,N}(\theta_{\tau_r}^{i,j,k,N})\right] \\
        &= \sum_{k=k_0}^n\left[\mathcal{L}^{i,j,k,N}(\theta_{\sigma_r}^{i,j,k,N}) - \mathcal{L}^{i,j,k,N}(\theta_{\tau_{r}}^{i,j,k,N}) + \mathcal{L}^{i,j,k,N}(\theta_{\tau_{r+1}}^{i,j,k,N}) - \mathcal{L}^{i,j,k,N}(\theta_{\sigma_{r}}^{i,j,k,N}) \right] \\[2mm]
        &\leq (n+1 - r_0) (-\beta + \beta_1). 
    \end{align}
    Since $-\beta+\beta_1<0$, this display implies that $\smash{\mathcal{L}^{i,j,k,N}(\theta_{\tau_{n+1}}^{i,j,k,N})\rightarrow - \infty}$ as $n\rightarrow\infty$ a.s. But this is a contradiction since $\mathcal{L}^{i,j,k,N}(\theta)$ is bounded below for all $\theta\in\Theta$ (see Lemma~\ref{lemma:main-theorem-lemma-2-a}). It follows that there must a.s. exist a finite number of stopping times $\tau_r$. Thus, in particular, there exists a finite time $t_0$ such that $\smash{\|\partial_{\theta}\mathcal{L}^{i,j,k,N}(\theta_t^{i,j,k,N})\|<\kappa}$ a.s. for all $t\geq t_0$. Since $\kappa>0$ was chosen arbitrarily, this establishes that $\lim_{t\rightarrow\infty} \|\partial_{\theta}\mathcal{L}^{i,j,k,N}(\theta_t^{i,j,k,N})\| = 0$ a.s. 
    It remains to prove that $\lim_{t\rightarrow\infty}\|\partial_{\theta}\mathcal{L}^{i,N}(\bar\theta_t^{i,N})\| = 0$ a.s. The proof in this case is entirely analogous, noting that all of the required lemmas (i.e., Lemma \ref{lemma:main-theorem-lemma-3} and Lemma \ref{lemma:main-theorem-lemma-4}) also hold for this estimator.
\end{proof}

\begin{proof}[Proof of Theorem \ref{theorem:main-result-1}]
We prove the claim for $(\theta_t^{i,j,k,N})_{t\ge 0}$. Fix $N\in\mathbb N$ and $t\ge 0$. By the triangle inequality, we have that $\smash{\|\partial_{\theta}\mathcal{L}(\theta_t^{i,j,k,N})\|\le \|\partial_{\theta}\mathcal{L}(\theta_t^{i,j,k,N})-\partial_{\theta}\mathcal{L}^{i,j,k,N}(\theta_t^{i,j,k,N})\|+\|\partial_{\theta}\mathcal{L}^{i,j,k,N}(\theta_t^{i,j,k,N})\|}$. Using also the fact that $\theta_t^{i,j,k,N}\in\Theta$ a.s., it follows that
\begin{align}
\|\partial_{\theta}\mathcal{L}(\theta_t^{i,j,k,N})\|
&\le
\sup_{\theta\in\Theta}\big\|\partial_{\theta}\mathcal{L}(\theta)
-\partial_{\theta}\mathcal{L}^{i,j,k,N}(\theta)\big\|
+
\|\partial_{\theta}\mathcal{L}^{i,j,k,N}(\theta_t^{i,j,k,N})\|.
\label{eq:exp-bound}
\end{align}
By Proposition \ref{prop:inf-t-convergence-1}, we have $\lim_{t\rightarrow\infty}\|\partial_{\theta}\mathcal{L}^{i,j,k,N}(\theta_t^{i,j,k,N})\| = 0$ a.s., for each fixed $N\in\mathbb{N}$. Meanwhile, by Proposition \ref{prop:inf-n-convergence-1}, we have that $\smash{\lim_{N\rightarrow\infty}\sup_{\theta\in\Theta}\|\partial_{\theta}\mathcal{L}(\theta)
-\partial_{\theta}\mathcal{L}^{i,j,k,N}(\theta)\|=0}$. Taking $\lim_{N\rightarrow\infty}\limsup_{t\rightarrow\infty}$ in \eqref{eq:exp-bound}, and using both of these bounds, it follows that
\begin{equation}
\lim_{N\to\infty}\,\limsup_{t\to\infty}\|\partial_{\theta}\mathcal{L}(\theta_t^{i,j,k,N})\|=0.
\end{equation}
It remains to prove the corresponding claim for $(\bar{\theta}_t^{i,N})_{t\geq 0}$, i.e., that $\smash{\lim_{N\to\infty}\,\limsup_{t\to\infty}\|\partial_{\theta}\mathcal{L}(\bar\theta_t^{i,N})\|=0}$. Similar to before, the proof is essentially identical, noting once more that all of the relevant results (i.e., Proposition~\ref{prop:inf-n-convergence-1} and Proposition~\ref{prop:inf-t-convergence-1}) also hold for this estimator.
\end{proof}

\subsubsection{Proofs for Section \ref{sec:main-results-convergence-rates}}
\label{app:additional-proofs-main-results-convergence-rates}

\begin{proof}[Proof of Theorem \ref{theorem:main-theorem-1-2-finite-n}]
    We begin by proving \eqref{eq:average-l2-rate-1} and \eqref{eq:non-average-l2-rate-1}. In particular, we will prove \eqref{eq:non-average-l2-rate-1} (i.e., the result for the non-averaged estimator), detailing subsequently how to adapt the proof to obtain \eqref{eq:average-l2-rate-1} (i.e., the results for the averaged estimator). We follow the approach used in the proof of \cite[Theorem 1, Proposition 1]{sirignano2020stochastic}. We begin by writing the update equation in the following form:
\begin{align}
\mathrm{d}\theta_t^{i,j,k,N} &= -\underbrace{\gamma_t\partial_{\theta}\mathcal{L}^{i,j,k,N}(\theta_t^{i,j,k,N})\mathrm{d}t}_{\text{true descent term}} - \underbrace{\gamma_t (h^{i,j,k,N}(\theta_t^{i,j,k,N},\boldsymbol{x}_t^N)-\partial_{\theta}\mathcal{L}^{i,j,k,N}(\theta_t^{i,j,k,N}))\mathrm{d}t}_{\text{fluctuations term}}  \\*
&~~~~~+ \underbrace{\gamma_t g^{i,j,N}(\theta_t^{i,j,k,N},\boldsymbol{x}_t^N)\sigma^{-\top} \mathrm{d}w_t^{i,N}}_{\text{noise term}} \label{theta_ideal_2-a-i-ii-finite-n}
\end{align}
Let $\theta_0^{i,j,k,N}$ denote the (unique) minimiser of $\mathcal{L}^{i,j,k,N}$. Then, using a first order Taylor expansion, and the fact that $\partial_{\theta}\mathcal{L}^{i,j,k,N}(\theta_0^{i,j,k,N})=0$, we have that 
\begin{align}
\label{eq:l2-convergence-proof-1-start}
\partial_{\theta}\mathcal{L}^{i,j,k,N}(\theta_t^{i,j,k,N}) &= \partial_{\theta}\mathcal{L}^{i,j,k,N}(\theta_0^{i,j,k,N}) + \partial_{\theta}^2\mathcal{L}^{i,j,k,N}({\tilde{\theta}}_t^{i,j,k,N})(\theta_t^{i,j,k,N}- \theta_{0}^{i,j,k,N}) \\
&=  \partial_{\theta}^2\mathcal{L}^{i,j,k,N}(\tilde{\theta}_t^{i,j,k,N})(\theta_t^{i,j,k,N}- \theta_{0}^{i,j,k,N}) \label{eq_4_107-finite-n} 
\end{align}
where $\partial_{\theta}^2{\mathcal{L}}^{i,j,k,N}(\cdot)$ denotes the Hessian, and ${\tilde{\theta}}_t^{i,j,k,N}$ is a point in the segment connecting $\theta_t^{i,j,k,N}$ and $\theta_0^{i,j,k,N}$. Substituting \eqref{eq_4_107-finite-n} into \eqref{theta_ideal_2-a-i-ii-finite-n}, we obtain the following equations for $z_t^{i,j,k,N} = \theta_t^{i,j,k,N} - \theta_{0}^{i,j,k,N}$
\begin{align}
\mathrm{d}z_t^{i,j,k,N}&= -\gamma_t\partial^2_{\theta}\mathcal{L}^{i,j,k,N}(\tilde{\theta}_t^{i,j,k,N})z_t^{i,j,k,N}\mathrm{d}t - \gamma_t\big(h^{i,j,k,N}(\theta_t^{i,j,k,N},\boldsymbol{x}_t^{N})-\partial_{\theta}\mathcal{L}^{i,j,k,N}(\theta_t^{i,j,k,N})\big)\mathrm{d}t  \\
&~~~~+ \gamma_tg^{i,j,N}(\theta_t^{i,j,k,N},\boldsymbol{x}_t^N)\sigma^{-\top}\mathrm{d}w_t^{i,N}.
\end{align}
Applying It\^o's formula to the function $\|\cdot\|^2$,
and using the strong convexity of $\mathcal{L}^{i,j,k,N}$ (Assumption \ref{assumption:convexity-finite-n}), it follows that
\begin{align}
\mathrm{d}\|z_t^{i,j,k,N}\|^2 + 2\eta^{i,j,k,N}\gamma_t\|z_t^{i,j,k,N}\|^2\mathrm{d}t  &\leq -2\gamma_t\big\langle z_t^{i,j,k,N}, h^{i,j,k,N}(\theta_t^{i,j,k,N},\boldsymbol{x}_t^N)-\partial_{\theta}\mathcal{L}^{i,j,k,N}(\theta_t^{i,j,k,N})\big\rangle\mathrm{d}t \label{eq_4_111-finite-n} \\
&\hspace{-27.5mm}+ 2\gamma_t\big\langle z_t^{i,j,k,N}, g^{i,j,N}(\theta_t^{i,j,k,N},\boldsymbol{x}_t^N)\sigma^{-\top}\mathrm{d}w_t^{i,N}\big\rangle + \gamma_t^2 \big\|g^{i,j,N}(\theta_t^{i,j,k,N},\boldsymbol{x}_t^N)\sigma^{-\top}\big\|_{F}^2\mathrm{d}t 
\end{align}
where $\|\cdot\|_{F}$ denotes the Frobenius norm. Let us define the function $\Phi_{s,t}= \exp[-2\eta^{i,j,k,N}\int_{s}^{t}\gamma_u\mathrm{d}u]$, with $\partial_{s}\Phi_{s,t} = 2\eta^{i,j,k,N}\gamma_s \Phi_{s,t}$. Using the product rule, and the previous display, we obtain
\begin{align}
&\mathrm{d}\big[\Phi_{s,t}\|z_s^{i,j,k,N}\|^2\big] = \Phi_{s,t}\big[\mathrm{d}\|z_s^{i,j,k,N}\|^2 + 2\eta^{i,j,k,N}\gamma_s\|z_s^{i,j,k,N}\|^2\mathrm{d}s\big] \\
&\leq -2\gamma_s\Phi_{s,t}\langle z_s^{i,j,k,N}, h^{i,j,k,N}(\theta_s^{i,j,k,N},\boldsymbol{x}_s^N)-\partial_{\theta}\mathcal{L}^{i,j,k,N}(\theta_s^{i,j,k,N})\rangle\mathrm{d}s  \\
&~~~+ 2\gamma_s\Phi_{s,t}\langle z_s^{i,j,k,N}, g^{i,j,N}(\theta_s^{i,j,k,N},\boldsymbol{x}_s^{N})\sigma^{-\top}\mathrm{d}w_s^{i,N}\rangle + \gamma_s^2\Phi_{s,t}\|g^{i,j,N}(\theta_s^{i,j,k,N},\boldsymbol{x}_s^{N})\sigma^{-\top}\|_{F}^2\mathrm{d}s 
\label{eq:l2-convergence-proof-1-end}
\end{align} 
Rewriting this inequality in integral form, and taking expectations, we arrive at
\begin{align}
\label{eq_defs-finite-n}
\mathbb{E}\big[\|z_t^{i,j,k,N}\|^2\big] &\leq  \mathbb{E}\big[\Phi_{1,t}\|z_1^{i,j,k,N}\|^2\big] + \mathbb{E}\big[\int_1^t \gamma_s^2\Phi_{s,t} \big\|g^{i,j,N}(\theta_s^{i,j,k,N},\boldsymbol{x}_s^N)\sigma^{-\top}\big\|_{F}^2\mathrm{d}s\big] \\[-2mm]
&\hspace{1mm}+ \mathbb{E}\big[\int_1^t 2\gamma_s\Phi_{s,t}\big\langle z_s^{i,j,k,N}, \partial_{\theta}\mathcal{L}^{i,j,k,N}(\theta_s^{i,j,k,N}) - h^{i,j,k,N}(\theta_s^{i,j,k,N},\boldsymbol{x}_s^N)\big\rangle\mathrm{d}s\big]  \\
&:=\mathbb{E}\big[\Omega_{t,i,j,k,N}^{(1)}\big] + \mathbb{E}\big[\Omega_{t,i,j,k,N}^{(2)}\big] + \mathbb{E}\big[\Omega_{t,i,j,k,N}^{(3)}\big]  \label{eq_defs_2-finite-n}
\end{align}
We will deal with each of these terms separately, beginning with $\smash{\Omega_{t,i,j,k,N}^{(1)}}$. For this term, we have that, for sufficiently large $t\geq0$, 
\begin{equation}
\mathbb{E}\left[\Omega_{t,i,j,k,N}^{(1)}\right] 
= \Phi_{1,t}\mathbb{E}\left[\|z_1^{i,j,k,N}\|^2\right] \leq K^{(1)} \gamma_t \label{eq3155-finite-n}
\end{equation}
where the inequality follows from the uniform-in-time moment bounds for the online parameter estimate, and Assumption~\ref{assumption:learning-rate-v2} (i.e., the conditions on the learning rate). We next consider $\smash{\Omega_{t,i,j,k,N}^{(2)}}$. By Corollary \ref{cor:empirical-drift-grad-lip} (i.e., the polynomial growth of $g^{i,j,N}$), Theorem~\ref{thm:moment-bounds} (i.e., the moment bounds for the IPS), the moment bounds for the online parameter estimate, and Assumption~\ref{assumption:learning-rate-v2} (i.e., the conditions on the learning rate), we have
\begin{equation}
\mathbb{E}\left[\Omega_{t,i,j,k,N}^{(2)}\right] = \mathbb{E}\left[\int_1^t\gamma_s^2\Phi_{s,t} \big\|g^{i,j,N}(\theta_s^{i,j,k,N},\boldsymbol{x}_s^{N})\big\|_{F}^2\mathrm{d}s\right]\leq K\int_1^t\gamma_s^2\Phi_{s,t}\mathrm{d}s\leq K^{(2)}\gamma_t. \label{eq3169-finite-n}
\end{equation}
Finally, we consider $\smash{\Omega_{t,i,j,k,N}^{(3)}}$. We will analyse this term by constructing an appropriate Poisson equation. Let us define 
\begin{equation}
    R^{i,j,k,N}(\theta,{\boldsymbol{x}}^N) = \langle \theta-\theta_{0}^{i,j,k,N},\partial_{\theta}\mathcal{L}^{i,j,k,N}(\theta) - h^{i,j,k,N}(\theta,{\boldsymbol{x}}^N)\rangle,
\end{equation} 
By Corollary \ref{cor:empirical-H-lip} (i.e., $\boldsymbol{x}^N\mapsto h^{i,j,k,N}(\theta,{\boldsymbol{x}}^N)$ and its derivatives are locally Lipschitz with polynomial growth), and Lemma~\ref{lemma:main-theorem-lemma-2-a} (i.e., the boundedness of the asymptotic log-likelihood and its derivatives), for $l=0,1,2$, $\smash{|\partial_{\theta}^l  R^{i,j,k,N}(\theta,\boldsymbol{x}^{N})  - \partial_{\theta}^l  R^{i,j,k,N}(\theta,\boldsymbol{y}^{N})|}$ satisfies a bound of the type given in Corollary~\ref{cor:empirical-H-lip}, with an additional multiplicative factor of $[1+\|\theta\|]$. In addition, by definition, this function is centered with respect to $\pi_{\theta_0}^N$. Thus, using (a minor variation of) Lemma 17 in \cite{sharrock2023online} (with $r=1$), the Poisson equation
\begin{equation}
\mathcal{A}_{\boldsymbol{x}^N} v^{i,j,k,N}(\theta,{\boldsymbol{x}}^N) = R^{i,j,k,N}(\theta,{\boldsymbol{x}}^N)~~~,~~~\int_{(\mathbb{R}^{d})^N} v^{i,j,k,N}(\theta,{\boldsymbol{x}}^N)\pi_{\theta_0}^{N}(\mathrm{d}{\boldsymbol{x}}^N)=0
\end{equation}
has a unique twice differentiable solution which satisfies $\sum_{\ell=0}^{2} \|\frac{\partial^\ell }{\partial \theta^{\ell}}v^{i,j,k,N}(\theta,{\boldsymbol{x}}^N)\| + \|\frac{\partial^2 }{\partial \theta\partial {\boldsymbol{x}}^N}v^{i,j,k,N}(\theta,{\boldsymbol{x}}^N)\| \leq K\big[1+\|\theta\|\big]  \big[1+\sum_{a\in\{i,j,k\}}\|x^{a,N}\|^{q} + \frac{1}{N}\sum_{a=1}^N \|x^{a,N}\|^{q}\big]$. Using It\^o's formula, we have that
\begin{align}
&v^{i,j,k,N}(\theta_t^{i,j,k,N},{\boldsymbol{x}}_t^N) - v^{i,j,k,N}(\theta_s^{i,j,k,N},{\boldsymbol{x}}_s^N) \\
&~~~~~~~= \int_s^t \mathcal{A}_{\theta} v^{i,j,k,N}(\theta_u^{i,j,k,N},{\boldsymbol{x}}^N_u)\mathrm{d}u + \int_s^{t} \mathcal{A}_{{\boldsymbol{x}}^N} v^{i,j,k,N}(\theta_u^{i,j,k,N},{\boldsymbol{x}}^N_u)\mathrm{d}u \hspace{-5mm} \label{eq_4147-finite-n} \\
&~~~~~~~+\int_s^t \gamma_u\partial_{\theta}v^{i,j,k,N}(\theta_u^{i,j,k,N},{\boldsymbol{x}}^N_u)g^{i,j,N}(\theta_u^{i,j,k,N},\boldsymbol{x}_u^N)\sigma^{-\top}\mathrm{d}w_u^{i,N}  \\
&~~~~~~~+\int_s^t \langle \partial_{\boldsymbol{x}^N}v^{i,j,k,N}(\theta_u^{i,j,k,N},{\boldsymbol{x}}_u^N),(I_N\otimes\sigma)\mathrm{d}{w}_u^N\rangle \\
&~~~~~~~+\int_s^{t}\gamma_u\partial_{\theta}\partial_{x^{i,N}}v^{i,j,k,N}(\theta_u^{i,j,k,N},{\boldsymbol{x}}_u^N)g^{i,j,N}(\theta_u^{i,j,k,N},{\boldsymbol{x}}_u^N)\mathrm{d}u
\end{align}
where ${w}_u^N = (w_u^{1,N},\dots,w_u^{N,N})^{\top}$ is the vector-valued Brownian motion defined in \eqref{eq:vector-sde}. It follows, writing $v_t^{i,j,k,N}:= v^{i,j,k,N}(\theta_t^{i,j,k,N},{\boldsymbol{x}}_t^N)$, that
\begin{align}
&R^{i,j,k,N}(\theta_t^{i,j,k,N},{\boldsymbol{x}}_t^N)\mathrm{d}t = \mathcal{A}_{{\boldsymbol{x}}^N} v^{i,j,k,N}(\theta_t^{i,j,k,N},{\boldsymbol{x}}_t^N)\mathrm{d}t \\
&=\mathrm{d}v_t^{i,j,k,N} - \mathcal{A}_{\theta}v^{i,j,k,N}(\theta_t^{i,j,k,N},{\boldsymbol{x}}_t^N)\mathrm{d}t - \gamma_t\partial_{\theta}v^{i,j,k,N}(\theta_t^{i,j,k,N},{\boldsymbol{x}}_t^N)g^{i,j,N}(\theta_t^{i,j,k,N},{\boldsymbol{x}}_t^N)\sigma^{-\top}\mathrm{d}w_t^{i,N} \\
&- \partial_{\boldsymbol{x}^N}v^{i,j,k,N}(\theta_t^{i,j,k,N},{\boldsymbol{x}}_t^N)\sigma\mathrm{d}{w}_t^N 
- \gamma_t\partial_{\theta}\partial_{{x}^{i,N}}v^{i,j,k,N}(\theta_t^{i,j,k,N},{\boldsymbol{x}}_t^N)g^{i,j,N}(\theta_t^{i,j,k,N},{\boldsymbol{x}}_t^N)\mathrm{d}t. 
\end{align}
Using this identity, we can rewrite $\Omega_{t,i,j,k,N}^{(3)}$ as
\begin{align}
\Omega_{t,i,j,k,N}^{(3)} &= \int_1^t 2\gamma_s\Phi_{s,t}\underbrace{\langle \theta_s^{i,j,k,N}-\theta_{0}^{i,j,k,N}, \partial_{\theta}\mathcal{L}^{i,j,k,N}(\theta_s^{i,j,k,N}) - h^{i,j,k,N}(\theta_s^{i,j,k,N},{\boldsymbol{x}}_s^N)\rangle\mathrm{d}s}_{R^{i,j,k,N}(\theta_s,{\boldsymbol{x}}_s^N)\mathrm{d}s}  \\[-.5mm]
&=\int_1^t 2\gamma_s \Phi_{s,t}\mathrm{d}v_s^{i,j,k,N}- \int_{1}^t2\gamma_s\Phi_{s,t}\mathcal{A}_{\theta}v^{i,j,k,N}(\theta_s^{i,j,k,N},{\boldsymbol{x}}_s^N)\mathrm{d}s \label{eq3163-finite-n} \\
&-  \int_{1}^t2\gamma_s^2\Phi_{s,t}\partial_{\theta}v^{i,j,k,N}(\theta_s^{i,j,k,N},{\boldsymbol{x}}_s^N)g^{i,j,N}(\theta_s^{i,j,k,N},{\boldsymbol{x}}_s^N)\sigma^{-\top}\mathrm{d}w_s^{i,N}  \\
&-  \int_{1}^t2\gamma_s\Phi_{s,t}\partial_{\boldsymbol{x}^N}v^{i,j,k,N}(\theta_s^{i,j,k,N},{\boldsymbol{x}}_s^N)\sigma \mathrm{d}{w}^N_s  \\
 &-  \int_{1}^t2\gamma^2_s\Phi_{s,t}\partial_{\theta}\partial_{x^{i,N}}v^{i,j,k,N}(\theta_s^{i,j,k,N},{\boldsymbol{x}}_s^N)g^{i,j,N}(\theta_s^{i,j,k,N},{\boldsymbol{x}}_s^N)\mathrm{d}s.
\end{align}
We can further rewrite the first term in this expression by applying It\^o's formula to $f(s,v_s^{i,j,k,N}) = 2\gamma_s\Phi_{s,t}v_s^{i,j,k,N}$. In particular, this yields 
\begin{equation}
2\gamma_t\Phi_{t,t}v_t^{i,j,k,N} - 2\gamma_1\Phi_{1,t}v_1^{i,j,k,N} = \int_1^t 2\gamma_s\Phi_{s,t}\mathrm{d}v_s^{i,j,k,N} + \int_{1}^t2\dot{\gamma}_s\Phi_{s,t}v_s^{i,j,k,N}\mathrm{d}s + \int_{1}^t4\eta^{i,j,k,N}\gamma^2_s\Phi_{s,t}v_s^{i,j,k,N}\mathrm{d}s. 
\end{equation}
Rearranging, substituting into \eqref{eq3163-finite-n}, and then taking expectations (upon which the stochastic integrals vanish), we have that
\begin{align}
\mathbb{E}\big[\Omega_{t,i,j,k,N}^{(3)}\big] 
&= 2\gamma_t\mathbb{E}\big[v^{i,j,k,N}(\theta_t^{i,j,k,N},{\boldsymbol{x}}_t^N)\big]- 2\gamma_1\Phi_{1,t} \mathbb{E}\big[v^{i,j,k,N}(\theta_1^{i,j,k,N},\boldsymbol{x}_1^N)\big] \\
&- 2 \int_{1}^t\dot{\gamma}_s\Phi_{s,t}\mathbb{E} \big[v^{i,j,k,N}(\theta_s^{i,j,k,N},{\boldsymbol{x}}_s^N)\big]\mathrm{d}s - 4\eta^{i,j,k,N} \int_{1}^t\gamma^2_s\Phi_{s,t}\mathbb{E}\big[v^{i,j,k,N}(\theta_s^{i,j,k,N},{\boldsymbol{x}}_s^N)\big]\mathrm{d}s  \\
&- 2 \int_{1}^t\gamma_s\Phi_{s,t}\mathbb{E}\big[\mathcal{A}_{\theta}v^{i,j,k,N}(\theta_s^{i,j,k,N},{\boldsymbol{x}}_s^N)\big]\mathrm{d}s \\
&- 2  \int_{1}^t\gamma^2_s\Phi_{s,t}\mathbb{E}\big[\partial_{\theta}\partial_{{x}^{i,N}}v^{i,j,k,N}(\theta_s^{i,j,k,N},{\boldsymbol{x}}_s^N) g^{i,j,N}(\theta_s^{i,j,k,N},{\boldsymbol{x}}_s^N)\big]\mathrm{d}s  \\
&\leq K\big[\gamma_t + \int_1^t \big( |\dot{\gamma}_s|+\gamma_s^2\big)\Phi_{s,t}\mathrm{d}s\big] \leq K^{(3)}\gamma_t, \label{eq3168-finite-n}
\end{align}
where in the penultimate inequality we have used the polynomial growth of $\boldsymbol{x}^N\mapsto v^{i,j,k,N}(\theta,{\boldsymbol{x}}^N)$ and $\boldsymbol{x}^N\mapsto\partial_{\theta}\partial_{x^{i,N}}v^{i,j,k,N}(\theta,{\boldsymbol{x}}^N)$,  Corollary \ref{cor:empirical-drift-grad-lip} (i.e., the polynomial growth of $g^{i,j,N}$), Theorem~\ref{thm:moment-bounds} (i.e., the moment bounds for the IPS), the moment bounds for the parameter estimator; and in the final inequality we have used Assumption~\ref{assumption:learning-rate-v2} (i.e., the conditions on the learning rate).

Combining \eqref{eq3155-finite-n}, \eqref{eq3169-finite-n}, and \eqref{eq3168-finite-n}, and setting $\smash{K_{1}^{\dagger}= 2\max\{K^{(1)},K^{(3)}\}}$ and $\smash{K_{2}^{\dagger} = K^{(2)}}$, we obtain the bound in \eqref{eq:non-average-l2-rate-1}.
The proof of \eqref{eq:average-l2-rate-1} is essentially identical, replacing $\smash{\theta_t^{i,j,k,N} \mapsto \bar\theta_t^{i,N}}$, $\smash{g^{i,j,N}\mapsto G^{i,N}}$ and $\smash{h^{i,j,k,N}\mapsto H^{i,N}}$, $\smash{\mathcal{L}^{i,j,k,N}\mapsto \mathcal{L}^{i,N}}$ and noting that we can apply the same arguments since $G^{i,N}$ and $H^{i,N}$ satisfy appropriate polynomial growth conditions (see Corollary~\ref{cor:empirical-drift-grad-lip}, Corollary~\ref{cor:empirical-H-lip}), and $\mathcal{L}^{i,N}$ is strongly convex (by assumption). 

It remains to establish \eqref{eq:average-l2-rate-2} and \eqref{eq:non-average-l2-rate-2}, i.e., the convergence rates w.r.t. the true parameter. We begin with \eqref{eq:average-l2-rate-2}. First note that $\mathcal{L}^{i,N}(\theta)\geq 0$ for all $\theta\in\Theta$, with equality iff $\theta=\theta_0$. Thus, $\theta_0$ is a global minimiser of $\mathcal{L}^{i,N}$. In addition, since $\mathcal{L}^{i,N}$ is $\eta$-strongly convex on $\Theta$, it has at most one minimiser. Since $\theta_0$ is \emph{a} minimiser, it must be \emph{the} minimiser. That is, $\theta_0^{i,N}=\theta_0$. The bound in \eqref{eq:average-l2-rate-2} now follows immediately from \eqref{eq:average-l2-rate-1}, i.e., the $\mathrm{L}^2$ convergence rate just established for the averaged estimator. 

We now turn our attention to \eqref{eq:non-average-l2-rate-2}. In this case, using the inequality $(a+b)^2 \leq 2(a^2+b^2)$ and \eqref{eq:non-average-l2-rate-1}, i.e., the $\mathrm{L}^2$ convergence rate just proven for the non-averaged estimator, we have 
    \begin{align}
    \mathbb{E}\left[\|\theta_t^{i,j,k,N}-\theta_0\|^2\right] &\leq 2\mathbb{E}\left[\|\theta_t^{i,j,k,N}-\theta_{0}^{i,j,k,N}\|^2\right] + 2\mathbb{E}\left[\|\theta_0^{i,j,k,N}-\theta_0\|^2\right]  \\
    &\leq2(K_{1}^{\dagger}+K_{2}^{\dagger})\gamma_t + 2\|\theta_0^{i,j,k,N}-\theta_0\|^2, \label{eq:non-average-l2-rate-2-decomp}
    \end{align}
    It remains to bound the Euclidean distance between the minimiser $\theta_0^{i,j,k,N}$ and the true parameter $\theta_0$. First note that, due to strong convexity, $\smash{\langle \partial_{\theta}\mathcal{L}^{i,j,k,N}(\theta_0) - \partial_{\theta}\mathcal{L}^{i,j,k,N}(\theta_0^{i,j,k,N}), \theta_0 - \theta_0^{i,j,k,N}\rangle \geq \eta^{i,j,k,N}\| \theta_0 - \theta_0^{i,j,k,N} \|^2}$. Using in addition the fact that $\partial_{\theta}\mathcal{L}^{i,j,k,N}(\theta_0^{i,j,k,N})=0$, it follows that $\eta^{i,j,k,N}\| \theta_0 - \theta_0^{i,j,k,N} \|^2 \leq \langle \partial_{\theta}\mathcal{L}^{i,j,k,N}(\theta_0) , \theta_0 - \theta_0^{i,j,k,N}\rangle \leq \|\partial_{\theta}\mathcal{L}^{i,j,k,N}(\theta_0)\|\,\|\theta_0 - \theta_0^{i,j,k,N}\|$. Dividing both sides by $\smash{\eta^{i,j,k,N}}$, adding and subtracting $\partial_{\theta}\mathcal{L}(\theta_0)$, again using the inequality $(a+b)^2\leq 2(a^2+b^2)$, and finally the fact that $\partial_{\theta}\mathcal{L}(\theta_0)=0$, it follows that 
     \begin{align}
        \| \theta_0 - \theta_0^{i,j,k,N} \|^2 &\leq \frac{1}{(\eta^{i,j,k,N})^2}  \|\partial_{\theta}\mathcal{L}^{i,j,k,N}(\theta_0)\|^2 \\
        &\leq \frac{2}{(\eta^{i,j,k,N})^2} \left[\|\partial_{\theta}\mathcal{L}^{i,j,k,N}(\theta_0) - \partial_{\theta}\mathcal{L}(\theta_0)\|^2 + \|\partial_{\theta}\mathcal{L}(\theta_0)\|^2\right ] \\
        &= \frac{2}{(\eta^{i,j,k,N})^2} \|\partial_{\theta}\mathcal{L}^{i,j,k,N}(\theta_0) - \partial_{\theta}\mathcal{L}(\theta_0)\|^2. \label{eq:theta-l-bound}
    \end{align}
    By Proposition \ref{prop:inf-n-convergence-1}, and one final use of the inequality $(a+b)^2\leq 2a^2 + 2b^2$, 
    there exist constants $K_3^{\dagger},K_4^{\dagger}<\infty$
    such that $\smash{\| \partial_{\theta}\mathcal{L}^{i,j,k,N}(\theta) - \partial_{\theta} \mathcal{L}(\theta)\|^2 \leq K_3^{\dagger}\rho^2(N) + K_4^{\dagger}N^{-\frac{1}{1+\alpha}}}$. Substituting this into \eqref{eq:theta-l-bound}, and allowing the constants $K_3^{\dagger},K_4^{\dagger}$ to absorb the factor $2$, we then have
    \begin{equation}
        \| \theta_0 - \theta_0^{i,j,k,N} \|^2 \leq \frac{1}{(\eta^{i,j,k,N})^2}\Big[K_3^{\dagger}\rho^2(N) + \frac{K_4^{\dagger}}{N^{\frac{1}{1+\alpha}}}\Big]. \label{eq:l2-param-bound}
    \end{equation}
    Finally, substituting \eqref{eq:l2-param-bound} into \eqref{eq:non-average-l2-rate-2-decomp}, we have the required result.
\end{proof}

\begin{proof}[Proof of Corollary~\ref{corollary:main-theorem-1-2-finite-n-with-m}]
The proof is a direct modification of the proof of Theorem~\ref{theorem:main-theorem-1-2-finite-n}. We will thus highlight only the relevant differences. Once again, we focus on the non-averaged estimator; the averaged case is analogous. 

We begin, similar to before, by writing the update equation in the form
\begin{align}
\mathrm{d}\theta_t^{N,M}
&= -\underbrace{\gamma_t\,\partial_{\theta}\mathcal{L}^{i,j,k,N}(\theta_t^{N,M})\,\mathrm{d}t}_{\text{true descent term}}
-\underbrace{\gamma_t\Big(\tfrac{1}{M}\textstyle\sum_{(i,j,k)\in\mathcal C(\Pi)}
\big[h^{i,j,k,N}(\theta_t^{N,M},\boldsymbol{x}_t^N)-\partial_{\theta}\mathcal{L}^{i,j,k,N}(\theta_t^{N,M})\big]\Big)\mathrm{d}t}_{\text{fluctuations term}}
\\*
&\qquad
+\underbrace{\gamma_t\Big(\tfrac{1}{M}\textstyle\sum_{(i,j,k)\in\mathcal C(\Pi)}
g^{i,j,N}(\theta_t^{N,M},\boldsymbol{x}_t^N)\Big)\sigma^{-\top}\mathrm{d}w_t^{i,N}}_{\text{noise term}} .
\label{theta_ideal_2-a-i-ii-finite-n-with-M}
\end{align}
Let $z_t^{N,M}:=\theta_t^{N,M}-\theta_0^{i,j,k,N}$. Then, repeating the steps in \eqref{eq:l2-convergence-proof-1-start} - \eqref{eq:l2-convergence-proof-1-end} (i.e., considering Taylor expansion around the minimiser, applying It\^o's formula to $\|\cdot\|^2$, using strong convexity of the asymptotic log-likelihood $\mathcal{L}^{i,j,k,N}$), we arrive at
\begin{align}
\mathbb E\big[\|z_t^{N,M}\|^2\big]
&\le \mathbb{E}\big[\Phi_{1,t}\|z_1^{N,M}\|^2\big] +\mathbb{E}\big[\int_1^t \gamma_s^2\,\Phi_{s,t}\,
\big\|
\frac1M\sum_{(i,j,k)\in\mathcal C(\Pi)}
g^{i,j,N}(\theta_s^{N,M},\boldsymbol x_s^N)\sigma^{-\top}
\big\|_F^2\,\mathrm ds\big] \\
&+\mathbb{E}\big[\int_1^t 2\gamma_s\,\Phi_{s,t}\,
\big\langle
z_s^{N,M},\,
\partial_\theta\mathcal L^{i,j,k,N}(\theta_s^{N,M})
-\frac1M\sum_{(i,j,k)\in\mathcal C(\Pi)}
h^{i,j,k,N}(\theta_s^{N,M},\boldsymbol x_s^N)
\big\rangle\,\mathrm ds\big]
\\
&:=\mathbb E[\Omega^{(1)}_{t,N,M}] + \mathbb E[\Omega^{(2)}_{t,N,M}] + \mathbb E[\Omega^{(3)}_{t,N,M}], \label{eq:new-z-t-m}
\end{align}
where, as in the previous proof, we have defined $\smash{\Phi_{s,t}= \exp[-2\eta^{i,j,k,N}\int_{s}^{t}\gamma_u\mathrm{d}u]}$. It remains to bound the three terms on the RHS. 

The bound for $\Omega_{t,N,M}^{(1)}$ follows identically to the bound for $\Omega_{t,i,j,k,N}^{(1)}$ in the proof of Theorem~\ref{theorem:main-theorem-1-2-finite-n}. In particular, we have
\begin{equation}
\label{eq:omega-1-M}
\mathbb E[\Omega_{t,N,M}^{(1)}]
=\mathbb E\!\left[\Phi_{1,t}\|z_1^{N,M}\|^2\right]
=\Phi_{1,t}\,\mathbb E\!\left[\|z_1^{N,M}\|^2\right]
\le K^{(1)}\gamma_t
\end{equation}
where the final inequality uses the assumed uniform moment bounds on $(\theta_t^{N,M})_{t\ge0}$, and Assumption~\ref{assumption:learning-rate-v2} (i.e., the conditions on the learning rate). 

The term $\smash{\Omega^{(2)}_{t,N,M}}$ now contains an average over noise terms. It follows, using the elementary inequality $\smash{
\|\frac1M\sum_{r=1}^M A_r\|_F^2 \le \frac1M\sum_{r=1}^M \|A_r\|_F^2}$, and arguing as in \eqref{eq3169-finite-n}, that
\begin{align}
\label{eq3169-finite-n-with-M}
\mathbb{E}\!\left[\Omega_{t,N,M}^{(2)}\right]
&\le \frac{1}{M}\,
\mathbb{E}\!\left[\int_{1}^{t}\gamma_s^{2}\Phi_{s,t}
\sum_{(i,j,k)\in\mathcal{C}(\Pi)}
\big\|g^{i,j,N}(\theta_s^{N,M},\boldsymbol{x}_s^{N})\sigma^{-\top}\big\|_{F}^{2}\,\mathrm{d}s\right] \\
&\leq \frac{K}{M}\int_1^t\gamma_s^2\Phi_{s,t}\mathrm{d}s\leq \frac{K^{(2)}}{M}\gamma_t. 
\end{align}

The argument for $\Omega_{t,N,M}^{(3)}$ is exactly the same as for $\Omega_{t,i,j,k,N}^{(3)}$ in the proof of Theorem~\ref{theorem:main-theorem-1-2-finite-n}, with the only change being that $h^{i,j,k,N}$ is replaced by its average over $(i,j,k)\in\mathcal C(\Pi)$. In particular, defining $R^{N,M}(\theta,\boldsymbol x^N)
:=\langle \theta-\theta_0^{i,j,k,N},\,
\partial_\theta\mathcal L^{i,j,k,N}(\theta)
-\frac1M\sum_{(i,j,k)\in\mathcal C(\Pi)} h^{i,j,k,N}(\theta,\boldsymbol x^N)
\rangle$, 
we have $\int R^{N,M}(\theta,\boldsymbol x^N)\,\pi_{\theta_0}^N(\mathrm d\boldsymbol x^N)=0$ for each $\theta$, and $R^{N,M}$ satisfies the same polynomial-growth bounds as before (since it is an average of $M$ terms with the same bounds). Thus, the same Poisson equation construction applies, and the resulting algebraic manipulation yields
\begin{equation}
\mathbb E[\Omega_{t,N,M}^{(3)}]
\le K\Big[\gamma_t+\int_1^t\big(|\dot\gamma_s|+\gamma_s^2\big)\Phi_{s,t}\,\mathrm ds\Big]
\le K^{(3)}\gamma_t,
\label{eq:original-bounds}
\end{equation}
where the final inequality uses Assumption~\ref{assumption:learning-rate-v2}. Finally, substituting the bounds in \eqref{eq:omega-1-M}, \eqref{eq3169-finite-n-with-M}, and \eqref{eq:original-bounds} into \eqref{eq:new-z-t-m}, and once more setting $\smash{K_{1}^{\dagger}= 2\max\{K^{(1)},K^{(3)}\}}$ and $\smash{K_{2}^{\dagger} = K^{(2)}}$, we obtain the bound in \eqref{eq:non-average-l2-rate-1-with-m}. 

The proof of the second half of the theorem, i.e., the bounds in \eqref{eq:average-l2-rate-2-with-m} and \eqref{eq:non-average-l2-rate-2-with-m}, follows verbatim from the final part of the proof of Theorem~\ref{theorem:main-theorem-1-2-finite-n}. 
\end{proof}

\begin{proof}[Proof of Theorem \ref{theorem:main-theorem-1-2}]
The proof of this result follows closely the proof of the previous theorem. In this case, however, since we assume convexity of the mean-field negative log-likelihood $\mathcal{L}$, rather than finite-particle pseudo negative log-likelihood $\mathcal{L}^{i,N}$ or $\mathcal{L}^{i,j,k,N}$, we will need to obtain bounds for some additional terms. Once again, we will focus on proving the case for the non-averaged estimator, later detailing how to adapt our proof for the averaged estimator. We begin by recalling the update equation for this estimator, now in the following form:
\begin{align}
\mathrm{d}\theta_t^{i,j,k,N} 
&= -\underbrace{\gamma_t\partial_{\theta}\mathcal{L}(\theta_t^{i,j,k,N})\mathrm{d}t}_{\text{true descent term}} -\underbrace{\gamma_t(\partial_{\theta}\mathcal{L}^{i,j,k,N}(\theta_t^{i,j,k,N}) - \partial_{\theta}\mathcal{L}(\theta_t^{i,j,k,N}))\mathrm{d}t}_{\text{finite particle fluctuation term}} \\
&~~~~~-\underbrace{\gamma_t (h^{i,j,k,N}(\theta_t^{i,j,k,N},\boldsymbol{x}_t^N)-\partial_{\theta}\mathcal{L}^{i,j,k,N}(\theta_t^{i,j,k,N}))\mathrm{d}t}_{\text{finite time fluctuation term}} + \underbrace{\gamma_t g^{i,j,N}(\theta_t^{i,j,k,N},\boldsymbol{x}_t^N)\sigma^{-\top} \mathrm{d}w_t^{i,N}}_{\text{noise term}} \label{eq:parameter-update-full-expansion}
\end{align}
We proceed similarly to the proof of the previous theorem, but now using a first order Taylor expansion for $\mathcal{L}$ around the true parameter $\theta_0$. Arguing as before, cf. \eqref{eq:l2-convergence-proof-1-start} - \eqref{eq:l2-convergence-proof-1-end}, we can show that 
\begin{align}
\mathbb{E}\big[\|z_t^{i,j,k,N}\|^2\big] &\leq  \mathbb{E}\big[\Phi_{1,t}\|z_1^{i,j,k,N}\|^2\big]  + \mathbb{E}\big[\int_1^t \gamma_s^2\Phi_{s,t} \big\|g^{i,j,N}(\theta_s^{i,j,k,N},\boldsymbol{x}_s^{N})\sigma^{-\top}\big\|_{F}^2\mathrm{d}s\big] \label{eq_defs} \\[-.5mm]
&\hspace{2mm}+ \mathbb{E}\big[\int_1^t 2\gamma_s\Phi_{s,t}\big\langle z_s^{i,j,k,N}, \partial_{\theta}\mathcal{L}^{i,j,k,N}(\theta_s^{i,j,k,N}) - h^{i,j,k,N}(\theta_s^{i,j,k,N},\boldsymbol{x}_s^{N})\big\rangle\mathrm{d}s\big] \\[-.5mm]
&\hspace{2mm} + \mathbb{E}\big[\int_1^t 2\gamma_s\Phi_{s,t}\big\langle z_s^{i,j,k,N}, \partial_{\theta}\mathcal{L}(\theta_s^{i,j,k,N}) - \partial_{\theta}\mathcal{L}^{i,j,k,N}(\theta_s^{i,j,k,N})\big\rangle\mathrm{d}s\big] \\ 
&:=\mathbb{E}\big[\Omega_{t,i,j,k,N}^{(1)}\big] + \mathbb{E}\big[\Omega_{t,i,j,k,N}^{(2)}\big] + \mathbb{E}\big[\Omega_{t,i,j,k,N}^{(3)}\big] + \mathbb{E}\big[\Omega_{t,i,j,k,N}^{(4)}\big]. \label{eq_defs_2}
\end{align}
where now $\smash{z_t^{i,j,k,N}=\theta_t^{i,j,k,N} - \theta_0}$ and $\Phi_{s,t}= \exp[-2\eta\int_{s}^{t}\gamma_u\mathrm{d}u]$. This is essentially identical to the bound which appeared in the previous proof, cf. \eqref{eq_defs-finite-n} - \eqref{eq_defs_2-finite-n}, except for the additional final term. The bounds for $\smash{\Omega_{t,i,j,k,N}^{(1)}}$, $\smash{\Omega_{t,i,j,k,N}^{(2)}}$, and $\smash{\Omega_{t,i,j,k,N}^{(3)}}$ follow exactly as before, now with $\eta$ in place of $\eta^{i,j,k,N}$. In particular, we have that
\begin{equation}
    \mathbb{E}\big[\Omega_{t,i,j,k,N}^{(1)}\big] + \mathbb{E}\big[\Omega_{t,i,j,k,N}^{(2)}\big] + \mathbb{E}\big[\Omega_{t,i,j,k,N}^{(3)}\big] \leq (K_1^{\dagger}+K_2^{\dagger})\gamma_t. 
    \label{eq:omega-1-2-3}
\end{equation}
Thus, we just need to bound the additional final term. To do so, we begin by writing
\begin{align}
\mathbb{E}\big[\Omega_{t,i,j,k,N}^{(4)}\big] &\leq 2 \int_{1}^t \gamma_s\Phi_{s,t} \mathbb{E}\left[\|z_s^{i,j,k,N}\| \right] \, \|\partial_{\theta}\mathcal{L}^{i,j,k,N}(\theta_s^{i,j,k,N}) - \partial_{\theta}\mathcal{L}(\theta_s^{i,j,k,N})\| \mathrm{d}s. \label{eq:omega-4-decomp}
\end{align}
From Proposition~\ref{prop:inf-n-convergence-1}, there exists $K_3^{\dagger},K_4^{\dagger}<\infty$ such that $\smash{\|\partial_{\theta}\mathcal{L}(\theta) - \partial_{\theta}\mathcal{L}^{i,j,k,N}(\theta)\| \leq K_3^{\dagger}\rho(N) + {K_4^{\dagger}}{N^{-\frac{1}{2(1+\alpha)}}}}$ for all $\theta\in\Theta$. Since $\mathbb{P}(\theta_t\in\Theta ~\forall t\geq 0) = 1$ by assumption, it thus holds that
\begin{equation}
\|\partial_{\theta}\mathcal{L}(\theta_s^{i,j,k,N}) - \partial_{\theta}\mathcal{L}^{i,j,k,N}(\theta_s^{i,j,k,N})\| \leq K_3^{\dagger}\rho(N) + \frac{K_4^{\dagger}}{N^{\frac{1}{2(1+\alpha)}}}
\end{equation}
for almost all $s\geq 0$. Substituting this bound into \eqref{eq:omega-4-decomp}, and using also the assumption that the online parameter estimate has bounded moments, uniform-in-time, it follows that 
\begin{align}
\mathbb{E}\big[\Omega_{t,i,j,k,N}^{(4)}\big] &\leq K \Big[K_3^{\dagger}\rho(N) + \frac{K_4^{\dagger}}{N^{\frac{1}{2(1+\alpha)}}}\Big] \int_{1}^t \gamma_s\Phi_{s,t} \mathrm{d}s \leq K_3^{\dagger}\rho(N) + \frac{K_4^{\dagger}}{N^{\frac{1}{2(1+\alpha)}}}\label{eq:omega-4-decomp-final}
\end{align}
where in the final bound we have used Assumption~\ref{assumption:learning-rate-v2} (i.e., the conditions on the learning rate), and allowed the values of the constant $K_3^{\dagger}$ and $K_4^{\dagger}$ to increase from the previous display, absorbing all other constants. Substituting \eqref{eq:omega-1-2-3} and \eqref{eq:omega-4-decomp-final} into \eqref{eq_defs} - \eqref{eq_defs_2}, we have
\begin{align}
\mathbb{E}\left[\|\theta_t-\theta_{0}\|^2\right]
&\leq (K_{1}^{\dagger}+K_{2}^{\dagger})\gamma_t + K_3^{\dagger}\rho(N) + \frac{K_4^{\dagger}}{N^{\frac{1}{2(1+\alpha)}}}, \label{eq_final}
\end{align}
which completes the proof for the non-averaged estimator. Once again, the proof for the averaged estimator proceeds in essentially the same way, replacing $\theta_t^{i,j,k,N} \mapsto \bar{\theta}_t^{i,N}$, $g^{i,j,N}\mapsto G^{i,N}$ and $h^{i,j,k,N}\mapsto H^{i,N}$, $\mathcal{L}^{i,j,k,N}\mapsto \mathcal{L}^{i,N}$.

It remains to prove the second part of theorem, i.e., the convergence rates in \eqref{eq:average-l2-rate-2} - \eqref{eq:non-average-l2-rate-2}. We will show that, under our additional assumptions, $\mathcal{L}^{i,N}$ and $\mathcal{L}^{i,j,k,N}$ are themselves strongly convex, with constants $\eta - \delta_{i,N}$ and $\eta - \delta_{i,j,k,N}$, respectively. In this case, the desired rates follow as an immediate consequence of Theorem~\ref{theorem:main-theorem-1-2-finite-n}. We prove the result for $\mathcal{L}^{i,N}$, with the proof for $\mathcal{L}^{i,j,k,N}$ entirely analogous. Fix $\theta\in\Theta$ and let $u\in\mathbb{R}^p$ with $\|u\|=1$. Then
\begin{align}
u^\top \partial_{\theta}^2 \mathcal{L}^{i,N}(\theta) u
&= u^\top \partial_{\theta}^2 \mathcal{L}(\theta) u
  + u^\top\big(\partial_{\theta}^2 \mathcal{L}^{i,N}(\theta)-\partial_{\theta}^2\mathcal{L}(\theta)\big)u \\
&\ge \eta + u^\top\big(\partial_{\theta}^2 \mathcal{L}^{i,N}(\theta)-\partial_{\theta}^2\mathcal{L}(\theta)\big)u \ge \eta - \big\|\partial_{\theta}^2 \mathcal{L}^{i,N}(\theta)-\partial_{\theta}^2\mathcal{L}(\theta)\big\|_{\mathrm{op}} \ge \eta-\delta_{i,N},
\end{align}
where in the third line we have used the bound
$|u^\top A u|\le \|A\|_{\mathrm{op}}$ for symmetric $A$, and in the final line the assumption that $\|\partial_{\theta}^2\mathcal{L}^{i,N}(\theta) - \partial_{\theta}^2\mathcal{L}(\theta)\|_{\mathrm{op}}\leq \delta_{i,N}$. Since this holds for all unit vectors $u$, it follows that
$\partial_{\theta}^2 \mathcal{L}^{i,N}(\theta)\succeq (\eta-\delta_{i,N})I_p$ for all $\theta\in\Theta$. Thus, in particular, $\mathcal{L}^{i,N}$ is $(\eta-\delta_{i,N})$-strongly convex on $\Theta$.
\end{proof}

\subsubsection{Proofs for Section~\ref{sec:main-results-clt}}
\label{app:additional-proofs-main-results-clt}

\begin{proof}[Proof of Theorem~\ref{theorem:clt}]
Similar to elsewhere, we will prove the result for the non-averaged estimator, before detailing how to adapt the proof for the averaged estimator. Our proof is adapted from the proof of \cite[Theorem 2, Proposition 1]{sirignano2020stochastic}. Once again, we begin by recalling the update equation for this estimator in the following form:
\begin{align}
\mathrm{d}\theta_t^{i,j,k,N} &= -\gamma_t\partial_{\theta}\mathcal{L}^{i,j,k,N}(\theta_t^{i,j,k,N})\mathrm{d}t - \gamma_t (h^{i,j,k,N}(\theta_t^{i,j,k,N},\boldsymbol{x}_t^N)-\partial_{\theta}\mathcal{L}^{i,j,k,N}(\theta_t^{i,j,k,N}))\mathrm{d}t \\*
&~~~~~+ \gamma_t g^{i,j,N}(\theta_t^{i,j,k,N},\boldsymbol{x}_t^N)\sigma^{-\top} \mathrm{d}w_t^{i,N} \label{theta_ideal_2-a-i-ii-clt}
\end{align}
We will now use a second order Taylor expansion. In particular, using the fact that $\partial_{\theta}\mathcal{L}^{i,j,k,N}(\theta_0^{i,j,k,N}) =0$, we have that 
\begin{align}
\label{eq_4_107_clt}
\partial_{\theta}\mathcal{L}^{i,j,k,N}(\theta_t^{i,j,k,N}) 
&=  \partial_{\theta}^2\mathcal{L}^{i,j,k,N}(\theta_0^{i,j,k,N})(\theta_t^{i,j,k,N}- \theta_{0}^{i,j,k,N})   \\
&+ \tfrac{1}{2} \partial_{\theta}^3\mathcal{L}^{i,j,k,N}(\tilde{\theta}_t^{i,j,k,N})(\theta_t^{i,j,k,N} - \theta_0^{i,j,k,N})(\theta_t^{i,j,k,N} - \theta_0^{i,j,k,N})^{\top} 
\end{align}
where $\partial_{\theta}^2{\mathcal{L}}^{i,j,k,N}(\cdot)$ denotes the Hessian, the last term is a tensor-matrix product, and ${\tilde{\theta}}_t^{i,j,k,N}$ is a point in the segment connecting $\theta_t^{i,j,k,N}$ and $\theta_0^{i,j,k,N}$. Substituting \eqref{eq_4_107_clt} into \eqref{theta_ideal_2-a-i-ii-clt}, and rearranging, we obtain the following equations for $z_t^{i,j,k,N} = \theta_t^{i,j,k,N} - \theta_{0}^{i,j,k,N}$
\begin{align}
\mathrm{d}z_t^{i,j,k,N} +\gamma_t\partial^2_{\theta}\mathcal{L}^{i,j,k,N}(\theta_0^{i,j,k,N})z_t^{i,j,k,N}\mathrm{d}t &=- \tfrac{1}{2}\gamma_t \partial_{\theta}^3\mathcal{L}^{i,j,k,N}(\tilde{\theta}_t^{i,j,k,N})z_t^{i,j,k,N}z_t^{i,j,k,N,\top}\mathrm{d}t \\
&~~~~- \gamma_t\big(h^{i,j,k,N}(\theta_t^{i,j,k,N},\boldsymbol{x}_t^N)-\partial_{\theta}\mathcal{L}^{i,j,k,N}(\theta_t^{i,j,k,N})\big)\mathrm{d}t   \\
&~~~~+ \gamma_tg^{i,j,N}(\theta_t^{i,j,k,N},\boldsymbol{x}_t^N)\sigma^{-\top}\mathrm{d}w_t^{i,N}. 
\end{align}
Define $\smash{\Phi^{*,i,j,k,N}_{s,t}= \exp[-\partial_{\theta}^2\mathcal{L}^{i,j,k,N}(\theta_0^{i,j,k,N})\int_{s}^{t}\gamma_u\mathrm{d}u]}$, with $\smash{\partial_{s}\Phi^{*,i,j,k,N}_{s,t} =  \gamma_s \partial_{\theta}^2\mathcal{L}^{i,j,k,N}(\theta_0^{i,j,k,N})\Phi^{*,i,j,k,N}_{s,t}}$ and $\smash{\Phi_{t,t}^{*,i,j,k,N} = \mathbf{I}_p}$. Under our assumption of strong convexity, it then holds that \citep[e.g.,][]{sirignano2020stochastic}
\begin{align}
&\| \Phi^{*,i,j,k,N}_{s,t} \|^2 \leq K e^{-2\eta^{i,j,k,N}\int_{s}^{t} \gamma_u\mathrm{d}u} = K\Phi_{s,t} \label{eq:phi-star-bounds} \\
&\|\partial_{t}\Phi^{*,i,j,k,N}_{s,t}\|^2\leq K\gamma_{t}^2 e^{-2\eta^{i,j,k,N} \int_{s}^{t} \gamma_u\mathrm{d}u}=K\gamma_{t}^2 \Phi_{s,t} 
\end{align}
where, as defined previously, $\Phi_{s,t} = e^{-2\eta^{i,j,k,N}\int_{s}^{t}\gamma_u\mathrm{d}u}$ (see, e.g., the proof of Theorem~\ref{theorem:main-theorem-1-2-finite-n}). Returning to the previous display, we have that
\begin{align}
\label{eq:clt-proof-1-end}
\mathrm{d}\big[\Phi_{s,t}^{*,i,j,k,N} z_s^{i,j,k,N} \big] &= \Phi_{s,t}^{*,i,j,k,N}\big[\mathrm{d}z_s^{i,j,k,N} + \gamma_s \partial_{\theta}^2\mathcal{L}^{i,j,k,N}(\theta_0^{i,j,k,N}) z_s^{i,j,k,N}\mathrm{d}s \big]   \\
&=- \tfrac{1}{2}\gamma_s \Phi_{s,t}^{*,i,j,k,N}\partial_{\theta}^3\mathcal{L}^{i,j,k,N}(\tilde{\theta}_s^{i,j,k,N})z_s^{i,j,k,N}z_s^{i,j,k,N,\top}\mathrm{d}s \\
&~~~~- \gamma_s\Phi_{s,t}^{*,i,j,k,N}\big(h^{i,j,k,N}(\theta_s^{i,j,k,N},\boldsymbol{x}_s^N)-\partial_{\theta}\mathcal{L}^{i,j,k,N}(\theta_s^{i,j,k,N})\big)\mathrm{d}s  \nonumber \\
&~~~~+ \gamma_s \Phi_{s,t}^{*,i,j,k,N}g^{i,j,N}(\theta_s^{i,j,k,N},\boldsymbol{x}_s^N)\sigma^{-\top}\mathrm{d}w_s^{i,N}. \nonumber
\end{align}
Rewriting this in integral form, and rearranging, it follows that
\begin{align}
\label{main-bound-clt-start}
z_t^{i,j,k,N} &=  \Phi^{*,i,j,k,N}_{1,t} z_1^{i,j,k,N}    - \int_1^t \Phi^{*,i,j,k,N}_{s,t}  \tfrac{1}{2}\gamma_s \partial_{\theta}^3 \mathcal{L}^{i,j,k,N} (\tilde{\theta}^{i,j,k,N}_s) z_s^{i,j,k,N} z_s^{i,j,k,N,\top} \mathrm{d}s   \\[-1mm]
&-   \int_1^t \Phi^{*,i,j,k,N}_{s,t} \gamma_s \big(h^{i,j,k,N}(\theta_s^{i,j,k,N},\boldsymbol{x}_s^N)-\partial_{\theta}\mathcal{L}^{i,j,k,N}(\theta_s^{i,j,k,N})\big) \mathrm{d}s \\[-1mm]
&+ \int_1^t   \Phi^{*,i,j,k,N}_{s,t}   \gamma_s g^{i,j,N}(\theta_s^{i,j,k,N},\boldsymbol{x}_s^N) \sigma^{-\top} \mathrm{d} w_s^{i,N}   \\
&= \Omega_{t,i,j,k,N}^{(1)} + \Omega_{t,i,j,k,N}^{(2)} + \Omega_{t,i,j,k,N}^{(3)} + \Omega_{t,i,j,k,N}^{(4)} 
\label{main-bound-clt}
\end{align}
We will consider each of these terms in turn, pre-multiplied by a factor of $\gamma_t^{-\frac{1}{2}}$. For the first term, using our previous bounds in \eqref{eq:phi-star-bounds}, we have
\begin{align}
        \|\gamma_{t}^{-\frac{1}{2}}\,\Omega_{t,i,j,k,N}^{(1)}\| 
        \leq \gamma_{t}^{-\frac{1}{2}} \|\Phi_{1,t}^{*,i,j,k,N}\| \, \|z_1^{i,j,k,N}\| 
        &\leq K\gamma_t^{-\frac{1}{2}}\Phi_{1,t}^{\frac{1}{2}} \|z_1^{i,j,k,N}\|. 
\end{align}
By Assumption \ref{assumption:learning-rate-v2} (i.e., our additional conditions on the learning rate), we have that $\smash{\Phi_{1,t}^{\frac{1}{2}} = o(\gamma_t^{\frac{1}{2}})}$. It thus follows that 
\begin{equation}
    \gamma_t^{-\frac{1}{2}}\,\Omega_{t,i,j,k,N}^{(1)} \stackrel{\mathrm{a.s.}}{\longrightarrow}0
    \label{eq:omega-t-1}
\end{equation}
as $t\rightarrow\infty$, and thus also in probability. We now turn our attention to the second term in \eqref{main-bound-clt}. In this case, working from the definition, we have that
\begin{align}
    \mathbb{E}\left[\|\gamma_{t}^{-\frac{1}{2}} \, \Omega_{t,i,j,k,N}^{(2)}\|_1\right] &\leq \mathbb{E} \Big[\gamma_{t}^{-\tfrac{1}{2}}\int_1^t \big\|\Phi^{*,i,j,k,N}_{s,t}  \tfrac{1}{2}\gamma_s \partial_{\theta}^3 \mathcal{L}^{i,j,k,N} (\tilde{\theta}_s^{i,j,k,N}) z_s^{i,j,k,N} z_s^{i,j,k,N,\top} \big\|_1 \mathrm{d}s\Big] \\
    &\leq K\gamma_{t}^{-\frac{1}{2}} \left[ \int_1^t \|\Phi^{*,i,j,k,N}_{s,t}\|  \gamma_s \,\mathbb{E}\left[\|z_s^{i,j,k,N}\|^2\right] \mathrm{d}s \right] \\
    &\leq K\gamma_{t}^{-\frac{1}{2}}\left[\int_1^t \Phi_{s,t}^{\frac{1}{2}} \gamma_s\, (K_1^{\dagger}+K_2^{\dagger})\gamma_s \,  \mathrm{d}s  \right] \leq K\gamma_{t}^{-\frac{1}{2}}\left[\int_1^t  \Phi_{s,t}^{\frac{1}{2}}\,\gamma_s^2  \mathrm{d}s  \right]
\end{align}
where in the second inequality we have used Lemma~\ref{lemma:main-theorem-lemma-2-a} (i.e., the boundedness of third derivative of the asymptotic log-likelihood), and in the final inequality we have used Theorem~\ref{theorem:main-theorem-1-2-finite-n} (i.e., our $\mathrm{L}^2$ convergence rate) and the bound on $\smash{\|\Phi_{s,t}^{*,i,j,k,N}\|}$ implied by \eqref{eq:phi-star-bounds}. By Assumption~\ref{assumption:learning-rate-v2} (i.e., our conditions on the learning rate), we have that $\smash{\int_1^t  \Phi_{s,t}^{\frac{1}{2}}\,\gamma_s^2  \mathrm{d}s = o(\gamma_t^{\frac{1}{2}})}$ as $t\rightarrow\infty$. It follows that 
\begin{equation}
    \gamma_{t}^{-\frac{1}{2}} \, \Omega_{t,i,j,k,N}^{(2)}\stackrel{\mathbb{L}_1}{\longrightarrow}0
    \label{eq:omega-t-2}
\end{equation}
as $t\rightarrow\infty$, and hence also in probability. We now turn our attention to $\smash{\Omega_{t,i,j,k,N}^{(3)}}$. We will analyse this term by constructing an appropriate Poisson equation, as in some of our earlier proofs. In this case, let us define 
\begin{equation}
    S^{i,j,k,N}(\theta,{\boldsymbol{x}}^N) = \partial_{\theta}\mathcal{L}^{i,j,k,N}(\theta) - h^{i,j,k,N}(\theta,{\boldsymbol{x}}^N),
\end{equation} 
Due to Corollary \ref{cor:empirical-H-lip} (i.e., the local Lipschitz and polynomial growth of $h^{i,j,k,N}(\theta,{\boldsymbol{x}}^N)$ and its derivatives) and Lemma~\ref{lemma:main-theorem-lemma-2-a} (i.e., the boundedness of the asymptotic log-likelihood and its derivatives), for $l=0,1,2$, $\smash{\|\partial_{\theta}^l  S^{i,j,k,N}(\theta,\mathbf{x}^{N})  - \partial_{\theta}^l  S^{i,j,k,N}(\theta,\mathbf{y}^{N}) \|}$ satisfies a bound of the type given in Corollary~\ref{cor:empirical-H-lip}. Moreover, by definition of $\partial_{\theta}\mathcal{L}^{i,j,k,N}$, this function is centered with respect to $\pi_{\theta_0}^N$. Thus, using (a minor variation of) Lemma 17 in \cite{sharrock2023online} (now with $r=0$), the Poisson equation
\begin{equation}
\mathcal{A}_{\boldsymbol{x}^N} v^{i,j,k,N}(\theta,{\boldsymbol{x}}^N) = S^{i,j,k,N}(\theta,{\boldsymbol{x}}^N)~~~,~~~\int_{(\mathbb{R}^{d})^N} v^{i,j,k,N}(\theta,{\boldsymbol{x}}^N)\pi_{\theta_0}^{N}(\mathrm{d}{\boldsymbol{x}}^N)=0
\end{equation}
has a unique twice differentiable solution which satisfies $\sum_{\ell=0}^{2} \|\frac{\partial^\ell }{\partial \theta^{\ell}}v^{i,j,k,N}(\theta,{\boldsymbol{x}}^N)\| + \|\frac{\partial^2 }{\partial \theta\partial {\boldsymbol{x}}^N}v^{i,j,k,N}(\theta,{\boldsymbol{x}}^N)\|  \leq K [1+\sum_{a\in\{i,j,k\}}\|x^{a,N}\|^{q} + \frac{1}{N}\sum_{a=1}^N \|x^{a,N}\|^{q}]$. Arguing similarly to before (see, e.g., the proof of Theorem~\ref{theorem:main-theorem-1-2-finite-n}), it is possible to rewrite $\Omega_{t,i,j,k,N}^{(3)}$ in terms of this (vector-valued) solution as
\begin{align}
\gamma_{t}^{-\frac{1}{2}} \Omega_{t,i,j,k,N}^{(3)} &= \gamma_{t}^{-\frac{1}{2}}\int_1^t \gamma_s\Phi^{*,i,j,k,N}_{s,t}\underbrace{\left(\partial_{\theta}\mathcal{L}^{i,j,k,N}(\theta_s^{i,j,k,N}) - h^{i,j,k,N}(\theta_s^{i,j,k,N},{\boldsymbol{x}}_s^N)\right)\mathrm{d}s}_{S^{i,j,k,N}(\theta_s^{i,j,k,N},{\boldsymbol{x}}_s^N)\mathrm{d}s} \label{eq3163-recall} \\[-2.5mm]
&=\gamma_{t}^{-\frac{1}{2}} \int_1^t \gamma_s \Phi^{*,i,j,k,N}_{s,t}\mathrm{d}v_s^{i,j,k,N}- \gamma_{t}^{-\frac{1}{2}} \int_{1}^t \gamma_s\Phi^{*,i,j,k,N}_{s,t}\mathcal{A}_{\theta}v^{i,j,k,N}(\theta_s^{i,j,k,N},{\boldsymbol{x}}_s^N)\mathrm{d}s \\
&- \gamma_{t}^{-\frac{1}{2}} \int_{1}^t \gamma_s^2 \Phi^{*,i,j,k,N}_{s,t}\partial_{\theta}v^{i,j,k,N}(\theta_s^{i,j,k,N},{\boldsymbol{x}}_s^N)g^{i,j,N}(\theta_s^{i,j,k,N},{\boldsymbol{x}}_s^N)\sigma^{-\top}\mathrm{d}w_s^{i,N} \\
&-  \gamma_{t}^{-\frac{1}{2}}\int_{1}^t \gamma_s\Phi^{*,i,j,k,N}_{s,t}\partial_{\boldsymbol{x}^N}v^{i,j,k,N}(\theta_s^{i,j,k,N},{\boldsymbol{x}}_s^N) (I_N\otimes \sigma) \mathrm{d}{w}^N_s \\
 &-  \gamma_{t}^{-\frac{1}{2}}\int_{1}^t\gamma^2_s\Phi^{*,i,j,k,N}_{s,t}\partial_{\theta}\partial_{x^{i,N}}v^{i,j,k,N}(\theta_s^{i,j,k,N},{\boldsymbol{x}}_s^N) g^{i,j,N}(\theta_s^{i,j,k,N},{\boldsymbol{x}}_s^N)\mathrm{d}s \\
 &:= \gamma_{t}^{-\frac{1}{2}}\Pi_{t,i,j,k,N}^{(1)} + \gamma_{t}^{-\frac{1}{2}}\Pi_{t,i,j,k,N}^{(2)} + \gamma_{t}^{-\frac{1}{2}}\Pi_{t,i,j,k,N}^{(3)} + \gamma_{t}^{-\frac{1}{2}}\Pi_{t,i,j,k,N}^{(4)} + \gamma_{t}^{-\frac{1}{2}}\Pi_{t,i,j,k,N}^{(5)} \label{eq:main-bound-clt-2}
\end{align}
Following very  similar steps to those used in the proof of Theorem~\ref{theorem:main-theorem-1-2-finite-n} (e.g., using the polynomial growth of $g^{i,j,N}$ from Corollary~\ref{cor:empirical-drift-grad-lip}, the uniform-in-time moment bounds from Theorem~\ref{thm:moment-bounds}, and the conditions on the learning rate from Assumption~\ref{assumption:learning-rate-v2}), we have that 
\begin{equation}
    \gamma_{t}^{-\frac{1}{2}}\left(\Pi_{t,i,j,k,N}^{(1)}  + \Pi_{t,i,j,k,N}^{(2)} + \Pi_{t,i,j,k,N}^{(3)} + \Pi_{t,i,j,k,N}^{(5)}\right) \stackrel{\mathbb{L}_1}{\longrightarrow} 0 \label{eq:omega-t-3-subset}
\end{equation}
as $t\rightarrow\infty$, and thus also in probability. Given the results in \eqref{eq:omega-t-1}, \eqref{eq:omega-t-2}, and \eqref{eq:omega-t-3-subset}, it remains to analyse $\smash{\gamma_t^{-\frac{1}{2}}(\Pi_{t,i,j,k,N}^{(4)} + \Omega_{t,i,j,k,N}^{(4)})}$, which will be responsible for the covariance of the limiting Gaussian random variable. From the definitions, we have 
\begin{align}
    &\gamma_t^{-\frac{1}{2}}\left[\Pi_{t,i,j,k,N}^{(4)} + \Omega_{t,i,j,k,N}^{(4)}\right] \\
    &= \gamma_{t}^{-\frac{1}{2}}\left[\int_1^t  \gamma_s  \Phi^{*,i,j,k,N}_{s,t} \left( g^{i,j,N}(\theta_s^{i,j,k,N},{\boldsymbol{x}}_s^N) (\sigma\sigma^{\top})^{-1} E_i^{\top} -\partial_{\boldsymbol{x}^N}v^{i,j,k,N}(\theta_s^{i,j,k,N},{\boldsymbol{x}}_s^N)\right)\left(I_N\otimes \sigma\right)\mathrm{d}{w}^N_s\right] \notag
\end{align}
where $E_i\in\mathbb{R}^{Nd\times d}$ denotes the block-selector matrix such that
$\mathrm{d}w_s^{i,N}=E_i^\top\,\mathrm{d}w_s^{N}$. The quadratic variation is thus given by
\begin{align}
    \Sigma_{t}^{i,j,k,N} &:= \gamma_{t}^{-1}\int_1^t  \gamma_s^2  \Phi^{*,i,j,k,N}_{s,t} \Gamma^{i,j,k,N}(\theta_s^{i,j,k,N},\boldsymbol{x}_s^{N})\Phi^{*,i,j,k,N,\top}_{s,t} \mathrm{d} s
\end{align}
where $\Gamma^{i,j,k,N}(\theta,\boldsymbol{x}^N):=(g^{i,j,N}(\theta,{\boldsymbol{x}}^N)(\sigma\sigma^{\top})^{-1}E_i^{\top}  - \partial_{\boldsymbol{x}^N}v^{i,j,k,N}(\theta,{\boldsymbol{x}}^N) ) (I_N\otimes (\sigma\sigma^{\top})) (g^{i,j,N}(\theta,{\boldsymbol{x}}^N)(\sigma\sigma^{\top})^{-1}E_i^{\top}  - \partial_{\boldsymbol{x}^N}v^{i,j,k,N}(\theta,{\boldsymbol{x}}^N) )^{\top}$.
We will establish the convergence of this covariation matrix in two steps. In particular, we will first show that there exists a limiting covariance matrix $\bar{\Sigma}^{i,j,k,N}$ such that
\begin{equation}
    \|\bar{\Sigma}_t^{i,j,k,N} - \bar{\Sigma}^{i,j,k,N}\|_{1} \longrightarrow 0 \label{eq:sigma-limit-1}
\end{equation}
as $t\rightarrow\infty$, where $\smash{\bar{\Sigma}_t^{i,j,k,N}}$ is a proxy for $\smash{\Sigma_t^{i,j,k,N}}$, in which the middle term has been replaced by its ergodic average evaluated at the minimizer, viz $\smash{\bar{\Sigma}_t^{i,j,k,N}= \gamma_{t}^{-1}\int_1^t  \gamma_s^2  \Phi^{*,i,j,k,N}_{s,t} \bar{\Gamma}^{i,j,k,N}(\theta_0^{i,j,k,N})\Phi^{*,i,j,k,N,\top}_{s,t} \mathrm{d} s}$ with $\smash{
\bar{\Gamma}^{i,j,k,N}(\theta) = \int_{(\mathbb{R}^d)^N}\Gamma^{i,j,k,N}(\theta,\boldsymbol{x}^{N}) \pi_{\theta_0}^N(\mathrm{d}\boldsymbol{x}^N)}$. We will subsequently also show that 
\begin{equation}
\mathbb{E}[\|\Sigma_t^{i,j,k,N} - \bar{\Sigma}_t^{i,j,k,N}\|_{1}]\rightarrow0
\end{equation}
as $t\rightarrow\infty$, and hence conclude that $\smash{\mathbb{E}[\|\Sigma_t^{i,j,k,N} - \bar{\Sigma}^{i,j,k,N}\|_{1}]\rightarrow0}$ as $t\rightarrow\infty$ using the triangle inequality. For now, to establish \eqref{eq:sigma-limit-1}, following the approach in \cite{sirignano2020stochastic}, we begin by rewriting $\smash{\Phi_{s,t}^{*,i,j,k,N}}$ in the form
\begin{equation}
\label{eq:limit-cov-start}
    \Phi_{s,t}^{*,i,j,k,N} = V\, e^{-\mathcal{K} \int_{s}^t\gamma_u\mathrm{d}u} \, V^{\top}:=V \mathcal{K}_{s,t} V^{\top}
\end{equation}
where $\mathcal{K} = \mathrm{diag}(\kappa_1,\dots,\kappa_p)$ is the diagonal matrix with (positive) eigenvalues $\kappa_i>0$, $i=1,\dots,p$, and $V=[v_1,\dots,v_p]$ is the corresponding matrix of orthogonal eigenvectors $v_1,\dots,v_p\in\mathbb{R}^p$, and $\mathcal{K}_{s,t}:=e^{-\mathcal{K} \int_s^t \gamma_u\mathrm{d}u} = \mathrm{diag}(e^{-\kappa_1\int_s^t \gamma_u\mathrm{d}u},\dots,e^{-\kappa_p\int_s^t \gamma_u\mathrm{d}u}):=\mathrm{diag}(\kappa_{s,t}^{1},\dots,\kappa_{s,t}^{p})$. It follows, in particular, that the $(m,n)^{\mathrm{th}}$ element of the matrix $\Phi^{*,i,j,k,N}_{s,t}$ takes the form
\begin{equation}
    [\Phi_{s,t}^{*,i,j,k,N}]_{m,n} = \sum_{p_1=1}^p \kappa_{s,t}^{p_1} v_{m}^{p_1} v_{n}^{p_1}, \qquad [\Phi_{s,t}^{*,i,j,k,N,\top}]_{m,n} = \sum_{p_3=1}^p \kappa_{s,t}^{p_3} v_{m}^{p_3} v_{n}^{p_3}
\end{equation}
We can now obtain an expression for the $(m,n)^{\mathrm{th}}$ element of the matrix $\smash{\bar{\Sigma}_t^{i,j,k,N}}$. In particular, substituting these identities into the definition, we have
\begin{align}
    [\bar{\Sigma}_{t}^{i,j,k,N}]_{m,n} &= \left[\gamma_{t}^{-1}\int_1^t  \gamma_s^2  \Phi^{*,i,j,k,N}_{s,t} \bar{\Gamma}^{i,j,k,N}(\theta_0^{i,j,k,N})\Phi^{*,\top}_{s,t} \mathrm{d} s\right]_{m,n} \\
    &=\sum_{p_0,p_1,p_2,p_3=1}^p \gamma_t^{-1}\int_1^t \gamma_s^2  \kappa_{s,t}^{p_1} v_{m}^{p_1} v_{p_0}^{p_1} [\bar{\Gamma}^{i,j,k,N}(\theta_0^{i,j,k,N})]_{p_0,p_2} \kappa_{s,t}^{p_3} v_{p_2}^{p_3} v_{n}^{p_3} \mathrm{d}s
\end{align}
It follows, in particular, that 
\begin{align}
[\bar{\Sigma}^{i,j,k,N}]_{m,n}&:= \lim_{t\rightarrow\infty} [\bar{\Sigma}_{t}^{i,j,k,N}]_{m,n} \\
&= \sum_{p_0,p_1,p_2,p_3=1}^p  \lim_{t\rightarrow\infty}\left[\gamma_t^{-1}\int_1^t \gamma_s^2  \kappa_{s,t}^{p_1} \kappa_{s,t}^{p_3} \mathrm{d}s\right] v_{m}^{p_1} v_{p_0}^{p_1} [\bar{\Gamma}^{i,j,k,N}(\theta_0^{i,j,k,N})]_{p_0,p_2}  v_{p_2}^{p_3} v_{n}^{p_3} \\
&= \sum_{p_0,p_1=1}^p  v_{m}^{p_1} v_{p_0}^{p_1} \sum_{p_2=1}^p [\bar{\Gamma}^{i,j,k,N}(\theta_0^{i,j,k,N})]_{p_0,p_2}\sum_{p_3=1}^p \lim_{t\rightarrow\infty}\left[\gamma_t^{-1}\int_1^t \gamma_s^2  e^{-(\kappa_{p_1} + \kappa_{p_3})\int_{s}^t \gamma_{u}\mathrm{d}u} \mathrm{d}s\right]   v_{p_2}^{p_3} v_{n}^{p_3}.
\label{eq:limit-cov-end}
\end{align}
It remains to show that $\smash{\mathbb{E}\|\Sigma_t^{i,j,k,N} - \bar{\Sigma}_t^{i,j,k,N}\|_{1}\rightarrow0}$ as $t\rightarrow\infty$. To do this, we will use the decomposition
\begin{align}
    \|\Sigma_{t}^{i,j,k,N} - \bar{\Sigma}_t^{i,j,k,N}\|_1 
    &\leq  \Big\| \gamma_{t}^{-1} \int_1^t  \gamma_s^2  \Phi^{*,i,j,k,N}_{s,t} \left[\Gamma^{i,j,k,N}(\theta_s^{i,j,k,N},\boldsymbol{x}_s^{N}) -  \bar{\Gamma}^{i,j,k,N}(\theta_s^{i,j,k,N})\right]\Phi^{*,i,j,k,N,\top}_{s,t} \mathrm{d} s \Big\|_1  \notag \\[-1mm] 
    & +  \Big\| \gamma_{t}^{-1}  \int_1^t  \gamma_s^2  \Phi^{*,i,j,k,N}_{s,t} \left[\bar{\Gamma}^{i,j,k,N}(\theta_s^{i,j,k,N})- \bar{\Gamma}^{i,j,k,N}(\theta_0^{i,j,k,N})\right]\Phi^{*,i,j,k,N,\top}_{s,t} \mathrm{d} s \Big\|_1 \label{eq:delta-triangle} \\
    &:=\Xi_{t,i,j,k,N}^{(1)} + \Xi_{t,i,j,k,N}^{(2)}
\end{align}
Similar to elsewhere, we can analyse $\Xi_{t,i,j,k,N}^{(1)}$ using an appropriate Poisson equation. In particular, we now define
\begin{equation}
    T^{i,j,k,N}(\theta,\boldsymbol{x}^N) = \Gamma^{i,j,k,N}(\theta,\boldsymbol{x}^N) -  \bar{\Gamma}^{i,j,k,N}(\theta)
\end{equation}
Due to Corollary~\ref{cor:empirical-drift-grad-lip} (i.e., the polynomial growth of $\boldsymbol{x}^N\mapsto g^{i,j,N}(\theta,\boldsymbol{x}^N)$ and its derivatives) and the polynomial growth of $\boldsymbol{x}^N\mapsto v^{i,j,k,N}(\theta,\boldsymbol{x}^N)$ and its derivatives, this function (and its derivatives) is locally Lipschitz with polynomial growth. Moreover, by definition, it is centered with respect to $\pi_{\theta_0}^N$. Thus, once more, we can apply (a variant) of Lemma~17 in \cite{sharrock2023online} (with $r=0$) to conclude that the Poisson equation 
\begin{equation}
\mathcal{A}_{\boldsymbol{x}^N} [w^{i,j,k,N}]_{m,n}(\theta,{\boldsymbol{x}}^N) = [T^{i,j,k,N}]_{m,n}(\theta,{\boldsymbol{x}}^N)~~~,~~~\int_{(\mathbb{R}^{d})^N} [w^{i,j,k,N}]_{m,n}(\theta,{\boldsymbol{x}}^N)\pi_{\theta_0}^{N}(\mathrm{d}{\boldsymbol{x}}^N)=0
\end{equation}
where $[A]_{m,n}$ denotes the $(m,n)^{\mathrm{th}}$ element of the matrix $A\in\mathbb{R}^{p\times p}$, has a solution which (element-wise) satisfies a polynomial growth property, similar to $v^{i,j,k,N}(\theta,\boldsymbol{x}^N)$. Thus, arguing as before (e.g., using the It\^o isometry and our moment bounds), it is possible to show that $\smash{\mathbb{E}[|[\Xi_{t,i,j,k,N}^{(1)}]_{m,n}|] \rightarrow 0}$
as $t\rightarrow\infty$. Thus, in particular, it follows that 
\begin{equation}
    \mathbb{E}[\| \Xi_{t,i,j,k,N}^{(1)}\|_{1}]\rightarrow 0 \label{eq:delta-1}
\end{equation}
as $t\rightarrow\infty$. We now turn our attention to $\Xi_{t,i,j,k,N}^{(2)}$. In this case, observe that the $(m,n)^{\text{th}}$ element of the matrix can be written as
\begin{align}
&[\Xi_{t,i,j,k,N}^{(2)}]_{m,n}\\
&= \gamma_t^{-1} \int_1^t  \gamma_s^2  \left[\Phi^{*,i,j,k,N}_{s,t} \left(\bar{\Gamma}^{i,j,k,N}(\theta_s^{i,j,k,N})- \bar{\Gamma}^{i,j,k,N}(\theta_0^{i,j,k,N})\right)\Phi^{*,i,j,k,N,\top}_{s,t}\right]_{m,n} \mathrm{d} s \\
&=\gamma_t^{-1}  \int_1^t  \gamma_s^2  \sum_{p_0=1}^p [\Phi^{*,i,j,k,N}_{s,t}]_{m,p_0}\sum_{p_1=1}^p\left[\bar{\Gamma}^{i,j,k,N}(\theta_s^{i,j,k,N})- \bar{\Gamma}^{i,j,k,N}(\theta_0^{i,j,k,N})\right]_{p_0,p_1}
[\Phi^{*,i,j,k,N,\top}_{s,t}]_{p_1,n} \mathrm{d} s \\
&=\gamma_t^{-1}  \int_1^t  \gamma_s^2  \sum_{p_0=1}^p [\Phi^{*,i,j,k,N}_{s,t}]_{m,p_0}\sum_{p_1=1}^p[\partial_{\theta}^{\top}\bar{\Gamma}^{i,j,k,N}(\tilde{\theta}_s^{i,j,k,N})]_{p_0,p_1}(\theta_s^{i,j,k,N} - \theta_0^{i,j,k,N})
[\Phi^{*,i,j,k,N,\top}_{s,t}]_{p_1,n} \mathrm{d} s
\end{align}
where $\tilde{\theta}_{s}^{i,j,k,N}$ is a point on the line segment connection $\theta_{s}^{i,j,k,N}$ and $\theta_0^{i,j,k,N}$. Due to Corollary~\ref{cor:empirical-drift-grad-lip} (i.e., the polynomial growth of $\boldsymbol{x}^N\mapsto g^{i,j,N}(\theta,\boldsymbol{x}^N)$ and its derivatives), and the polynomial growth of $\boldsymbol{x}^N\mapsto v^{i,j,k,N}(\theta,\boldsymbol{x}^N)$ and its derivatives, the function $\boldsymbol{x}^N\mapsto \partial_{\theta}{\Gamma}^{i,j,k,N}(\theta,\boldsymbol{x}^N)$ satisfies a polynomial growth property, uniformly in $\theta\in\Theta$. Thus, by Theorem~\ref{thm:moment-bounds} (i.e., the uniform-in-time moment bounds for the IPS), the function $\smash{\partial_{\theta}\bar{\Gamma}^{i,j,k,N}(\theta) = \int \partial_{\theta}\Gamma^{i,j,k,N}(\theta,\boldsymbol{x}^N)\pi_{\theta_0}^N(\mathrm{d}\boldsymbol{x}^N)}$ is bounded, uniformly in $\theta\in\Theta$. 

Using this, the Cauchy--Schwarz inequality, and the results of Theorem~\ref{theorem:main-theorem-1-2-finite-n} (i.e., the $\mathrm{L}^2$ convergence rate), it follows that 
\begin{align}
&\mathbb{E}\left[|[\Xi_{t,i,j,k,N}^{(2)}]_{m,n}|\right] \\
&\leq \gamma_t^{-1}  \int_1^t  \gamma_s^2  \mathbb{E}\left[\left|\sum_{p_0=1}^p [\Phi^{*,i,j,k,N}_{s,t}]_{m,p_0}\sum_{p_1=1}^p[\partial_{\theta}^{\top}\bar{\Gamma}^{i,j,k,N}(\tilde{\theta}_s^{i,j,k,N})]_{p_0,p_1}(\theta_s^{i,j,k,N} - \theta_0^{i,j,k,N})
[\Phi^{*,i,j,k,N,\top}_{s,t}]_{p_1,n} \right|\right]\mathrm{d} s \notag \\
&\leq K\gamma_t^{-1}  \int_1^t  \gamma_s^2  \|\Phi^{*,i,j,k,N}_{s,t}\|^2\mathbb{E}\Big[\|\theta_s^{i,j,k,N} - \theta_0^{i,j,k,N}\|^2\Big]^{\frac{1}{2}}
\mathrm{d} s 
\leq K\gamma_t^{-1}  \int_1^t  \gamma_s^2  e^{-2\eta^{i,j,k,N} \int_{s}^t \gamma_u\mathrm{d}u} \Big[(K_1^{\dagger} + K_2^{\dagger})\gamma_s \Big]^{\frac{1}{2}}\mathrm{d} s \notag \\
&\leq K\gamma_t^{-1}  \int_1^t  \gamma_s^{\frac{5}{2}}  \Phi_{s,t} \mathrm{d}s
\end{align}
where, as usual, we allow the value of the constants to increase from line to line. From Assumption \ref{assumption:learning-rate-v2} (our conditions on the learning rate), we have that $\smash{\int_1^t \gamma_s^{\frac{5}{2}} \Phi_{s,t}\mathrm{d}s = o(\gamma_t)}$. Thus, substituting this into the previous display, it follows that $\smash{\mathbb{E}[|[\Xi_{t,i,j,k,N}^{(2)}]_{m,n}|]\rightarrow 0}$ as $t\rightarrow\infty$ and thus, in particular, 
\begin{equation}
    \mathbb{E}[\|\Xi_{t,i,j,k,N}^{(2)}\|_1]\rightarrow 0 \label{eq:delta-2}
\end{equation}
as $t\rightarrow \infty$. Substituting \eqref{eq:delta-1} and \eqref{eq:delta-2} into \eqref{eq:delta-triangle}, it follows immediately that $\smash{\mathbb{E}\left[\|\Sigma_{t}^{i,j,k,N} - \bar{\Sigma}_{t}^{i,j,k,N}\|_1\right] \rightarrow 0}$ as $t\rightarrow \infty$. Using this result, the limit  previously established for $\bar{\Sigma}_t^{i,j,k,N}$ in \eqref{eq:sigma-limit-1}, and one final application of the triangle inequality, we finally arrive at
\begin{equation}
    \mathbb{E}\left[\|\Sigma_t^{i,j,k,N} - \bar{\Sigma}^{i,j,k,N}\|_1\right]\leq
    \mathbb{E}\left[\|\Sigma_t^{i,j,k,N} - \bar{\Sigma}_t^{i,j,k,N}\|_1\right]
    + \mathbb{E}\left[\|\bar{\Sigma}_t^{i,j,k,N} - \bar{\Sigma}^{i,j,k,N}\|_1\right] \longrightarrow 0
\end{equation}
as $t\rightarrow\infty$. This implies, in particular, that $\smash{\Sigma_{t}^{i,j,k,N} \stackrel{\mathbb{P}}{\longrightarrow}\bar{\Sigma}^{i,j,k,N}}$. That is, we have shown that the quadratic variation of the random variable $\smash{\gamma_t^{-\frac{1}{2}}[\Pi_{t,i,j,k,N}^{(4)} + \Omega_{t,i,j,k,N}^{(4)}]}$ converges in probability to $\smash{\bar{\Sigma}^{i,j,k,N}}$ as $t\rightarrow\infty$. It follows using standard results \citep[e.g.,][Section~1.2.2]{kutoyants2004statistical} that 
\begin{equation}
\gamma_t^{-\frac{1}{2}}\left[\Pi_{t,i,j,k,N}^{(4)} + \Omega_{t,i,j,k,N}^{(4)}\right]\stackrel{\mathrm{d}}{\longrightarrow}\mathcal{N}(0,\bar{\Sigma}^{i,j,k,N}).
\end{equation}
This result, combined with the decomposition in \eqref{main-bound-clt}, the decomposition in \eqref{eq:main-bound-clt-2}, and the convergence in probability of all other terms to zero, yields (via Slutsky's theorem) the result in \eqref{eq:clt-non-average-1}. The proof of \eqref{eq:clt-average-1} is essentially identical, replacing any quantities relating to the non-averaged estimator with their analogues for the averaged-estimator, and noting that all relevant results (e.g., solutions of the relevant Poisson equation, $\mathrm{L}^2$ convergence rate) continue to hold. 
\end{proof}

\begin{proof}[Proof of Theorem~\ref{theorem:clt-inf-n}]
    Once again, we will prove the result for the non-averaged estimator, before detailing how to adapt the proof for the averaged-estimator. Our proof will share some similarities with the proof of Theorem~\ref{theorem:clt}. Now, however, we will have to deal with additional terms arising in the $\mathrm{L}^2$ convergence rate w.r.t.\ the true parameter (see Theorem~\ref{theorem:main-theorem-1-2}). We begin, once again, by recalling the parameter update equation in a convenient form, namely
\begin{align}
\mathrm{d}\theta_t^{i,j,k,N} &= -\gamma_t\partial_{\theta}\mathcal{L}(\theta_t^{i,j,k,N})\mathrm{d}t - \gamma_t (h^{i,j,k,N}(\theta_t^{i,j,k,N},\boldsymbol{x}_t^N) - \partial_{\theta}\mathcal{L}(\theta_t^{i,j,k,N}))\mathrm{d}t
\label{eq:parameter-update-full-expansion_clt} \\*
&~~~~~+ \gamma_t g^{i,j,N}(\theta_t^{i,j,k,N},\boldsymbol{x}_t^N)\sigma^{-\top} \mathrm{d}w_t^{i,N} \notag
\end{align}
We proceed as in the previous proof, but now using a second order Taylor expansion for $\partial_{\theta}\mathcal{L}(\theta)$ around the true minimizer $\theta_0$. In particular, following similar arguments to those in \eqref{eq_4_107_clt} - \eqref{eq:clt-proof-1-end}, we can show that
\begin{align}
z_t^{i,j,k,N} &=  \Phi^{*}_{1,t} z_1^{i,j,k,N}    - \int_1^t \Phi^{*}_{s,t}  \frac{1}{2}\gamma_s \partial_{\theta}^3 \mathcal{L} (\tilde{\theta}_s^{i,j,k,N}) z_s^{i,j,k,N} z_s^{i,j,k,N,\top} \mathrm{d}s   \\[-1mm]
&- \int_1^t \Phi^{*}_{s,t} \gamma_s \big(h^{i,j,k,N}(\theta_s^{i,j,k,N},\boldsymbol{x}_s^N)-\partial_{\theta}\mathcal{L}(\theta_s^{i,j,k,N})\big) \mathrm{d}s  + \int_1^t   \Phi^{*}_{s,t}   \gamma_s g^{i,j,N}(\theta_s^{i,j,k,N},\boldsymbol{x}_s^N) \sigma^{-\top} \mathrm{d} w_s^{i,N}   \\
&= \Omega_{t,i,j,k,N}^{(1)} + \Omega_{t,i,j,k,N}^{(2)} + \Omega_{t,i,j,k,N}^{(3)} + \Omega_{t,i,j,k,N}^{(4)} 
\label{main-bound-clt-inf-n}
\end{align}
where now $z_t^{i,j,k,N}:=\theta_t^{i,j,k,N} - \theta_0$ and $\smash{\Phi^{*}_{s,t}= \exp[-\partial_{\theta}^2\mathcal{L}(\theta_0)\int_{s}^{t}\gamma_u\mathrm{d}u]}$. We will begin by bounding $\smash{\Omega_{t,i,j,k,N}^{(1)}}$. In this case, similar to before, we have that 
\begin{equation}
    \|\gamma_{t}^{-\frac{1}{2}}\,\Omega_{t,i,j,k,N}^{(1)}\| = \gamma_{t}^{-\frac{1}{2}}\|\Phi_{1,t}^{*}z_1^{i,j,k,N}\| \leq \gamma_{t}^{-\frac{1}{2}} \|\Phi_{1,t}^{*}\| \, \|z_1^{i,j,k,N}\| \leq K \gamma_{t}^{-\frac{1}{2}} \Phi_{1,t}^{\frac{1}{2}} \|z_1^{i,j,k,N}\|. 
\end{equation}
By Assumption \ref{assumption:learning-rate-v2} (i.e., our conditions on the learning rate), we have that $\smash{\Phi_{1,t}^{\frac{1}{2}} = o(\gamma_t^{\frac{1}{2}})}$. It follows from this and the previous display that
\begin{equation}
    \gamma_t^{-\frac{1}{2}}\,\Omega_{t,i,j,k,N}^{(1)} \stackrel{\mathrm{a.s.}}{\longrightarrow}0
    \label{eq:omega-t-1-recall}
\end{equation}
as $t\rightarrow\infty$, and thus also in probability. We now turn our attention to $\smash{\Omega_{t,i,j,k,N}^{(2)}}$. Arguing similarly to before, but now using the $\mathrm{L}^2$ convergence rate from Theorem~\ref{theorem:main-theorem-1-2}, we have
\begin{align}
    \mathbb{E}\left[\|\gamma_{t}^{-\frac{1}{2}} \, \Omega_{t,i,j,k,N}^{(2)}\|_1\right] 
    &\leq K\gamma_{t}^{-\frac{1}{2}}\Big[\int_1^t \Phi_{s,t}^{\frac{1}{2}}\,\gamma_s^2  \mathrm{d}s + \Big(\rho(N) + \frac{1}{N^{\frac{1}{2(1+\alpha)}}}\Big) \int_1^t \Phi_{s,t}^{\frac{1}{2}}\,\gamma_s \,\mathrm{d}s \Big]. \label{eq:tilde-omega-1}
\end{align}
By Assumption~\ref{assumption:learning-rate-v2} (i.e., our additional conditions on the learning rate), we have that $\smash{\int_1^t  \Phi_{s,t}^{\frac{1}{2}}\,\gamma_s^2  \mathrm{d}s = o(\gamma_t^{\frac{1}{2}})}$ and $\smash{\int_1^t  \Phi_{s,t}^{\frac{1}{2}}\,\gamma_s  \mathrm{d}s = O(1)}$ as $t\rightarrow\infty$. In addition, under our standing assumption, we have that $N=N(t)\rightarrow\infty$ as $t\rightarrow\infty$ at the rate $\smash{\rho(N) + N^{-\frac{1}{2(1+\alpha)}} = o(\gamma_t^{\frac{1}{2}})}$. These facts, together with \eqref{eq:tilde-omega-1}, imply that
\begin{equation}
    \gamma_{t}^{-\frac{1}{2}} \, \Omega_{t,i,j,k,N}^{(2)}\stackrel{\mathbb{L}_1}{\longrightarrow}0
    \label{eq:omega-t-2-recall}
\end{equation}
as $t\rightarrow\infty$, and hence also in probability. We now turn our attention to $\smash{\Omega_{t,i,j,k,N}^{(3)}}$. For this term, we will require a different strategy from the corresponding term in the previous proof, since it is not clear that the finite-particle Poisson equation (or its solution) are well-defined in the limit as $N\rightarrow\infty$. We begin by decomposing this term into two parts, namely, 
    \begin{align}
    \gamma_{t}^{-\frac{1}{2}}\Omega_{t,i,j,k,N}^{(3)}
     &=-   \gamma_{t}^{-\frac{1}{2}}\int_1^t \Phi^{*}_{s,t} \gamma_s \big(h(\theta_s^{i,j,k,N},x_s^{i,N},x_s^{j,N},x_s^{k,N},\mu_s^N)-h(\theta_s^{i,j,k,N},\bar{x}_s^{i},\bar{x}_s^{j},\bar{x}_s^{k},\
    \pi_{\theta_0})\big) \mathrm{d}s \\
    &~~~~-   \gamma_{t}^{-\frac{1}{2}}\int_1^t \Phi^{*}_{s,t} \gamma_s \big(h(\theta_s^{i,j,k,N},\bar{x}_s^{i},\bar{x}_s^{j},\bar{x}_s^{k},\pi_{\theta_0})-\partial_{\theta}\mathcal{L}(\theta_s^{i,j,k,N})\big) \mathrm{d}s \\
    &:=\gamma_{t}^{-\frac{1}{2}}\Psi_{t,i,j,k,N}^{(1)} + \gamma_{t}^{-\frac{1}{2}}\Psi_{t,i,j,k,N}^{(2)} 
\end{align}
where, similar to elsewhere, $(x_s^{i})_{s\geq 0}$, $(x_s^{j})_{s\geq 0}$ and $(x_s^{k})_{s\geq 0}$ are independent solutions of the MVSDE, driven by the same Brownian motions and with the same initial conditions as the particles $(x_s^{i,N})_{t\geq 0}$, $(x_s^{j,N})_{s\geq 0}$ and $(x_s^{k,N})_{s\geq 0}$, and $(\bar{x}_s^{i})_{s\geq 0}$, $(\bar{x}_s^{j})_{s\geq 0}$, and $(\bar{x}_s^k)_{s\geq 0}$ are solutions of the MVSDE, driven by the same Brownian motions, but initialised at the stationary distribution $\pi_{\theta_0}$. Similar to before, we assume that $(x_0^{a},\bar{x}_0^{a})\sim \gamma_0^{*}$, for $a\in\{i,j,k\}$, where $\gamma_0^{*}\in\Gamma(\mu_0,\pi_{\theta_0})$ denotes the optimal coupling between $\mu_0$ and $\pi_{\theta_0}$ w.r.t. the quadratic cost. By Lemma~\ref{cor:empirical-H-lip} (i.e., the function $h$ is locally Lipschitz with polynomial growth), the Cauchy--Schwarz inequality, and Theorem~\ref{thm:moment-bounds} (i.e., uniform-in-time moment bounds for the IPS and the MVSDE), we have that
\begin{align}
\label{eq:h-decomp}
    &\mathbb{E}\Big[\| h(\theta_s^{i,j,k,N},x_s^{i,N},x_s^{j,N},x_s^{k,N},\mu_s^N)- h(\theta_s^{i,j,k,N},\bar{x}_s^{i},\bar{x}_s^{j},\bar{x}_s^{k},\pi_{\theta_0})\|\Big] \\
    &\leq K \Big[(\mathbb{E}[\mathsf{W}_2^2(\mu_s^N,\pi_{\theta_0})])^{\frac{1}{2}}  + \textstyle \sum_{a\in\{i,j,k\}} (\mathbb{E}\big[\|x_s^{a,N} - \bar{x}_s^{a}\|^2\big])^{\frac{1}{2}} \Big] 
\end{align}
Meanwhile, using the triangle inequality, the elementary inequality $(a+b+c)^2\leq 3(a^2+b^2+c^2)$, and Theorem~\ref{thm:poc} (i.e., uniform-in-time propagation of chaos), Theorem~1 in \cite{fournier2015rate} (i.e., bounds on the W2 distance to the empirical measure), and Theorem~\ref{thm:invariant-distribution-2} (i.e., convergence to the invariant distribution), we have that
\begin{align}
\mathbb{E}[\mathsf{W}_2^2(\mu_s^N,\pi_{\theta_0})] &\leq 3 \mathbb{E}[\mathsf{W}_2^2(\mu_s^N,\mu_s^{[N]})] + 3 \mathbb{E}[\mathsf{W}_2^2(\mu_s^{[N]},\mu_s)] + 3 \mathbb{E}[\mathsf{W}_2^2(\mu_s,\pi_{\theta_0})] \\
&\leq K\Big[\frac{1}{N^{\frac{1}{1+\alpha}}} + \rho^2(N)  + a_s(\mathsf{W}_2(\mu_0,\pi_{\theta_0}) )\Big]
\end{align}
where $a_s:\mathbb{R}_{+}\rightarrow\mathbb{R}_{+}$ is the function defined in Theorem~\ref{thm:invariant-distribution-2}, and $\rho:\mathbb{N}\rightarrow\mathbb{R}_{+}$ is the function defined in \eqref{eq:rho}. Using similar arguments, we also have
\begin{align}
    \mathbb{E}\big[\|x_s^{a,N} - \bar{x}_s^{a}\|^2\big] \leq 2\mathbb{E}[\|x_s^{a,N} - x_s^{a}\|^2] + 2\mathbb{E}[\|x_s^{a} - \bar{x}_s^a\|^2] \leq K\left[\frac{1}{N^{\frac{1}{1+\alpha}}} + a_s(\mathsf{W}_2(\mu_0,\pi_{\theta_0}))\right]
\end{align}
Substituting these two bounds into \eqref{eq:h-decomp}, and setting $a_s:=a_s(\mathsf{W}_2(\mu_0,\pi_{\theta_0}))$, we have that
\begin{align}
    \mathbb{E}\left[\| h(\theta_s^{i,j,k,N},x_s^{i,N},x_s^{j,N},x_s^{k,N},\mu_s^N)- h(\theta_s^{i,j,k,N},\bar{x}_s^{i},\bar{x}_s^{j},\bar{x}_s^{k},\pi_{\theta_0})\|\right] \leq K \Big[\rho(N) + \frac{1}{N^{\frac{1}{2(1+\alpha)}}} + a_s^{\frac{1}{2}}\Big] 
\end{align}
Substituting this back into the definition of $\Psi_{t,i,j,k,N}^{(1)}$, and using our previous bounds on $\|\Phi_{s,t}^{*}\|$,  we thus have that
\begin{align}
    \mathbb{E}\left[\|\gamma_t^{-\frac{1}{2}}\Psi_{t,i,j,k,N}^{(1)}\|_1\right] &\leq K \gamma_t^{-\frac{1}{2}} \int_1^t \|\Phi_{s,t}^{*}\|\,\gamma_s \mathbb{E}\left[\| h(\theta_s^{i,j,k,N},x_s^{i,N},x_s^{j,N},x_s^{k,N}, \mu_s^N)-h(\theta_s^{i,j,k,N},\bar{x}_s^{i},\bar{x}_s^{j},\bar{x}_s^{k},\pi_{\theta_0}) \| \right]\mathrm{d}s \nonumber \\
    &= K \Big(\rho(N) + \frac{1}{N^{\frac{1}{2(1+\alpha)}}}\Big)\gamma_t^{-\frac{1}{2}} \int_1^t \Phi_{s,t}^{\frac{1}{2}}\, \gamma_s \, \mathrm{d}s + K\gamma_t^{-\frac{1}{2}} \int_1^t \Phi_{s,t}^{\frac{1}{2}}\, \gamma_s \,a_s^{\frac{1}{2}} \, \mathrm{d}s
\end{align}
By Assumption~\ref{assumption:learning-rate-v2}, we have that $\smash{\int_1^t  \Phi_{s,t}^{\frac{1}{2}}\,\gamma_s  \mathrm{d}s = O(1)}$ and that $\smash{\int_1^t  \Phi_{s,t}^{\frac{1}{2}}\,\gamma_s\,a_s^{\frac{1}{2}}\, \mathrm{d}s = o(\gamma_t^{\frac{1}{2}})}$ as $t\rightarrow\infty$. In addition, $N=N(t)\rightarrow\infty$ as $t\rightarrow\infty$ at a rate such that $\smash{\rho(N) + N^{-\frac{1}{2(1+\alpha)}} = o(\gamma_t^{\frac{1}{2}})}$. It follows immediately that 
\begin{equation}
    \gamma_t^{-\frac{1}{2}} \Psi_{t,i,j,k,N}^{(1)} \stackrel{\mathbb{L}_1}{\longrightarrow}0
\end{equation}
as $t\rightarrow\infty$, and hence also in probability. We now consider $\Psi_{t,i,j,k,N}^{(2)}$. We will analyse this term by constructing an appropriate Poisson equation, this time specified in terms of the (linearized) mean-field equation, rather than its finite-particle counterpart. Let us define
\begin{equation}
    R^{i,j,k}(\theta,\bar{\boldsymbol{x}}^{(i,j,k)}) = \partial_{\theta}\mathcal{L}(\theta) - h_{\pi_{\theta_0}}(\theta,\bar{\boldsymbol{x}}^{(i,j,k)}), \qquad h_{\pi_{\theta_0}}(\theta,\bar{\boldsymbol{x}}^{(i,j,k)}):=h(\theta,\bar{x}^i,\bar{x}^j,\bar{x}^k,\pi_{\theta_0}).
\end{equation} 
where $\bar{\boldsymbol{x}}^{(i,j,k)} = (\bar{x}^i, \bar{x}^j, \bar{x}^k)^{\top}$. Using the definition of $\partial_{\theta}\mathcal{L}$, this function is centered w.r.t. $\pi_{\theta_0}^{\otimes 3}$. Thus, by Lemma~\ref{lemma:main-theorem-lemma-2-a} (i.e., the boundedness of the asymptotic log-likelihood and its derivatives), and Lemma~\ref{lem:Phi-hH-stability} (i.e., the local Lipschitz and polynomial growth of $h$ and its derivatives), this function satisfies all of the conditions required by (a minor modification of) Lemma~17 in \cite{sharrock2023online}.  Thus, the Poisson equation 
\begin{equation}
\mathcal{A}_{\boldsymbol{x}} v^{i,j,k}(\theta,\boldsymbol{x}^{(i,j,k)}) = R^{i,j,k}(\theta,\bar{\boldsymbol{x}}^{(i,j,k)})~~~,~~~\int_{(\mathbb{R}^{d})^3} v^{i,j,k}(\theta,\bar{\boldsymbol{x}}^{(i,j,k)})\pi_{\theta_0}^{\otimes 3}(\mathrm{d}\bar{\boldsymbol{x}}^{(i,j,k)})=0
\end{equation}
has a unique twice differentiable solution satisfying $\sum_{k=0}^{2} \|\frac{\partial^k }{\partial \theta^{k}}v^{i,j,k}(\theta,\bar{\boldsymbol{x}}^{(i,j,k)})\| + \|\frac{\partial^2 }{\partial \theta\partial \bar{\boldsymbol{x}}^{(i,j,k)}}v^{i,j,k}(\theta,\bar{\boldsymbol{x}}^{(i,j,k)})\| \leq K [1+\|\bar{x}^i\|^{q}+\|\bar{x}^j\|^{q}+\|\bar{x}^k\|^{q} ]$. Arguing similar to before (e.g., using It\^{o}'s formula, then rearranging), it is possible to rewrite $\Psi_{t,i,j,k,N}^{(2)}$ in terms of this solution as
\begin{align}
\gamma_{t}^{-\frac{1}{2}} \Psi_{t,i,j,k,N}^{(2)} &= \gamma_{t}^{-\frac{1}{2}}\int_1^t \gamma_s\Phi^{*}_{s,t}\underbrace{\left(\partial_{\theta}\mathcal{L}(\theta_s^{i,j,k,N}) - h_{\pi_{\theta_0}}(\theta_s^{i,j,k,N},\bar{\boldsymbol{x}}_s^{(i,j,k)})\right)\mathrm{d}s}_{R^{i,j,k}(\theta_s,\boldsymbol{x}^{(i,j,k)})\mathrm{d}s} \label{eq3163-recall-2} \\[-1mm]
&=\gamma_{t}^{-\frac{1}{2}} \int_1^t \gamma_s \Phi^{*}_{s,t}\mathrm{d}v_s^{i,j,k}- \gamma_{t}^{-\frac{1}{2}} \int_{1}^t \gamma_s\Phi^{*}_{s,t}\mathcal{A}_{\theta}v^{i,j,k}(\theta_s^{i,j,k,N},\bar{\boldsymbol{x}}_s^{(i,j,k)})\mathrm{d}s \\
&- \gamma_{t}^{-\frac{1}{2}} \int_{1}^t \gamma_s^2 \Phi^{*}_{s,t}\partial_{\theta}v^{i,j,k}(\theta_s^{i,j,k,N},\bar{\boldsymbol{x}}_s^{(i,j,k)})g(\theta_s^{i,j,k,N},\bar{x}_s^{i},\bar{x}_s^{j})\sigma^{-\top}\mathrm{d}w_s^{i,N} \\
&-  \gamma_{t}^{-\frac{1}{2}}\int_{1}^t \gamma_s\Phi^{*}_{s,t}\partial_{\boldsymbol{x}}v^{i,j,k}(\theta_s,\bar{\boldsymbol{x}}_s^{(i,j,k)}) (I_3\otimes \sigma)\mathrm{d}{w}^{(i,j,k)}_s \\
 &-  \gamma_{t}^{-\frac{1}{2}}\int_{1}^t2\gamma^2_s\Phi^{*}_{s,t}\partial_{\theta}\partial_{x}v^{i,j,k}(\theta_s,\bar{\boldsymbol{x}}_s^{(i,j,k)}) g(\theta_s^{i,j,k,N},\bar{x}_s^{i},\bar{x}_s^{j}) \mathrm{d}s \\
 &:= \gamma_{t}^{-\frac{1}{2}}\Pi_{t,i,j,k,N}^{(1)} + \gamma_{t}^{-\frac{1}{2}}\Pi_{t,i,j,k,N}^{(2)} + \gamma_{t}^{-\frac{1}{2}}\Pi_{t,i,j,k,N}^{(3)} + \gamma_{t}^{-\frac{1}{2}}\Pi_{t,i,j,k,N}^{(4)} + \gamma_{t}^{-\frac{1}{2}}\Pi_{t,i,j,k,N}^{(5)}\label{eq:main-bound-clt-2-inf-n}
\end{align}
Following steps similar to those used in the proof of Theorem~\ref{theorem:main-theorem-1-2-finite-n} (e.g., using the PGP of $(x,y)\mapsto g(\theta,x,y)$ from Corollary~\ref{cor:empirical-drift-grad-lip}, the uniform-in-time moment bounds from Theorem~\ref{thm:moment-bounds}, and the conditions on the learning rate from Assumption~\ref{assumption:learning-rate-v2}), we have that 
\begin{equation}
    \gamma_{t}^{-\frac{1}{2}}\left(\Pi_{t,i,j,k,N}^{(1)}  + \Pi_{t,i,j,k,N}^{(2)} + \Pi_{t,i,j,k,N}^{(3)} + \Pi_{t,i,j,k,N}^{(5)}\right) \stackrel{\mathbb{L}_1}{\longrightarrow} 0
\end{equation}
as $t\rightarrow\infty$, and thus also in probability. We will later return to $\Pi_{t,i,j,k,N}^{(4)}$. For now, let us consider $\Omega_{t,i,j,k,N}^{(4)}$. In this case, we will once more make use of a further decomposition, namely,  
\begin{align}
\gamma_t^{-\frac{1}{2}}\Omega_{t,i,j,k,N}^{(4)} &= \gamma_t^{-\frac{1}{2}}\int_1^t   \Phi^{*}_{s,t}   \gamma_s (g(\theta_s^{i,j,k,N},\boldsymbol{x}_s^{i,j,N}) - g(\theta_s^{i,j,k,N},\bar{\boldsymbol{x}}_s^{(i,j)})) \sigma^{-\top} \mathrm{d} w_s^{i,N} 
\\
&+\gamma_t^{-\frac{1}{2}}\int_1^t   \Phi^{*}_{s,t}   \gamma_s g(\theta_s^{i,j,k,N},\bar{\boldsymbol{x}}_s^{(i,j)}) \sigma^{-\top} \mathrm{d} w_s^{i,N} :=\gamma_t^{-\frac{1}{2}}\Phi_{t,i,j,k,N}^{(1)} + \gamma_t^{-\frac{1}{2}}\Phi_{t,i,j,k,N}^{(2)}
\end{align}
We can bound the first term in this decomposition using Lemma~\ref{cor:empirical-drift-grad-lip} (i.e., the function $g$ is locally Lipschitz with polynomial growth). In particular, using this lemma, the Cauchy--Schwarz inequality, and Theorem~\ref{thm:moment-bounds} (i.e., uniform-in-time moment bounds for the IPS and the MVSDE), we have that 
\begin{align}
\label{eq:g-decomp}
    &\mathbb{E}\left[\| g(\theta_s^{i,j,k,N},x_s^{i,N},x_s^{j,N})- g(\theta_s^{i,j,k,N},\bar{x}_s^{i},\bar{x}_s^{j})\|\right] \\
    &\leq K \Big[ \textstyle \sum_{a\in\{i,j,k\}} (\mathbb{E}\left[\|x_s^{a,N} - x_s^{a}\|^2\right])^{\frac{1}{2}} + \sum_{a\in\{i,j,k\}} 
    \left(\mathbb{E}[\|x_s^{a}-\bar{x}_s^{a}\|^2]\right)^{\frac{1}{2}}\Big] 
\end{align}
Using Theorem~\ref{thm:poc} (i.e., uniform-in-time propagation of chaos) and Theorem~\ref{thm:invariant-distribution-2} (i.e., convergence to the invariant distribution), it follows that 
\begin{align*}
    &\mathbb{E}\left[\| g(\theta_s^{i,j,k,N},x_s^{i,N},x_s^{j,N})- g(\theta_s^{i,j,k,N},\bar{x}_s^{i},\bar{x}_s^{j})\|\right] \leq K \Big[\frac{1}{N^{\frac{1}{2(1+\alpha)}}} + a_s^{\frac{1}{2}}\Big].
\end{align*}
where once again we use the shorthand $a_s:=a_s(\mathsf{W}_2(\mu_0,\pi_{\theta_0}))$. Using the $\mathbb{L}_p$ interpolation inequality, and once more  Theorem~\ref{thm:moment-bounds} (i.e., uniform-in-time moment bounds for the IPS and the MVSDE), it follows that, for any $0<\varepsilon<1$, there exists $K=K(\varepsilon)<\infty$ such that 
\begin{align*}
    &\mathbb{E}\left[\| g(\theta_s^{i,j,k,N},x_s^{i,N},x_s^{j,N})- g(\theta_s^{i,j,k,N},\bar{x}_s^{i},\bar{x}_s^{j})\|^2\right] \leq K \Big[\frac{1}{N^{\frac{1-\varepsilon}{2(1+\alpha)}}} + a_s^{\frac{1}{2}(1-\varepsilon)}\Big].
\end{align*}
Recalling the definition of $\Phi_{t,i,j,k,N}^{(1)}$, using It\^{o}'s lemma, and also our previous bounds on $\|\Phi_{s,t}^{*}\|$, we then have that
\begin{align}
    \mathbb{E}\left[\|\gamma_t^{-\frac{1}{2}}\Phi_{t,i,j,k,N}^{(1)}\|^2\right] &\leq K \gamma_t^{-1} \int_1^t \|\Phi_{s,t}^{*}\|^2\,\gamma_s^2 \mathbb{E}\left[\| g(\theta_s^{i,j,k,N},x_s^{i,N},x_s^{j,N})-g(\theta_s^{i,j,k,N},\bar{x}_s^{i},\bar{x}_s^{j}) \|_{\sigma\sigma^{\top}}^2 \right]\mathrm{d}s \nonumber \\
    &= K \Big(\frac{1}{N^{\frac{1-\varepsilon}{2(1+\alpha)}}}\Big)\gamma_t^{-1} \int_1^t \Phi_{s,t}\, \gamma_s^2 \, \mathrm{d}s + K\gamma_t^{-1} \int_1^t \Phi_{s,t}\, \gamma_s^2 \,a_s^{\frac{1}{2}(1-\varepsilon)} \, \mathrm{d}s
\end{align}
By Assumption~\ref{assumption:learning-rate-v2} (i.e., our additional conditions on the learning rate), we have that $\smash{\int_1^t  \Phi_{s,t}\,\gamma_s^2  \mathrm{d}s = O(\gamma_t)}$ as $t\rightarrow\infty$. Moreover, $\smash{\int_1^t  \Phi_{s,t}\,\gamma_s^2\,a_s^{\frac{1}{2}}\, \mathrm{d}s = o(\gamma_t)}$ as $t\rightarrow\infty$, i.e., $\smash{\int_1^t  \Phi_{s,t}\,\gamma_s^2\,a_s^{\frac{1}{2}(1-\varepsilon)}\, \mathrm{d}s = o(\gamma_t)}$ with $\varepsilon=\frac{1}{2}$. It follows, using also the fact that $N=N(t)\rightarrow\infty$ as $t\rightarrow\infty$, that
\begin{equation}
    \gamma_t^{-\frac{1}{2}} \Phi_{t,i,j,k,N}^{(1)} \stackrel{\mathrm{L}^2}{\longrightarrow}0
\end{equation}
as $t\rightarrow\infty$, and so this term converges in probability to zero. It remains to analyse $\smash{\gamma_t^{-\frac{1}{2}}[\Pi_{t,i,j,k,N}^{(4)} + \Phi_{t,i,j,k,N}^{(2)}]}$, which will be responsible for the covariance of the limiting Gaussian random variable. From the definitions, arguing similarly to before, we have 
\begin{align}
    &\gamma_t^{-\frac{1}{2}}[\Pi_{t,i,j,k,N}^{(4)} + \Omega_{t,i,j,k,N}^{(4)}] \\
    &= \gamma_{t}^{-\frac{1}{2}}\left[\int_1^t  \gamma_s  \Phi^{*}_{s,t}    g(\theta_s^{i,j,k,N},\bar{\boldsymbol{x}}_s^{(i,j)}) \sigma^{-\top} \mathrm{d} w_s^{i} - \int_{1}^t \gamma_s\Phi^{*}_{s,t}\partial_{\boldsymbol{y}}v^{i,j,k}(\theta_s^{i,j,k,N},\bar{\boldsymbol{x}}_s^{(i,j,k)}) (I_3\otimes \sigma)\mathrm{d}{w}^{(i,j,k)}_s\right] \\
    &= \gamma_{t}^{-\frac{1}{2}}\left[\int_1^t  \gamma_s  \Phi^{*}_{s,t}  \left(g(\theta_s^{i,j,k,N},\bar{\boldsymbol{x}}_s^{(i,j)})(\sigma\sigma^{\top})^{-1} D_i^{\top} - \partial_{\boldsymbol{y}}v^{i,j,k}(\theta_s^{i,j,k,N},\bar{\boldsymbol{x}}_s^{(i,j,k)})\right)(I_3\otimes \sigma)\mathrm{d}{w}^{(i,j,k)}_s\right]
\end{align}
where $D_i\in\mathbb{R}^{3d\times d}$ is the block-selector matrix such that $\mathrm{d}w_s^{i} = D_i^{\top} \mathrm{d}w_s^{(i,j,k)}$. It is then straightforward to compute the quadratic covariation matrix of these terms as 
\begin{align}
    \Sigma_{t}^{i,j,k} &= \gamma_{t}^{-1}\int_1^t  \gamma_s^2  \Phi^{*}_{s,t} \Gamma^{i,j,k}(\theta_s^{i,j,k,N},\bar{\boldsymbol{x}}_s^{(i,j,k)})\Phi^{*,\top}_{s,t} \mathrm{d} s
\end{align}
where $\Gamma^{i,j,k}(\theta,\bar{\boldsymbol{x}}^{(i,j,k)})=(g(\theta,\bar{\boldsymbol{x}}^{(i,j)}) (\sigma\sigma^{\top})^{-1} D_i^{\top} - \partial_{\boldsymbol{x}}v^{i,j,k}(\theta,\bar{\boldsymbol{x}}^{(i,j,k)}) ) (I_3\otimes (\sigma\sigma^{\top})) (g(\theta,\bar{\boldsymbol{x}}^{(i,j)})(\sigma\sigma^{\top})^{-1} D_i^{\top}  - \partial_{\boldsymbol{x}}v^{i,j,k}(\theta,\bar{\boldsymbol{x}}^{(i,j,k)}))^{\top}$. Similar to before, we will establish the convergence of this covariation matrix in two steps. To be specific, we will first show that there exists a limiting covariance matrix $\bar{\Sigma}^{i,j,k}$ such that
\begin{equation}
\|\bar{\Sigma}_t^{i,j,k} - \bar{\Sigma}^{i,j,k}\|_{1}\rightarrow 0 \label{eq:limiting-cov-inf-n}
\end{equation}
as $t\rightarrow\infty$, where $\bar{\Sigma}_t^{i,j,k}$ is an approximation for $\Sigma_t^{i,j,k}$, in which the central term in the integrand has been replaced by its ergodic average, evaluated at the true parameter, viz $\smash{\bar{\Sigma}_t^{i,j,k} = \gamma_{t}^{-1}\int_1^t  \gamma_s^2  \Phi^{*}_{s,t} \bar{\Gamma}^{i,j,k}(\theta_0)\Phi^{*,\top}_{s,t} \mathrm{d} s}$ with $\smash{\bar{\Gamma}^{i,j,k}(\theta) = \int_{(\mathbb{R}^d)^3}\Gamma^{i,j,k}(\theta,\bar{\boldsymbol{x}}^{(i,j,k)}) \pi^{\otimes 3}_{\theta_0}(\mathrm{d}\bar{\boldsymbol{x}}^{(i,j,k)})}$. We will later also show that 
\begin{equation}
\mathbb{E}[\|\Sigma_t^{i,j,k} - \bar{\Sigma}_t^{i,j,k}\|_1]\rightarrow 0
\end{equation}
as $t\rightarrow\infty$, and hence conclude that $\mathbb{E}[\|\Sigma_t^{i,j,k} - \bar{\Sigma}^{i,j,k}\|_1]\rightarrow 0$ as $t\rightarrow\infty$ via the triangle inequality. For now, to establish \eqref{eq:limiting-cov-inf-n}, we can argue exactly as in the previous proof. In particular, following the steps in \eqref{eq:limit-cov-start} - \eqref{eq:limit-cov-end}, we can show that the limiting covariance matrix is given by
\begin{align}
[\bar{\Sigma}^{i,j,k}]_{m,n} &:= \lim_{t\rightarrow\infty} [\bar{\Sigma}_{t}^{i,j,k}]_{m,n} \\
&= \sum_{p_0,p_1=1}^p  v_{m}^{p_1} v_{p_0}^{p_1} \sum_{p_2=1}^p [\bar{\Gamma}^{i,j,k}(\theta_0)]_{p_0,p_2}\sum_{p_3=1}^p \lim_{t\rightarrow\infty}\left[\gamma_t^{-1}\int_1^t \gamma_s^2  e^{-(\kappa_{p_1} + \kappa_{p_3})\int_{s}^t \gamma_{u}\mathrm{d}u} \mathrm{d}s\right]   v_{p_2}^{p_3} v_{n}^{p_3},
\end{align}
which exists due to our conditions on the learning rate. It remains to show that $\mathbb{E}\|\Sigma_t^{i,j,k} - \bar{\Sigma}_t^{i,j,k}\|_{1}\rightarrow0$ as $t\rightarrow\infty$. To do this, we will consider a similar decomposition to before, namely, 
\begin{align}
    \|\Sigma_{t}^{i,j,k} - \bar{\Sigma}_t^{i,j,k}\| 
    &\leq  \Big\|  \gamma_{t}^{-1} \int_1^t  \gamma_s^2  \Phi^{*}_{s,t} \left[\Gamma^{i,j,k}(\theta_s^{i,j,k,N},\bar{\boldsymbol{x}}_s^{(i,j,k)}) -  \bar{\Gamma}^{i,j,k}(\theta_s^{i,j,k,N})\right]\Phi^{*,\top}_{s,t} \mathrm{d} s \Big\|  \label{eq:delta-triangle-inf-n} \\[-2mm] 
    & +  \Big\| \gamma_{t}^{-1}  \int_1^t  \gamma_s^2  \Phi^{*}_{s,t} \left[\bar{\Gamma}^{i,j,k}(\theta_s^{i,j,k,N})- \bar{\Gamma}^{i,j,k}(\theta_0)\right]\Phi^{*,\top}_{s,t} \mathrm{d} s \Big\| \\
    &:=\Xi_{t,i,j,k}^{(1)} + \Xi_{t,i,j,k}^{(2)}
\end{align}
Similar to the proof of the previous result, we can analyse $\Xi_{t,i,j,k}^{(1)}$ using an appropriate Poisson equation, now given in terms of the linearized, mean-field SDE. In particular, we now define
\begin{equation}
    T^{i,j,k}(\theta,\boldsymbol{x}^{(i,j,k)}) = \Gamma^{i,j,k}(\theta,\boldsymbol{x}^{(i,j,k)}) -  \bar{\Gamma}^{i,j,k}(\theta)
\end{equation}
Due to Corollary~\ref{cor:empirical-drift-grad-lip} (i.e., the polynomial growth of $(x,y)\mapsto g(\theta,{x},y)$ and its derivatives) and the bounds on our existing solution of the Poisson equation (i.e., the polynomial growth of $\bar{\boldsymbol{x}}^{(i,j,k)}\mapsto v^{i,j,k}(\theta,\bar{\boldsymbol{x}}^{(i,j,k)})$ and its derivatives), this function (and its derivatives) are locally Lipschitz with polynomial growth. Moreover, by definition, it is centered with respect to $\pi_{\theta_0}^{\otimes 3}$. Thus, the Poisson equation 
\begin{equation}
\mathcal{A}_{\boldsymbol{y}} [w^{i,j,k}]_{m,n}(\theta,{\boldsymbol{y}}^{(i,j,k)}) = [T^{i,j,k}]_{m,n}(\theta,{\boldsymbol{y}}^{(i,j,k)})~~~,~~~\int_{(\mathbb{R}^{d})^3} [w^{i,j,k}]_{m,n}(\theta,{\boldsymbol{y}}^{(i,j,k)})\pi_{\theta_0}^{\otimes 3}(\mathrm{d}{\boldsymbol{y}}^{(i,j,k)})=0
\end{equation}
where $[A]_{m,n}$ denotes the $(m,n)^{\mathrm{th}}$ element of the matrix $A\in\mathbb{R}^{p\times p}$, has a solution which (element-wise) satisfies a polynomial growth property. Thus, arguing as before (e.g., using the It\^o isometry and our moment bounds), it is possible to show that $\smash{\mathbb{E}[|[\Xi_{t,i,j,k,N}^{(1)}]_{m,n}|] \rightarrow 0}$
as $t\rightarrow\infty$. Thus, in particular, it follows that 
\begin{equation}
    \mathbb{E}[\| \Xi_{t,i,j,k}^{(1)}\|_{1}]\rightarrow 0 \label{eq:delta-1-inf-n}
\end{equation}
as $t\rightarrow\infty$. We now turn our attention to $\Xi_{t,i,j,k}^{(2)}$. In this case, observe that the $(m,n)^{\text{th}}$ element of the matrix can be written as
\begin{align}
[\Xi_{t,i,j,k}^{(2)}]_{m,n}&= \gamma_t^{-1} \int_1^t  \gamma_s^2  \left[\Phi^{*}_{s,t} \left(\bar{\Gamma}^{i,j,k}(\theta_s^{i,j,k,N})- \bar{\Gamma}^{i,j,k}(\theta_0)\right)\Phi^{*,\top}_{s,t}\right]_{m,n} \mathrm{d} s \\
&=\gamma_t^{-1}  \int_1^t  \gamma_s^2  \sum_{p_0=1}^p [\Phi^{*}_{s,t}]_{m,p_0}\sum_{p_1=1}^p[\partial_{\theta}^{\top}\bar{\Gamma}^{i,j,k}(\tilde{\theta}_s^{i,j,k,N})]_{p_0,p_1}(\theta_s^{i,j,k,N} - \theta_0)
[\Phi^{*,\top}_{s,t}]_{p_1,n} \mathrm{d} s
\end{align}
where $\tilde{\theta}_{s}^{i,j,k,N}$ is a point on the line segment connecting $\theta_{s}^{i,j,k,N}$ and $\theta_0^{i,j,k,N}$. Due to Corollary~\ref{cor:empirical-drift-grad-lip} (i.e., the polynomial growth of $(x,y)\mapsto g(\theta,x,y)$ and its derivatives), and the bounds on the solution of the earlier Poisson equation (i.e., the polynomial growth of $\bar{\boldsymbol{x}}^{(i,j,k)}\mapsto v^{i,j,k}(\theta,\bar{\boldsymbol{x}}^{(i,j,k)})$ and its derivatives), the function $\bar{\boldsymbol{x}}^{(i,j,k)}\mapsto \partial_{\theta}{\Gamma}^{i,j,k}(\theta,\bar{\boldsymbol{x}}^{(i,j,k)})$ satisfies a polynomial growth property, uniformly in $\theta\in\Theta$. Thus, by Theorem~\ref{thm:moment-bounds} (i.e., the uniform-in-time moment bounds for the IPS), the function $\smash{\partial_{\theta}\bar{\Gamma}^{i,j,k}(\theta) = \int \partial_{\theta}\Gamma^{i,j,k}(\theta,\bar{\boldsymbol{x}}^{(i,j,k)})\pi_{\theta_0}^{\otimes 3}(\mathrm{d}\bar{\boldsymbol{x}}^{(i,j,k)})}$ is bounded, uniformly in $\theta\in\Theta$. Using this, the Cauchy--Schwarz inequality, and the results of Theorem~\ref{theorem:main-theorem-1-2} (i.e., the $\mathrm{L}^2$ convergence rate), and otherwise arguing as before, it follows that 
\begin{align}
&\mathbb{E}\left[|[\Xi_{t,i,j,k}^{(2)}]_{m,n}|\right] 
\leq K \gamma_t^{-1}  \int_1^t  \gamma_s^{\frac{5}{2}} \Phi_{s,t} \mathrm{d}s + K\Big(\rho(N) + \frac{1}{N^{\frac{1}{2(1+\alpha)}}}\Big)^{\frac{1}{2}} \gamma_t^{-1}  \int_1^t \gamma_s^2 \Phi_{s,t} \mathrm{d}s
\end{align}
where, as usual, we allow the value of the constants to increase from line to line. From Assumption \ref{assumption:learning-rate-v2} (our conditions on the learning rate), we have that $\smash{\int_1^t \gamma_s^{\frac{5}{2}} \Phi_{s,t}\mathrm{d}s = o(\gamma_t)}$ and $\smash{\int_0^t \gamma_s^2\Phi_{s,t}\mathrm{d}s = O(\gamma_t)}$. Thus, using also the fact that $N=N(t)\rightarrow\infty$ as $t\rightarrow\infty$, we have that 
\begin{equation}
    \mathbb{E}[\|\Xi_{t,i,j,k}^{(2)}\|_1]\rightarrow 0 \label{eq:delta-2-inf-n}
\end{equation}
as $t\rightarrow \infty$. Thus, substituting \eqref{eq:delta-1-inf-n} and \eqref{eq:delta-2-inf-n} into \eqref{eq:delta-triangle-inf-n}, we have shown that $\smash{\mathbb{E}[\|\Sigma_{t}^{i,j,k} - \bar{\Sigma}_{t}^{i,j,k}\|_1] \rightarrow 0}$ as $t\rightarrow \infty$. Using this result, the limit \eqref{eq:limiting-cov-inf-n} previously established for $\bar{\Sigma}_t^{i,j,k}$, and the triangle inequality, it follows that
\begin{equation}
    \mathbb{E}\left[\|\Sigma_t^{i,j,k} - \bar{\Sigma}^{i,j,k}\|\right]\leq
    \mathbb{E}\left[\|\Sigma_t^{i,j,k} - \bar{\Sigma}_t^{i,j,k}\|_1\right]
    + \mathbb{E}\left[\|\bar{\Sigma}_t^{i,j,k} - \bar{\Sigma}^{i,j,k}\|_1\right] \longrightarrow 0
\end{equation}
as $t\rightarrow\infty$. This implies, in particular, that $\Sigma_{t}^{i,j,k} \stackrel{\mathbb{P}}{\longrightarrow}\bar{\Sigma}^{i,j,k}$ as $t\rightarrow\infty$. That is, we have established that the quadratic variation of the random variable $\smash{\gamma_t^{-\frac{1}{2}}[\Pi_{t,i,j,k,N}^{(4)} + \Phi_{t,i,j,k,N}^{(2)}]}$ converges in probability to $\bar{\Sigma}^{i,j,k}$ in the limit as $t\rightarrow\infty$ and $N\rightarrow\infty$ (at the required rate). It now follows using standard results \citep[e.g.,][Section~1.2.2]{kutoyants2004statistical} that 
\begin{equation}
\gamma_t^{-\frac{1}{2}}\left[\Pi_{t,i,j,k,N}^{(4)} + \Phi_{t,i,j,k,N}^{(2)}\right]\stackrel{\mathrm{d}}{\longrightarrow}\mathcal{N}(0,\bar{\Sigma}^{i,j,k})
\end{equation}
as $t\rightarrow\infty$ and $N\rightarrow\infty$ (at the required rate). This result, combined with the decomposition in \eqref{main-bound-clt-inf-n}, the further decomposition in \eqref{eq:main-bound-clt-2-inf-n}, and the convergence of all other terms to zero in probability, implies that 
\begin{align}
    \gamma_t^{-\frac{1}{2}}(\theta_t^{i,j,k,N} - \theta_0)
    \stackrel{\mathrm{d}}{\longrightarrow}\mathcal{N}(0,\bar{\Sigma}^{i,j,k}).
\end{align}
\end{proof}


\section{Additional Results}
\label{app:additional-results}

\subsection{Additional Lemmas for Proposition \ref{prop:inf-n-convergence-1}}
\label{app:additional-results-prop-inf-n}

\begin{proposition}
\label{prop:ips-likelihood-t-limit-recall}
    Suppose that Assumption \ref{assumption:moments}, Assumption \ref{assumption:drift}, and Assumption \ref{assumption:drift-grad} (with $k=0$) hold. Then, as $t\rightarrow\infty$, it holds that
    \begin{align}
         \frac{1}{t}[\mathcal{L}_t^{i,N}(\theta) - \mathcal{L}_t^{i,N}(\theta_0)] &\xrightarrow[\mathrm{L}^1]{\mathrm{a.s.}} -\frac{1}{2}\int_{(\mathbb{R}^d)^{N}}  \left\|B^{i,N}(\theta,\boldsymbol{x}^{N}) - B^{i,N}(\theta_0,\boldsymbol{x}^{N})\right\|_{\sigma\sigma^{\top}}^2 \pi_{\theta_0}^N(\mathrm{d}\boldsymbol{x}^N)  \label{eq:l-i-asymptotic} \\
    \frac{1}{t}\left[\mathcal{L}_t^{i,j,k,N}(\theta) - \mathcal{L}_t^{i,j,k,N}(\theta_0)\right] &\xrightarrow[\mathrm{L}^1]{\mathrm{a.s.}} -\frac{1}{2}\int_{(\mathbb{R}^d)^{N}}  \big\langle b(\theta,{x}^{i,N},{x}^{j,N}) - B^{i,N}(\theta_0,\boldsymbol{x}^N), \nonumber \\[-2mm]
    &\hspace{28mm}b(\theta,{x}^{i,N},{x}^{k,N}) - B^{i,N}(\theta_0,\boldsymbol{x}^{N})\big\rangle_{\sigma\sigma^{\top}} \pi_{\theta_0}^{N}(\mathrm{d}\boldsymbol{x}^N). \label{eq:l-ijk-asymptotic}
\end{align}
\end{proposition}

\begin{proof}
    The result in \eqref{eq:l-i-asymptotic} was established in the proof of Proposition~\ref{prop:ips-likelihood-t-limit}. It remains to prove the result in \eqref{eq:l-ijk-asymptotic}. The proof will follow essentially the same template. Recalling the definition of $\mathcal{L}_t^{i,j,k,N}$ from \eqref{eq:IPS-incomplete-likelihood-2}, we have that 
    \begin{align}
        &\frac{1}{t}\left[\mathcal{L}_t^{i,j,k,N}(\theta) - \mathcal{L}_t^{i,j,k,N}(\theta_0)\right]  
= -\frac{1}{t}\int_{0}^t \ell^{i,j,k,N}(\theta,\boldsymbol{x}_s^N) \mathrm{d}s + \frac{1}{t}\int_0^t \big\langle  \Delta b(\theta,x_s^{i,N},x_s^{j,N}), \sigma \mathrm{d}w_s^{i,N}\big\rangle_{\sigma\sigma^{\top}}. \label{eq:IPS-incomplete-likelihood-2-recall-2} 
    \end{align}
    By Theorem~\ref{thm:invariant-distribution}, the IPS is ergodic, and admits a unique invariant measure $\smash{\pi_{\theta_0}^N \in \mathcal{P}((\mathbb{R}^d)^N)}$. By Theorem~\ref{thm:moment-bounds} (i.e., uniform-in-time moment bounds for the IPS) and Corollary~\ref{cor:empirical-L-lip} (i.e., the polynomial growth of $\ell^{i,j,k,N}$),  $\ell^{i,j,k,N}(\theta,\boldsymbol{x}^N) \in L^1(\pi_{\theta_0}^N)$. It follows by the ergodic theorem \citep[e.g.,][Theorem 4.2]{khasminskii2012stochastic} that 
    \begin{align}
        & \frac{1}{t}\Big[ \int_0^t \ell^{i,j,k,N}(\theta,\boldsymbol{x}_s^N)\mathrm{d}s \Big] 
        \stackrel{\mathrm{a.s.}}{\longrightarrow}
        \int_{(\mathbb{R}^d)^{N}}  \ell^{i,j,k,N}(\theta,\boldsymbol{x}^N) \pi_{\theta_0}^N(\mathrm{d}\boldsymbol{x}^N) \label{eq:ell-ergodic}
    \end{align}
    as $t\rightarrow\infty$. We now show that the second term in \eqref{eq:IPS-incomplete-likelihood-2-recall-2} converges a.s. to zero. To do so, we first define the continuous local martingales $\smash{M_{i,j,t}^N 
    =\int_0^t \langle  \Delta b^{i,j,N}(\theta,\boldsymbol{x}_s^N) , \sigma \mathrm{d}w_s^{i,N}\big\rangle_{\sigma\sigma^{\top}}}$. It is straightforward to compute the quadratic variation of these martingales as $\smash{\langle M_{i,j}^N\rangle_{t} = \int_0^t \|\Delta b^{i,j,N}(\theta,\boldsymbol{x}_s^N)\|_{\sigma\sigma^{\top}}^2 \mathrm{d}s}$. Reasoning similarly to before, the integrand belongs to $L^1(\pi_{\theta_0}^N)$. Thus, by the ergodic theorem, we have a.s. that $\smash{\frac{1}{t}\langle M_{i,j}^N\rangle_{t} \rightarrow \int_{(\mathbb{R}^d)^N} \|\Delta b^{i,j,N}(\theta,\boldsymbol{x}^N) \|_{\sigma\sigma^{\top}}^2 \pi_{\theta_0}^N(\mathrm{d}\boldsymbol{x}^N) <\infty}$ as $t\rightarrow\infty$. It follows, using the strong law of large numbers for continuous local martingales \citep[e.g.,][Theorem 1.3.4]{mao2008stochastic} that, as $t\rightarrow\infty$,  
    \begin{equation} 
    \frac{1}{t}M_{i,j,t}^N:=\frac{1}{t}\int_0^t \langle  \Delta b(\theta,x_s^{i,N},x_s^{j,N}), \sigma \mathrm{d}w_s^{i,N}\big\rangle_{\sigma\sigma^{\top}}\stackrel{\mathrm{a.s.}}{\longrightarrow} 0. \label{eq:martingale-finish} 
    \end{equation}
    Substituting \eqref{eq:ell-ergodic} and \eqref{eq:martingale-finish} into \eqref{eq:IPS-incomplete-likelihood-2-recall-2} establishes the a.s. statement in \eqref{eq:l-ijk-asymptotic}. We now show that the convergence in \eqref{eq:l-ijk-asymptotic} also holds in $\mathrm{L}^1$. Using Corollary~\ref{cor:empirical-L-lip} (i.e., polynomial growth of $\ell^{i,j,k,N}$) and Theorem~\ref{thm:moment-bounds} (i.e., uniform-in-time moment bounds for the IPS), for each $\delta>0$ there exists $K_{\delta}<\infty$ such that $\sup_{s\ge0}\mathbb{E}\big[|\ell^{i,j,k,N}(\theta,\boldsymbol x_s^N)|^{1+\delta}\big] <K_{\delta}$. Thus, according to Jensen's inequality, it holds, uniformly in $t\geq 1$, that
\begin{equation}
\mathbb{E}\Big[\Big|\frac1t\int_0^t \ell^{i,j,k,N}(\theta,\boldsymbol{x}_s^N)\,\mathrm ds\Big|^{1+\delta}\Big]
\le \frac1t\int_0^t \mathbb{E}\big[|\ell^{i,j,k,N}(\theta,\boldsymbol x_s^N)|^{1+\delta}\big]\,\mathrm ds
\le K_{\delta},
\end{equation}
It follows that the family of random variables $\{\frac{1}{t}\int_0^t \ell^{i,j,k,N}(\theta,\boldsymbol{x}_s^N)\,\mathrm ds\}_{t\geq 1}$ is uniformly integrable. This, combined with the a.s. convergence already established in \eqref{eq:ell-ergodic}, and Vitali's theorem, yields
\begin{align}
        & \frac{1}{t}\Big[ \int_0^t \ell^{i,j,k,N}(\theta,\boldsymbol{x}_s^N)\mathrm{d}s \Big] 
        \stackrel{\mathrm{L}^1}{\longrightarrow}
        \int_{(\mathbb{R}^d)^{N}}  \ell^{i,j,k,N}(\theta,\boldsymbol{x}^N) \pi_{\theta_0}^N(\mathrm{d}\boldsymbol{x}^N). \label{eq:ell-ergodic-l1}
    \end{align}
For the martingale term, using the BDG inequality, together with Jensen's inequality, we have that $\smash{\mathbb{E}[|\frac{1}{t}M_{i,j,t}^N|]\le \frac{K}{t}\mathbb{E}[\langle M_{i,j}^N\rangle_t^{1/2}] \le \frac{K}{t}(\mathbb{E}[\langle M_{i,j}^N\rangle_t])^{1/2} = \frac{K}{t}(\int_0^t \mathbb{E}[\|\Delta b^{i,j,N}(\theta,\boldsymbol{x}_s^N)\|^2]\,\mathrm ds)^{1/2}}$. By Corollary~\ref{cor:empirical-drift-lip} (i.e., polynomial growth of $b^{i,j,N}$) and Theorem~\ref{thm:moment-bounds} (i.e., uniform-in-time moment bounds for the IPS), there exist $K<\infty$ such that $\smash{\sup_{s\ge0}\mathbb{E}[\|\Delta b^{i,j,N}(\theta,\boldsymbol{x}_s^N)\|^2]<K}$.
Thus, combining with the previous display, it follows that $\smash{\mathbb{E}[|\frac{1}{t}M_{i,j,t}^N|]\le \frac{K}{t} t^{\frac{1}{2}} = Kt^{-\frac{1}{2}}}$, which implies that
\begin{equation}
\frac{1}{t}M_{i,j,t}^N\stackrel{\mathrm{L}^1}{\longrightarrow}0. \label{eq:mart-ijk-l1}
\end{equation}
Substituting \eqref{eq:ell-ergodic-l1} and \eqref{eq:mart-ijk-l1} into \eqref{eq:IPS-incomplete-likelihood-2-recall-2} establishes the $\mathrm{L}^1$ convergence statement in \eqref{eq:l-ijk-asymptotic}. 
\end{proof}

\begin{lemma}
\label{prop:asymptotic-partial-log-lik-grad-l1-convergence}
    Suppose that Assumption \ref{assumption:moments}, Assumption \ref{assumption:drift}, and Assumption \ref{assumption:drift-grad} hold. Then, for all $N\in\mathbb{N}$, and for all distinct $i,j,k\in[N]$, as $t\rightarrow\infty$, it holds that
\begin{align}
    -\frac{1}{t}\partial_{\theta}\mathcal{L}_t^{i,N}(\theta)&\xrightarrow[\mathrm{L}^1]{\mathrm{a.s.}}\partial_{\theta}\mathcal{L}^{i,N}(\theta) \label{eq:lt-i-n-limit} \\
    -\frac{1}{t}\partial_{\theta}\mathcal{L}_t^{i,j,k,N}(\theta)&\xrightarrow[\mathrm{L}^1]{\mathrm{a.s.}}\partial_{\theta}\mathcal{L}^{i,j,k,N}(\theta) \label{eq:lt-ijk-n-limit}
\end{align}
\end{lemma}

\begin{proof}
This proof will closely resemble the proof of Proposition \ref{prop:ips-likelihood-t-limit-recall}.  We begin by establishing \eqref{eq:lt-ijk-n-limit}. Working from the definition in \eqref{eq:IPS-incomplete-likelihood-2}, and simplifying, we have  
    \begin{align}
    -\frac{1}{t}\partial_{\theta}\mathcal{L}_t^{i,j,k,N}(\theta) 
    &:=  \frac{1}{t}\int_0^t \frac{1}{2}\left(h^{i,j,k,N}(\theta,\boldsymbol{x}_s^{N}) + h^{i,k,j,N}(\theta,\boldsymbol{x}_s^N)\right)\mathrm{d}s -\frac{1}{t}\int_0^t g^{i,j,N}(\theta,\boldsymbol{x}_s^N) (\sigma\sigma^{\top})^{-1} \sigma \mathrm{d}w_s^{i,N}.  \label{eq:grad-log-lik-ijk-1-v0}
\end{align}
By Theorem~\ref{thm:invariant-distribution}, the IPS is ergodic, and admits a unique invariant
measure $\smash{\pi_{\theta_0}^N}$. By
Corollary~\ref{cor:empirical-H-lip} (i.e., polynomial growth for $h^{i,j,k,N}$) and Theorem~\ref{thm:moment-bounds} (i.e., uniform-in-time moment bounds), the functions $\smash{\boldsymbol x^N\mapsto h^{i,j,k,N}(\theta,\boldsymbol x^N)}$ and $\smash{\boldsymbol x^N\mapsto h^{i,k,j,N}(\theta,\boldsymbol x^N)}$ belong to $\smash{L^1(\pi_{\theta_0}^N)}$. Hence, by the ergodic theorem, 
\begin{align}
\frac1t\int_0^t \frac{1}{2}\left(h^{i,j,k,N}(\theta,\boldsymbol x_s^N) + h^{i,k,j,N}(\theta,\boldsymbol x_s^N)\right)\,\mathrm ds
&\stackrel{a.s.}{\longrightarrow}
\int_{(\mathbb R^d)^N} \frac{1}{2}\left(h^{i,j,k,N}(\theta,\boldsymbol x^N)+h^{i,k,j,N}(\theta,\boldsymbol x^N)\right)\,\pi_{\theta_0}^N(\mathrm d\boldsymbol x^N) \notag
\\
&~~~=\int_{(\mathbb R^d)^N} h^{i,j,k,N}(\theta,\boldsymbol x^N)\,\pi_{\theta_0}^N(\mathrm d\boldsymbol x^N).  \label{eq:h-ergodic}
\end{align}
where the second line follows from the definition of $\smash{h^{i,j,k,N}}$ and the exchangeability of $\pi_{\theta_0}^N$. By Proposition~\ref{prop:asymptotic-partial-log-lik-grad}, the integral in  \eqref{eq:h-ergodic} is equal to $\smash{\partial_{\theta}\mathcal{L}^{i,j,k,N}(\theta)}$. Thus, we have established that the first term in \eqref{eq:grad-log-lik-ijk-1-v0} converges a.s. to the desired limit. We next show that the second term in \eqref{eq:grad-log-lik-ijk-1-v0} converges a.s. to zero. To do so, let us first define the continuous local martingale $\smash{M_{i,j,t}^N = \int_0^t g^{i,j,N}(\theta,\boldsymbol{x}_s^N) (\sigma\sigma^{\top})^{-1} \sigma \mathrm{d}w_s^{i,N}}$, with quadratic variation $\smash{\langle M_{i,j}^N\rangle_t= \int_0^t \|g^{i,j,N}(\theta,\boldsymbol x_s^N)\|_{\sigma\sigma^\top}^2\,\mathrm ds}$. By Corollary~\ref{cor:empirical-drift-grad-lip} (i.e., polynomial growth of $g^{i,j,N}$) and Theorem~\ref{thm:moment-bounds} (i.e., uniform-in-time moment bounds for the IPS), the integrand belongs to $L^1(\pi_{\theta_0}^N)$. Thus, once more applying the ergodic theorem, we have that $\smash{\frac1t\langle M_{i,j}^N\rangle_t
\stackrel{\mathrm{a.s.}}{\longrightarrow}
\int_{(\mathbb R^d)^N} \|g^{i,j,N}(\theta,\boldsymbol x^N)\|_{\sigma\sigma^\top}^2\,
\pi_{\theta_0}^N(\mathrm d\boldsymbol x^N)<\infty}$ as $t\rightarrow\infty$. It follows, via the strong law of large numbers for continuous local martingales \citep[e.g.,][Theorem 1.3.4]{mao2008stochastic}, that
\begin{equation}
\frac{M_{i,j,t}^{N}}{t}:=\frac{1}{t}\int_0^t g^{i,j,N}(\theta,\boldsymbol{x}_s^N) (\sigma\sigma^{\top})^{-1} \sigma \mathrm{d}w_s^{i,N}\stackrel{\mathrm{a.s.}}{\longrightarrow}0.
\label{eq:martingale-sl}
\end{equation}
Substituting \eqref{eq:h-ergodic} and \eqref{eq:martingale-sl} into \eqref{eq:grad-log-lik-ijk-1-v0}, establishes the a.s. convergence result in \eqref{eq:lt-ijk-n-limit}. We will now show that this convergence also holds in $\mathrm{L}^1$. First, using Jensen, Cauchy--Schwarz, and the elementary inequality $(a+b)^2\leq 2(a^2+b^2)$, we have $\mathbb{E}[\|\frac1t\int_0^t \frac{1}{2}(h^{i,j,k,N}(\theta,\boldsymbol x_s^N) + h^{i,k,j,N}(\theta,\boldsymbol{x}_s^N))\,\mathrm ds\|^2] \le \frac{1}{2t}\int_0^t [\mathbb{E}[\|h^{i,j,k,N}(\theta,\boldsymbol x_s^N)\|^2] + \mathbb{E}[\|h^{i,k,j,N}(\theta,\boldsymbol x_s^N)\|^2] ]\,\mathrm ds$. By Lemma~\ref{lem:Phi-hH-stability} (i.e., $h$ satisfies a polynomial growth property) and Theorem~\ref{thm:moment-bounds} (i.e., uniform-in-time moment bounds), we have $\sup_{s\ge0}\mathbb{E}[\|h^{i,j,k,N}(\theta,\boldsymbol x_s^N)\|^2]<\infty$ and $\sup_{s\ge0}\mathbb{E}[\|h^{i,k,j,N}(\theta,\boldsymbol x_s^N)\|^2]<\infty$. Thus, the family $\smash{\{\frac1t\int_0^t \frac{1}{2}\left(h^{i,j,k,N}(\theta,\boldsymbol x_s^N)+h^{i,k,j,N}(\theta,\boldsymbol x_s^N)\right)\mathrm ds\}_{t\ge1}}$ is bounded in $\mathrm{L}^2$, and so uniformly integrable. Together with the a.s. convergence from before, Vitali's theorem then implies
\begin{align}
\frac1t\int_0^t \frac{1}{2}\left(h^{i,j,k,N}(\theta,\boldsymbol x_s^N) + h^{i,k,j,N}(\theta,\boldsymbol x_s^N)\right)\,\mathrm ds
\stackrel{\mathrm{L}^1}{\longrightarrow}
\partial_{\theta}\mathcal{L}^{i,j,k,N}(\theta).
\label{eq:ergodic-h-l1}
\end{align}
We now consider the stochastic integral. By the  It\^{o} isometry,  $\smash{\mathbb E[\|\frac{1}{t}M_{i,j,t}^{N}\|^2]=\frac{1}{t^2}\,\mathbb E[\langle M_{i,j}^N\rangle_t]}$. Meanwhile, by Corollary~\ref{cor:empirical-drift-grad-lip} (i.e., $g^{i,j,N}$ satisfies a polynomial growth property) and Theorem~\ref{thm:moment-bounds} (i.e., uniform-in-time moment bounds for the IPS), there exists $K<\infty$ such that $\smash{\mathbb E[\langle M_{i,j}^N\rangle_t]=\mathbb{E}[\int_0^t \|g^{i,j,N}(\theta,\boldsymbol x_s^N)\|_{\sigma\sigma^\top}^2\,\mathrm ds]\leq Kt}$. It follows that $\smash{\mathbb E[\|\frac{1}{t}M_{i,j,t}^{N}\|^2]\leq \frac{K}{t}\rightarrow0}$, i.e., $\frac{1}{t}M_{i,j,t}^{N}(\theta)\to0$ in $\mathbb \mathrm{L}^2$ as $t\rightarrow\infty$. But this immediately implies that
\begin{equation}
    \frac{1}{t}M_{i,j,t}^{N}:=\frac{1}{t}\int_0^t g^{i,j,N}(\theta,\boldsymbol{x}_s^N) (\sigma\sigma^{\top})^{-1} \sigma \mathrm{d}w_s^{i,N}\stackrel{\mathrm{L}^1}{\longrightarrow}0 \label{eq:martingale-l1}
\end{equation}
Substituting \eqref{eq:ergodic-h-l1}  and \eqref{eq:martingale-l1} into \eqref{eq:grad-log-lik-ijk-1-v0}, we obtain the $\mathrm{L}^1$ convergence in \eqref{eq:lt-ijk-n-limit}. This completes the proof of \eqref{eq:lt-ijk-n-limit}. 

It remains to establish \eqref{eq:lt-i-n-limit}. From the definitions, we have that 
$\smash{\partial_{\theta} \mathcal L_t^{i,N}(\theta)
=
\frac{1}{N^2}\sum_{j,k=1}^N \partial_{\theta} \mathcal L_t^{i,j,k,N}(\theta)}$ and so $\smash{-\frac{1}{t}\partial_{\theta}\mathcal{L}_t^{i,N}(\theta) = -\frac{1}{t}[\frac{1}{N^2}\sum_{j,k=1}^N \partial_{\theta} \mathcal{L}_t^{i,j,k,N}(\theta)]= \frac{1}{N^2}\sum_{j,k=1}^N [-\frac{1}{t}\partial_{\theta} \mathcal{L}_t^{i,j,k,N}(\theta)]}$. Using this identity, the fact that the sum is finite (so we may interchange limits and sums), and the convergence results just established, we have as required that $\smash{-\frac{1}{t}\partial_{\theta}\mathcal{L}_t^{i,N}(\theta) =\frac{1}{N^2}\sum_{j,k=1}^N [-\frac{1}{t} \partial_{\theta} \mathcal{L}_t^{i,j,k,N}(\theta)] \rightarrow \frac{1}{N^2}\sum_{j,k=1}^N\partial_{\theta}\mathcal{L}^{i,j,k,N}(\theta) =\partial_{\theta}\mathcal{L}^{i,N}(\theta)}$ both a.s. and in $\mathrm{L}^1$.
\end{proof}

\begin{lemma}
\label{prop:asymptotic-partial-log-lik-grad-l1-convergence-mvsde}
    Suppose that Assumption \ref{assumption:moments}, Assumption \ref{assumption:drift}, and Assumption \ref{assumption:drift-grad} (with $k=0,1$) hold. Then, as $t\rightarrow\infty$, it holds that
\begin{align}
    -\frac{1}{t}\partial_{\theta}\mathcal{L}_t^{i}(\theta)&\xrightarrow[]{\mathrm{L}^1}\partial_{\theta}\mathcal{L}(\theta) \label{eq:lt-i-limit} \\
    -\frac{1}{t}\partial_{\theta}\mathcal{L}_t^{i,j,k}(\theta)&\xrightarrow[]{\mathrm{L}^1}\partial_{\theta}\mathcal{L}(\theta) \label{eq:lt-ijk-limit}
\end{align}
\end{lemma}

\begin{proof}
We begin by establishing \eqref{eq:lt-i-limit}. Recalling the definition of the function $\mathcal{L}_t^i$  in \eqref{eq:mckean-incomplete-likelihood-1}, differentiating, and then simplifying, we have that
\begin{align}
    -\frac{1}{t}\partial_{\theta}\mathcal{L}_t^{i}(\theta) 
    &= \frac{1}{t}\int_0^t H(\theta,x_s^{i},\mu_s^{i}) \mathrm{d}s   - \frac{1}{t}\int_0^t G(\theta,x_s^{i},\mu_s^{i})  \mathrm{d}w_s^{i} 
    \label{eq:l-i-decomp} 
\end{align}
From here, the proof is very similar to the proof of Proposition \ref{prop:mvsde-likelihood-t-limit}. We first show that, as $t\rightarrow\infty$, $\smash{\mathbb{E}[|\frac{1}{t}\int_0^t  H(\theta,x_s^{i},\mu_s^{i})\mathrm{d}s - \int_{\mathbb{R}^d} H(\theta,x,\pi_{\theta_0})\pi_{\theta_0}(\mathrm{d}x)|] \rightarrow 0 }$.
To establish this limit, we use the triangle inequality to write 
    \begin{align}
        \mathbb{E}\Big[\Big|\frac{1}{t}\int_0^t H(\theta,x_s^{i},\mu_s^{i})\mathrm{d}s - \int_{\mathbb{R}^d} H(\theta,x,\pi_{\theta_0})\pi_{\theta_0}(\mathrm{d}x)\Big|\Big] \leq \mathbb{E}[H_{t,i}^{(1)}] + \mathbb{E}[H_{t,i}^{(2)}] + \mathbb{E}[H_{t,i}^{(3)}] \label{eq:H-start-decomp}
    \end{align}
    where $\smash{H_{t,i}^{(1)}=\frac{1}{t}\int_0^t |H(\theta,x_s^{i},\mu_s^{i}) -H(\theta,x_s^{i},\pi_{\theta_0})|\mathrm{d}s}$, $\smash{H_{t,i}^{(2)}=\frac{1}{t} \int_0^t|H(\theta,x_s^{i},\pi_{\theta_0})- H(\theta,\bar{x}_s^{i},\pi_{\theta_0})|\mathrm{d}s}$, and $\smash{H_{t,i}^{(3)}=}$ $\smash{|\frac{1}{t}\int_0^t H(\theta,\bar{x}_s^{i},\pi_{\theta_0}) \mathrm{d}s - \int_{\mathbb{R}^d} H(\theta,{x},\pi_{\theta_0})\pi_{\theta_0}(\mathrm{d}x)|}$.
    We can bound these terms by arguing exactly as in the proof of Proposition~\ref{prop:mvsde-likelihood-t-limit}. In particular, following the steps in \eqref{eq:j-exp-bound} - \eqref{eq:h3-final-limit}, we have $\smash{\mathbb E[H_{t,i}^{(1)}]+\mathbb E[H_{t,i}^{(2)}]+\mathbb E[H_{t,i}^{(3)}]\stackrel{t\rightarrow\infty}{\longrightarrow} 0}$.
This, together with \eqref{eq:H-start-decomp}, establishes that 
 \begin{align}
        \frac{1}{t}\int_0^t H(\theta,x_s^{i},\mu_s^{i})\mathrm{d}s &\stackrel{\mathrm{L}^1}{\longrightarrow} \int_{\mathbb{R}^d}H(\theta,x,\pi_{\theta_0})\pi_{\theta_0}(\mathrm{d}x) 
        \label{eq:h-limit-1} 
    \end{align}
It remains to establish $\mathrm{L}^1$ convergence of the second term in \eqref{eq:l-i-decomp} to zero. Let $\smash{M_{i,t} := \int_0^t G(\theta,x_s^{i},\mu_s^{i})\,dw_s^{i}}$. Using It\^{o}'s isometry, and Lemma~\ref{lem:weighted-lip-g}, the expectation of the quadratic variation of these martingales is finite for all $t\geq 0$. Thus, as in the proof of Proposition~\ref{prop:mvsde-likelihood-t-limit}, we have as required that
\begin{align}
    \frac{1}{t}M_{i,t} = \frac{1}{t}\int_0^t G(\theta,x_s^{i},\mu_s^{i})\,dw_s^{i} &\stackrel{\mathrm{L}^1}{\longrightarrow} 0 
\end{align}
This completes the proof of \eqref{eq:lt-i-limit}. We now turn our attention to \eqref{eq:lt-ijk-limit}. Similar to before, recalling the definition of the function $\mathcal{L}_t^{i,j,k}$ in \eqref{eq:mckean-incomplete-likelihood-2}, differentiating, and then simplifying, we have that
\begin{align}
    -\frac{1}{t}\partial_{\theta}\mathcal{L}_t^{i,j,k}(\theta) 
    &=\frac{1}{t}\int_0^t h^{\mathrm{sym}}(\theta,x_s^{i},x_s^{j},x_s^{k}, \mu_s^{i}) \mathrm{d}s  - \frac{1}{t}\int_0^t g(\theta,x_s^{i},x_s^{j}) \mathrm{d}w_s^{i} 
    \label{eq:l-ijk-decomp}
\end{align}
where we have defined $h^{\mathrm{sym}}(\theta,x,y,z,\mu) = \frac{1}{2}(h(\theta,x,y,z,\mu) + h(\theta,x,z,y,\mu))$. Arguing exactly as above, we can show that
\begin{align}
            \frac{1}{t}\int_0^t h^{\mathrm{sym}}(\theta,x_s^{i},x_s^{j},x_s^{k},\mu_s^{i})\mathrm{d}s &\stackrel{\mathrm{L}^1}{\longrightarrow} \int_{(\mathbb{R}^d)^3}h^{\mathrm{sym}}(\theta,x,y,z,\pi_{\theta_0})\pi_{\theta_0}^{\otimes 3}(\mathrm{d}x,\mathrm{d}y,\mathrm{d}z) \\
            &~~~= \int_{(\mathbb{R}^d)^3}h(\theta,x,y,z,\pi_{\theta_0})\pi_{\theta_0}^{\otimes 3}(\mathrm{d}x,\mathrm{d}y,\mathrm{d}z)
            \label{eq:h-limit-2} 
\end{align}
    where the second equality follows from the fact that, by symmetry, $ \smash{\int_{\mathbb{R}^d} h(\theta,{x},y,z,\pi_{\theta_0})\pi_{\theta_0}^{\otimes 3}(\mathrm{d}x,\mathrm{d}y,\mathrm{d}z)=}$ $\smash{\int_{\mathbb{R}^d} h(\theta,{x},z,y,\pi_{\theta_0})\pi_{\theta_0}^{\otimes 3}(\mathrm{d}x,\mathrm{d}y,\mathrm{d}z)}$. Similarly, by the same arguments as before, we have that
\begin{align}
    \frac{1}{t}M_{i,j,k,t} = \frac{1}{t}\int_0^t g(\theta,x_s^{i},x_s^{j})\,dw_s^{i}&\stackrel{\mathrm{L}^1}{\longrightarrow} 0. \label{eq:m-ijk-limit}
\end{align}
Finally, substituting \eqref{eq:h-limit-2} and \eqref{eq:m-ijk-limit} into \eqref{eq:l-ijk-decomp} yields the claimed result in \eqref{eq:lt-ijk-limit}. 
\end{proof}

\begin{lemma}
\label{lemma:main-theorem-1-2-lemma-1}
    Suppose that Assumption \ref{assumption:moments}, Assumption \ref{assumption:drift}, and Assumption \ref{assumption:drift-grad} hold. Then, for all $\theta\in\Theta$, for all $t>0$, for all $N\in\mathbb{N}$, for all distinct $i,j,k\in\{1,\dots,N\}$, and for all $0<\varepsilon<1$, there exists finite constants $K>0$ and $K'=K'(\varepsilon)>0$ such that 
    \begin{align} 
        \mathbb{E}\Big[\Big\| \frac{1}{t}\partial_{\theta}\mathcal{L}_t^{i,N}(\theta) - \frac{1}{t}\partial_{\theta}\mathcal{L}_t^{[i,N]}(\theta)\Big\|\Big] &\leq \frac{K}{N^{\frac{1}{2(1+\alpha)}}} + \frac{K'}{N^{\frac{1-\varepsilon}{2(1+\alpha)}}}\frac{1}{\sqrt{t}} \label{eq:lt-quant-v1} \\
        \mathbb{E}\Big[\Big\| \frac{1}{t}\partial_{\theta}\mathcal{L}_t^{i,j,k,N}(\theta) - \frac{1}{t}\partial_{\theta}\mathcal{L}_t^{[i,j,k,N]}(\theta)\Big\|\Big] &\leq \frac{K}{N^{\frac{1}{2(1+\alpha)}}} + \frac{K'}{N^{\frac{1-\varepsilon}{2(1+\alpha)}}}\frac{1}{\sqrt{t}}. \label{eq:lt-quant-v2}
    \end{align}
\end{lemma}

\begin{proof}
    We will first prove \eqref{eq:lt-quant-v2}. We start by recalling the relevant definitions, viz
\begin{align}
\frac{1}{t}\partial_{\theta}{\mathcal{L}}^{i,j,k,N}_t(\theta)&= 
\frac{1}{t}\int_0^t h(\theta,x_s^{i,N},x_s^{j,N},x_s^{k,N},\mu_s^N)\mathrm{d}s
+\frac{1}{t}\int_0^t \left\langle g(\theta,x_s^{i,N},x_s^{j,N}),\mathrm{d}w_s^{i,N}\right\rangle
\label{eq:l-i-j-k-n-recall} \\
\frac{1}{t}\partial_{\theta}{\mathcal{L}}^{[i,j,k,N]}_t(\theta)&= 
\frac{1}{t}\int_0^t h(\theta,x_s^{i},x_s^{j},x_s^{k},\mu_s^{[N]})\mathrm{d}s
+\frac{1}{t}\int_0^t \left\langle  g(\theta,x_s^{i},x_s^{j}),\mathrm{d}w_s^{i}\right\rangle
\label{eq:l-i-j-k-n-indep-recall}
\end{align}
We first bound the difference in the ``deterministic'' integrals in \eqref{eq:l-i-j-k-n-recall} - \eqref{eq:l-i-j-k-n-indep-recall}. Using Lemma \ref{lem:Phi-hH-stability} (i.e., $h$ is locally Lipschitz with polynomial growth),  the Cauchy--Schwarz inequality, Theorem~\ref{thm:moment-bounds} (i.e., uniform-in-time moment bounds),
and then Theorem~\ref{thm:poc} (i.e., uniform-in-time propagation of chaos), we have that 

\begin{align}
    &\sup_{s\geq 0}\mathbb{E}\big[\big\|h(\theta,x_s^{i,N},x_s^{j,N},x_s^{k,N},\mu_s^N) - h(\theta,x_s^{i},x_s^{j},x_s^{k},\mu_s^{[N]})\big\|\big] \\
    &\leq K \sup_{s\geq 0} \big(\textstyle\sum_{a\in\{i,j,k\}}\mathbb{E}\big[\|x_s^{a,N} - x_s^{a}\|^2\big] + \frac{1}{N}\sum_{a=1}^N \mathbb{E}\big[\|x_s^{a,N} - x_s^{a}\|^2\big]\big)^{\frac{1}{2}} 
    \displaystyle \leq \frac{K}{N^{\frac{1}{2(1+\alpha)}}} 
\end{align}
for some constant $K<\infty$, which has been allowed to increase from line to line. It follows, using also the triangle inequality (in integral form), that for all $\theta\in\Theta$ and for all $t\geq 0$, 
\begin{align}
    \mathbb{E}\big[\big\|\frac{1}{t}\int_0^t \big( h(\theta,x_s^{i,N},x_s^{j,N},x_s^{k,N},\mu_s^N) - h(\theta,x_s^{i},x_s^{j},x_s^{k},\mu_s^{[N]})\big)\mathrm{d}s\big\|\big] 
    &
    \leq \frac{K}{N^{\frac{1}{2(1+\alpha)}}}. \label{eq:i1-final-bound}
\end{align}
We now seek an $\mathrm{L}^1$ bound for the difference between the two stochastic integrals.  By Lemma \ref{lem:weighted-lip-g} (i.e., $g$ is   
locally Lipschitz with polynomial growth), there exists a  constant $K_1<\infty$ such that, for all $s\geq 0$, 
\begin{align}
    \|g(\theta,x_s^{i,N},x_s^{j,N}) - g(\theta,x_s^{i},x_s^{j})\|^2 &\leq K_1\Big(\|x_s^{i,N} - x_s^{i}\|^2 + \|x_s^{j,N} - x_s^{j}\|^2\Big) \label{eq:deltag-square-prelocal} \\
    &~~~~~\times\big(1 + \|x_s^{i,N}\|^{q} + \|x_s^{j,N}\|^{q} + \|x_s^{i}\|^{q} + \|x_s^{j}\|^{q} \big)^2. \nonumber
\end{align}
For $M\ge 1$, define the event $A_s^M$ according to $\smash{A_s^M:=\{\|x_s^{i,N}\|\vee \|x_s^{i}\|\vee \|x_s^{j,N}\|\vee \|x_s^{j}\|\le M\}}$. On this event, we have $1 + \|x_s^{i,N}\|^{q} + \|x_s^{j,N}\|^{q} + \|x_s^{i}\|^{q} + \|x_s^{j}\|^{q}\le 1+4M^q\le K_2(1+M^q)$. Substituting this into \eqref{eq:deltag-square-prelocal}, taking expectations, and using Theorem~\ref{thm:poc} (i.e., uniform-in-time propagation-of-chaos), it follows that
\begin{align}
 \mathbb{E}\big[\|g(\theta,x_s^{i,N},x_s^{j,N}) - g(\theta,x_s^{i},x_s^{j})\|^2\boldsymbol{1}_{A_s^M}\big]
\le \frac{K_3(1+M^{2q})}{N^{\frac{1}{1+\alpha}}}, \label{eq:deltag-local-term}
\end{align}
Next, using Lemma~\ref{lem:weighted-lip-g} (i.e., the polynomial growth of $g$), and Theorem~\ref{thm:moment-bounds} (i.e., uniform-in-time moment bounds for the IPS and the MVSDE), there exists $K_4<\infty$ such that
$\smash{\mathbb{E}[\|g(\theta,x_s^{i,N},x_s^{j,N}) - g(\theta,x_s^{i},x_s^{j})\|^{2r}]\le K_4}$ for all $s\geq 0$. Using H\"older's inequality and this bound, it then follows that, for all $s\geq 0$, 
\begin{align}
\mathbb{E}\big[\|g(\theta,x_s^{i,N},x_s^{j,N}) - g(\theta,x_s^{i},x_s^{j})\|^2\boldsymbol{1}_{(A_s^M)^c}\big]
&\le \mathbb{E}\big[\|g(\theta,x_s^{i,N},x_s^{j,N}) - g(\theta,x_s^{i},x_s^{j})\|^{2r}\big]^{\frac1r}\,
\mathbb{P}((A_s^M)^c)^{1-\frac1r}
\\
&\le K_4\,\mathbb{P}((A_s^M)^c)^{1-\frac1r}. \label{eq:g-delta-new-bound}
\end{align}
Meanwhile, by a union bound, Markov's inequality, and Theorem~\ref{thm:moment-bounds} (i.e., uniform-in-time moment bounds), we have that for any $\ell>0$, there exists $K_5=K_5(\ell)<\infty$ such that $\mathbb{P}((A_s^M)^c)
\le 4\,\sup_{u\in\{x_s^{i,N},x_s^{i},x_s^{j,N},x_s^{j}\}}\mathbb{P}(\|u\|>M)
\le \frac{K_5}{M^\ell}$. Combining this with the bound in \eqref{eq:g-delta-new-bound}, it follows that for any $\gamma>0$, there exists $\ell=\ell(\gamma,r)$ and $K_6=K_6(\gamma)$ such that
\begin{equation}
\label{eq:deltag-tail-term}
\mathbb{E}\big[\|g(\theta,x_s^{i,N},x_s^{j,N}) - g(\theta,x_s^{i},x_s^{j})\|^2\boldsymbol{1}_{(A_s^M)^c}\big]\le \frac{K_6}{M^\gamma}.
\end{equation}
Finally, combining \eqref{eq:deltag-local-term} and \eqref{eq:deltag-tail-term}, it follows that for all $M\ge 1$, the following upper bound holds
\begin{equation}
\label{eq:deltag-splitting}
\mathbb{E}\big[\|g(\theta,x_s^{i,N},x_s^{j,N}) - g(\theta,x_s^{i},x_s^{j})\|^2\big]
\le \frac{K_3(1+M^{2q})}{N^{\frac{1}{1+\alpha}}}+\frac{K_6}{M^\gamma}.
\end{equation}
Fix $\varepsilon>0$. Let $M=M(N):=N^{\eta}$, with $\eta:=\frac{\varepsilon}{2q(1+\alpha)}$.\footnote{We note that if the growth exponent $q=0$ (i.e., the relevant function is bounded), then we may take $M=\infty$, and skip the localisation step entirely.} Then the first term on the right-hand side of \eqref{eq:deltag-splitting} is bounded by
\begin{equation}
\frac{K_3(1+M^{2q})}{N^{\frac{1}{1+\alpha}}}
\le \frac{K_3}{N^{\frac{1}{1+\alpha}}}+\frac{K_3}{N^{\frac{1-\varepsilon}{1+\alpha}}}
\le \frac{K_3}{N^{\frac{1-\varepsilon}{1+\alpha}}}. \label{eq:k3-m-bound}
\end{equation}
Meanwhile, by choosing $\gamma>0$ sufficiently large, we can ensure that $M^{-\gamma}=N^{-\eta\gamma}\le N^{-\frac{1-\varepsilon}{1+\alpha}}$, so the tail term in \eqref{eq:deltag-splitting} is of the same (or smaller) order. Thus,  there exists $K_7 = K_{7}(\varepsilon)<\infty$ such that
\begin{equation}
\label{eq:deltag-final}
\mathbb{E}\big[\|g(\theta,x_s^{i,N},x_s^{j,N}) - g(\theta,x_s^{i},x_s^{j})\|^2\big]\le \frac{K_7}{N^{\frac{1-\varepsilon}{1+\alpha}}}.
\end{equation}
It follows, using also the It\^{o} isometry and Fubini's Theorem, that for all $\theta\in\Theta$, and for all $t\geq 0$, it holds that 
\begin{align}
 &\mathbb{E}\Big[\Big\|\frac{1}{t}\int_0^t \left\langle g(\theta,x_s^{i,N},x_s^{j,N}),\mathrm{d}w_s^{i,N}\right\rangle - \frac{1}{t}\int_0^t \left\langle g(\theta,x_s^{i},x_s^{j}),\mathrm{d}w_s^{i}\right\rangle\Big\|^2\Big] \\
 &= \frac{1}{t^2} \int_0^t \mathbb{E}\left[ \left\| g(\theta,x_s^{i,N},x_s^{j,N}) -  g(\theta,x_s^{i},x_s^{j})\right\|^2\right] \mathrm{d}s 
  \leq \frac{K_{7}}{N^{\frac{1-\varepsilon}{(1+\alpha)}}t}. \label{eQ32-recall}
 \end{align}
Thus, applying the Cauchy--Schwarz inequality one final time, and defining a new constant $K' = K_7^{\frac{1}{2}}$, we have that
 \begin{equation}
\mathbb{E}\Big[\Big\|\frac{1}{t}\int_0^t \left\langle g(\theta,x_s^{i,N},x_s^{j,N}),\mathrm{d}w_s^{i,N}\right\rangle - \frac{1}{t}\int_0^t \left\langle g(\theta,x_s^{i},x_s^{j}),\mathrm{d}w_s^{i}\right\rangle\Big\|\Big] \leq \left[\frac{K_7}{N^{\frac{1-\varepsilon}{1+\alpha}}t} \right]^{\frac{1}{2}} \leq \frac{K'}{N^{\frac{1-\varepsilon}{2(1+\alpha)}}}\frac{1}{\sqrt{t}}. \label{eQ346}
\end{equation}
Combining the bounds in \eqref{eq:i1-final-bound} and \eqref{eQ346}, and using the triangle inequality one final time, yields the bound in \eqref{eq:lt-quant-v2}.
The proof of \eqref{eq:lt-quant-v1} is now straightforward. In particular, working from the definitions, and using the result just proved, we have
\begin{align}
\label{eq:lt-i-n-start}
\mathbb{E}\Big[\big\| \frac{1}{t}\partial_{\theta}\mathcal{L}_t^{i,N}(\theta) - \frac{1}{t}\partial_{\theta}\mathcal{L}_t^{[i,N]}(\theta)\big\|\Big] &=\mathbb{E}\Big[\Big\| \frac{1}{N^2}\sum_{j,k=1}^N \big[\frac{1}{t} \partial_{\theta}\mathcal{L}_t^{i,j,k,N}(\theta) - \frac{1}{t}\partial_{\theta}\mathcal{L}_t^{[i,j,k,N]}(\theta) \big]\Big\|\Big] \\
&\leq \frac{1}{N^2}\sum_{j,k=1}^N \mathbb{E}\Big[\Big\| \frac{1}{t} \partial_{\theta}\mathcal{L}_t^{i,j,k,N}(\theta) - \frac{1}{t}\partial_{\theta}\mathcal{L}_t^{[i,j,k,N]}(\theta) \Big\|\Big] \\
&\leq \frac{1}{N^2}\sum_{j,k=1}^N \Big[\frac{K}{N^{\frac{1}{2(1+\alpha)}}} + \frac{K'}{N^{\frac{1-\varepsilon}{2(1+\alpha)}}}\frac{1}{\sqrt{t}}\Big]  \leq \frac{K}{N^{\frac{1}{2(1+\alpha)}}} + \frac{K'}{N^{\frac{1-\varepsilon}{2(1+\alpha)}}}\frac{1}{\sqrt{t}}.
\label{eq:lt-i-n-end}
\end{align}
\end{proof}

\begin{lemma}
\label{lemma:main-theorem-1-2-lemma-2}
    Suppose that Assumption \ref{assumption:moments}, Assumption \ref{assumption:drift}, and Assumption \ref{assumption:drift-grad} hold. Then, for all $t>0$, for all $N\in\mathbb{N}$, and for all distinct $i,j,k\in\{1,\dots,N\}$, there exists a finite constant $K>0$ such that
    \begin{align} 
        \mathbb{E}\left[\left\| \frac{1}{t}\partial_{\theta}\mathcal{L}_t^{[i,N]}(\theta) - \frac{1}{t}\partial_{\theta}{\mathcal{L}}_t^{i}(\theta)\right\|\right] &\leq K\rho(N)\left(1+\frac{1}{\sqrt{t}}\right) \label{eq:lt-quant-v1-v2} \\
        \mathbb{E}\left[\left\| \frac{1}{t}\partial_{\theta}\mathcal{L}_t^{[i,j,k,N]}(\theta) - \frac{1}{t}\partial_{\theta}{\mathcal{L}}_t^{i,j,k}(\theta)\right\|\right] &\leq K\rho(N). \label{eq:lt-quant-v2-v2}
    \end{align}
    where $\rho:\mathbb{N}\rightarrow\mathbb{R}_{+}$ is the function defined in \eqref{eq:rho}.
\end{lemma}

\begin{proof}
    The proof follows the same template as the previous one, with some minor modifications. In this case, we begin by establishing \eqref{eq:lt-quant-v1-v2}. Recall that 
\begin{align}
\frac{1}{t}\partial_{\theta}{\mathcal{L}}^{[i,N]}_t(\theta)&= 
\frac{1}{t}\int_0^t H(\theta,x_s^{i},\mu_s^{[N]})\mathrm{d}s
+ 
\frac{1}{t}\int_0^t \big\langle G(\theta,x_s^{i},\mu_s^{[N]}),\mathrm{d}w_s^{i}\big\rangle
\label{eq:l-i-n-recall} \\
\frac{1}{t}\partial_{\theta}{\mathcal{L}}^{i}_t(\theta)&= 
\frac{1}{t}\int_0^t  H(\theta,x_s^{i},\mu_s)\mathrm{d}s
+ 
\frac{1}{t}\int_0^t \big\langle  G(\theta,x_s^{i},\mu_s),\mathrm{d}w_s^{i}\big\rangle
\label{eq:l-i-recall}
\end{align}
We first bound the difference in the ``deterministic'' integrals in \eqref{eq:l-i-n-recall} - \eqref{eq:l-i-recall}. By Lemma \ref{lem:Phi-hH-stability} (i.e., the function $H$ is locally Lipschitz with polynomial growth),  the Cauchy--Schwarz inequality, Theorem~\ref{thm:moment-bounds} (i.e., uniform-in-time moment bounds), and Theorem~\ref{thm:poc} (i.e., uniform-in-time propagation of chaos), we have 
$\smash{\sup_{s\geq 0} \mathbb{E}\big[\big\|H(\theta,x_s^{i},\mu_s^{[N]}) - H(\theta,x_s^{i},\mu_s)\big\|\big] \leq K\sup_{s\geq 0} \big(\mathbb{E}\big[\mathsf{W}_2^2(\mu_s^{[N]},\mu_s)\big]\big)^{\frac{1}{2}} \leq K\rho(N)}$ for some constant $K<\infty$ which is allowed to increase between displays, where $\rho:\mathbb{N}\rightarrow\mathbb{R}_{+}$ is the function defined in \eqref{eq:rho}. It follows, using also the triangle inequality (in integral form), that for all $\theta\in\Theta$ and for all $t\geq 0$, 
\begin{align}
    \mathbb{E}\Big[\big\|\frac{1}{t}\int_0^t \big( H(\theta,x_s^{i},\mu_s^{[N]}) - H(\theta,x_s^{i},\mu_s)\big)\mathrm{d}s\big\|\Big] 
    \leq K \rho(N). \label{eq:i1-final-bound-v2}
\end{align}
We now consider the difference in the stochastic integrals. Similar to above, by Lemma~\ref{lem:weighted-lip-g} (i.e., $g$ is locally Lipschitz with polynomial growth), the Cauchy--Schwarz inequality, Theorem~\ref{thm:moment-bounds} (i.e., uniform-in-time moment bounds), and Theorem~1 in \cite{fournier2015rate} (i.e., bound on the $\mathsf{W}_2$ distance to the empirical measure),\footnote{To be precise, Theorem 1 in \cite{fournier2015rate} provides a bound of the form $\smash{\sup_{s\geq 0} \mathbb{E}[\mathsf{W}_2^2(\mu_s^{[N]},\mu_s)] \leq C\rho^2(N)}$. This, combined with the concentration inequality in Theorem 2 in \cite{fournier2015rate}, which provides a bound on $\sup_{s\geq 0} \mathbb{P}(\mathsf{W}_2^2(\mu_s^{[N]},\mu_s)>x)$ for any $x\in(0,\infty)$, yields the stated bound for $\sup_{s\geq 0}\mathbb{E}[\mathsf{W}_2^4(\mu_s^{[N]},\mu_s)]$.} we have that 
$\smash{\sup_{s\geq 0} \mathbb{E}\big[\big\|G(\theta,x_s^{i},\mu_s^{[N]}) - G(\theta,x_s^{i},\mu_s)\big\|^2\big] \leq K \sup_{s\geq 0}\big(\mathbb{E}\big[\mathsf{W}_2^4(\mu_s^{[N]},\mu_s)\big]\big)^{\frac{1}{2}}  \leq K \rho^2(N)}$. We thus have, using also the It\^{o} isometry and Fubini's Theorem, that for all $\theta\in\Theta$, and for all $t\geq 0$, 
\begin{align}
 &\mathbb{E}\Big[\Big\|\frac{1}{t}\int_0^t \big\langle G(\theta,x_s^{i},\mu_s^{[N]}) -  G(\theta,x_s^{i},\mu_s),\mathrm{d}w_s^{i}\big\rangle\Big\|^2\Big] 
  \leq \frac{K}{t}\rho^2(N). \label{eQ32}
 \end{align}
Applying the Cauchy--Schwarz inequality one final time, and once more allowing $K$ to increase between displays, we have that
 \begin{equation}
\mathbb{E}\Big[\Big\|\frac{1}{t}\int_0^t \big\langle G(\theta,x_s^{i},\mu_s^{[N]}) -  G(\theta,x_s^{i},\mu_s),\mathrm{d}w_s^{i}\big\rangle\Big\|\Big]\leq \left[\frac{K}{t} \rho^2(N) \right]^{\frac{1}{2}} \leq \frac{K}{\sqrt{t}} \rho(N). \label{eQ346-alt-alt}
\end{equation}
Combining inequalities \eqref{eq:i1-final-bound-v2} and \eqref{eQ346-alt-alt}, and making use of the triangle inequality, completes the proof of \eqref{eq:lt-quant-v1-v2}. It remains to establish \eqref{eq:lt-quant-v2-v2}. Recall that
\begin{align}
\frac{1}{t}\partial_{\theta}{\mathcal{L}}^{[i,j,k,N]}_t(\theta)&= 
\frac{1}{t}\int_0^t h(\theta,x_s^{i},x_s^{j},x_s^{k},\mu_s^{[N]})\mathrm{d}s
+\frac{1}{t}\int_0^t \left\langle g(\theta,x_s^{i},x_s^{j}),\mathrm{d}w_s^{i}\right\rangle
\\
\frac{1}{t}\partial_{\theta}{\mathcal{L}}^{i,j,k}_t(\theta)&= 
\frac{1}{t}\int_0^t  h(\theta,x_s^{i},x_s^{j},x_s^{k},\mu_s)\mathrm{d}s
+\frac{1}{t}\int_0^t \left\langle  g(\theta,x_s^{i},x_s^{j}),\mathrm{d}w_s^{i}\right\rangle
\end{align}
First consider the difference in the ``deterministic integrals''. Using Lemma \ref{lem:Phi-hH-stability} (i.e., $h$ is locally Lipschitz with polynomial growth),  the Cauchy--Schwarz inequality, Theorem~\ref{thm:moment-bounds} (i.e., uniform-in-time moment bounds), and then Theorem~1 in \cite{fournier2015rate} (i.e., bounds on the $\mathsf{W}_2$ distance to the empirical measure), we have that
\begin{align}
    \sup_{s\geq 0} \mathbb{E}\big[\big\|h(\theta,x_s^{i},x_s^{j},x_s^{k},\mu_s^{[N]}) - h(\theta,x_s^{i},x_s^{j},x_s^{k},\mu_s)\big\|\big] 
    &\leq K \sup_{s\geq 0} \big(\mathbb{E}\big[\mathsf{W}_2^2(\mu_s^{[N]},\mu_s)\big]\big)^{\frac{1}{2}}  \leq K \rho(N),
\end{align}
for some constant $K<\infty$ which has been allowed to increase between displays, where $\rho:\mathbb{N}\rightarrow\mathbb{R}_{+}$ is the function defined in \eqref{eq:rho}. It follows that
\begin{align}
    \mathbb{E}\Big[\big\|\frac{1}{t}\int_0^t \big( h(\theta,x_s^{i},x_s^{j},x_s^{k},\mu_s^{[N]}) - h(\theta,x_s^{i},x_s^{j},x_s^{k},\mu_s)\big)\mathrm{d}s\big\|\Big] 
    \leq K\rho(N). \label{eq:i1-final-bound-recall}
\end{align}
This, combined with the fact that the difference in the stochastic integrals is null, completes the proof of \eqref{eq:lt-quant-v2-v2}. 
\end{proof}

\subsection{Additional Lemmas for Proposition \ref{prop:inf-t-convergence-1}}

\begin{lemma}
\label{lemma:main-theorem-lemma-2-a}
Suppose that Assumption~\ref{assumption:moments}, Assumption~\ref{assumption:drift}, and Assumption~\ref{assumption:drift-grad} hold. Then, for each $m\in\{0,1,2,3\}$, there exists a constant $C_m<\infty$ such that, for all $N\in\mathbb{N}$, for all distinct  $i,j,k\in[N]$, and for all $\theta\in\Theta$, $ \|\partial_{\theta}^m\mathcal{L}^{i,N}(\theta)\| \leq C_m$ and $ \|\partial_{\theta}^m\mathcal{L}^{i,j,k,N}(\theta) \|\leq C_m$.
\end{lemma}

\begin{proof}
We prove the result for $\mathcal{L}^{i,j,k,N}$, with the result for $\mathcal{L}^{i,N}$ proved similarly. Fix $m\in\{0,1,2,3\}$. Recall that $\smash{\mathcal{L}^{i,j,k,N}(\theta)=\int_{(\mathbb{R}^d)^N} \ell^{i,j,k,N}(\theta,\boldsymbol{x}^N) \pi_{\theta_0}^N(\mathrm{d}\boldsymbol{x}^N)}$. Due to Theorem~\ref{thm:moment-bounds} (i.e., uniform-in-time moment bounds for the IPS) and Theorem~\ref{thm:invariant-distribution} (i.e., ergodicity of the IPS), it holds that $\int \|x^{a,N}\|^q\pi_{\theta_0}^N(\mathrm{d}\boldsymbol{x}^N)<K_q$ for all $q\geq 1$, for all $N\in\mathbb{N}$, and for all $a\in[N]$. In addition, by Corollary~\ref{cor:empirical-L-lip} and Corollary~\ref{cor:empirical-H-lip}, there exist $K_m>0$ and $q_m\geq 1$, independent of $\theta\in\Theta$, such that $\|\partial_{\theta}^m \ell^{i,j,k,N}(\theta,\boldsymbol{x}^N)\|\leq K_m(1+\|x^{i,N}\|^{q_m} + \|x^{j,N}\|^{q_m} + \|x^{k,N}\|^{q_m} + \frac{1}{N}\sum_{a=1}^N \|x^{a,N}\|^{q_m})$. It follows, using the DCT to differentiate under the integral, the triangle inequality, and these bounds, that
\begin{align}
&\| \partial_{\theta}^m\mathcal{L}^{i,j,k,N}(\theta) \| = \big\|\textstyle\int_{(\mathbb{R}^d)^N} \partial_{\theta}^m  \ell^{i,j,k,N}(\theta,\boldsymbol{x}^N) \pi_{\theta_0}^N(\mathrm{d}\boldsymbol{x}^N)\big\| \leq \int_{(\mathbb{R}^d)^N} \big\|\partial_{\theta}^m  \ell^{i,j,k,N}(\theta,\boldsymbol{x}^N)\big\| \pi_{\theta_0}^N(\mathrm{d}\boldsymbol{x}^N) \\
   & \leq K_m\textstyle\int_{(\mathbb{R}^d)^N} \big(1+\sum_{a\in\{i,j,k\}}\|x^{a,N}\|^{q_m} + \frac{1}{N}\sum_{a=1}^N \|x^{a,N}\|^{q_m}\big)\pi_{\theta_0}^N(\mathrm{d}\boldsymbol{x}^N)  \leq K_m(1+4K_{q_m}):=C_m.
\end{align}
\end{proof}

\begin{lemma}
    \label{lemma:main-theorem-lemma-1}
        Suppose that Assumption~\ref{assumption:moments}, Assumption~\ref{assumption:drift}, Assumption~\ref{assumption:drift-grad} (with $k=0,1,2$), and Assumption~\ref{assumption:learning-rate} hold. Define 
        \begin{align}
            \Gamma_{r,\eta}^{i,N} &= \int_{\tau_r}^{\sigma_{r,\eta}} \gamma_s \left( H^{i,N}(\bar{\theta}_s^{i,N}, {\boldsymbol{x}}_s^N) - \partial_{\theta}\mathcal{L}^{i,N}(\bar{\theta}_s^{i,N})\right) \mathrm{d}s \label{eq:gamma-k-1} \\
            \Gamma_{r,\eta}^{i,j,k,N} &= \int_{\tau_r}^{\sigma_{r,\eta}} \gamma_s \left( h^{i,j,k,N}(\theta_s^{i,j,k,N}, {\boldsymbol{x}}_s^N) - \partial_{\theta}\mathcal{L}^{i,j,k,N}(\theta_s^{i,j,k,N})\right) \mathrm{d}s.  \label{eq:gamma-k-2}
        \end{align}
        where $(\tau_r)_{r\geq 1}$ and $(\sigma_{r,\eta})_{r\geq 0}$ are the stopping times defined in \eqref{eq:stop-times-i-eq-1} - \eqref{eq:stop-times-i-eq-2} and \eqref{eq:stop-times-ijk-eq-1} - \eqref{eq:stop-times-ijk-eq-2}, respectively, and where $\sigma_{r,\eta} = \sigma_r + \eta$ for some $\eta>0$. Then $\|\Gamma_{r,\eta}^{i,N}\|\rightarrow 0$ and $\|\Gamma_{r,\eta}^{i,j,k,N}\|\rightarrow 0$ a.s. as $r\rightarrow\infty$.
    \end{lemma}

\begin{proof}
We will prove the result for \eqref{eq:gamma-k-2}, with the result for \eqref{eq:gamma-k-1} proved similarly. Consider the function 
\begin{equation}
S^{i,j,k,N}(\theta,{\boldsymbol{x}}^N)
=h^{i,j,k,N}(\theta,{\boldsymbol{x}}^N) - \partial_{\theta}{\mathcal{L}}^{i,j,k,N}(\theta).
\end{equation}
This function is centered with respect to the invariant measure $\pi_{\theta_0}^N$, owing to the definition of $\partial_{\theta}{\mathcal{L}}^{i,j,k,N}(\cdot)$ (see Proposition~\ref{prop:asymptotic-partial-log-lik-grad}). In addition, due to Lemma~\ref{lemma:main-theorem-lemma-2-a} (the boundedness of the asymptotic log-likelihood and its derivatives) and Corollary~\ref{cor:empirical-H-lip} (the local Lipschitz and polynomial growth of ${h}^{i,j,k,N}(\theta,{\boldsymbol{x}}^N)$ and its derivatives), for $l=0,1,2$, $\smash{\|\partial_{\theta}^l  S^{i,j,k,N}(\theta,\boldsymbol{x}^{N})  - \partial_{\theta}^l  S^{i,j,k,N}(\theta,\boldsymbol{y}^{N}) \|}$ satisfies a bound of the type given in Corollary~\ref{cor:empirical-H-lip}. Thus, the function $S^{i,j,k,N}(\theta,{\boldsymbol{x}}^N)$ satisfies the conditions of (a minor modification of) Lemma 17 in \cite{sharrock2023online} with $r=0$. It follows that, for any fixed distinct $i,j,k\in[N]$, the Poisson equation
\begin{equation}
\mathcal{A}_{{\boldsymbol{x}^N}} v^{i,j,k,N}(\theta,{\boldsymbol{x}}^N) = S^{i,j,k,N}(\theta,{\boldsymbol{x}}^N)~~~,~~~\int_{(\mathbb{R}^d)^N} v^{i,j,k,N}(\theta,{\boldsymbol{x}}^N)\pi_{\theta_0}^N(\mathrm{d}\boldsymbol{x}^N)=0 
\end{equation}
has a unique twice differentiable solution which satisfies $\smash{\sum_{l=0}^{2} \|\frac{\partial^l }{\partial \theta^{l}}v^{i,j,k,N}(\theta,{\boldsymbol{x}}^N)\| + \|\frac{\partial^2 }{\partial \theta\partial {\boldsymbol{x}}^N}v^{i,j,k,N}(\theta,{\boldsymbol{x}}^N)\| \leq}$ $\smash{K [1+\sum_{a\in\{i,j,k\}}\|x^{a,N}\|^{q} + \frac{1}{N}\sum_{a=1}^N \|x^{a,N}\|^{q}]}$, where the constant $K>0$ and the integer $q\geq 1$ are also independent of $N$. Suppose now we define $u^{i,j,k,N}(t,\theta,{\boldsymbol{x}}^N) = \gamma_tv^{i,j,k,N}(\theta,{\boldsymbol{x}}^N)$. Then, applying It\^{o}'s formula to each component of this vector-valued function, we obtain, for $m=1,\dots,p$, 
\begin{align}
&u^{i,j,k,N}_m(t_2,\theta^{i,j,k,N}_{t_2},{\boldsymbol{x}}^N_{t_2}) - u^{i,j,k,N}_m(t_1,\theta^{i,j,k,N}_{t_1},{\boldsymbol{x}}^N_{t_1}) \\
&= \int_{t_1}^{t_2}\partial_s u^{i,j,k,N}_m(s,\theta^{i,j,k,N}_{s},{\boldsymbol{x}}_s^N)\mathrm{d}s  + \int_{t_1}^{t_2}\mathcal{A}_{{\boldsymbol{x}^N}}u^{i,j,k,N}_m(s,\theta^{i,j,k,N}_{s},{\boldsymbol{x}}_s^N)\mathrm{d}s \\
&+ \int_{t_1}^{t_2}\mathcal{A}_{\theta}u^{i,j,k,N}_m(s,\theta^{i,j,k,N}_{s},{\boldsymbol{x}}_s^N)\mathrm{d}s + \int_{t_1}^{t_2}\gamma_s \mathrm{Tr}\bigg[G^{i,N}(\theta_s^{i,j,k,N},{\boldsymbol{x}}_s^N)\partial_{\theta}\partial_{{\boldsymbol{x}^N}}u^{i,j,k,N}_m(s,\theta^{i,j,k,N}_{s},{\boldsymbol{x}}_s^N)\bigg]\mathrm{d}s  \\
&+\int_{t_1}^{t_2} \langle \partial_{{\boldsymbol{x}^N}}u^{i,j,k,N}_m(s,\theta^{i,j,k,N}_{s},{\boldsymbol{x}}_s^N), \sigma\otimes I_N \mathrm{d}w_s^N \rangle \\
&+\int_{t_1}^{t_2}\gamma_s \langle \partial_{\theta}u^{i,j,k,N}_m(s,\theta^{i,j,k,N}_{s},{\boldsymbol{x}}_s^N),G^{i,N}(\theta_s^{i,j,k,N},{\boldsymbol{x}}_s^N)\sigma^{-1}\mathrm{d}w_s^{i,N}\rangle 
\end{align}
where $\mathcal{A}_{{\boldsymbol{x}^N}}$ and $\mathcal{A}_{\theta}$ are the infinitesimal generators of ${\boldsymbol{x}}^N$ and $\theta^{i,j,k,N}$. Rearranging this identity, and also recalling that $v^{i,j,k,N}(\theta,{\boldsymbol{x}}^N)$ is the solution of the Poisson equation, we obtain
\begin{align}
\hspace{-5mm} \Gamma_{r,\eta} &= \int_{\tau_r}^{\sigma_{r,\eta}}\gamma_s\mathcal{A}_{{\boldsymbol{x}^N}}v^{i,j,k,N}(\theta^{i,j,k,N}_{s},{\boldsymbol{x}}_s^N)\mathrm{d}s \\
&= \gamma_{\sigma_{r,\eta}}v^{i,j,k,N}(\theta^{i,j,k,N}_{\sigma_{r,\eta}},{\boldsymbol{x}}^N_{\sigma_{r,\eta}}) - \gamma_{\tau_r}v^{i,j,k,N}(\theta^{i,j,k,N}_{\tau_r},{\boldsymbol{x}}^N_{\tau_r})  - \int_{\tau_r}^{\sigma_{r,\eta}}\dot{\gamma}_sv^{i,j,k,N}(\theta^{i,j,k,N}_{s},{\boldsymbol{x}}_s^N)\mathrm{d}s  \hspace{-3mm} \\
& - \int_{\tau_r}^{\sigma_{r,\eta}}\gamma_s\mathcal{A}_{\theta}v^{i,j,k,N}(\theta^{i,j,k,N}_{s},{\boldsymbol{x}}_s^N)\mathrm{d}s - \int_{\tau_r}^{\sigma_{r,\eta}}\gamma_s^2 \mathrm{Tr}\big[G^{i,N}(\theta_s^{i,j,k,N},\boldsymbol{x}_s^{N})\partial_{\theta}\partial_{{\boldsymbol{x}^N}}v^{i,j,k,N}(\theta^{i,j,k,N}_{s},{\boldsymbol{x}}_s^N)\big]\mathrm{d}s \hspace{-3mm}  \\
&-\int_{\tau_r}^{\sigma_{r,\eta}}\gamma_s \langle \partial_{{\boldsymbol{x}^N}}v^{i,j,k,N}(\theta^{i,j,k,N}_{s},{\boldsymbol{x}}_s^N),\sigma \otimes I_N\mathrm{d}w_s^N\rangle \\
&-\int_{\tau_r}^{\sigma_{r,\eta}}\gamma_s^2\langle\partial_{\theta}v^{i,j,k,N}(\theta^{i,j,k,N}_{s},{\boldsymbol{x}}_s^N), G^{i,N}(\theta_{s}^{i,j,k,N},{\boldsymbol{x}}_s^N)\sigma^{-1}\mathrm{d}w^{i,N}_s \rangle 
\end{align} 
First consider $J_{t,i,j,k,N}^{(1)} = \gamma_t \|v^{i,j,k,N}(\theta_t,{\boldsymbol{x}}_t^N)\|$. We have, using the polynomial growth of $v^{i,j,k,N}(\theta,{\boldsymbol{x}}^N)$, and Theorem~\ref{thm:moment-bounds} (i.e., uniform-in-time moment bounds for the IPS), that
\begin{align}
\mathbb{E}[|J_{t,i,j,k,N}^{(1)}|^2]
\leq K \gamma_t^2\Big(1+\sum_{a\in\{i,j,k\}}\mathbb{E}[\|{x}^{a,N}_t\|^{q}]+ \frac{1}{N}\sum_{a=1}^N \mathbb{E}[\|x^{a,N}_t\|^{q}]\Big)\leq K\gamma_t^2.
\end{align}
Applying the Borel--Cantelli argument as in \cite[Appendix B]{sirignano2020stochastic}, it follows that $\smash{J_{t,i,j,k,N}^{(1)} \rightarrow 0}$ as $t\rightarrow\infty$ with probability one. We next consider the term
\begin{align}
J_{0,t,i,j,k,N}^{(2)} &= \int_{0}^t\dot{\gamma}_sv^{i,j,k,N}(\theta^{i,j,k,N}_s,{\boldsymbol{x}}^N_s)\mathrm{d}s+\int_{0}^t\gamma_s\mathcal{A}_{\theta}v^{i,j,k,N}(\theta^{i,j,k,N}_{s},{\boldsymbol{x}}_s^N)\mathrm{d}s \\
&+\int_0^{t}\gamma_s^2 \mathrm{Tr}\big[G^{i,N}(\theta_{s}^{i,N},{\boldsymbol{x}}_s^N)\partial_{\theta}\partial_{x}v^{i,j,k,N}(\theta^{i,j,k,N}_{s},{\boldsymbol{x}}_s^N)\big]\mathrm{d}s. \nonumber
\end{align}
In this case, using the growth properties of the $v^{i,j,k,N}(\theta,{\boldsymbol{x}}^N)$, Theorem~\ref{thm:moment-bounds} (i.e., uniform-in-time moment bounds for the IPS), and Assumption~\ref{assumption:learning-rate} (the properties of the learning rate), we obtain the bound
\begin{align}
\sup_{t>0}\mathbb{E}[|J_{0,t,i,j,k,N}^{(2)}|]&\leq K\int_0^{\infty} (|\dot{\gamma}_s|+\gamma_s^2)(1+\sum_{a\in\{i,j,k\}}\mathbb{E}[\|{x}^{a,N}_s\|^{q}]+\frac{1}{N}\sum_{a=1}^N \mathbb{E}[\|x^{a,N}_s\|^{q}])\mathrm{d}s<\infty. 
\end{align} 
Thus, there exists a finite random variable $\smash{J_{0,\infty,i,j,k,N}^{(2)}}$ such that, with probability one, $\smash{J_{0,t,i,j,k,N}^{(2)}\rightarrow J_{0,\infty,i,j,k,N}^{(2)}}$ as $t\rightarrow\infty$. The last term to consider is the stochastic integral
\begin{align}
J_{0,t,i,j,k,N}^{(3)} &= \int_{0}^{t}\gamma_s\partial_{{\boldsymbol{x}^N}}v^{i,j,k,N}(\theta^{i,j,k,N}_{s},{\boldsymbol{x}}_s^N)\cdot \mathrm{d}w_s^N \\
&+\int_{0}^{t}\gamma_s^2\partial_{\theta}v^{i,j,k,N}(\theta^{i,j,k,N}_{s},{\boldsymbol{x}}_s^N)\cdot G^{i,N}(\theta_{s}^{i,N},{\boldsymbol{x}}_s^N)\mathrm{d}w^{i,N}_s.
\end{align}
In this case, using the BDG inequality, and the same bounds as above, we have
\begin{align}
\mathbb{E}\left[|J_{0,t,i,j,k,N}^{(3)}|^2\right]
&\leq K\int_0^{\infty} (\gamma_s^2+\gamma_s^4) \bigg[1+\sum_{a\in\{i,j,k\}}\mathbb{E}[\|x^{a,N}_s\|^{q}]
+\frac{1}{N}\sum_{a=1}^N \mathbb{E}[\|x^{a,N}_s\|^{q}]\bigg]\mathrm{d}s \leq K\int_0^{\infty}\gamma_s^2\mathrm{d}s<\infty.
\end{align}
Thus, by Doob's martingale convergence theorem, there exists a square integrable random variable $\smash{J_{0,\infty}^{(3)}}$ such that, both almost surely and in $\mathrm{L}^2$, $\smash{J_{0,t,i,j,k,N}^{(3)}\rightarrow J_{0,\infty,i,j,k,N}^{(3)}}$ as $t\rightarrow\infty$. Combining these results, we have 
\begin{equation}
\|\Gamma_{r,\eta}\|\leq J_{\sigma_{r,\eta},i,j,k,N}^{(1)}+ J_{\tau_r,i,j,k,N}^{(1)}+J^{(2)}_{\tau_r,\sigma_{r,\eta},i,j,k,N}+J^{(3)}_{\tau_r,\sigma_{r,\eta},i,j,k,N}\stackrel{r\rightarrow\infty}{\rightarrow}0.
\end{equation}
This completes the proof of \eqref{eq:gamma-k-2}. The proof of \eqref{eq:gamma-k-1} is essentially identical, noting that all of the relevant results (e.g., polynomial growth property, solution of the Poisson equation) continue to hold when $h^{i,j,k,N}$ is replaced by $H^{i,N}$.
\end{proof}

    \begin{lemma}
    \label{lemma:main-theorem-lemma-2}
        Suppose that Assumption~\ref{assumption:moments}, Assumption~\ref{assumption:drift}, Assumption~\ref{assumption:drift-grad} (with $k=0,1,2$), and Assumption~\ref{assumption:learning-rate} hold. Let $L$ denote the Lipschitz constant of $\partial_{\theta}\mathcal{L}^{i,N}$ or $\partial_{\theta}\mathcal{L}^{i,j,k,N}$. Let $\lambda>0$ be such that, for a given $\kappa>0$, it holds that $3\lambda + \frac{\lambda}{4\kappa} = \frac{1}{2L}$. Then, for $r$ sufficiently large and $\eta>0$ sufficiently small (potentially random, and depending on $r$), it holds that
        \begin{align}
            \int_{\tau_r}^{\sigma_{r,\eta}} \gamma_s \mathrm{d}s> \lambda, \qquad \frac{\lambda}{2}\leq \int_{\tau_r}^{\sigma_r}\gamma_s\mathrm{d}s \leq \lambda ~~~\text{a.s.}
        \end{align}
        where $(\tau_r)_{r\geq 1}$ and $(\sigma_{r,\eta})_{r\geq 0}$ are the stopping times defined in either \eqref{eq:stop-times-i-eq-1} - \eqref{eq:stop-times-i-eq-2} or \eqref{eq:stop-times-ijk-eq-1} - \eqref{eq:stop-times-ijk-eq-2}, and where $\sigma_{r,\eta} = \sigma_{r}+\eta$, for some $\eta>0$.
    \end{lemma}

    \begin{proof}
    We prove the result in the case where the stopping times are defined by  \eqref{eq:stop-times-ijk-eq-1} - \eqref{eq:stop-times-ijk-eq-2}, with the other case proved in the same way. Our proof follows closely that of \cite[][Lemma 3.2]{sirignano2017stochastic}, with the appropriate modifications. We will argue by contradiction. Let us assume that $\smash{\int_{\tau_r}^{\sigma_{r,\eta}}\gamma_s\mathrm{d}s\leq \lambda}$. Choose  $\varepsilon>0$ such that $\varepsilon\leq \frac{\lambda}{8}$. Then, using the It\^{o} isometry, Corollary~\ref{cor:empirical-drift-grad-lip} (the polynomial growth of $g^{i,j,N}(\theta,\boldsymbol{x}^N)$), Theorem~\ref{thm:moment-bounds} (the bounded moments of the IPS), and Assumption~\ref{assumption:learning-rate} (the properties of the learning rate), we have that
\begin{align}
&\sup_{t\geq 0}\mathbb{E}\big\|\int_0^t \gamma_s \frac{\kappa}{\|\partial_{\theta}\mathcal{L}^{i,j,k,N}(\theta^{i,j,k,N}_{\tau_r})\|}g^{i,j,N}(\theta_s^{i,j,k,N},\boldsymbol{x}_s^N)\sigma^{-1}\mathrm{d}w_s^{i,N}\big\|^2 \\[.5mm]
&\leq \sup_{t\geq 0}\mathbb{E}\big\|\int_0^t \gamma_s g^{i,j,N}(\theta_s^{i,j,k,N},\boldsymbol{x}_s^N)\sigma^{-1}\mathrm{d}w_s^{i,N}\big\|^2 \\
&\leq \int_0^t K\gamma_s^2\Big(1+\mathbb{E}[\|{x}^{i,N}_s\|^{q}] + \mathbb{E}[\|{x}^{j,N}_s\|^{q}]+\frac{1}{N}\sum_{k=1}^N \mathbb{E}[\|x^{k,N}_s\|^{q}]\Big)\mathrm{d}s<\infty.
\end{align}
Thus, appealing to Doob's martingale convergence theorem, there exists a finite random variable $M$ such that, both almost surely and in $\mathrm{L}^2$, $\int_0^t [\cdots]\mathrm{d}w_s^{i,N}\rightarrow M$ and thus, for the chosen $\varepsilon>0$, there exists $r$ such that 
\begin{equation}
\int_{\tau_r}^{\sigma_{r,\eta}}\gamma_s \frac{\kappa}{\|\partial_{\theta}\mathcal{L}^{i,j,k,N}(\theta^{i,j,k,N}_{\tau_r})\|}g^{i,j,N}(\theta_s^{i,j,k,N},{\boldsymbol{x}_s^N})\sigma^{-1}\mathrm{d}w_s^{i,N}<\varepsilon. \label{eq391}
\end{equation}
Let us now also assume that, for the given $r$, $\eta$ is small enough such that for all $s\in[\tau_r,\sigma_{r,\eta}]$, we have $\|\partial_{\theta}\mathcal{L}^{i,j,k,N}(\theta_s^{i,j,k,N})\|\leq 3\|\partial_{\theta}\mathcal{L}^{i,j,k,N}(\theta_{\tau_r}^{i,j,k,N})\|$. We can then compute
\begin{align} 
&\|\theta^{i,j,k,N}_{\sigma_{r,\eta}}-\theta^{i,j,k,N}_{\tau_r}\| = \big\|\int_{\tau_r}^{\sigma_{r,\eta}}\gamma_sh^{i,j,k,N}(\theta_s^{i,j,k,N},{\boldsymbol{x}}_s^N)\mathrm{d}s+\int_{\tau_r}^{\sigma_{r,\eta}}\gamma_s g^{i,j,N}(\theta_s^{i,j,k,N},{\boldsymbol{x}}_s^N) \sigma^{-1}\mathrm{d}w_s^{i,N} \big\| \\
&\leq 3\|\partial_{\theta}\mathcal{L}^{i,j,k,N}(\theta_{\tau_r}^{i,N})\|\int_{\tau_r}^{\sigma_{r,\eta}}\gamma_s\mathrm{d}s + 
\big\|\int_{\tau_r}^{\sigma_{r,\eta}}\gamma_s(h^{i,j,k,N}(\theta_s^{i,j,k,N},{\boldsymbol{x}}_s^N)-\partial_{\theta}\mathcal{L}^{i,j,k,N}(\theta_s^{i,j,k,N}))\mathrm{d}s\big\| \notag  \\
&~~~+\frac{\|\partial_{\theta}\mathcal{L}^{i,j,k,N}(\theta^{i,j,k,N}_{\tau_r})\|}{\kappa}
\big\|\int_{\tau_r}^{\sigma_{r,\eta}}\gamma_s\frac{\kappa}{\|\partial_{\theta}\mathcal{L}^{i,j,k,N}(\theta^{i,j,k,N}_{\tau_r})\|} g^{i,j,N}(\theta_s^{i,j,k,N},{\boldsymbol{x}}_s^N) \sigma^{-1} \mathrm{d}w_s^{i,N}\big\|  \\[2mm]
&\leq 3\|\partial_{\theta}\mathcal{L}^{i,j,k,N}(\theta_{\tau_r}^{i,N})\|\lambda + \varepsilon + \frac{\|\partial_{\theta}\mathcal{L}^{i,j,k,N}(\theta^{i,j,k,N}_{\tau_r})\|}{\kappa}\varepsilon  \leq \|\partial_{\theta}\mathcal{L}^{i,j,k,N}(\theta_{\tau_r}^{i,N})\|\Big[{3\lambda}+\frac{\lambda}{4\kappa}\Big] \label{d105}
\end{align}
where in the penultimate line we have used Lemma \ref{lemma:main-theorem-lemma-1} and our previous bound in \eqref{eq391}, and in the final line we have used the fact that our choice of $\varepsilon$ satisfies $\varepsilon\leq \frac{\lambda}{8}$. We thus obtain
\begin{equation}
\|\theta^{i,j,k,N}_{\sigma_{r,\eta}}-\theta^{i,j,k,N}_{\tau_r}\|\leq \|\partial_{\theta}\mathcal{L}^{i,j,k,N}(\theta_{\tau_r}^{i,j,k,N})\|\Big[{3\lambda}+\frac{\lambda}{4\kappa}\Big]\leq \|\partial_{\theta}\mathcal{L}^{i,j,k,N}(\theta_{\tau_r}^{i,j,k,N})\|\frac{1}{2L}.\label{d106}
\end{equation}
It follows, using also the definition of the Lipschitz constant $L$, that
\begin{equation}
\|\partial_{\theta}\mathcal{L}^{i,j,k,N}(\theta_{\sigma_{r,\eta}}^{i,j,k,N})-\partial_{\theta}\mathcal{L}^{i,j,k,N}(\theta_{\tau_r}^{i,j,k,N})\|\leq L\|\theta^{i,j,k,N}_{\sigma_{r,\eta}}-\theta^{i,j,k,N}_{\tau_r}\|\leq \frac{1}{2}\|\partial_{\theta}\mathcal{L}^{i,j,k,N}(\theta_{\tau_r}^{i,j,k,N})\|
\end{equation}
which, in turn, yields
\begin{equation}
\frac{1}{2}\|\partial_{\theta}\mathcal{L}^{i,j,k,N}(\theta_{\tau_r}^{i,j,k,N})\|\leq \|\partial_{\theta}\mathcal{L}^{i,j,k,N}(\theta_{\sigma_{r,\eta}}^{i,j,k,N})\|\leq 2\|\partial_{\theta}\mathcal{L}^{i,j,k,N}(\theta_{\tau_r}^{i,j,k,N})\|. \label{d108}
\end{equation}
But this implies that $\sigma_{r,\eta}\in[\tau_r,\sigma_r]$, which is a contradiction, since $\sigma_{r,\eta}:=\sigma_{r}+\eta>\sigma_{r}$. Thus, we must have $\int_{\tau_r}^{\sigma_{r,\eta}}\gamma_s\mathrm{d}s>\lambda$. It remains to prove the second part of the Lemma. In fact, this is a straightforward consequence of the result just proven. By definition of the stopping times, we have that $\smash{\int_{\tau_r}^{\sigma_r}\gamma_s\mathrm{d}s\leq \lambda}$. Thus, it remains only to show that $\smash{\frac{\lambda}{2}\leq \int_{\tau_r}^{\sigma_r}\gamma_s\mathrm{d}s}$. From the first part of the Lemma, we have that $
\smash{\int_{\tau_r}^{\sigma_{r,\eta}}\gamma_s\mathrm{d}s>\lambda}$. Moreover, for $r$ sufficiently large and $\eta$ sufficiently small, we must have $\smash{\int_{\sigma_r}^{\sigma_{r,\eta}}\gamma_s\mathrm{d}s\leq \frac{\lambda}{2}}$. We thus obtain $\int_{\tau_r}^{\sigma_r}\gamma_s\mathrm{d}s\geq \lambda - \int_{\sigma_r}^{\sigma_{r,\eta}}\gamma_s\mathrm{d}s\geq \lambda - \frac{\lambda}{2} = \frac{\lambda}{2}$.
\end{proof}

\begin{lemma}
\label{lemma:main-theorem-lemma-3}
    Suppose that Assumption~\ref{assumption:moments}, Assumption~\ref{assumption:drift}, Assumption~\ref{assumption:drift-grad} (with $k=0,1,2$), and Assumption~\ref{assumption:learning-rate} hold. Suppose that $\theta_t\in\Theta$ for all $t\geq 0$ and that there are an infinite number of intervals $[\tau_r,\sigma_r)$. Then there exists a constant $\beta:=\beta(\kappa)>0$ such that, for all $r>r_0$, 
    \begin{align}
        \mathcal{L}^{i,N}(\bar\theta_{\sigma_r}^{i,N}) - \mathcal{L}^{i,N}(\bar\theta_{\tau_r}^{i,N})\leq - \beta ~~~\text{a.s.}
        \label{eq:l1-diff-1} \\
    \mathcal{L}^{i,j,k,N}(\theta_{\sigma_r}^{i,j,k,N}) - \mathcal{L}^{i,j,k,N}(\theta_{\tau_r}^{i,j,k,N})\leq -\beta ~~~\text{a.s.} \label{eq:l1-diff-2}
    \end{align}
\end{lemma}

\begin{proof}
We will prove \eqref{eq:l1-diff-2}, with \eqref{eq:l1-diff-1} proved in an identical fashion. By It\^{o}'s formula, we have that
\begin{align}
&\mathcal{L}^{i,j,k,N}(\theta^{i,j,k,N}_{\sigma_r}) - \mathcal{L}^{i,j,k,N}(\theta_{\tau_r}^{i,j,k,N}) \label{eq_d117} \\[2mm]
&= -\int_{\tau_r}^{\sigma_r}\gamma_s \|\partial_{\theta}\mathcal{L}^{i,j,k,N}(\theta_s^{i,j,k,N})\|^2\mathrm{d}s + \int_{\tau_r}^{\sigma_r}\gamma_s\langle \partial_{\theta}\mathcal{L}^{i,j,k,N}(\theta_s^{i,j,k,N}), g^{i,j,N}(\theta_s^{i,j,k,N},{\boldsymbol{x}}_s^N)\sigma^{-1}\mathrm{d}w_s^{i,N}\rangle \\ 
&~~~~~+ \int_{\tau_r}^{\sigma_r}\gamma_s\langle \partial_{\theta}\mathcal{L}^{i,j,k,N}(\theta_s^{i,j,k,N}), \partial_{\theta}\mathcal{L}^{i,j,k,N}(\theta_s^{i,j,k,N}) - h^{i,j,k,N}(\theta_s^{i,j,k,N},{\boldsymbol{x}}_s^N)\rangle \mathrm{d}s \\
&~~~~~+ \int_{\tau_r}^{\sigma_r}\frac{1}{2}\gamma_s^2\mathrm{Tr}\left[g^{i,j,N}(\theta_s^{i,j,k,N},{\boldsymbol{x}}_s^N)g^{i,j,N}(\theta_s^{i,j,k,N},{\boldsymbol{x}}_s^N)^T\partial^2_{\theta}\mathcal{L}^{i,j,k,N}(\theta_s^{i,j,k,N})\right]\mathrm{d}s  \\
&:= -A_{r,i,j,k,N}^{(1)} + A_{r,i,j,k,N}^{(2)} + A_{r,i,j,k,N}^{(3)} + A_{r,i,j,k,N}^{(4)}
\end{align}
We will deal with each of the four terms on the RHS separately. First consider $A_{r,i,j,k,N}^{(1)}$. For this term, we have that
\begin{align}
A_{r,i,j,k,N}^{(1)}
&= \int_{\tau_r}^{\sigma_r}\gamma_s\|\partial_{\theta}\mathcal{L}^{i,j,k,N}(\theta_s^{i,j,k,N})\|^2\mathrm{d}s \\
&\geq \frac{\|\partial_{\theta}\mathcal{L}^{i,j,k,N}(\theta_{\tau_r}^{i,j,k,N})\|^2}{4} \int_{\tau_r}^{\sigma_r}\gamma_s\mathrm{d}s \geq \frac{\|\partial_{\theta}\mathcal{L}^{i,j,k,N}(\theta_{\tau_r}^{i,j,k,N})\|^2}{8} \lambda   
\end{align}
where, in the first inequality, we have used the definition of the $\{\tau_r\}_{r\geq 0}$, which implies that  $\|\partial_{\theta}\mathcal{L}^{i,j,k,N}(\theta_s^{i,j,k,N})\|\geq \frac{1}{2}\|\partial_{\theta}\mathcal{L}^{i,j,k,N}(\theta_{\tau_r}^{i,j,k,N})\|$ for all $s\in[\tau_r,\sigma_r]$, and in the second inequality we have used Lemma \ref{lemma:main-theorem-lemma-2}. We next consider $\smash{A_{r,i,j,k,N}^{(2)}}$. Using It\^{o}'s isometry, Lemma~\ref{lemma:main-theorem-lemma-2-a} (the bound on the asymptotic log-likelihood of the IPS), Corollary~\ref{cor:empirical-drift-grad-lip} (the polynomial growth of $\smash{g^{i,j,N}(\theta,{\boldsymbol{x}^N})}$), 
Theorem~\ref{thm:moment-bounds} (uniform-in-time moment bounds for the IPS), and Assumption~\ref{assumption:learning-rate} (the square summability of the learning rate), we have
\begin{align}
&\sup_{t\geq 0}\mathbb{E}\Big[\Big|\int_{0}^{t}\gamma_s\langle \partial_{\theta}\mathcal{L}^{i,j,k,N}(\theta_s^{i,j,k,N}), g^{i,j,N}(\theta_s^{i,j,k,N},{\boldsymbol{x}}_s^N)\mathrm{d}w_s^{i,N}\rangle\Big|^2\Big] \\
&\leq K\mathbb{E} \int_0^{\infty}\gamma_s^2\|g^{i,j,N}(\theta_s^{i,j,k,N},{\boldsymbol{x}}_s^N)\|^2\mathrm{d}s \leq K\int_0^{\infty} \gamma_s^2 (1+\sum_{a\in\{i,j\}}\mathbb{E}\left[\|{x}_s^{a,N}\|^q\right] + \frac{1}{N}\sum_{a=1}^N \mathbb{E}\left[\|{x}_s^{a,N}\|^q\right]\mathrm{d}s<\infty.
\end{align} 
Thus, by Doob's martingale convergence theorem, there exists a finite random variable $\smash{A_{\infty,i,j,k,N}^{(2)}}$ such that, both a.s. and in $\mathrm{L}^2$, $\smash{\int_{0}^{t}[\cdots]\rightarrow A_{\infty,i,j,k,N}^{(2)}}$ as $t\rightarrow\infty$. It follows that $A_{r,i,j,k,N}^{(2)}\rightarrow 0$ a.s. as $r\rightarrow\infty$. We now consider $\smash{A_{r,i,j,k,N}^{(3)}}$. Define 
\begin{equation}
T^{i,j,k,N}(\theta,{\boldsymbol{x}}^N) = \langle \partial_{\theta}\mathcal{L}^{i,j,k,N}(\theta),h^{i,j,k,N}(\theta,{\boldsymbol{x}}^N) - \partial_{\theta}\mathcal{L}^{i,j,k,N}(\theta)\rangle.
\end{equation}
Due to Lemma~\ref{lemma:main-theorem-lemma-2-a} (the boundedness of the asymptotic log-likelihood and its derivatives) and Corollary~\ref{cor:empirical-H-lip} (the local Lipschitz and polynomial growth of ${h}^{i,j,k,N}(\theta,{\boldsymbol{x}}^N)$ and its derivatives), for $l=0,1,2$, $\smash{\|\partial_{\theta}^l  T^{i,j,k,N}(\theta,\boldsymbol{x}^{N})  - \partial_{\theta}^l  T^{i,j,k,N}(\theta,\boldsymbol{y}^{N}) \|}$ satisfies a bound of the type given in Corollary~\ref{cor:empirical-H-lip}. In addition, this function is centered w.r.t. the invariant distribution $\pi_{\theta_0}^N$. Thus, by (a minor variation on) Lemma 17 in \cite{sharrock2023online} with $r=0$, the Poisson equation
\begin{equation}
\mathcal{A}_{{\boldsymbol{x}^N}} v^{i,j,k,N}(\theta,{\boldsymbol{x}}^N) = T^{i,j,k,N}(\theta,{\boldsymbol{x}}^N)~~~,~~~\int_{(\mathbb{R}^d)^N} v^{i,j,k,N}(\theta,{\boldsymbol{x}}^N)\pi_{\theta_0}^{N}(\mathrm{d}{\boldsymbol{x}}^N)=0
\end{equation}
has a unique twice differentiable solution which satisfies $\sum_{l=0}^{2} \|\frac{\partial^l }{\partial \theta^{l}}v^{i,j,k,N}(\theta,{\boldsymbol{x}}^N)\| + \|\frac{\partial^2 }{\partial \theta\partial {\boldsymbol{x}}^N}v^{i,j,k,N}(\theta,{\boldsymbol{x}}^N)\| \leq K [1+\sum_{a\in\{i,j,k\}}\|x^{a,N}\|^{q} + \frac{1}{N}\sum_{a=1}^N \|x^{a,N}\|^{q}]$, for a constant $K>0$ and an integer $q\geq 1$ which are independent of $N$. Arguing as in Lemma \ref{lemma:main-theorem-lemma-1}, it follows that, a.s., $\smash{\| A_{r,i,j,k,N}^{(3)}\|\rightarrow 0}$ as $r\rightarrow\infty$.

Finally, we turn our attention to $A_{r,i,j,k,N}^{(4)}$. Once more using Lemma~\ref{lemma:main-theorem-lemma-2-a} (the bound on the asymptotic log-likelihood of the IPS), Corollary~\ref{cor:empirical-drift-grad-lip} (the polynomial growth of the function $\smash{g^{i,j,N}(\theta,{\boldsymbol{x}^N})}$), Theorem~\ref{thm:moment-bounds} (the uniform-in-time moment bounds for solutions of the IPS), and Assumption~\ref{assumption:learning-rate} (the square summability of the learning rate), we have that
\begin{align}
&\sup_{t\geq 0}\mathbb{E}\Big\|\int_{0}^{t}\frac{1}{2}\gamma_s^2\mathrm{Tr}\left[g^{i,j,N}(\theta_s^{i,j,k,N},{\boldsymbol{x}}_s^N)g^{i,j,N}(\theta_s^{i,j,k,N},{\boldsymbol{x}}_s^N)^T\partial^2_{\theta}\mathcal{L}^{i,j,k,N}(\theta_s^{i,j,k,N})\right]\mathrm{d}s\Big\| \\
&\leq K\int_0^{\infty} \gamma_s^2(1+\mathbb{E}\left[\|{x}_s^{i,N}\|^q\right] +\mathbb{E}\left[\|{x}_s^{j,N}\|^q\right] + \frac{1}{N}\sum_{k=1}^N \mathbb{E}\left[\|x_s^{k,N}\|^q\right])\mathrm{d}s<\infty, 
\end{align}
It follows that the random variable $\int_{0}^{\infty}[\frac{1}{2}\gamma_s^2\cdots]\mathrm{d}s$ is finite a.s., which in turn implies that there exists a finite random variable $A_{\infty,i,j,k,N}^{(4)}$ such that $\int_{0}^{\infty}[\frac{1}{2}\gamma_s^2\cdots]\mathrm{d}s\rightarrow A_{\infty,i,j,k,N}^{(4)}$ a.s. as $t\rightarrow\infty$. This implies, in particular, that $A_{r,i,j,k,N}^{(4)}=\int_{\tau_r}^{\sigma_r}\frac{1}{2}\gamma_s^2[\cdots]\mathrm{d}s\rightarrow 0$ as $r\rightarrow\infty$. Putting all of these results together, it follows that, for all $\varepsilon>0$, there exists $k$ such that 
\begin{align}
\mathcal{L}^{i,j,k,N}(\theta^{i,j,k,N}_{\sigma_r}) - \mathcal{L}^{i,j,k,N}(\theta_{\tau_r}^{i,j,k,N}) 
&\leq -A_{r,i,j,k,N}^{(1)} +\|A_{r,i,j,k,N}^{(2)}\| + \|A_{r,i,j,k,N}^{(3)}\| + \|A_{r,i,j,k,N}^{(4)}\| 
\\[1mm]
&=-\frac{\|\partial_{\theta}\mathcal{L}^{i,j,k,N}(\theta_{\tau_r}^{i,j,k,N})\|^2}{8} \lambda + 3 \varepsilon
\end{align}
The claim follows by setting $\varepsilon = \frac{\lambda(\kappa)\kappa^2}{32}$ and $\beta=\frac{\lambda(\kappa)\kappa^2}{32}$. This completes the proof of \eqref{eq:l1-diff-2}. The proof of \eqref{eq:l1-diff-1} is essentially unchanged, noting that, up to minor variations, all of the relevant results (e.g., polynomial growth property, solution of the associated Poisson equation) still hold when $g^{i,j,N}$ and $h^{i,j,k,N}$ are replaced by $G^{i,N}$ and $H^{i,N}$ (up to minor differences in the form of the polynomial growth).
\end{proof}

    \begin{lemma}
    \label{lemma:main-theorem-lemma-4}
        Suppose that Assumption~\ref{assumption:moments}, Assumption~\ref{assumption:drift}, Assumption~\ref{assumption:drift-grad} (with $k=0,1,2$), and Assumption~\ref{assumption:learning-rate} hold.  Suppose that $\theta_t\in\Theta$ for all $t\geq 0$ and that there are an infinite number of intervals $[\tau_r,\sigma_r)$. Then there exists a constant $\beta_1:=\beta_1(\kappa)>0$ satisfying $0<\beta_1<\beta$ such that, for all $r>r_0$, 
        \begin{align}
            \mathcal{L}^{i,N}(\bar\theta_{\tau_r}^{i,N}) - \mathcal{L}^{i,N}(\bar\theta_{\sigma_{r-1}}^{i,N})\leq\beta_1 ~~~\text{a.s.} \label{eq:l1-change-v2-1} \\
        \mathcal{L}^{i,j,k,N}(\theta_{\tau_r}^{i,j,k,N}) - \mathcal{L}^{i,j,k,N}(\theta_{\sigma_{r-1}}^{i,j,k,N})\leq \beta_1 ~~~\text{a.s.}\label{eq:l1-change-v2-2}
        \end{align}
    \end{lemma}

    \begin{proof}
Similar to the previous result, we will prove \eqref{eq:l1-change-v2-2}, with \eqref{eq:l1-change-v2-1} proved similarly. Using It\^{o}'s formula, and discarding the non-positive term, we have that
\begin{align}
&{\mathcal{L}}^{i,j,k,N}(\theta^{i,j,k,N}_{\tau_r}) - {\mathcal{L}}^{i,j,k,N}(\theta_{\sigma_{r-1}}^{i,j,k,N})  \\
&\leq \int_{\sigma_{r-1}}^{\tau_r}\gamma_s\langle \partial_{\theta}\mathcal{L}^{i,j,k,N}(\theta_s^{i,j,k,N}), \partial_{\theta}\mathcal{L}^{i,j,k,N}(\theta_s^{i,j,k,N}) - h^{i,j,k,N}(\theta_s^{i,j,k,N},{\boldsymbol{x}}_s^N)\rangle \mathrm{d}s \\
&+ \int_{\sigma_{r-1}}^{\tau_r}\gamma_s\langle \partial_{\theta}\mathcal{L}^{i,j,k,N}(\theta_s^{i,j,k,N}), g^{i,j,N}(\theta_s^{i,j,k,N},{\boldsymbol{x}}_s^N)\sigma^{-1}\mathrm{d}w_s^{i,N}\rangle \\
&+ \int_{\sigma_{r-1}}^{\tau_r}\frac{1}{2}\gamma_s^2\mathrm{Tr}\left[g^{i,j,N}(\theta_s^{i,j,k,N},{\boldsymbol{x}}_s^N)g^{i,j,N}(\theta_s^{i,j,k,N},{\boldsymbol{x}}_s^N)^T\partial^2_{\theta}\mathcal{L}^{i,j,k,N}(\theta_s^{i,j,k,N})\right]\mathrm{d}s
\end{align}
Arguing as in the proof of Lemma \ref{lemma:main-theorem-lemma-3}, the magnitude of each of the terms converges to zero a.s. as $r\rightarrow\infty$. This is sufficient for the conclusion.
\end{proof}

\subsection{Auxiliary Lemmas}

In this appendix, we present some auxiliary growth estimates which follow from Assumption~\ref{assumption:drift-grad}. The proofs of these results follows from basic algebraic inequalities and are thus omitted in the interest of space.

\begin{lemma}
\label{lem:weighted-lip-B}
Suppose that Assumption \ref{assumption:drift-grad} (with $k=0$) holds. Then there exists a constant $K<\infty$ and an integer $q\geq 1$, depending only on $C,m$, such that for all $\theta\in\Theta$, the following hold. For all $x,y\in\mathbb{R}^d$, and for all $\mu\in\mathcal{P}(\mathbb{R}^d)$, 
\begin{align}
    \|b(\theta,x,y)\| &\leq K\big(1+\|x\|^{q} + \|y\|^{q}\big), \qquad 
    \|B(\theta,x,\mu)\| \leq K\big(1 + \|x\|^{q} + \mu(\|\cdot\|^{q})\big). 
\end{align}
In addition, for all $x,y,w,z\in\mathbb{R}^d$, and for all $\mu,\nu\in\mathcal{P}(\mathbb{R}^d)$, 
\begin{align}
\label{eq:B-lip-0}
\|b(\theta,x,w)-b(\theta,y,z)\|
&\le K \big(\|x-y\| + \|w-z\|\big)
\big(1+\textstyle\sum_{a\in\{x,y,w,z\}} \|a\|^q \big) \\
\label{eq:B-lip}
\|B(\theta,x,\mu)-B(\theta,y,\nu)\|
&\le K \big(\|x-y\| + \mathsf{W}_2(\mu,\nu)\big) 
\big(1+\textstyle \sum_{a\in\{x,y\}} \|a\|^q + \textstyle \sum_{\eta\in\{\mu,\nu\}}\eta(\|\cdot\|^{q}) \big).
\end{align}
\end{lemma}

\begin{lemma}
\label{lem:weighted-lip-g}
Suppose that Assumption \ref{assumption:drift-grad} (with $k=1$) holds. Then there exists a constant $K<\infty$, and an integer $q\geq 1$, depending only on $C,m$, such that for all $\theta\in\Theta$, the following hold. For all $x,y\in\mathbb{R}^d$, and for all $\mu\in\mathcal{P}(\mathbb{R}^d)$, 
\begin{align}
    \label{eq:g-pgp}
    \|g(\theta,x,y)\| &\leq K\big(1+\|x\|^{q} + \|y\|^{q}\big), \qquad
    \|G(\theta,x,\mu)\| \leq K\big(1 + \|x\|^{q} + \mu(\|\cdot\|^{q})\big). 
\end{align}
In addition, for all $x,y,w,z\in\mathbb{R}^d$, and for all $\mu,\nu\in\mathcal{P}(\mathbb{R}^d)$, 
\begin{align}
\label{eq:G-lip}
\|g(\theta,x,w)-g(\theta,y,z)\|
&\le K \big(\|x-y\| + \|w-z\|\big)
\big(1+\textstyle\sum_{a\in\{x,y,w,z\}} \|a\|^q \big) \\
\label{eq:g-lip}
\|G(\theta,x,\mu)-G(\theta,y,\nu)\|
&\le K \big(\|x-y\| + \mathsf{W}_2(\mu,\nu)\big)  
\big(1+\textstyle\sum_{a\in\{x,y\}} \|a\|^q + \textstyle\sum_{\eta\in\{\mu,\nu\}}\eta(\|\cdot\|^{q}) \big).
\end{align}
\end{lemma}

\begin{lemma}
\label{lem:Phi-L1-stability}
Suppose that Assumption \ref{assumption:drift-grad} (with $k=0$) holds. Then there exists a constant $K<\infty$, and an integer $q\geq 1$, depending only on $C,m,\sigma$, such that for all $\theta\in\Theta$, the following hold. For all $x,y,z\in\mathbb{R}^d$, and for all $\mu\in\mathcal{P}(\mathbb{R}^d)$, 
\begin{align}
    | \ell(\theta,x,y,z,\mu) | &\leq K\big(1+\|x\|^{q} + \|y\|^q + \|z \|^q + \mu(\|\cdot\|^q)\big), \qquad 
    | L(\theta,x,\mu) |  \leq K\big(1+\|x\|^{q} + \mu(\|\cdot\|^q)\big). \label{eq:L-pgp}
\end{align}
In addition, for all $x,y,w,z,v,s\in\mathbb R^d$, and for all $\mu,\nu\in\mathcal P(\mathbb R^d)$,
\begin{align}
\label{eq:l-stab-pointwise}
|\ell(\theta,x,w,v,\mu)-\ell(\theta,y,z,s,\nu)| 
&\le K\big(\textstyle\sum_{(a,b)\in\{(x,y),(w,z),(v,s)\}}\|a-b\| + \mathsf{W}_2(\mu,\nu)\big)\\
&~~~\times\big(1+\textstyle\sum_{a\in\{x,y,w,z,v,s\}} \|a\|^q+ \sum_{\eta\in\{\mu,\nu\}} \eta(\|\cdot\|^q)  \big)  \nonumber \\[2mm]
\label{eq:L-stab-pointwise}
|L(\theta,x,\mu)-L(\theta,y,\nu)| 
&\le K\big(\|x-y\|+\mathsf{W}_2(\mu,\nu)\big)  \\
&~~~\times\big(1+\textstyle\sum_{a\in\{x,y\}} \|a\|^q + \textstyle\sum_{\eta\in\{\mu,\nu\}}\eta(\|\cdot\|^{q}) \big). \nonumber
\end{align}
\end{lemma}

\begin{lemma}
\label{lem:Phi-hH-stability}
Suppose that Assumption \ref{assumption:drift-grad} (with $k=0,1,2,3$) holds. Then there exists a constant $K<\infty$, and an integer $q\geq 1$, depending only on $C,m,\sigma$, such that for all $\theta\in\Theta$, and for $l=0,1,2$, the following hold. For all $x,y,z\in\mathbb{R}^d$, and for all $\mu\in\mathcal{P}(\mathbb{R}^d)$, 
\begin{align}
    \big\| \partial_{\theta}^lh(\theta,x,y,z,\mu) \big\| &\leq K\big(1+\|x\|^{q} + \|y\|^q + \|z \|^q + \mu(\|\cdot\|^q)\big), \quad 
    \big\| \partial_{\theta}^lH(\theta,x,\mu) \big\|  \leq K\big(1+\|x\|^{q} + \mu(\|\cdot\|^q)\big) \label{eq:H-pgp}
\end{align}
In addition, for all $x,y,w,z,v,s\in\mathbb R^d$, and for all $\mu,\nu\in\mathcal P(\mathbb R^d)$,
\begin{align}
\label{eq:h-stab-pointwise}
\|\partial_{\theta}^l h(\theta,x,w,v,\mu)-\partial_{\theta}^l h(\theta,y,z,s,\nu)\|
&\le K\big(\textstyle\sum_{(a,b)\in\{(x,y),(w,z),(v,s)\}}\|a-b\| + \mathsf{W}_2(\mu,\nu)\big)\\
&~~~~\times\big(1+\textstyle\sum_{a\in\{x,y,w,z,v,s\}} \|a\|^q+ \sum_{\eta\in\{\mu,\nu\}} \eta(\|\cdot\|^q)  \big) \nonumber \\[2mm]
\label{eq:H-stab-pointwise}
\|\partial_{\theta}^l H(\theta,x,\mu)-\partial_{\theta}^lH(\theta,y,\nu)\|
&\le K\big(\|x-y\|+\mathsf{W}_2(\mu,\nu)\big)  \\
&~~~~\times\big(1+\textstyle\sum_{a\in\{x,y\}} \|a\|^q + \textstyle\sum_{\eta\in\{\mu,\nu\}}\eta(\|\cdot\|^{q}) \big). \nonumber
\end{align}
\end{lemma}

\begin{corollary}
\label{cor:empirical-drift-lip}
Suppose that Assumption~\ref{assumption:drift-grad} (with $k=0$) holds. Then there exists a constant $K<\infty$, and an integer $q\geq 1$, depending only on $C,m$, such that for all
$\theta\in\Theta$, all $N\in\mathbb N$, all $i,j\in[N]$, the following hold. For all $\mathbf x^N\in(\mathbb{R}^d)^N$, 
\small
\begin{align}
    \|b^{i,j,N}(\theta,\boldsymbol{x}^N)\| &\leq K\big(1+\|x^{i,N}\|^{m+1}+\|x^{j,N}\|^{q}\big) \\ 
    \|B^{i,N}(\theta,\boldsymbol{x}^N)\| &\leq K\big(1 + \|x^{i,N}\|^{q} + \textstyle\frac{1}{N}\textstyle\sum_{j=1}^N \|x^{j,N}\|^{q}\big) \label{eq:b-pgp-2}
\end{align}
\normalsize
In addition, for all $\mathbf x^N,\boldsymbol{y}^N\in(\mathbb{R}^d)^N$,
\small
\begin{align}
\big\|b^{i,j,N}(\theta,\boldsymbol{x}^{N})  - b^{i,j,N}(\theta,\boldsymbol{y}^{N}) \big\| &\leq K\big(\|x^{i,N} - y^{i,N}\| + \|x^{j,N} - y^{j,N}\|\big) \label{eq:B-bound-1} \\ 
&~~~~\times \big(1+\|x^{i,N}\|^q + \|y^{i,N}\|^q + \|x^{j,N}\|^q + \|y^{j,N}\|^q\big) \nonumber \\[1mm]
\big\|B^{i,N}(\theta,\mathbf x^N)-B^{i,N}(\theta,\mathbf y^N)\big\|
&\le K \big(\|x^{i,N}-y^{i,N}\|+\big(\textstyle\frac1N\textstyle\sum_{j=1}^N\|x^{j,N}-y^{j,N}\|^2\big)^{\frac{1}{2}}\big) \label{eq:B-bound-2} \\
&~~~~\times\big(1+\|x^{i,N}\|^q+\|y^{i,N}\|^q+\textstyle\frac{1}{N}\sum_{j=1}^N \|x^{j,N}\|^{q} + \textstyle\frac{1}{N}\sum_{j=1}^N \|y^{j,N}\|^{q}\big). \nonumber
\end{align}
\normalsize
\end{corollary}

\begin{corollary}
\label{cor:empirical-drift-grad-lip}
Suppose that Assumption~\ref{assumption:drift-grad} (with $k=1$) holds. Then there exists a constant $K<\infty$ and an integer $q\geq 1$, depending only on $C,m$, such that for all
$\theta\in\Theta$, all $N\in\mathbb N$, all $i,j\in[N]$, the following hold. For all $\boldsymbol{x}^N\in(\mathbb{R}^d)^N$,
\small
\begin{align}
    \|g^{i,j,N}(\theta,\boldsymbol{x}^N)\| &\leq K\big(1+\|x^{i,N}\|^{q}+\|x^{j,N}\|^{q}\big) \\
    \|G^{i,N}(\theta,\boldsymbol{x}^N)\| &\leq K\big(1 + \|x^{i,N}\|^{q} + \textstyle\frac{1}{N}\sum_{j=1}^N \|x^{j,N}\|^{q}\big). 
\end{align}
\normalsize
In addition, for all $\boldsymbol{x}^N, \boldsymbol{y}^N \in(\mathbb{R}^d)^N$,
\small
\begin{align}
\big\|g^{i,j,N}(\theta,\boldsymbol{x}^{N})  - g^{i,j,N}(\theta,\boldsymbol{y}^{N}) \big\| &\leq K\big(\|x^{i,N} - y^{i,N}\| + \|x^{j,N} - y^{j,N}\|\big) \label{eq:g-bound-1}  \\ 
&~~~~\times \big(1+\|x^{i,N}\|^q + \|y^{i,N}\|^q + \|x^{j,N}\|^q + \|y^{j,N}\|^q\big) \nonumber \\[1mm]
\big\|G^{i,N}(\theta,\mathbf x^N)-G^{i,N}(\theta,\mathbf y^N)\big\|
&\le K \big(\|x^{i,N}-y^{i,N}\|+\big(\textstyle\frac1N\sum_{j=1}^N\|x^{j,N}-y^{j,N}\|^2\big)^{\frac{1}{2}}\big) \label{eq:G-bound-2} \\
&~~~~\times \big(1+\|x^{i,N}\|^q+\|y^{i,N}\|^q+\textstyle\frac{1}{N}\sum_{j=1}^N \|x^{j,N}\|^{q} + \textstyle\frac{1}{N}\sum_{j=1}^N \|y^{j,N}\|^{q}\big). \nonumber
\end{align}
\normalsize
\end{corollary}

\begin{corollary}
\label{cor:empirical-L-lip}
Suppose that Assumption~\ref{assumption:drift-grad} (with $k=0$) holds. Then there exists a constant $K<\infty$ and an integer $q\geq 1$, depending only on $C,m$, such that for all
$\theta\in\Theta$, all $N\in\mathbb N$, all $i,j,k\in[N]$, the following hold. For all $\boldsymbol{x}^N\in(\mathbb{R}^d)^N$,   
\small
\begin{align}
    |\ell^{i,j,k,N}(\theta,\boldsymbol{x}^N)| &\leq K\big(1+\textstyle\sum_{a\in\{i,j,k\}}\|x^{a,N}\|^{q}+ \textstyle\frac{1}{N}\sum_{a=1}^N\|x^{a,N}\|^{q}\big) \label{eq:l-pgp-1} \\
    |L^{i,N}(\theta,\boldsymbol{x}^N)| &\leq K\big(1 + \|x^{i,N}\|^{q} + \textstyle\frac{1}{N}\sum_{j=1}^N \|x^{j,N}\|^{q}\big) \label{eq:L-pgp-2}
\end{align}
\normalsize
In addition, for all $\boldsymbol{x}^N, \boldsymbol{y}^N \in(\mathbb{R}^d)^N$,
\small
\begin{align}
\big|\ell^{i,j,k,N}(\theta,\boldsymbol{x}^{N})  - \ell^{i,j,k,N}(\theta,\boldsymbol{y}^{N}) \big| &\leq K\big(\textstyle\sum_{a\in\{i,j,k\}}\|x^{a,N} - y^{a,N}\| + \big(\textstyle\frac1N\sum_{a=1}^N\|x^{a,N}-y^{a,N}\|^2\big)^{\frac{1}{2}}\big) \label{eq:l-bound} \hspace{-10mm} \\ 
&~~~~\times\big(1+\textstyle\sum_{a\in\{i,j,k\}}\left(\|x^{a,N}\|^q +\|y^{a,N}\|^q\right) +   \textstyle\frac{1}{N}\sum_{a=1}^N\left( \|x^{a,N}\|^{q} + \|y^{a,N}\|^{q}\right)\big) \notag \\[1mm]
\big|L^{i,N}(\theta,\mathbf x^N)-L^{i,N}(\theta,\mathbf y^N)\big| &\le K \big(\|x^{i,N}-y^{i,N}\|+\big(\textstyle\frac1N\sum_{j=1}^N\|x^{j,N}-y^{j,N}\|^2\big)^{\frac{1}{2}}\big) \label{eq:L-bound} \\
&~~~~\times\big(1+\|x^{i,N}\|^q+\|y^{i,N}\|^q+\textstyle\frac{1}{N}\sum_{j=1}^N\left( \|x^{j,N}\|^{q} +  \|y^{j,N}\|^{q}\right)\big).  \notag 
\end{align}
\normalsize
\end{corollary}

\begin{corollary}
\label{cor:empirical-H-lip}
Suppose that Assumption~\ref{assumption:drift-grad} (with $k=0,1,2,3$) holds. Then there exists a constant $K<\infty$ and an integer $q\geq 1$, depending only on $C,m$, such that for all
$\theta\in\Theta$, all $N\in\mathbb N$, all $i,j,k\in[N]$, and $l=0,1,2$, the following hold. For all $\boldsymbol{x}^N\in(\mathbb{R}^d)^N$,  
\small
\begin{align}
    \|\partial_{\theta}^l h^{i,j,k,N}(\theta,\boldsymbol{x}^N)\| &\leq K\big(1+\textstyle \sum_{a\in\{i,j,k\}}\|x^{a,N}\|^{q}+ \textstyle \frac{1}{N}\sum_{a=1}^N\|x^{a,N}\|^{q}\big) \label{eq:h-pgp-1} \\
    \| \partial_{\theta}^l  H^{i,N}(\theta,\boldsymbol{x}^N)\| &\leq K\big(1 + \|x^{i,N}\|^{q} + \textstyle \frac{1}{N}\sum_{j=1}^N \|x^{j,N}\|^{q}\big) \label{eq:H-pgp-2}
\end{align}
\normalsize
In addition, for all $\boldsymbol{x}^N, \boldsymbol{y}^N \in(\mathbb{R}^d)^N$,
\small
\begin{align}
\big\|\partial_{\theta}^l  h^{i,j,k,N}(\theta,\boldsymbol{x}^{N})  - \partial_{\theta}^l  h^{i,j,k,N}(\theta,\boldsymbol{y}^{N}) \big\| &\leq K\big(\textstyle \sum_{a\in\{i,j,k\}}\|x^{a,N} - y^{a,N}\| + \big(\textstyle \frac1N\sum_{a=1}^N\|x^{a,N}-y^{a,N}\|^2\big)^{\frac{1}{2}}\big)  \hspace{-10mm} \\  
&~~~~\times\big(1+\textstyle \sum_{a\in\{i,j,k\}}\left(\|x^{a,N}\|^q +\|y^{a,N}\|^q\right) +   \textstyle \frac{1}{N}\sum_{a=1}^N\left( \|x^{a,N}\|^{q} + \|y^{a,N}\|^{q}\right)\big) \notag \\[1mm]
\big\|\partial_{\theta}^l  H^{i,N}(\theta,\mathbf x^N)-\partial_{\theta}^l  H^{i,N}(\theta,\mathbf y^N)\big\| &\le K \big(\|x^{i,N}-y^{i,N}\|+\big(\textstyle \frac1N\sum_{j=1}^N\|x^{j,N}-y^{j,N}\|^2\big)^{\frac{1}{2}}\big) \label{eq:H-bound} \\
&~~~~\times\big(1+\|x^{i,N}\|^q+\|y^{i,N}\|^q+\textstyle \frac{1}{N}\sum_{j=1}^N\left( \|x^{j,N}\|^{q} +  \|y^{j,N}\|^{q}\right)\big).\notag 
\end{align} 
\normalsize
\end{corollary}

\end{document}